\documentclass[11pt,a4paper]{report}

\usepackage{amsmath, amssymb, amsthm}
\usepackage{graphicx,color}
\usepackage[left=1.5in, right=1in, top=1in, bottom=1in, includefoot, headheight=13.6pt]{geometry}
\usepackage[square, comma, numbers, sort&compress]{natbib}
\usepackage[T1]{fontenc}

\usepackage{mathpazo}
\usepackage[hang, small, bf, margin=20pt, tableposition=top]{caption} 
\usepackage[cmtip,all]{xy}

\numberwithin{equation}{section}

\usepackage{fancyhdr}
\makeatletter
\renewcommand{\chaptermark}[1]{\markboth{\textsc{\@chapapp}\ \thechapter:\ #1}{}}
\makeatother

\usepackage{sectsty}
\chapterfont{\large\sc\centering}
\chaptertitlefont{\LARGE\centering}
\subsubsectionfont{\centering}

\newtheorem{theorem}{Theorem}[section]
\newtheorem{lemma}[theorem]{Lemma}
\newtheorem{corollary}[theorem]{Corollary}

\theoremstyle{definition}
\newtheorem{definition}[theorem]{Definition}
\newtheorem{example}[theorem]{Example}
\newtheorem{remark}[theorem]{Remark}

\def\notation{\\ \vspace{1mm}\\{\bf{Notation}}\\ \vspace{-3mm}\\}

\usepackage[figure,table]{hypcap}
\hypersetup{
	bookmarksnumbered,
	pdfstartview={FitH},
	citecolor={black},
	linkcolor={black},
	urlcolor={black},
	pdfpagemode={UseOutlines}
}
\makeatletter
\newcommand\org@hypertarget{}
\let\org@hypertarget\hypertarget
\renewcommand\hypertarget[2]{%
  \Hy@raisedlink{\org@hypertarget{#1}{}}#2%
} 
\makeatother 

\begin{document}
	
	\title{
		\huge{\textbf{Uniform Algebras Over Complete\\
		Valued Fields}} \\[1.2cm]
		\Large{Jonathan W. Mason, MMath.} \\[1.2cm]
		\Large{Thesis submitted to The University of Nottingham \\ 
		for the degree of Doctor of Philosophy} \\ \vspace{1cm}
		\Large{March 2012} 
	} 
	\author{} \date{}
	\pdfbookmark[0]{Titlepage}{title}
	\maketitle
	
 	\newpage \vspace*{8cm} 
	\pdfbookmark[0]{Dedication}{dedication}
	\begin{center} 
		\large For Olesya \\
		\emph{Follow the Romany patteran\\
		West to the sinking sun,\\
		Till the junk-sails lift through the houseless drift.\\
		And the east and west are one.\footnote{From Rudyard Kipling's poem The Gipsy Trail.}}
 	\end{center}

	\newpage 
	\pdfbookmark[0]{Abstract}{abstract}
	\chapter*{Abstract}
	\textsc{Uniform} algebras have been extensively investigated because of their importance in the theory of uniform approximation and as examples of complex Banach algebras. An interesting question is whether analogous algebras exist when a complete valued field other than the complex numbers is used as the underlying field of the algebra. In the Archimedean setting, this generalisation is given by the theory of real function algebras introduced by S. H. Kulkarni and B. V. Limaye in the 1980s. This thesis establishes a broader theory accommodating any complete valued field as the underlying field by involving Galois automorphisms and using non-Archimedean analysis. The approach taken keeps close to the original definitions from the Archimedean setting.
\vspace{3mm}\\
Basic function algebras are defined and generalise real function algebras to all complete valued fields whilst retaining the obligatory properties of uniform algebras.\\
Several examples are provided. A basic function algebra is constructed in the non-Archimedean setting on a $p$-adic ball such that the only globally analytic elements of the algebra are constants.\\
Each basic function algebra is shown to have a lattice of basic extensions related to the field structure. In the non-Archimedean setting it is shown that certain basic function algebras have residue algebras that are also basic function algebras.\\
A representation theorem is established. Commutative unital Banach $F$-algebras with square preserving norm and finite basic dimension are shown to be isometrically $F$-isomorphic to some subalgebra of a Basic function algebra. The condition of finite basic dimension is always satisfied in the Archimedean setting by the Gel'fand-Mazur Theorem. The spectrum of an element is considered.\\
The theory of non-commutative real function algebras was established by K. Jarosz in 2008. The possibility of their generalisation to the non-Archimedean setting is established in this thesis and also appeared in a paper by J. W. Mason in 2011.\\
In the context of complex uniform algebras, a new proof is given using transfinite induction of the Feinstein-Heath Swiss cheese ``Classicalisation'' theorem. This new proof also appeared in a paper by J. W. Mason in 2010.

 	\chapter*{Acknowledgements}
	\pdfbookmark[0]{Acknowledgements}{acknowledgements}
	I would particular like to thank my supervisor J. F. Feinstein for his guidance and enthusiasm over the last four years. Through his expert knowledge of Banach algebra theory he has helped me to identify several productive lines of research whilst always allowing me the freedom required to make the work my own.\\
	In addition to my supervisor, I. B. Fesenko also positively influenced the direction of my research. During my doctoral training I undertook a postgraduate training module on the theory of local fields given by I. B. Fesenko. With extra reading, this enabled me to work both in the Archimedean and non-Archimedean settings as implicitly suggested by my thesis title.\\
	It was a pleasure to know my friends in the algebra and analysis group at Nottingham and I thank them for their interest in my work and hospitality.\\
	I am grateful to the School of Mathematical sciences at the University of Nottingham for providing funds in support of my conference participation and visits.\\
	Similarly I appreciate the support given to me by the EPSRC through a Doctoral Training Grant.\\
	This PhD thesis was examined by A. G. O'Farrell and J. Zacharias who I thank for their time and interest in my work.
	
	\tableofcontents 
	\newpage	
	\pagenumbering{arabic}
 	
	\chapter[Introduction]{Introduction}
\label{cha:IN}
This short chapter provides an informal overview of the material in this thesis. Justification of the statements made in this chapter can therefore be found in the main body of the thesis which starts at Chapter \ref{cha:CVF}.
\section[Background and Overview]{Background and Overview}
\label{sec:INBO}
Complex uniform algebras have been extensively investigated because of their importance in the theory of uniform approximation and as examples of complex Banach algebras. Let $C_{\mathbb{C}}(X)$ denote the complex Banach algebra of all continuous complex-valued functions defined on a compact Hausdorff space $X$. A complex uniform algebra $A$ is a subalgebra of $C_{\mathbb{C}}(X)$ that is complete with respect to the sup norm, contains the constant functions making it a unital complex Banach algebra and separates the points of X in the sense that for all $x_{1},x_{2}\in X$ with $x_{1}\not=x_{2}$ there is $f\in A$ satisfying $f(x_{1})\not=f(x_{2})$. Attempting to generalise this definition to other complete valued fields simply by replacing $\mathbb{C}$ with some other complete valued field $L$ produces very limited results. This is because the various versions of the Stone-Weierstrass theorem restricts our attention to $C_{L}(X)$ in this case.\\
However the theory of real function algebras introduced by S. H. Kulkarni and B. V. Limaye in the 1980s does provide an interesting generalisation of complex uniform algebras. One important departure in the definition of these algebras from that of complex uniform algebras is that they are real Banach algebras of continuous complex-valued functions. Similarly the elements of the algebras introduced in this thesis are also continuous functions that take values in some complete valued field or division ring extending the field of scalars over which the algebra is a vector space.\\
A prominent aspect of the emerging theory is that it has a lot to do with representation. As a very simple example the field of complex numbers itself is isometrically isomorphic to a real function algebra, all be it on a two point space.\\
When considering the generalisation of complex uniform algebras over all complete valued fields I naturally wanted the complex uniform algebras and real function algebras to appear directly as instances of the new theory. This resulted in the definition of basic function algebras involving the use of a Galois automorphism and homeomorphic endofunction that interact in a useful way, see Definition \ref{def:CGBFA}. In retrospect these particular algebras should more appropriately be referred to as cyclic basic function algebras since the functions involved take values in some cyclic extension of the underlying field of scalars of the algebra.\\
Necessarily this thesis starts by surveying complete valued fields and their properties. The transition from the Archimedean setting to the non-Archimedean setting preserves in places several of the nice properties that complete Archimedean fields have. However all complete non-Archimedean fields are totally disconnected, some of them are not locally compact and there is no non-Archimedean analog of the Gel'fand-Mazur Theorem.\\
On the other hand some complete non-Archimedean fields have interesting properties that only appear in the non-Archimedean setting. Consider for example the closed unit disc of the complex plane. It is closed under multiplication but not with respect to addition. In the non-Archimedean setting the closed unit ball $\mathcal{O}_{F}$, of a complete valued field $F$, is a ring since in this case the valuation involved observes the strong version of the triangle inequality, see Definition \ref{def:CVFV}. The set $\mathcal{M}_{F}=\{a\in F:|a|_{F}<1\}$ is a maximal ideal of $\mathcal{O}_{F}$ from which the residue field $\overline{F}=\mathcal{O}_{F}/\mathcal{M}_{F}$ is obtained. The residue field is of great importance in the study of such fields.\\
Similarly in the non-Archimedean setting we will see that certain basic function algebras have residue algebras that are also basic function algebras. In the process of proving this result an interesting fact is shown concerning a large class of complete non-Archimedean fields. For such a field $F$ and every finite extension $L$ of $F$, extending $F$ as a valued field, it is shown that for each Galois automorphism $g\in\mbox{Gal}(^{L}/_{F})$ there exists a set $\mathcal{R}_{L,g}\subseteq\mathcal{O}_{L}$ of residue class representatives such that the restriction of $g$ to $\mathcal{R}_{L,g}$ is an endofunction, i.e. a self map, on $\mathcal{R}_{L,g}$. This fact is probably known to certain number theorists.\\
This thesis also includes several examples of basic function algebras and these are considered at depth. A new proof of an existing theorem in the setting of complex uniform algebras is given and theory in the non-commutative setting is also considered.\\
With respect to commutative Banach algebra theory, Chapter \ref{cha:RT} presents an interesting new Gel'fand representation result extending those of the Archimedean setting. In particular we have the following theorem where the condition of finite basic dimension is automatically satisfied in the Archimedean setting and compensates for the lack of a Gel'fand-Mazur Theorem in the non-Archimedean setting. See Chapter \ref{cha:RT} for full details.
\begin{theorem}
\label{thr:INREPCRLF}
Let $F$ be a locally compact complete valued field with nontrivial valuation. Let $A$ be a commutative unital Banach $F$-algebra with $\|a^{2}\|_{A}=\|a\|_{A}^{2}$ for all $a\in A$ and finite basic dimension. Then:
\begin{enumerate}
\item[(i)]
if $F$ is the field of complex numbers then $A$ is isometrically $F$-isomorphic to a complex uniform algebra on some compact Hausdorff space $X$;
\item[(ii)]
if $F$ is the field of real numbers then $A$ is isometrically $F$-isomorphic to a real function algebra on some compact Hausdorff space $X$;
\item[(iii)]
if $F$ is non-Archimedean then $A$ is isometrically $F$-isomorphic to a non-Archimedean analog of the real function algebras on some Stone space $X$ where by a Stone space we mean a totally disconnected compact Hausdorff space.
\end{enumerate}
In particular $A$ is isometrically $F$-isomorphic to some subalgebra $\hat{A}$ of a basic function algebra and $\hat{A}$ separates the points of $X$.
\end{theorem}
Note that (i) and (ii) of Theorem \ref{thr:INREPCRLF} are the well known results from the Archimedean setting. This brings us to the following summary.
\section[Summary]{Summary}
\label{sec:INSU}
\begin{enumerate}
\item[Chapter 2:]
The relevant background concerning complete valued fields is provided. Several examples are given and the topological properties of complete valued fields are compared and discussed. A particularly useful and well known way of expressing the extension of a valuation is considered and the relevant Galois theory is introduced.
\item[Chapter 3:]
Some background concerning functional analysis over complete valued fields is given. Analytic functions are discussed. Banach $F$-algebras are introduced and the spectrum of an element is considered.
\item[Chapter 4:]
Complex uniform algebras are introduced. In the context of complex uniform algebras, a new proof is given using transfinite induction of the Feinstein-Heath Swiss cheese ``Classicalisation'' theorem. This new proof also appeared in a paper by J. W. Mason in 2010. This is followed by a preliminary discussion concerning non-complex analogs of uniform algebras. Real function algebras are introduced.
\item[Chapter 5:]
Basic function algebras are defined providing the required generalisation of real function algebras to all complete valued fields. A generalisation theorem proves that Basic function algebras have the obligatory properties of uniform algebras.\\
Several examples are provided. Complex uniform algebras and real function algebras now appear as instances of the new theory. A basic function algebra is constructed in the non-Archimedean setting on a $p$-adic ball such that the only globally analytic elements of the algebra are constants.\\
Each basic function algebra is shown to have a lattice of basic extensions related to the field structure. Further, in the non-Archimedean setting it is shown that certain basic function algebras have residue algebras that are also basic function algebras. To prove this each Galois automorphism, for certain field extensions, is shown to restrict to an endofunction on some set of residue class representatives.
\item[Chapter 6:]
A representation theorem is established in the context of locally compact complete fields with nontrivial valuation. For such a field $F$, commutative unital Banach $F$-algebras with square preserving norm and finite basic dimension are shown to be isometrically $F$-isomorphic to some subalgebra of a Basic function algebra.  The condition of finite basic dimension is automatically satisfied in the Archimedean setting by the Gel'fand-Mazur Theorem.
\item[Chapter 7:]
The theory of non-commutative real function algebras was established by K. Jarosz in 2008. The possibility of their generalisation to the non-Archimedean setting is established in this thesis having been originally pointed out in a paper by J. W. Mason in 2011. The thesis concludes with a list of open questions highlighting the potential for further interesting developments of this theory.
\end{enumerate}
	\chapter[Complete valued fields]{Complete valued fields}
\label{cha:CVF}
In this chapter we survey some of the basic facts and definitions concerning complete valued fields. Whilst also providing a background, most of the material presented here is required by later chapters and has been selected accordingly. 
\section{Introduction}
\label{sec:CVFINT}
We begin with some definitions.
\begin{definition}
We adopt the following terminology:
\begin{enumerate}
\item[(i)]
Let $F$ be a field. We will call a multiplicative norm $|\cdot|_{F}:F\rightarrow \mathbb{R}$ a {\em valuation} on $F$ and $F$ together with $|\cdot|_{F}$ a {\em valued field}.
\item[(ii)]
Let $F$ be a valued field. If the valuation on $F$ satisfies the strong triangle inequality,
\begin{equation*}
|a-b|_{F}\leq\mbox{max}(|a|_{F},|b|_{F})\mbox{ for all }a,b\in F,
\end{equation*}
then we call $|\cdot|_{F}$ a {\em non-Archimedean valuation} and $F$ a {\em non-Archimedean field}. Else we call $|\cdot|_{F}$ an {\em Archimedean valuation} and $F$ an {\em Archimedean field}.
\item[(iii)]
If a valued field is complete with respect to the metric obtained from its valuation then we call it a {\em complete valued field}. Similarly we have {\em complete valuation} and {\em complete non-Archimedean field} etc.
\item[(iv)]
More generally, a metric space $(X,d)$ is called an {\em ultrametric space} if the metric $d$ satisfies the strong triangle inequality,
\begin{equation*}
d(x,z)\leq\mbox{max}(d(x,y),d(y,z))\mbox{ for all }x,y,z\in X.
\end{equation*}
\end{enumerate}
\label{def:CVFV}
\end{definition}
The following theorem is a characterisation of non-Archimedean fields, courtesy of \cite[p18]{Schikhof}.
\begin{theorem}
Let $F$ be a valued field. Then $F$ is a non-Archimedean field if and only if $|2|_{F}\leq1$.
\label{thr:CVFCHAR}
\end{theorem}
\begin{remark}
Whilst it is clear from the definition of the strong triangle inequality that an Archimedean field can't be extended as a valued field to a non-Archimedean field, Theorem \ref{thr:CVFCHAR} also shows that a non-Archimedean field can't be extended to an Archimedean field.
\label{rem:CVFE}
\end{remark}
\begin{theorem}
Let $F$ be a valued field. Let $\mathfrak{C}$, with pointwise operations, be the ring of Cauchy sequences of elements of $F$ and let $\mathfrak{N}$ denote its ideal of null sequences. Then the completion $\mathfrak{C}/\mathfrak{N}$ of $F$ with the function
\begin{equation*}
|(a_{n})+\mathfrak{N}|_{\mathfrak{C}/\mathfrak{N}}:=\lim_{n\to\infty}|a_{n}|_{F},
\end{equation*}
for $(a_{n})+\mathfrak{N}\in\mathfrak{C}/\mathfrak{N}$, is a complete valued field extending $F$ as a valued field.
\label{thr:CVFCOM}
\end{theorem}
\begin{proof}
We will only highlight one important part of the proof since further details can be found in \cite[p80]{McCarthy}. We first note that since the valuation $|\cdot|_{F}$ is multiplicative we have $|a^{-1}|_{F}=|a|^{-1}_{F}$ for all units $a\in F^{\times}$. Let $(a_{n})$ be a Cauchy sequence taking values in $F$ but not a null sequence. Then there exists $\delta>0$ and $N\in\mathbb{N}$ such that for all $n>N$ we have $|a_{n}|_{F}>\delta$. If $(a_{n})$ takes the value $0$ then a null sequence can be added to $(a_{n})$ such that the resulting sequence $(b_{n})$ does not takes the value $0$ and $(b_{n})$ agrees with $(a_{n})$ for all $n>N$. Hence for all $m>N$ and $n>N$ we have
\begin{equation*}
|b^{-1}_{m}-b^{-1}_{n}|_{F}=|b^{-1}_{m}|_{F}|b^{-1}_{n}|_{F}|b_{n}-b_{m}|_{F}<\frac{1}{\delta^{2}}|b_{n}-b_{m}|_{F}
\end{equation*}
and so the sequence $(b^{-1}_{n})$ is also a Cauchy sequence. This shows that the ideal of null sequences $\mathfrak{N}$ is maximal and $\mathfrak{C}/\mathfrak{N}$ is therefore a field opposed to merely a ring.
\end{proof}
\begin{definition}
Let $F$ be a valued field. We will call a function $\nu:F\rightarrow\mathbb{R}\cup\{\infty\}$ a {\em valuation logarithm} if and only if for an appropriate fixed $r>1$ we have $|a|_{F}=r^{-\nu(a)}$ for all $a\in F$.
\label{def:CVFVL}
\end{definition}
\begin{remark} We have the following basic facts.
\begin{enumerate}
\item[(i)]
With reference to Definition \ref{def:CVFV}, a valuation logarithm $\nu$ on a non-Archimedean field $F$ has the following properties. For $a, b\in F$ we have:
\begin{enumerate}
\item[(1)]
$\nu(a+b)\geq\mbox{min}(\nu(a),\nu(b))$;
\item[(2)]
$\nu(ab)=\nu(a)+\nu(b)$;
\item[(3)]
$\nu(1)=0$ and $\nu(a)=\infty$ if and only if $a=0$.
\end{enumerate}
\item[(ii)]
Every valued field $F$ has a valuation logarithm since we can take $r=e$, where $e:=\exp(1)$, and for $a\in F$ define
\begin{equation*}
\nu(a): = \left\{ \begin{array}{l@{\quad\mbox{if}\quad}l}-\log|a|_{F} & a\not=0 \\ \infty & a=0. \end{array} \right.
\end{equation*}
\item[(iii)]
If $\nu$ is a valuation logarithm on a valued field $F$ then so is $c\nu$ for any $c\in\mathbb{R}$ with $c>0$. However there will sometimes be a preferred choice. For example a valuation logarithm of {\em rank 1} is such that $\nu(F^{\times})=\mathbb{Z}$.
\end{enumerate}
\label{rem:CVFVL}
\end{remark}
\begin{lemma}
Let $F$ be a non-Archimedean field with valuation logarithm $\nu$. If $a,b\in F$ are such that $\nu(a)<\nu(b)$ then $\nu(a+b)=\nu(a)$.
\label{lem:CVFEQ}
\end{lemma}
\begin{proof}
Given that $\nu(a)<\nu(b)$ we have $\nu(a+b)\geq\mbox{min}(\nu(a),\nu(b))=\nu(a)$. Moreover $0=\nu(1)=\nu((-1)(-1))=2\nu(-1)$ therefore giving $\nu(-b)=\nu(-1)+\nu(b)=\nu(b)$. Hence $\nu(a)\geq\mbox{min}(\nu(a+b),\nu(-b))=\mbox{min}(\nu(a+b),\nu(b))$. But $\nu(a)<\nu(b)$ and so $\nu(a)\geq\nu(a+b)$ giving $\nu(a+b)=\nu(a)$.
\end{proof}
Before looking at specific examples of complete valued fields we first consider some of the theory concerning series representations of elements.
\subsection{Series expansions of elements of valued fields}
\label{subsec:CVFSE}
\begin{definition}
Let $F$ be a valued field. If 1 is an isolated point of $|F^{\times}|_{F}$, equivalently 0 is an isolated point of $\nu(F^{\times})$ for $\nu$ a valuation logarithm on $F$, then the valuation on $F$ is said to be {\em discrete}, else it is said to be {\em dense}.
\label{def:CVFDIS}
\end{definition}
\begin{lemma}
If a valued field $F$ has a discrete valuation then $\nu(F^{\times})$ is a discrete subset of $\mathbb{R}$ for $\nu$ a valuation logarithm on $F$.
\label{lem:CVFDV}
\end{lemma}
\begin{proof}
We show the contrapositive. Suppose there is a sequence $(a_{n})$ of elements of $F^{\times}$ such that $\nu(a_{n})$ converges to a point of $\mathbb{R}$ with $\nu(a_{n})\not=\lim_{m\to\infty}\nu(a_{m})$ for all $n\in\mathbb{N}$. We can take $(a_{n})$ to be such that $\nu(a_{n})\not=\nu(a_{m})$ for $n\not=m$. Then setting $b_{n}:=a_{n}a^{-1}_{n+1}$ defines a sequence $(b_{n})$ such that
\begin{equation*}
\nu(b_{n})=\nu(a_{n}a^{-1}_{n+1})=\nu(a_{n})+\nu(a^{-1}_{n+1})=\nu(a_{n})-\nu(a_{n+1})
\end{equation*}
which converges to 0.
\end{proof}
The following standard definitions are particularly important.
\begin{definition}
For $F$ a non-Archimedean field with valuation logarithm $\nu$, Define:
\begin{enumerate}
\item[(i)]
$\mathcal{O}_{F}:=\{a\in F:\nu(a)\geq0,\mbox{ equivalently }|a|_{F}\leq1\}$ the {\em ring of integers} of $F$ noting that this is a ring by the strong triangle inequality;
\item[(ii)]
$\mathcal{O}^{\times}_{F}:=\{a\in F:\nu(a)=0,\mbox{ equivalently }|a|_{F}=1\}$ the units of $\mathcal{O}_{F}$;
\item[(iii)]
$\mathcal{M}_{F}:=\{a\in F:\nu(a)>0,\mbox{ equivalently }|a|_{F}<1\}$ the maximal ideal of $\mathcal{O}_{F}$ of elements without inverses in $\mathcal{O}_{F}$;
\item[(iv)]
$\overline{F}:=\mathcal{O}_{F}/\mathcal{M}_{F}$ the {\em residue field} of $F$ of {\em residue classes}.
\end{enumerate}
\label{def:CVFRF}
\end{definition}
\begin{definition}
Let $F$ be a field with a discrete valuation and valuation logarithm $\nu$.
\begin{enumerate}
\item[(i)]
If $|F|_{F}=\{0,1\}$, equivalently $\nu(F)=\{0,\infty\}$, then the valuation is called {\em trivial}.
\item[(ii)]
If $|\cdot|_{F}$ is not trivial then an element $\pi\in F^{\times}$ such that $\nu(\pi)=\min\nu(F^{\times})\cap(0,\infty)$ is called a {\em prime} element since $\pi\not=ab$ for all $a,b\in\mathcal{O}_{F}\backslash\mathcal{O}^{\times}_{F}$ given above.
\end{enumerate}
\label{def:CVFPE}
\end{definition}
\begin{remark}
For a field $F$, as in part (ii) of Definition \ref{def:CVFPE}, it follows easily from Lemma \ref{lem:CVFDV} that $F$ has a prime element $\pi$ and from Remark \ref{rem:CVFVL} that $\nu(F^{\times})=\nu(\pi)\mathbb{Z}$ which we call the {\em value group}. Moreover for $a\in F^{\times}$ we have
\begin{equation*}
|a|_{F}=r^{-\nu(a)}=e^{\nu(\pi)\log(r)(-\nu(a)/\nu(\pi))}=e^{\log(|\pi|^{-1}_{F})(-\nu(a)/\nu(\pi))}=\left(|\pi|^{-1}_{F}\right)^{-\nu(a)/\nu(\pi)}
\end{equation*}
giving a rank 1 valuation logarithm $\frac{1}{\nu(\pi)}\nu$ noting that $|\pi|^{-1}_{F}>1$ since $r>1$.
\label{rem:CVFVG}
\end{remark}
\begin{theorem}
Let $F$ be a valued field with a non-trivial, discrete valuation. Let $\pi$ be a prime element of $F$ and let $\mathcal{R}\subseteq\mathcal{O}^{\times}_{F}\cup\{0\}$ be a set of residue class representatives with 0 representing $\bar{0}=\mathcal{M}_{F}$. Then every element $a\in F^{\times}$ has a unique series expansion over $\mathcal{R}$ of the form
\begin{equation*}
a=\sum_{i=n}^{\infty}a_{i}\pi^{i}\quad\mbox{for some }n\in\mathbb{Z}\mbox{ with }a_{n}\not=0.
\end{equation*}
Moreover if $F$ is complete then every series over $\mathcal{R}$ of the above form defines an element of $F^{\times}$.
\label{thr:CVFSE}
\end{theorem}
\begin{remark}
Concerning Theorem \ref{thr:CVFSE}.
\begin{enumerate}
\item[(i)]
A proof is given in \cite[p28]{Schikhof}, in fact a generalisation of Theorem \ref{thr:CVFSE} is also given that can be applied to non-Archimedean fields with a dense valuation.
\item[(ii)]
For $a=\sum_{i=n}^{\infty}a_{i}\pi^{i}$ as in Theorem \ref{thr:CVFSE} and using the rank 1 valuation logarithm of Remark \ref{rem:CVFVG} we have, for each $i\geq n$, $\nu(a_{i}\pi^{i})=\nu(a_{i})+i\nu(\pi)=i$ if $a_{i}\pi^{i}\not=0$ and $\nu(a_{i}\pi^{i})=\infty$ otherwise. Further by Lemma \ref{lem:CVFEQ} $b_{m}:=\sum_{i=n}^{m}a_{i}\pi^{i}$ defines a Cauchy sequence in $F$, with respect to $|\cdot|_{F}$, and its limit is $a$. Hence, since for each $m>n$ we have $|a|_{F}-|b_{m}|_{F}\leq |a-b_{m}|_{F}$, $|b_{m}|_{F}$ converges in $\mathbb{R}$ to $|a|_{F}$. But $\nu(b_{m})=n$ for all $m>n$ by Lemma \ref{lem:CVFEQ} and so $\nu(a)=n$.
\end{enumerate}
\label{rem:CVFSV}
\end{remark}
We will now look at some examples of complete valued fields and consider the availability of such structures in the Archimedean and non-Archimedean settings.
\subsection{Examples of complete valued fields}
\label{subsec:CVFEX}
\begin{example} Here are some non-Archimedean examples.
\begin{enumerate}
\item[(i)]
Let $F$ be any field. Then $F$ with the trivial valuation is a non-Archimedean field. It is complete noting that in this case each Cauchy sequences will be constant after some finite number of initial values. The trivial valuation induces the trivial topology on $F$ where every subset of $F$ is {\em clopen} i.e. both open and closed. Furthermore $F$ will coincide with its own residue field.
\item[(ii)]
There are examples of complete non-Archimedean fields of non-zero characteristic with non-trivial valuation. For each there is a prime $p$ such that the field is a transcendental extension of the finite field $\mathbb{F}_{p}$ of $p$ elements. The reason why such a field is not an algebraic extension of $\mathbb{F}_{p}$ follows easily from the fact that the only valuation on a finite field is the trivial valuation. One example of this sort is the valued field of formal Laurent series $\mathbb{F}_{p}\{\{T\}\}$ in one variable over $\mathbb{F}_{p}$ with termwise addition,
\begin{equation*}
\mbox{$\sum_{n\in\mathbb{Z}}a_{n}T^{n}+\sum_{n\in\mathbb{Z}}b_{n}T^{n}:=\sum_{n\in\mathbb{Z}}(a_{n}+b_{n})T^{n}$,}
\end{equation*}
multiplication in the form of the Cauchy product,
\begin{equation*}
\mbox{$(\sum_{n\in\mathbb{Z}}a_{n}T^{n})(\sum_{n\in\mathbb{Z}}b_{n}T^{n}):=\sum_{n\in\mathbb{Z}}(\sum_{i\in\mathbb{Z}}a_{i}b_{n-i})T^{n}$,}
\end{equation*}
and valuation given at zero by $|0|_{T}:=0$ and on the units $\mathbb{F}_{p}\{\{T\}\}^{\times}$ by,
\begin{equation*}
\mbox{$|\sum_{n\in\mathbb{Z}}a_{n}T^{n}|$}_{T}:=r^{-\mbox{min}\{n:a_{n}\not= 0\}}\mbox{ for any fixed }r>1.
\end{equation*}
The valuation on $\mathbb{F}_{p}\{\{T\}\}$ is discrete and its residue field is isomorphic to $\mathbb{F}_{p}$. The above construction also gives a complete non-Archimedean field if we replace $\mathbb{F}_{p}$ with any other field $F$, see \cite[p288]{Schikhof}.
\item[(iii)]
On the other hand complete valued fields of characteristic zero necessarily contain one of the completions of the rational numbers $\mathbb{Q}$. The Levi-Civita field $R$ is such a valued field, see \cite{Shamseddine}. Each element $a\in R$ can be represented as a formal power series of the form
\begin{equation*}
\mbox{$a=\sum_{q\in\mathbb{Q}}a_{q}T^{q}$ with $a_{q}\in\mathbb{R}$ for all $q\in\mathbb{Q}$}
\end{equation*}
such that for each $q\in\mathbb{Q}$ there are at most finitely many $q'<q$ with $a_{q'}\not=0$. Moreover addition, multiplication and the valuation for $R$ can all be obtained by analogy with example (ii) above. A total order can be put on the Levi-Civita field such that the order topology agrees with the topology induced by the field's valuation which is non-trivial. To verify this one shows that the order topology sub-base of open rays topologically generates the valuation topology sub-base of open balls and vice versa. This might be useful to those interested in generalising the theory of C*-algebras to new fields where there is a need to define positive elements. The completion of $\mathbb{Q}$ that the Levi-Civita field contains is in fact $\mathbb{Q}$ itself since the valuation when restricted to $\mathbb{Q}$ is trivial.
\end{enumerate}
\label{exa:CVFA}
\end{example}
We consider what examples of complete Archimedean fields there are. Since the only valuation on a finite field is the trivial valuation, it follows from Remark \ref{rem:CVFE} that every Archimedean field is of characteristic zero. Moreover every non-trivial valuation on the rational numbers is given by Ostrowski's Theorem, see \cite[p2]{Fesenko}\cite[p22]{Schikhof}.
\begin{theorem}
A non-trivial valuation on $\mathbb{Q}$ is either a power of the absolute valuation $|\cdot|_{\infty}^{c}$, with $0<c\leq1$, or a power of the $p$-adic valuation $|\cdot|_{p}^{c}$ for some prime $p\in\mathbb{N}$ with positive $c\in\mathbb{R}$.
\label{thr:CVFOS}
\end{theorem}
\begin{remark}
We will look at the $p$-adic valuations on $\mathbb{Q}$ and the $p$-adic numbers in Example \ref{exa:CVFPN}. We note that any two of the valuations mentioned in Theorem \ref{thr:CVFOS} that are not the same up to a positive power will also not be equivalent as norms. Further, since all of the $p$-adic valuations are non-Archimedean, Theorem \ref{thr:CVFOS} implies that every complete Archimedean field contains $\mathbb{R}$, with a positive power of the absolute valuation, as a valued sub-field. It turns out that almost all complete valued fields are non-Archimedean with $\mathbb{R}$ and $\mathbb{C}$ being the only two Archimedean exceptions up to isomorphism as topological fields, see \cite[p36]{Schikhof}. This in part follows from the Gel'fand-Mazur Theorem which depends on spectral analysis involving Liouville's Theorem and the Hahn-Banach Theorem in the complex setting. We will return to these issues in the more general setting of Banach $F$-algebras.
\label{rem:CVFOS}
\end{remark}
\begin{example}
\label{exa:CVFPN}
Let $p\in\mathbb{N}$ be a prime. Then with reference to Remark \ref{rem:CVFVL}, for $n\in\mathbb{Z}$, 
\begin{equation*}
\nu_{p}(n):=\left\{\begin{array}{l@{\quad\mbox{if}\quad}l}\mbox{max}\{i\in\mathbb{N}_{0}:p^{i}|n\} & n\not=0 \\ \infty & n=0 \end{array}\right. , \quad\mathbb{N}_{0}:=\mathbb{N}\cup\{0\},
\end{equation*}
extends uniquely to $\mathbb{Q}$ under the properties of a valuation logarithm. Indeed for $n\in\mathbb{N}$ we have
\begin{equation*}
0=\nu_{p}(1)=\nu_{p}(n/n)=\nu_{p}(n)+\nu_{p}(1/n)
\end{equation*}
giving $\nu_{p}(1/n)=-\nu_{p}(n)$ etc. The standard $p$-adic valuation of $a\in\mathbb{Q}$ is then given by $|a|_{p}:=p^{-\nu_{p}(a)}$. This is a discrete valuation on $\mathbb{Q}$ with respect to which $p$ is a prime element in the sense of Definition \ref{def:CVFPE}. Moreover $\mathcal{R}_{p}:=\{0,1,\cdots,p-1\}$ is one choice of a set of residue class representatives for $\mathbb{Q}$. This is because for $m,n\in\mathbb{N}$ with $p\nmid m$ and $p\nmid n$ we have that $m$, using the Division algorithm, can be expressed as $m=a_{1}+pb_{1}$ and $1/n$, using the extended Euclidean algorithm, can be expressed as $1/n=a_{2}+pb_{2}/n$ with $a_{1},a_{2}\in\{1,\cdots,p-1\}$ and $b_{1},b_{2}\in\mathbb{Z}$. Hence, with reference to Definition \ref{def:CVFRF}, $m/n$ can be expressed as $m/n=a_{3}+pb_{3}/n$ with $a_{3}\in\{1,\cdots,p-1\}$ and $pb_{3}/n\in\mathcal{M}_{p}$ as required. With these details in place we can apply Theorem \ref{thr:CVFSE} so that every element $a\in\mathbb{Q}^{\times}$ has a unique series expansion over $\mathcal{R}_{p}$ of the form
\begin{equation*}
a=\sum_{i=n}^{\infty}a_{i}p^{i}\quad\mbox{for some }n\in\mathbb{Z}\mbox{ with }a_{n}\not=0.
\end{equation*}
The completion of $\mathbb{Q}$ with respect to $|\cdot|_{p}$ is the field of $p$-adic numbers denoted $\mathbb{Q}_{p}$. The elements of $\mathbb{Q}_{p}^{\times}$ are all of the series of the above form when using the expansion over $\mathcal{R}_{p}$. Further, with reference to Remark \ref{rem:CVFSV}, for such an element $a=\sum_{i=n}^{\infty}a_{i}p^{i}$ with $a_{n}\not=0$ we have $\nu_{p}(a)=n$. As an example of such expansions for $p=5$ we have,
\begin{equation*}
\frac{1}{2}=3\cdot5^{0}+2\cdot5+2\cdot5^{2}+2\cdot5^{3}+2\cdot5^{4}+\cdots.
\end{equation*}
More generally the residue field of $\mathbb{Q}_{p}$ is the finite field $\mathbb{F}_{p}$ of $p$ elements. Each non-zero element of $\mathbb{F}_{p}$ has a lift to a $p-1$ root of unity in $\mathbb{Q}_{p}$, see \cite[p37]{Fesenko}. These roots of unity together with $0$ also constitute a set of residue class representatives for $\mathbb{Q}_{p}$. The ring that they generate embeds as a ring into the complex numbers, e.g. see Figure \ref{fig:CVFRL}.
\begin{figure}[h]
\begin{center}
\includegraphics{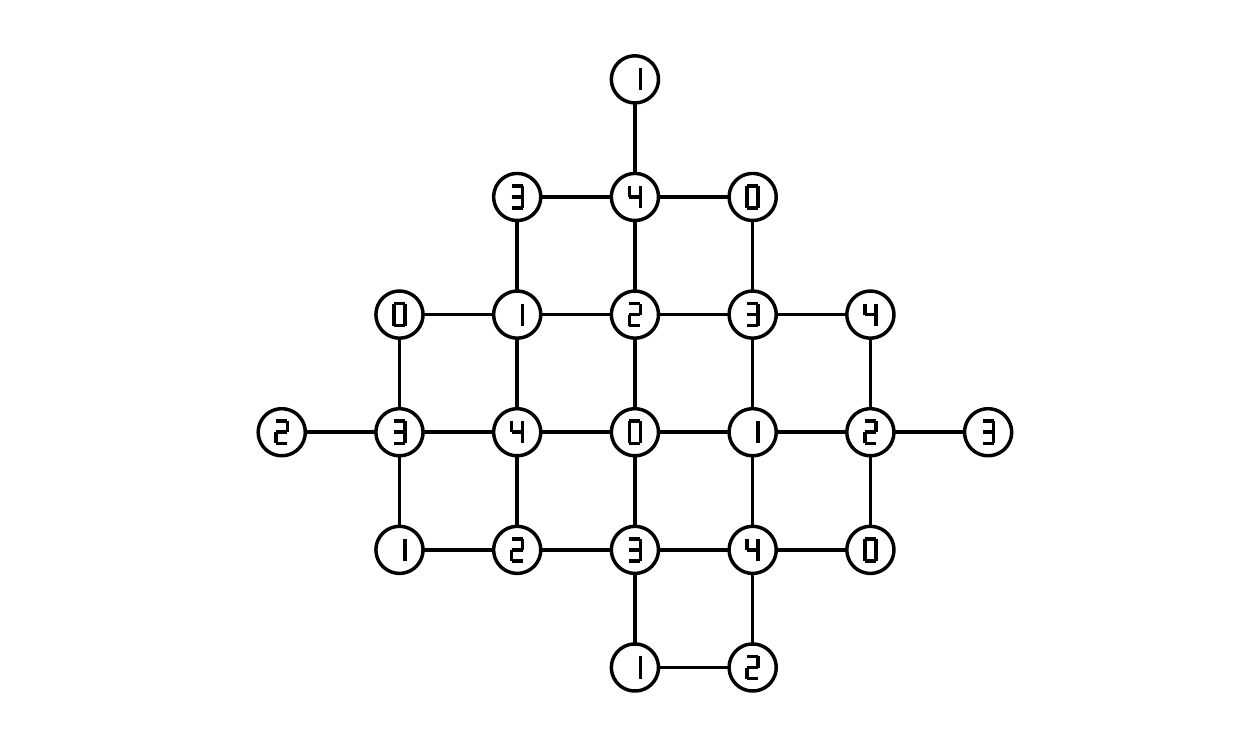}
\end{center}
\caption{Part of a ring in $\mathbb{C}$. The points are labeled with the first coefficient of their corresponding $5$-adic expansion over $\mathcal{R}_{5}$ under a ring isomorphism.}
\label{fig:CVFRL}
\end{figure}
Moreover as a field, rather than as a valued field, $\mathbb{Q}_{p}$ has an embedding into $\mathbb{C}$. The $p$-adic valuation on $\mathbb{Q}_{p}$ can then be extended to a complete valuation on the complex numbers which in this case as a valued field we denote as $\mathbb{C}_{p}$, see \cite[46]{Schikhof}\cite{Roquette}.
\end{example}
Finally it is interesting to note that the different standard valuations on $\mathbb{Q}$, when restricted to the units $\mathbb{Q}^{\times}$, are related by the equation $|\cdot|_{0}|\cdot|_{2}|\cdot|_{3}|\cdot|_{5}\cdots|\cdot|_{\infty}=1$ where $|\cdot|_{0}$ denotes the trivial valuation and $|\cdot|_{\infty}$ the absolute value function. See \cite[p3]{Fesenko}.
\subsection{Topological properties of complete valued fields}
\label{subsec:CVFTP}
In this subsection we consider the connectedness and local compactness of complete valued fields.
\begin{definition}
Let $X$ be a topological space and $Y\subseteq X$.
\begin{enumerate}
\item[(i)]
If $Y$ cannot be expressed as the disjoint union of two non-empty clopen subsets with respect to the relative topology then $Y$ is said to be a {\em connected subset} of $X$.
\item[(ii)]
If the only non-empty connected subsets of $X$ are singletons then $X$ is said to be {\em totally disconnected}.
\item[(iii)]
If for each pair of points $x,y\in X$ there exists a continuous map $f:I\rightarrow X$, $I:=[0,1]\subseteq\mathbb{R}$, such that $f(0)=x$ and $f(1)=y$ then $X$ is {\em path-connected}.
\item[(iv)]
A {\em neighborhood base} $\mathfrak{B}_{x}$ at a point $x\in X$ is a collection of neighborhoods of $x$ such that for every neighborhood $U$ of $x$ there is $V\in\mathfrak{B}_{x}$ with $V\subseteq U$.
\item[(v)]
We call $X$ {\em locally compact} if and only if each point in $X$ has a neighborhood base consisting of compact sets.
\end{enumerate}
\label{def:CVFCO}
\end{definition}
The following lemma is well known however I have provided a proof for the reader's convenience.
\begin{lemma}
Let $F$ be a non-Archimedean field. Then $F$ is totally disconnected.
\label{lem:CVFTD}
\end{lemma}
\begin{proof}
Let $F$ be a non-Archimedean field and $r\in\mathbb{R}$ with $r>0$. For $a,b\in F$, $a\sim b$ if and only if $|a-b|_{F}<r$ defines an equivalence relation on $F$ by the strong triangle inequality noting that for transitivity if $a\sim b$ and $b\sim c$ then
\begin{equation*}
|a-c|_{F}\leq\mbox{max}(|a-b|_{F},|b-c|_{F})<r.
\end{equation*}
Hence for $a\in F$ the $F$-ball $B_{r}(a)$ is an equivalence class and so every element of $B_{r}(a)$ is at its center because every element is an equivalence class representative. In particular if $b\in B_{r}(a)$ then $B_{r}(b)=B_{r}(a)$ but for $b\notin B_{r}(a)$ we have $B_{r}(b)\cap B_{r}(a)=\emptyset$, showing that $B_{r}(a)$ is clopen. Since this holds for every $r>0$, $a$ has a neighborhood base of clopen balls. Hence since $F$ is Hausdorff, $\{a\}$ is the only connected subset of $F$ with $a$ as an element and so $F$ is totally disconnected.
\end{proof}
\begin{remark} We make the following observations.
\begin{enumerate}
\item[(i)]
Every complete Archimedean field is path-connected whereas every complete non-Archimedean field is totally disconnected.
\item[(ii)]
In general a valued field being totally disconnected is not the same as it being discrete. For example $\mathbb{Q}$ with the absolute valuation is a totally disconnected Archimedean field but it is obviously neither discrete nor complete. Also it is easy to show that a valued field admits a non-constant path if and only if it is path-connected, see \cite[p197]{Willard} for the standard definitions used here.
\item[(iii)]
With reference to the proof of Lemma \ref{lem:CVFTD}, $a\sim b$ if and only if $|a-b|_{F}\leq r$, noting the change from the strict inequality, is again an equivalence relation on $F$. Hence every ball of positive radius in a non-Archimedean field is clopen although a ball $\bar{B}_{r}(a):=\{b\in F:|b-a|_{F}\leq r\}$ may contain elements in addition to those in $B_{r}(a)$ depending on whether $r\in|F^{\times}|_{F}$. To clarify then, in the non-Archimedean setting $\bar{B}_{r}(a)$ does not denote the closure of $B_{r}(a)$ with respect to the valuation.
\item[(iv)]
In section \ref{subsec:UASCS} concerning complex uniform algebras we will look at Swiss cheese sets. For a non-Archimedean field $F$ if $a,b\in F$ and $r_{1},r_{2}\in\mathbb{R}$ with $r_{1}\geq r_{2}>0$ then either $B_{r_{2}}(b)\subseteq B_{r_{1}}(a)$ or $B_{r_{2}}(b)\cap B_{r_{1}}(a)=\emptyset$ since either $B_{r_{1}}(b)=B_{r_{1}}(a)$ or $B_{r_{1}}(b)\cap B_{r_{1}}(a)=\emptyset$. Further if $S$ is an $F$-ball or the complement of an $F$-ball then the closure of $S$ with respect to $|\cdot|_{F}$ coincides with $S$ since $F$-balls are clopen. Hence a Swiss cheese set $X\subseteq F$ will be classical exactly when there exists a countable or finite collection $\mathcal{D}$ of $F$-balls, with finite radius sum, and an $F$-ball $\Delta$ such that each element of $\mathcal{D}$ is a subset of $\Delta$ and $X=\Delta\backslash\bigcup\mathcal{D}$. It follows that such a set $X$ can be empty in the non-Archimedean setting.
\end{enumerate}
\label{rem:CVFTD}
\end{remark}
\begin{theorem}
Let $X$ be a Hausdorff space. Then $X$ is locally compact if and only if each point in $X$ has a compact neighborhood.
\label{thr:CVFHL}
\end{theorem}
\begin{theorem}
Let $F$ be a complete non-Archimedean field that is not simultaneously both infinite and with the trivial valuation. Then the following are equivalent:
\begin{enumerate}
\item[(i)]
$F$ is locally compact;
\item[(ii)]
the residue field $\overline{F}$ is finite and the valuation on $F$ is discrete;
\item[(iii)]
each bounded sequence in $F$ has a convergent subsequence;
\item[(iv)]
each infinite bounded subset of $F$ has an accumulation point in $F$;
\item[(v)]
each closed and bounded subset of $F$ is compact.
\end{enumerate}
\label{thr:CVFHB}
\end{theorem}
Proofs of Theorem \ref{thr:CVFHL} and Theorem \ref{thr:CVFHB} can be found in \cite[p130]{Willard} and \cite[p29,p57]{Schikhof} respectively.
\begin{remark}
Concerning Theorem \ref{thr:CVFHB}.
\begin{enumerate}
\item[(i)]
Let $F$ be an infinite field with the trivial valuation. Then $F$ is locally compact since for $a\in F$ it follows that $\{\{a\}\}$ is a neighborhood base of compact sets for $a$. However $F$ does not have any of the other properties given in Theorem \ref{thr:CVFHB}. For example the residue field $\overline{F}$ is $F$ and so it is not finite and $F$ itself is closed and bounded but not compact etc.
\item[(ii)]
We will call a complete non-Archimedean field $F$ that satisfies (ii) of Theorem \ref{thr:CVFHB} a {\em local field}. Some authors weaken the condition on the residue field when defining local fields so that the residue field needs only to be of prime characteristic for some prime $p$ and perfect, that is the Frobenius endomorphism on $\overline{F}$, $\bar{a}\to \bar{a}^{p}$, is an automorphism.
\item[(iii)]
Since the only complete Archimedean fields are $\mathbb{R}$ and $\mathbb{C}$ all but property (ii) of Theorem \ref{thr:CVFHB} hold for complete Archimedean fields by the Heine-Borel Theorem and Bolzano-Weierstrass Theorem etc. In fact this provides one way to prove Theorem \ref{thr:CVFHB} since if $F$ is a local field then there is a homeomorphic embedding of $F$ into $\mathbb{C}$ as a closed unbounded subset.
\item[(iv)]
By the details given in Example \ref{exa:CVFPN} it is immediate that for each prime $p$ the field of $p$-adic numbers $\mathbb{Q}_{p}$ is a local field. However $\mathbb{C}_{p}$ is not a local field since its valuation is dense and its residue field is infinite, see \cite[p45]{Schikhof}.
\end{enumerate}
\label{rem:CVFHB}
\end{remark}
\section{Extending complete valued fields}
\label{sec:CVFEGT}
In this section and later chapters we will adopt the following notation.
\notation
If $F$ is a field and $L$ is a field extending $F$ then we will denote the Galois group of $F$-automorphisms on $L$, that is automorphisms on $L$ that fix the elements of $F$, by $\mbox{Gal}(^{L}/_{F})$. Further we will denote fixed fields by:
\begin{enumerate}
\item[(i)]
$L^{g}:=\{x\in L:g(x)=x\}$, for $g\in\mbox{Gal}(^{L}/_{F})$;
\item[(ii)]
$L^{G}:=\bigcap_{g\in G}L^{g}$, for a subgroup $G\leqslant\mbox{Gal}(^{L}/_{F})$.
\end{enumerate}
More generally if $S$ is a set and $G$ is a group of self maps $g:S\rightarrow S$, with group law composition, then we will denote:
\begin{enumerate}
\item[(1)]
$\mbox{ord}(g):=\mbox{min}\{n\in\mathbb{N}:g^{(n)}=\mbox{id}\}$, the order of an element $g\in G$ with finite order;
\item[(2)]
$\mbox{ord}(g,s):=\mbox{min}\{n\in\mathbb{N}:g^{(n)}(s)=s\}$, the order at an element $s\in S$, when finite, of an element $g\in G$;
\item[(3)]
$\mbox{ord}(g,S):=\{\mbox{ord}(g,s):s\in S\}$, the order set of an element $g\in G$ with finite order.
\end{enumerate}
In the rest of this section we will look mainly at extensions of valued fields, including their valuations, as well as some Galois theory used in later chapters.
\subsection{Extensions}
\label{subsec:CVFE}
The first theorem below is rather general in scope.
\begin{theorem}
Let $F$ be a complete non-Archimedean field. All non-Archimedean norms, that is norms that observe the strong triangle inequality, on a finite dimensional $F$-vector space $E$ are equivalent. Further $E$ is a Banach space, i.e. complete normed space, with respect to each norm.
\label{thr:CVFEN}
\end{theorem}
Theorem \ref{thr:CVFEN} also holds for complete Archimedean fields and in the Archimedean setting proofs often make use of the underlying field being locally compact. However the complete non-Archimedean field $F$ in Theorem \ref{thr:CVFEN} is not assumed to be locally compact and so a proof of the theorem from \cite{Schikhof} has been included below for interest.
\begin{proof}[Proof of Theorem \ref{thr:CVFEN}]
We use induction on $n:=\dim E$. The base case $n=1$ is immediate. Suppose Theorem \ref{thr:CVFEN} holds for $(n-1)$-dimensional spaces and let $E$ be such that $\dim E=n$. We choose a base $e_{1},\cdots,e_{n}$ for $E$ and define
\begin{equation*}
\|x\|_{\infty}=:\max_{i}|a_{i}|_{F}\quad\mbox{for }x=\sum_{i=1}^{n}a_{i}e_{i}\in E.
\end{equation*}
Note that $\|\cdot\|_{\infty}$ is a non-Archimedean norm on $E$ by $|\cdot|_{F}$ being a non-Archimedean valuation. Now let $\|\cdot\|$ be any other non-Archimedean norm on $E$. We show that $\|\cdot\|$ is equivalent to $\|\cdot\|_{\infty}$. For $x=\sum_{i=1}^{n}a_{i}e_{i}\in E$ we have
\begin{equation*}
\|x\|\leq\max_{i}|a_{i}|_{F}\|e_{i}\|\leq M\|x\|_{\infty}
\end{equation*}
where $M:=\max_{i}\|e_{i}\|$. Hence it remains to show that there is a positive constant $N$ such that for all $x\in E$ we have $\|x\|\geq N\|x\|_{\infty}$. Let $D$ be the linear subspace generated by $e_{1},\cdots,e_{n-1}$. By the inductive hypothesis there is $c>0$ such that for all $x\in D$ we have $\|x\|\geq c\|x\|_{\infty}$. Further $D$ is complete and hence closed in $E$ with respect to $\|\cdot\|$. Hence for
\begin{equation*}
c':=\|e_{n}\|^{-1}\inf\{\|e_{n}-y\|:y\in D\}
\end{equation*}
we have $0<c'\leq1$. Now set $N:=\min(c'c,c'\|e_{n}\|)$ and let $x\in E$. Then $x=y+a_{n}e_{n}$ for some $y\in D$ and $a_{n}\in F$. If $a_{n}\not=0$ then $\|x\|=|a_{n}|_{F}\|e_{n}+a_{n}^{-1}y\|\geq|a_{n}|_{F}\|e_{n}\|c'=c'\|a_{n}e_{n}\|$. But then we also get
\begin{equation*}
\|x\|\geq\max(c'\|x\|,c'\|a_{n}e_{n}\|)\geq c'\|x-a_{n}e_{n}\|=c'\|y\|
\end{equation*}
and this inequality also holds for $a_{n}=0$ since $0<c'\leq1$. We get 
\begin{equation*}
\|x\|\geq c'\max(\|y\|,\|a_{n}e_{n}\|)\geq c'\max(c\|y\|_{\infty},\|e_{n}\| |a_{n}|_{F})\geq N\max(\|y\|_{\infty},|a_{n}|_{F})
\end{equation*}
where $N\max(\|y\|_{\infty},|a_{n}|_{F})=N\|x\|_{\infty}$. Hence $\|\cdot\|$ and $\|\cdot\|_{\infty}$ are equivalent. Finally we note that a sequence in $E$ is a Cauchy sequence with respect to $\|\cdot\|_{\infty}$ if and only if each of its coordinate sequences is a Cauchy sequence with respect to $|\cdot|_{F}$. Hence $E$ is a Banach space with respect to each norm by the equivalence of norms and completeness of $F$.
\end{proof}
\begin{remark}
Let $L$ be a non-Archimedean field and let $F$ be a complete subfield of $L$. If $L$ is a finite extension of $F$ then $L$ is also complete by Theorem \ref{thr:CVFEN}. Now suppose $L$ is such a complete finite extension of $F$, then viewing $L$ as a finite dimensional $F$-vector space we note that convergence in $L$ is coordinate-wise since $|\cdot|_{L}$ is equivalent to $\|\cdot\|_{\infty}$ by Theorem \ref{thr:CVFEN}. Hence each element $g\in\mbox{Gal}(^{L}/_{F})$ is continuous since being linear over $F$. We will see later in this section that, for such complete finite extensions, each element of $\mbox{Gal}(^{L}/_{F})$ is in fact an isometry. Finally in all cases if $L$ is complete and $g$ is continuous then the fixed field $L^{g}$ is also complete. To see this let $(a_{n})$ be a Cauchy sequence in $L^{g}$ and let $a$ be its limit in $L$. For $\mathrm{id}$ the identity map on $L$, note that $g-\mathrm{id}$ is also continuous on $L$ and so $L^{g}=(g-\mathrm{id})^{-1}(0)$ is a closed subset of $L$. In particular we have $a\in L^{g}$ as required.
\label{rem:CVFEN}
\end{remark}
The following is Krull's extension theorem, a proof can be found in \cite[p34]{Schikhof}.
\begin{theorem}
Let $F$ be a subfield of a field $L$ and let $|\cdot|_{F}$ be a non-Archimedean valuation on $F$. Then there exists a non-Archimedean valuation on $L$ that extends $|\cdot|_{F}$.
\label{thr:CVFKET}
\end{theorem}
The following corollary to Theorem \ref{thr:CVFKET}, which also uses Theorem \ref{thr:CVFCOM}, contrasts with the Archimedean setting.
\begin{corollary}
For every complete non-Archimedean field $F$ there exists a proper extension $L$ of $F$ for which the complete valuation on $F$ extends to a complete valuation on $L$.
\label{cor:CVFEE}
\end{corollary}
Moreover an extension of a valuation is often unique.
\begin{theorem}
Let $F$ be a complete non-Archimedean field, let $L$ be an algebraic extension of $F$ and let $a\in L$. Then:
\begin{enumerate}
\item[(i)]
there is a unique valuation $|\cdot|_{L}$ on $L$ that extends the valuation on $F$;
\item[(ii)]
if $\|\cdot\|$ is an arbitrary norm on the $F$-vector space $L$ then $|a|_{L}=\lim_{n\to\infty}\sqrt[n]{\|a^{n}\|}$.
\end{enumerate}
\label{thr:CVFUT}
\end{theorem}
\begin{remark}
We make the following observations.
\begin{enumerate}
\item[(i)]
Part (i) of Theorem \ref{thr:CVFUT} follows easily from Theorem \ref{thr:CVFKET} and Theorem \ref{thr:CVFEN} applied respectively noting that if $a\in L$ then $a$ is also an element of a finite extension of $F$. See \cite[p39]{Schikhof} for the rest of the proof.
\item[(ii)]
It is worth emphasizing that Theorem \ref{thr:CVFEN}, Theorem \ref{thr:CVFKET} and Theorem \ref{thr:CVFUT} all hold for the case where the valuation on $F$ is trivial.
\item[(iii)]
Now for $L$ and $F$ conforming to the conditions of Theorem \ref{thr:CVFUT} we have that each $g\in\mbox{Gal}(^{L}/_{F})$ is indeed an isometry on $L$ since $|a|':=|g(a)|_{L}$, for $a\in L$, is a valuation on $L$ giving $|g(a)|_{L}=|a|_{L}$ by uniqueness.
\end{enumerate}
\label{rem:CVFUT}
\end{remark}
The following theory will often allow us to express the extension of a valuation in a particularly useful form. We begin with a standard theorem.
\begin{theorem}
Let $F$ be a field, let $L$ be an algebraic extension of $F$ and let $a\in L$. Then there is a unique monic irreducible polynomial $\mbox{Irr}_{F,a}(x)\in F[x]$ such that $\mbox{Irr}_{F,a}(a)=0$. Moreover, for the simple extension $F(a)$, we have $[F(a),F]=\mbox{degIrr}_{F,a}(x)$ where $[F(a),F]$ denotes the dimension of $F(a)$ as an $F$-vector space.
\label{thr:CVFIR}
\end{theorem}
\begin{definition}
Let $F$ be a field and let $L$ be an algebraic extension of $F$.
\begin{enumerate}
\item[(i)]
An element $a\in L$ is said to be {\em separable} over $F$ if $a$ is not a repeated root of its own irreducible polynomial $\mbox{Irr}_{F,a}(x)$.
\item[(ii)]
We call $L_{sc}:=\{a\in L:a\mbox{ is separable over }F\}$ the {\em separable closure} of $F$ in $L$.
\item[(iii)]
The extension $L$ is said to be a {\em separable extension} of $F$ if $L=L_{sc}$.
\item[(iv)]
Let $f(x)\in F[x]$. Then $L$ is called a {\em splitting field} of $f(x)$ over $F$ if $f(x)$ splits completely in $L[x]$ as a product of linear factors but not over any proper subfield of $L$ containing $F$.
\item[(v)]
We will call $L$ a {\em normal extension} of $F$ if $L$ is the splitting field over $F$ of some polynomial in $F[x]$.
\item[(vi)]
The field $L$ is called a {\em Galois extension} of $F$ if $L^{G}=F$ for $G:=\mbox{Gal}(^{L}/_{F})$.
\end{enumerate}
\label{def:CVFSN}
\end{definition}
\begin{remark}
Following Definition \ref{def:CVFSN} we note that the separable closure $L_{sc}$ of $F$ in $L$ is a field with $F\subseteq L_{sc}\subseteq L$. Moreover if $F$ is of characteristic zero then $L$ is a separable extension of $F$.
\label{rem:CVFSN}
\end{remark}
For proofs of the following two theorems and Remark \ref{rem:CVFSN} see \cite[p13-p19,p36]{McCarthy}.
\begin{theorem}
Let $F$ be a field and let $L$ be a finite extension of $F$. Then there is a normal extension $L_{ne}$ of $F$ which contains $L$ and which is the smallest such extension in the sense that if $K$ is a normal extension of $F$ which contains $L$ then there is a $L$-monomorphism of $L_{ne}$ into $K$, i.e. an embedding of $L_{ne}$ into $K$ that fixes $L$.
\label{thr:CVFNE}
\end{theorem}
\begin{theorem}
Let $F$ be a field and let $L$ be a finite extension of $F$. Then, with reference to Theorem \ref{thr:CVFNE} and Definition \ref{def:CVFSN}, there are exactly $[L_{sc}:F]$ distinct $F$-isomorphisms of $L$ onto subfields of $L_{ne}$. Further if $L=L_{ne}$ then $\#\mbox{Gal}(^{L}/_{F})=[L_{sc}:F]$. Moreover $L$ is a Galois extension of $F$ if and only if $L_{sc}=L=L_{ne}$ in which case $\#\mbox{Gal}(^{L}/_{F})=[L:F]$.
\label{thr:CVFGL}
\end{theorem}
\begin{definition}
Let $F$ be a field, let $L$ be a finite extension of $F$ and let $n_{0}:=[L_{sc}:F]$. By Theorem \ref{thr:CVFGL} there are exactly $n_{0}$ distinct $F$-isomorphisms $g_{1},\cdots,g_{n_{0}}$ of $L$ onto subfields of $L_{ne}$. The {\em norm map} $N_{^{L}/_{F}}:L\rightarrow F$ is defined as
\begin{equation*}
N_{^{L}/_{F}}(a):=\left(\prod_{i=1}^{n_{0}}g_{i}(a)\right)^{[L:L_{sc}]}\quad\mbox{for }a\in L.
\end{equation*}
\label{def:CVFNM}
\end{definition}
A proof showing that the norm map only takes values in the ground field can be found in \cite[p23,p24]{McCarthy}. Using the preceding theory we can now state and prove a theorem that will often allow us to express the extension of a valuation in a particularly useful form. The theorem is in the literature. However, having set out the preceding theory, the proof presented here is more immediate than the sources I have seen.
\begin{theorem}
Let $F$ be a complete non-Archimedean field with valuation $|a|_{F}=r^{-\nu(a)}$, for $a\in F$, where $\nu$ is a valuation logarithm on $F$. Let $L$ be a finite extension of $F$ as a field. Then, with reference to Theorem \ref{thr:CVFUT} and Theorem \ref{thr:CVFEN}, the unique extension of $|\cdot|_{F}$ to a complete valuation $|\cdot|_{L}$ on $L$ is given by
\begin{equation*}
|a|_{L}=\sqrt[n]{|N_{^{L}/_{F}}(a)|_{F}}=r^{-\omega(a)}\quad\mbox{for }a\in L,
\end{equation*}
where $n=[L:F]$ and $\omega:=\frac{1}{n}\nu\circ N_{^{L}/_{F}}$ is the corresponding extension of $\nu$ to $L$. If in addition the valuation $|\cdot|_{F}$ is discrete then $|\cdot|_{L}$ is also discrete. If further $|\cdot|_{F}$ is non-trivial and $\nu$ is the rank 1 valuation logarithm of remark \ref{rem:CVFVG} then $e\omega(L^{\times})=\mathbb{Z}$ for some $e\in\mathbb{N}$.
\label{thr:CVFEE}
\end{theorem}
\begin{proof}
Let $L_{ne}$ be the normal extension of $F$ containing $L$ of Theorem \ref{thr:CVFNE}. Since $L_{ne}$ is the splitting field of some polynomial in $F[x]$ it is a finite extension of $F$ and so also of $L$. Hence by Theorem \ref{thr:CVFUT} the valuation $|\cdot|_{L}$ extends uniquely to a valuation $|\cdot|_{L_{ne}}$ on $L_{ne}$. Let $n_{0}:=[L_{sc}:F]$ and let $g_{1},\cdots,g_{n_{0}}$ be the $n_{0}$ distinct $F$-isomorphisms of $L$ onto subfields of $L_{ne}$ as given by Theorem \ref{thr:CVFGL}. Then for each $i\in\{1,\cdots,n_{0}\}$ we have that $|a|_{i}:=|g_{i}(a)|_{L_{ne}}$, for $a\in L$, is a valuation on $L$ extending $|\cdot|_{F}$. Hence, by the uniqueness of $|\cdot|_{L}$ as an extension of $|\cdot|_{F}$ to $L$, each of $g_{1},\cdots,g_{n_{0}}$ is an isometry from $L$ onto a subfield of $L_{ne}$ with respect to $|\cdot|_{L_{ne}}$. Hence setting $n:=[L:F]$, $n_{1}:=[L:L_{sc}]$ and noting that the norm map $N_{^{L}/_{F}}$ takes values in $F$, we have for all $a\in L$
\begin{equation*}
|a|_{L}=\sqrt[n]{|a|_{L}^{[L:L_{sc}][L_{sc}:F]}}=\sqrt[n]{\left(\prod_{i=1}^{n_{0}}|g_{i}(a)|_{L_{ne}}\right)^{n_{1}}}=\sqrt[n]{\left|\left(\prod_{i=1}^{n_{0}}g_{i}(a)\right)^{n_{1}}\right|_{L_{ne}}}=\sqrt[n]{|N_{^{L}/_{F}}(a)|_{F}}.
\end{equation*}
Therefore we also have $\omega(L^{\times})=\frac{1}{n}\nu\circ N_{^{L}/_{F}}(L^{\times})\subseteq\frac{1}{n}\nu(F^{\times})$ and so if $|\cdot|_{F}$ is a discrete valuation then so is $|\cdot|_{L}$. Moreover $\omega$ is indeed an extension of $\nu$ since for $a\in F$ we have $\omega(a)=\frac{1}{n}\nu\circ N_{^{L}/_{F}}(a)=\frac{1}{n}\nu(a^{n})=\frac{1}{n}n\nu(a)=\nu(a)$. Now suppose that $\nu$ is a rank 1 valuation logarithm so that $\nu(F^{\times})=\mathbb{Z}$ and $\omega(L^{\times})\subseteq\frac{1}{n}\mathbb{Z}$. Then there are at most $n$ elements in $\omega(L^{\times})\cap(0,1]$ but also at least $1$ element since there is $\pi\in F^{\times}$ that is prime with respect to $\nu$ giving $\omega(\pi)=\nu(\pi)=1$. Hence let $e':=\min\omega(L^{\times})\cap(0,1]$ and $a\in L^{\times}$ such that $\omega(a)=e'$. We show that $\omega(L^{\times})=e'\mathbb{Z}$. Let $b\in L^{\times}$ giving $\omega(b)=ke'+\varepsilon$ for some $0\leq\varepsilon<e'$ and $k\in\mathbb{Z}$. Then since $a^{k},a^{-k}\in L^{\times}$ we have $\omega(ba^{-k})=\omega(b)-k\omega(a)=ke'+\varepsilon-ke'=\varepsilon$ giving $\varepsilon=0$ by the definition of $e'$. Hence $\omega(L^{\times})\subseteq e'\mathbb{Z}$. On the other hand for $k\in\mathbb{Z}$ we have $\omega(a^{k})=k\omega(a)=ke'$ so $e'\mathbb{Z}\subseteq\omega(L^{\times})$ giving $\omega(L^{\times})=e'\mathbb{Z}$. Finally since $\omega(\pi)=1$ we have $1\in e'\mathbb{Z}$ and so there is $e\in\mathbb{N}$ such that $e'e=1$ giving $e\omega(L^{\times})=\mathbb{Z}$ which completes the proof.
\end{proof}
\begin{remark}
Let $F$ and $L$ be as in Theorem \ref{thr:CVFEE} with non-trivial discrete valuations. Let $\nu$ be the rank 1 valuation logarithm on $F$ and let $\omega$ be the extension of $\nu$ to $L$.
\begin{enumerate}
\item[(i)]
With group law addition, $\omega(F^{\times})$ and $\omega(L^{\times})$ are groups. It is immediate from Theorem \ref{thr:CVFEE} that $e=[\omega(L^{\times}):\omega(F^{\times})]$, the index of $\omega(F^{\times})$ in $\omega(L^{\times})$.
\item[(ii)]
If $e=1$ then $L$ is called an {\em unramified} extension of $F$. If $e=[L:F]$ then $L$ is called a {\em totally ramified} extension of $F$. Other classifications are also in use in the literature.
\item[(iii)]
The value of $e$ has implications for the degree of the extension $\overline{L}$ of the residue field $\overline{F}$. For $F$ and $L$ as specified in these remarks we have $[L:F]=e[\overline{L}:\overline{F}]$, see \cite[p107,p108]{McCarthy} for details. Hence in this case, with reference to Theorem \ref{thr:CVFHB}, if $F$ is locally compact then $L$ is locally compact.
\end{enumerate}
\label{rem:CVFEE}
\end{remark}
\subsection{Galois theory}
\label{subsec:CVFGT}
The following is the fundamental theorem of Galois theory, see \cite[p36]{McCarthy}.
\begin{theorem}
Let $F$ and $E$ be fields such that $E$ is a finite Galois extension of $F$, that is $E^{G}=F$ for $G:=\mbox{Gal}(^{E}/_{F})$. Then we have the following one-one correspondence
\begin{equation*}
\{G':G'\leqslant G\mbox{ is a subgroup}\}\leftrightarrow\{E':E'\mbox{ is a field with }F\subseteq E'\subseteq E\}
\end{equation*}
given by the inverse maps $G'\mapsto E^{G'}$ and $E'\mapsto\mbox{Gal}(^{E}/_{E'})$.
\label{thr:CVFFTG}
\end{theorem}
\begin{corollary}
Let $F$ and $L$ be fields such that $L$ is a finite extension of $F$ and let $G:=\mbox{Gal}(^{L}/_{F})$. Then $L$ is a finite Galois extension of $L^{G}$ and so for $L$ and $L^{G}$ Theorem \ref{thr:CVFFTG} is applicable.
\label{cor:CVFLFTG}
\end{corollary}
\begin{proof}
We show that $L$ is a Galois extension of $L^{G}$. For $g\in\mbox{Gal}(^{L}/_{F})$ we have $g(a)=a$ for all $a\in L^{G}$ and so $g\in\mbox{Gal}(^{L}/_{L^{G}})$. On the other hand for $g\in\mbox{Gal}(^{L}/_{L^{G}})$ we have $g(a)=a$ for all $a\in F$ since $F\subseteq L^{G}$ and so $g\in\mbox{Gal}(^{L}/_{F})$. Therefore $\mbox{Gal}(^{L}/_{F})=\mbox{Gal}(^{L}/_{L^{G}})$ and so setting $G':=\mbox{Gal}(^{L}/_{L^{G}})$ gives  $L^{G'}=L^{G}$ as required.
\end{proof}
The following group theory result must be known. However we will provide a proof in lieu of a reference. 
\begin{lemma}
Let $(G,+)$ be a group and $g\in\mbox{Aut}(G)$ be a group automorphism on $G$. If $a,b\in G$ are such that $\mbox{gcd}(\mbox{ord}(g,a),\mbox{ord}(g,b))=1$ then $\mbox{ord}(g,a+b)=\mbox{ord}(g,a)\mbox{ord}(g,b)$.
\label{lem:CVFGCD}
\end{lemma}
\begin{proof}
We assume the conditions of Lemma \ref{lem:CVFGCD} and note that the result is immediate if one or more of $\mbox{ord}(g,a)$ and $\mbox{ord}(g,b)$ is equal to 1. So assuming otherwise, let $p_{1}^{k_{1}}p_{2}^{k_{2}}\cdots p_{i}^{k_{i}}$ and $q_{1}^{l_{1}}q_{2}^{l_{2}}\cdots q_{j}^{l_{j}}$ be the prime decompositions of $\mbox{ord}(g,a)$ and $\mbox{ord}(g,b)$ respectively. For $n:=\mbox{ord}(g,a)\mbox{ord}(g,b)$ we have
\begin{equation*}
g^{(n)}(a+b)=g^{(n)}(a)+g^{(n)}(b)=a+b.
\end{equation*}
Therefore $\mbox{ord}(g,a+b)|n$. Suppose towards a contradiction that $\mbox{ord}(g,a+b)<n$. Then $\mbox{ord}(g,a+b)|\frac{n}{r}$ for some $r\in\{p_{1},p_{2},\cdots,p_{i},q_{1},q_{2},\cdots,q_{j}\}$. If $r=p_{m}$ for some $m\in\{1,2,\cdots,i\}$ then $a+b=g^{(\frac{n}{r})}(a+b)=g^{(\frac{n}{r})}(a)+g^{(\frac{n}{r})}(b)=g^{(\frac{n}{r})}(a)+b$ giving, by right cancellation of $b$, $g^{(\frac{n}{r})}(a)=a$. It then follows that 
\begin{equation*}
p_{1}^{k_{1}}p_{2}^{k_{2}}\cdots p_{m}^{k_{m}}\cdots p_{i}^{k_{i}}|p_{1}^{k_{1}}p_{2}^{k_{2}}\cdots p_{m}^{k_{m}-1}\cdots p_{i}^{k_{i}}q_{1}^{l_{1}}q_{2}^{l_{2}}\cdots q_{j}^{l_{j}}
\end{equation*}
giving $p_{m}|q_{1}^{l_{1}}q_{2}^{l_{2}}\cdots q_{j}^{l_{j}}$ which is a contradiction since $\mbox{gcd}(\mbox{ord}(g,a),\mbox{ord}(g,b))=1$. A similar contradiction occurs for $r=q_{m}$ with $m\in\{1,2,\cdots,j\}$. Hence $\mbox{ord}(g,a+b)=n$ as required.
\end{proof}
\begin{lemma}
Let $F$ be a field with finite extension $L$ and let $g\in\mbox{Gal}(^{L}/_{F})$. If $n\in\mathbb{N}$ is such that $n|\mbox{ord}(g)$ then $n\in\mbox{ord}(g,L)$.
\label{lem:CVFOSET}
\end{lemma}
\begin{proof}
Suppose towards a contradiction that there is $n\in\mathbb{N}$ such that $n|\mbox{ord}(g)$ but $n\notin\mbox{ord}(g,L)$. We can take $n$ to be the least such element and note that $n\not= 1$ since $1\in\mbox{ord}(g,L)$. Express $n$ as $n=p^{k}r$ where $p$ is a prime, $k,r\in\mathbb{N}$ and $p\nmid r$. We thus have the following two cases.\\ 
{\bf{Case:}} $r\not= 1$. In this case by the definition of $n$ we have $p^{k},r\in\mbox{ord}(g,L)$ and so there are $a,b\in L$ with $\mbox{ord}(g,a)=p^{k}$, $\mbox{ord}(g,b)=r$ and
\begin{equation*}
\mbox{gcd}(\mbox{ord}(g,a),\mbox{ord}(g,b))=1.
\end{equation*}
Then by Lemma \ref{lem:CVFGCD} we have $\mbox{ord}(g,a+b)=\mbox{ord}(g,a)\mbox{ord}(g,b)$ which contradicts our assumption that $n\notin\mbox{ord}(g,L)$.\\
{\bf{Case:}} $r=1$. In this case $n=p^{k}$ and note that $\mbox{ord}(g)=nm$ for some $m\in\mathbb{N}$. Hence we have the following subgroups of $G$:
\begin{enumerate}
\item[(i)]
$\langle g^{(n)}\rangle:=(\{\mbox{id},g^{(n)},g^{(2n)},\cdots,g^{((m-1)n)}\},\circ)<G$;
\item[(ii)]
$\langle g^{(\frac{n}{p})}\rangle:=(\{\mbox{id},g^{(\frac{n}{p})},g^{(2\frac{n}{p})},\cdots,g^{((mp-1)\frac{n}{p})}\},\circ)\leqslant G$.
\end{enumerate}
Therefore $\#\langle g^{(n)}\rangle=m$, $\#\langle g^{(\frac{n}{p})}\rangle=mp$ and $\langle g^{(n)}\rangle$ is a proper normal subgroup of $\langle g^{(\frac{n}{p})}\rangle$. Hence by Corollary \ref{cor:CVFLFTG} we have the following tower of fields
\begin{equation*}
L^{G}\subseteq L^{\langle g^{(n/p)}\rangle}\subsetneqq L^{\langle g^{(n)}\rangle}\subseteq L.
\end{equation*}
Now it is immediate that $L^{g^{(n)}}=L^{\langle g^{(n)}\rangle}$ and $L^{g^{(n/p)}}=L^{\langle g^{(n/p)}\rangle}$ and so there is some $a\in L^{g^{(n)}}\backslash L^{g^{(n/p)}}$ with $\mbox{ord}(g,a)|n$ but $\mbox{ord}(g,a)\nmid\frac{n}{p}$. Therefore $\mbox{ord}(g,a)=p^{k}=n$ which again contradicts our assumption that $n\notin\mbox{ord}(g,L)$. In particular the lemma holds.
\end{proof}
The following lemma is well known but we will provide a proof in lieu of a reference.
\begin{lemma}
\label{lem:CVFPOL}
Let $F$ be a field, let $L$ be an algebraic extension of $F$ and let $a\in L$. For the simple extension $F(a)$ of $F$ and $F[X]$ the ring of polynomials over $F$ we have $F(a)=F[a]$.
\end{lemma}
\begin{proof}
It is immediate that $F[a]\subseteq F(a)$. Now by Theorem \ref{thr:CVFIR} there is a unique monic irreducible polynomial $\mbox{Irr}_{F,a}(x)\in F[X]$ such that $\mbox{Irr}_{F,a}(a)=0$. Further for any element $\frac{p(a)}{q(a)}\in F(a)$, given by $p(x),q(x)\in F[X]$, we have $q(a)\not=0$. Hence $\mbox{Irr}_{F,a}(x)$ and $q(x)$ are relatively prime, that is we have $\mbox{gcd}(\mbox{Irr}_{F,a},q)=1$. Therefore by Bezout's identity there are $s(x),t(x)\in F[X]$ such that $s(x)q(x)+t(x)\mbox{Irr}_{F,a}(x)=1$ giving $q(x)=\frac{1-t(x)\mbox{Irr}_{F,a}(x)}{s(x)}$. Finally then we have $q(a)=\frac{1}{s(a)}$ giving $\frac{p(a)}{q(a)}=p(a)s(a)$ which is an element of $F[a]$ as required.
\end{proof}
	\chapter[Functions and algebras]{Functions and algebras}
\label{cha:FAA}
In this chapter we build upon some of the basic facts and analysis of complete valued fields surveyed in Chapter \ref{cha:CVF}. The first section establishes particular facts in functional analysis over complete valued fields that will be used in later chapters. However it is not the purpose of the first section to provide an extensive introduction to the subject. The second section provides background on Banach $F$-algebras, Banach algebras over a complete valued field $F$. Whilst some of the details are included purely as background others also support the discussion from Remark \ref{rem:CVFOS} of Chapter \ref{cha:CVF}.
\section{Functional analysis over complete valued fields}
\label{sec:FAAFA}
We begin with the following lemma.
\begin{lemma}
Let $F$ be a non-Archimedean field and let $(a_{n})$ be a sequence of elements of $F$.
\begin{enumerate}
\item[(i)]
If $\lim_{n\to\infty}a_{n}=a$ for some $a\in F^{\times}$ then there exists $N\in\mathbb{N}$ such that for all $n\geq N$ we have $|a_{n}|_{F}=|a|_{F}$. We will call this {\em convergence from the side}, opposed to from above or below.
\item[(ii)]
If $F$ is also complete then $\sum{a_{n}}$ converges if and only if $\lim_{n\to\infty}a_{n}=0$, in sharp contrast to the Archimedean case. Further if $\sum{a_{n}}$ does converge then
\begin{equation*}
\left|\sum{a_{n}}\right|_{F}\leq\max\{|a_{n}|_{F}:n\in\mathbb{N}\}.
\end{equation*}
\end{enumerate}
\label{lem:FAACS}
\end{lemma}
\begin{proof}
For (i), since $a\not=0$ we have $|a|_{F}>0$ and so there is some $N\in\mathbb{N}$ such that, for all $n\geq N$, $|a_{n}-a|_{F}<|a|_{F}$. Hence for all $n\geq N$ we have by Lemma \ref{lem:CVFEQ} that $|a_{n}|_{F}=|(a_{n}-a)+a|_{F}=|a|_{F}$.\\
For (ii), suppose $\lim_{n\to\infty}a_{n}=0$ and let $\varepsilon>0$. Then there is $N\in\mathbb{N}$ such that for all $n\geq N$ we have $|a_{n}|_{F}<\varepsilon$. Hence for $n_{1},n_{2}\in\mathbb{N}$ with $N<n_{1}<n_{2}$ we have
\begin{equation*}
\left|\sum_{i=1}^{n_{2}}a_{i}-\sum_{i=1}^{n_{1}}a_{i}\right|_{F}=\left|\sum_{i=n_{1}+1}^{n_{2}}a_{i}\right|_{F}\leq\max\{|a_{n_{1}+1}|_{F},\cdots,|a_{n_{2}}|_{F}\}<\varepsilon.
\end{equation*}
Hence the sequence of partial sums is a Cauchy sequence in $F$ and so converges. The converse is immediate. Further suppose $\sum{a_{n}}$ does converge. For $\sum{a_{n}}\not=0$ we have by (i) that there is $N\in\mathbb{N}$ such that for all $n\geq N$
\begin{equation*}
\left|\sum_{i=1}^{\infty}a_{i}\right|_{F}=\left|\sum_{i=1}^{n}a_{i}\right|_{F}\leq\max\{|a_{1}|_{F},\cdots,|a_{n}|_{F}\}\leq\max\{|a_{i}|_{F}:i\in\mathbb{N}\}.
\end{equation*}
On the other hand for $\sum{a_{n}}=0$ the result is immediate.
\end{proof}
The following theorem appears in \cite[p59]{Schikhof} without proof.
\begin{theorem}
Let $F$ be a complete non-Archimedean field and let $a_{0},a_{1},a_{2},\cdots$ be a sequence of elements of $F$. Define the radius of convergence by
\begin{equation*}
\rho:=\frac{1}{\limsup_{n\to\infty}\sqrt[n]{|a_{n}|_{F}}}\quad\mbox{where by convention }0^{-1}=\infty\mbox{ and }\infty^{-1}=0.
\end{equation*}
Then the power series $\sum{a_{n}x^{n}}$, $x\in F$, converges if $|x|_{F}<\rho$ and diverges if $|x|_{F}>\rho$. Furthermore for each $t\in(0,\infty)$, $t<\rho$ the convergence is uniform on $\bar{B}_{t}(0):=\{a\in F:|a|_{F}\leq t\}$.
\label{thr:FAACP}
\end{theorem}
\begin{proof}
Note that the following equalities hold, except for when $|x|_{F}=\rho=0$,
\begin{equation}
\label{equ:FAACP}
\limsup_{n\to\infty}\sqrt[n]{|a_{n}x^{n}|_{F}}=\limsup_{n\to\infty}\sqrt[n]{|a_{n}|_{F}|x|_{F}^{n}}=\frac{|x|_{F}}{\rho}.
\end{equation}
Suppose $\sum{a_{n}x^{n}}$ is divergent. Then by part (ii) of Lemma \ref{lem:FAACS}, $\lim_{n\to\infty}a_{n}x^{n}$ is not $0$. Therefore there is some $\varepsilon\in(0,1]$ such that for each $m\in\mathbb{N}$ there is $n>m$ with $|a_{n}x^{n}|_{F}\geq\varepsilon$, in particular $\sqrt[n]{|a_{n}x^{n}|_{F}}\geq\sqrt[n]{\varepsilon}\geq\sqrt[m]{\varepsilon}$. Hence since $\lim_{m\to\infty}\sqrt[m]{\varepsilon}=1$ we have $\limsup_{n\to\infty}\sqrt[n]{|a_{n}x^{n}|_{F}}\geq1$. Therefore in this case $|x|_{F}\geq\rho$ by (\ref{equ:FAACP}). In particular for cases where $|x|_{F}<\rho$ the series $\sum{a_{n}x^{n}}$ converges.\\
On the other hand suppose $\sum{a_{n}x^{n}}$ converges. Then since $\lim_{n\to\infty}|a_{n}x^{n}|_{F}=0$ we have $\limsup_{n\to\infty}\sqrt[n]{|a_{n}x^{n}|_{F}}\leq\limsup_{n\to\infty}\sqrt[n]{\frac{1}{2}}=1$. Therefore in this case $|x|_{F}\leq\rho$ by (\ref{equ:FAACP}). In particular for cases where $|x|_{F}>\rho$ the series $\sum{a_{n}x^{n}}$ diverges.\\
Now suppose there is $t\in(0,\infty)$ with $t<\rho$ and let $\varepsilon>0$. If the valuation on $F$ is dense then $|F^{\times}|_{F}$ is dense in the positive reals and so there is some $x_{0}\in F^{\times}$ with $t<|x_{0}|_{F}<\rho$. Alternatively, if the valuation on $F$ is discrete, there is $x_{0}\in\bar{B}_{t}(0)$ with $|x_{0}|_{F}=\max\{|a|_{F}:a\in\bar{B}_{t}(0)\}$. In either case, since $|x_{0}|_{F}<\rho$, $\sum{a_{n}x_{0}^{n}}$ converges and so $\lim_{n\to\infty}a_{n}x_{0}^{n}=0$. Hence there is some $N\in\mathbb{N}$ such that for all $n>N$ we have $|a_{n}x_{0}^{n}|_{F}<\varepsilon$. Now by the last part of Lemma \ref{lem:FAACS} we have for all $m>N$ and $x\in \bar{B}_{t}(0)$ that
\begin{equation*}
|\sum_{n=1}^{\infty}a_{n}x^{n}-\sum_{n=1}^{m}a_{n}x^{n}|_{F}=|\sum_{n=m+1}^{\infty}a_{n}x^{n}|_{F}\leq\max\{|a_{n}x^{n}|_{F}:n\geq m+1\}.
\end{equation*}
But $|a_{n}x^{n}|_{F}=|a_{n}|_{F}|x|_{F}^{n}\leq |a_{n}|_{F}|x_{0}|_{F}^{n}=|a_{n}x_{0}^{n}|_{F}$ and so $\max\{|a_{n}x^{n}|_{F}:n\geq m+1\}<\varepsilon$ and the convergence is uniform on $\bar{B}_{t}(0)$.
\end{proof}
\begin{remark}
With reference to Theorem \ref{thr:FAACP}.
\begin{enumerate}
\item[(i)]
We note that the radius of convergence as defined in Theorem \ref{thr:FAACP} is the same as that used in the Archimedean setting when replacing $F$ with the complex numbers. However, unlike in the complex setting, if the valuation on $F$ is discrete then a power series $\sum{a_{n}x^{n}}$ may not have a unique choice for the definition of its radius of convergence since $|F^{\times}|_{F}$ is discrete in this case.
\item[(ii)]
We need to be careful when considering convergence of power series. Let $|\cdot|_{\infty}$ denote the absolute valuation on $\mathbb{R}$ and let $|\cdot|_{0}$ denote the trivial valuation on $\mathbb{R}$. All power series are convergent on $B_{1}(0):=\{a\in\mathbb{R}:|a|_{0}<1\}=\{0\}$ with respect to $|\cdot|_{0}$. Whereas the only power series that are convergent at a point $a\in\mathbb{R}^{\times}$ with respect to $|\cdot|_{0}$ are polynomials. On the other hand $\exp(x):=\sum_{n=1}^{\infty}\frac{x^{n}}{n!}$ converges everywhere on $\mathbb{R}$ with respect to $|\cdot|_{\infty}$. The function $\exp(x)$ defined with respect to $|\cdot|_{\infty}$ is a continuous function on all of $\mathbb{R}$ with respect to $|\cdot|_{0}$ but does not have a power series representation on $\mathbb{R}$ with respect to $|\cdot|_{0}$. Similarly $\sum_{n=1}^{\infty}\frac{x^{n}}{n!}$ does not converge everywhere on the $p$-adic  numbers $\mathbb{Q}_{p}$ with respect to $|\cdot|_{p}$, see \cite[p70]{Schikhof} for details in this case.
\item[(iii)]
Under the conditions of Theorem \ref{thr:FAACP}, suppose that the ball $B_{\rho}(0)$ is without isolated points where $\rho$ is the radius of convergence of $f(x):=\sum{a_{n}x^{n}}$. Then, with differentiation defined as in the Archimedean setting, the derivative of $f$ exists on $B_{\rho}(0)$ and it is $f'(x)=\sum{na_{n}x^{n-1}}$. We will not consider this in depth but note, for $x\in B_{\rho}(0)$, the series $\sum{a_{n}x^{n}}$ converges giving $\lim_{n\to\infty}c_{n}=0$ for $c_{n}:=a_{n}x^{n}$ by Lemma \ref{lem:FAACS}. Hence for all $n\in\mathbb{N}$, since $|n|_{F}=|1_{1}+\cdots+1_{n}|_{F}\leq\max\{|1_{1}|_{F},\cdots,|1_{n}|_{F}\}=1$, we have for $x\not=0$ that
\begin{equation*}
|na_{n}x^{n-1}|_{F}=|n|_{F}|x^{-1}|_{F}|a_{n}x^{n}|_{F}\leq|x^{-1}|_{F}|c_{n}|_{F}.
\end{equation*}
Therefore the series $\sum{na_{n}x^{n-1}}$ also converges on $B_{\rho}(0)$ by Lemma \ref{lem:FAACS}.
\end{enumerate}
\label{rem:FAACP}
\end{remark}
\subsection{Analytic functions}
\label{subsec:FAAAF}
Let $F$ be a complete valued field. In this subsection we consider $F$ valued functions that are analytic on the interior of some subset of $F$ that is without isolated points. In particular the situation concerning such analytic functions is some what different in the non-Archimedean setting to that in the Archimedean one, even though the standard results of differentiation such as the chain rule and Leibniz rule are the same, see \cite[p59]{Schikhof}. Recall that if a complex valued function $f$ is analytic on an open disc $D_{r}(a)\subseteq\mathbb{C}$ then $f$ can be represented by the convergent power series
\begin{equation*}
f(z)=\sum_{n=0}^{\infty}\frac{f^{(n)}(a)}{n!}(z-a)^{n}\quad\mbox{for }z\in D_{r}(a),
\end{equation*}
known as the {\em Taylor expansion} of $f$ about $a$, where $f^{(n)}(a)$ is the $n$th derivative of $f$ at $a$. Moreover if $b\in D_{r}(a)$ then $f$ can also be expanded about $b$. However this expansion need not be convergent on all of $D_{r}(a)$ merely on the largest open disc centered at $b$ contained in $D_{r}(a)$ since a lack of differentiability of $f$ at points on, or outside, the boundary of $D_{r}(a)$ will restrict the radius of convergence of such an expansion, see \cite[p449,p450]{Apostol}.\\
Now for $F$ a complete non-Archimedean field the same scenario in this case is such that if $f$ is analytic on a ball $B_{r}(a)\subseteq F$ and can be represented by a Taylor expansion about $a$ on all of $B_{r}(a)$ then $f$ can be represented by a Taylor expansion about any other point $b\in B_{r}(a)$ and this expansion will also be valid on all of $B_{r}(a)$, see \cite[p68]{Schikhof}. This is closely related to the fact that every point of $B_{r}(a)$ is at its center, see the proof of Lemma \ref{lem:CVFTD}. However in general a function $f$ analytic on $B_{r}(a)\subseteq F$ need not have a Taylor expansion about $a$ that is valid on all of $B_{r}(a)$. This is because $B_{r}(a)$ can be decomposed as a disjoint union of clopen balls, see Remark \ref{rem:CVFTD}, upon each of which $f$ can independently be defined. This leads to the following definitions.
\begin{definition}
Let $F$ be a complete valued field with non-trivial valuation. 
\begin{enumerate}
\item[(i)]
We will call a subset $X\subseteq F$ {\em strongly convex} if $X$ is either $F$, the empty set $\emptyset$, a ball or a singleton set.
\item[(ii)]
Let $X$ be an open strongly convex subset of $F$ and let $f:X\rightarrow F$ be a continuous $F$-valued function on $X$. If $f$ can be represented by a single Taylor expansion that is valid on all of $X$ then we say that $f$ is {\em globally analytic} on $X$.
\item[(iii)]
Let $X$ be an open subset of $F$ and let $f:X\rightarrow F$ be a continuous $F$-valued function on $X$. If for each $a\in X$ there is an open strongly convex neighborhood $V\subseteq X$ of $a$ such that $f|_{V}$ is globally analytic on $V$ then we say that $f$ is {\em locally analytic} on $X$.
\item[(iv)]
Let $X$ and $f$ be as in (iii). As usual, if the derivative
\begin{equation*}
f'(a):=\lim_{x\to a}\frac{f(x)-f(a)}{x-a}
\end{equation*}
exists at every $a\in X$ then we say that $f$ is {\em analytic} on $X$.
\item[(v)]
In the case where $X=F$ we similarly define {\em globally entire}, {\em locally entire} and {\em entire} functions on $X$.
\end{enumerate}
\label{def:FAAGL}
\end{definition}
\begin{remark}
Note that the condition in Definition \ref{def:FAAGL} that $F$ has a non-trivial valuation is there because it does not make sense to talk about analytic functions defined on a space without accumulation points. We also note that (ii), (iii) and (iv) of Definition \ref{def:FAAGL} are equivalent in the complex setting for $X$ an open strongly convex subset of $\mathbb{C}$, see \cite[p450]{Apostol}.
\label{rem:FAAGL}
\end{remark}
Now let $F$ be a complete non-Archimedean field and let $(f_{n})$ be a sequence of $F$-valued functions analytic on $\bar{B}_{1}(0)$ and converging uniformly on $\bar{B}_{1}(0)$ to a function $f$. We ask whether $f$ will also be analytic on $\bar{B}_{1}(0)$ in this case? It is very well known that the answer to the analog of this question involving the complex numbers is yes although in this case the functions are required to be continuous on $\bar{B}_{1}(0)$ and analytic only on the interior of $\bar{B}_{1}(0)$ since $\bar{B}_{1}(0)$ will not be clopen. In the case involving the real numbers the answer to the question is of course no since for example a function with a chevron shaped graph in $\mathbb{R}^{2}$ can be uniformly approximated by differentiable functions. In the non-Archimedean setting the following theorem provides insight for when $F$ is not locally compact and also gives a maximum principle result, see \cite[p122]{Schikhof} for proof.
\begin{theorem}
Let $F$ be a complete non-Archimedean field that is not locally compact and let $r\in|F^{\times}|_{F}$.
\begin{enumerate}
\item[(i)]
If $f_{1},f_{2},\cdots$ are globally analytic functions on $\bar{B}_{r}(0)$ and if $f:=\lim_{n\to\infty}f_{n}$ uniformly on $\bar{B}_{r}(0)$ then $f$ is also globally analytic on $\bar{B}_{r}(0)$.
\item[(ii)]
Let $f$ be a globally analytic function on $\bar{B}_{r}(0)$ with power series $f(x)=\sum_{n=0}^{\infty}a_{n}x^{n}$.\\
If the valuation $|\cdot|_{F}$ is dense then
\begin{equation*}
\sup\{|f(x)|_{F}:|x|_{F}\leq r\}=\sup\{|f(x)|_{F}:|x|_{F}<r\}=\max\{|a_{n}|_{F}r^{n}:n\geq0\}<\infty.
\end{equation*}
If the residue field $\overline{F}$ is infinite then
\begin{equation*}
\max\{|f(x)|_{F}:|x|_{F}\leq r\}=\max\{|f(x)|_{F}:|x|_{F}=r\}=\max\{|a_{n}|_{F}r^{n}:n\geq0\}<\infty.
\end{equation*}
\end{enumerate}
\label{thr:FAAMP}
\end{theorem}
\begin{remark}
In Theorem \ref{thr:FAAMP} $\bar{B}_{r}(0)$ is not compact since $F$ is not locally compact. In fact every ball of positive radius is not compact in this case and this follows from Theorem \ref{thr:CVFHB} noting that translations and non-zero scalings in $F$ are homeomorphisms on $F$. Now since we are progressing towards a study of uniform algebras and their generalisation over complete valued fields we note that in order to use the uniform norm, see Remark \ref{rem:UASA}, on such algebras of continuous functions we need the functions to be bounded. Hence to avoid imposing boundedness directly it is convenient to work on compact spaces.
\label{rem:FAAMP}
\end{remark}
For $\bar{B}_{1}(0)$ compact, i.e. in the $F$ locally compact case, I provide the following example to show that in this case the uniform limit of locally analytic, and hence analytic, functions on $\bar{B}_{1}(0)$ need not be analytic.
\begin{example}
Let $F$ be a locally compact, complete non-Archimedean field with non-trivial valuation. Then we have the following sequence of functions on $\bar{B}_{1}(0)\subseteq F$,
\begin{equation*}
f_{n}(x):=\left\{ \begin{array}{l@{\quad\mbox{if}\quad}l}\pi^{\nu(x)} & \nu(x)<n \\ 0 & \nu(x)\geq n \end{array} \right.\quad\mbox{for }x\in\bar{B}_{1}(0)
\end{equation*}
where $\pi$ is a prime element of $F$ and $\nu$ is the rank 1 valuation logarithm. For each $n\in\mathbb{N}$, $f_{n}$ is a locally constant function since convergence in $F$ is from the side, see Lemma \ref{lem:FAACS}, and so $f_{n}$ is locally analytic on $\bar{B}_{1}(0)$. Moreover the sequence $(f_{n})$ converges uniformly on $\bar{B}_{1}(0)$ to the continuous function
\begin{equation*}
f(x):=\left\{ \begin{array}{l@{\quad\mbox{if}\quad}l}\pi^{\nu(x)} & x\not=0 \\ 0 & x=0 \end{array} \right.\quad\mbox{for }x\in\bar{B}_{1}(0)
\end{equation*}
with $\lim_{x\to0}f(x)=0$ since $|f(x)|_{F}=|x|_{F}$ for all $x\in\bar{B}_{1}(0)$. We now show that $f$ is not differentiable at zero. let $a_{1},a_{2},\cdots$ and $b_{1},b_{2},\cdots$ be sequences in $F$ given by $a_{n}:=\pi^{n}$ and $b_{n}:=-\pi^{n}$. Both of these sequences tend to zero as $n$ tends to $\infty$. But then
\begin{equation*}
\frac{f(a_{n})-f(0)}{a_{n}}=(\pi^{n}-0)\pi^{-n}=1\quad\mbox{and }\frac{f(b_{n})-f(0)}{b_{n}}=(\pi^{n}-0)(-\pi^{-n})=-1
\end{equation*}
for all $n\in\mathbb{N}$ so that the limit $\lim_{x\to0}\frac{f(x)-f(0)}{x}$ does not exist as required. Alternatively we can obtain a similar example by redefining $f$ as
\begin{equation*}
f(x):=\left\{ \begin{array}{l@{\quad\mbox{if}\quad}l}\pi^{\frac{1}{2}\nu(x)} & \nu(x)\mbox{ is even} \\ \pi^{\frac{1}{2}(\nu(x)-1)} & \nu(x)\mbox{ is odd} \\ 0 & x=0 \end{array} \right.\quad\mbox{for }x\in\bar{B}_{1}(0).
\end{equation*}
In this case $\lim_{x\to0}\frac{f(x)-f(0)}{x}$ blows up with respect to $|\cdot|_{F}$ as demonstrated by the sequence $c_{1},c_{2},\cdots$ with $c_{n}:=\pi^{2n}$. 
\label{exa:FAANA}
\end{example}
Later when we look at non-complex analogs of uniform algebras we will see, from Kaplansky's non-Archimedean generalisation of the Stone-Weierstrass theorem, that the continuous functions in Example \ref{exa:FAANA} can be uniformly approximated by polynomials on $\bar{B}_{1}(0)$ given that $\bar{B}_{1}(0)$ is compact in this case. Hence Example \ref{exa:FAANA} also shows that, for $F$ locally compact, the uniform limit of globally analytic functions on $\bar{B}_{1}(0)$ need not be analytic in contrast to Theorem \ref{thr:FAAMP}.\\
In anticipation of topics in the next section we now consider Liouville's theorem. It is immediate that the standard Liouville theorem never holds in the non-Archimedean setting since for a complete non-Archimedean field $F$ with non-trivial valuation the indicator function $\chi_{_{B}}$ for $B:=\bar{B}_{1}(0)$ is a non-constant bounded locally analytic function from $F$ to $F$ noting that $\bar{B}_{1}(0)$ is a clopen subset of $F$. However the following is called the ultrametric Liouville theorem.
\begin{theorem}
Let $F$ be a complete non-Archimedean field with non-trivial valuation. Then every bounded globally analytic function from $F$ to $F$ is constant if and only if $F$ is not locally compact.
\label{thr:FAAULT}
\end{theorem}
\begin{proof}
See \cite[p124,p125]{Schikhof} for a full proof of Theorem \ref{thr:FAAULT}. However proof in the if direction is as follows. Let $f(x)=\sum_{n=0}^{\infty}a_{n}x^{n}$, for $x\in F$, be as in Theorem \ref{thr:FAAULT}. Since $f$ is bounded there is $M<\infty$ such that $|f(x)|_{F}\leq M$ for all $x\in F$. Let $m\in\mathbb{N}$. Since $F$ is not locally compact we can apply (ii) of Theorem \ref{thr:FAAMP} so that for $r\in|F^{\times}|_{F}$ we have
\begin{equation*}
|a_{m}|_{F}r^{m}\leq\max\{|a_{n}|_{F}r^{n}:n\geq0\}=\sup\{|f(x)|_{F}:|x|_{F}\leq r\}\leq M.
\end{equation*}
This holds for every $r\in|F^{\times}|_{F}$ and so $a_{m}=0$ leaving $f=a_{0}$ for all $x\in F$.
\end{proof}
For a field $F$ with the trivial valuation we note that $F$ is locally compact and that there are bounded non-constant polynomials from $F$ to $F$, where we take polynomials to be the analog of globally analytic functions in this case.
\section{Banach {\it F}-algebras}
\label{sec:FAABFA}
We begin this section with the following definitions.
\begin{definition}
Let $F$ be a complete valued field.
\begin{enumerate}
\item[(i)]
A {\em general Banach ring} is a normed ring $R$ that is complete with respect to its norm which is required to be sub-multiplicative, i.e.
\begin{equation*}
\|ab\|_{R}\leq\|a\|_{R}\|b\|_{R}\quad\mbox{for all }a,b\in R.
\end{equation*}
We do not assume that $R$ has a multiplicative identity or that its multiplication is commutative, we merely assume it is associative.
\item[(ii)]
A {\em Banach ring} is a general Banach ring $R$ that has a left/right multiplicative identity satisfying $\|1_{R}\|_{R}=1=\|-1_{R}\|_{R}$.
\item[(iii)]
A {\em Banach $F$-algebra} is a general Banach ring $A$ that is also a normed vector space over $F$, with respect to the ring's addition operation and norm, and such that the ring's multiplication operation is a bilinear map over $F$, i.e. respectively
\begin{equation*}
\|\alpha a\|_{A}=|\alpha|_{F}\|a\|_{A}\mbox{ and }(\alpha a)b=a(\alpha b)=\alpha(ab)\quad\mbox{for all }a,b\in A\mbox{ and }\alpha\in F.
\end{equation*}
\item[(iv)]
A {\em unital Banach $F$-algebra} is a Banach $F$-algebra that is also a Banach ring opposed to being merely a general Banach ring.
\item[(v)]
By {\em commutative general Banach ring} and {\em commutative Banach $F$-algebra} etc. we mean that the multiplication is commutative in these cases. By {\em $F$-algebra} we mean the structure of a Banach $F$-algebra but without the requirement of a norm.
\end{enumerate}
\label{def:FAABA}
\end{definition}
\begin{remark}
In Definition \ref{def:FAABA} we always require a multiplicative identity to be different to the additive identity. As standard we will usually dispense with the subscript when denoting elements of the structures defined in Definition \ref{def:FAABA} and in the Archimedean setting we will call a Banach $\mathbb{C}$-algebra a complex Banach algebra and a Banach $\mathbb{R}$-algebra a real Banach algebra.
\label{rem:FAABA}
\end{remark}
\subsection{Spectrum of an element}
\label{subsec:FAASE}
The following discussion concerns the spectrum of an element.
\begin{definition}
Let $F$ be a complete valued field and let $A$ be a unital Banach $F$-algebra. Then for $a\in A$ we call the set
\begin{equation*}
\mbox{Sp}(a):=\{\lambda\in F:\lambda-a\mbox{ is not invertible in }A\}
\end{equation*}
the spectrum of $a$.
\label{def:FAASP}
\end{definition}
\begin{theorem}
Every element of every unital complex Banach algebra has non-empty spectrum.
\label{thr:FAANES}
\end{theorem}
Theorem \ref{thr:FAANES} is very well known. A proof can be found in \cite[p11]{Stout} and relies on Liouville's theorem and the Hahn-Banach theorem in the complex setting. We will confirm that this result is unique among unital Banach $F$-algebras and I will give details of where the proof from the complex setting fails for other complete valued fields. Let us first recall the Gelfand-Mazur theorem which demonstrates the importance of Theorem \ref{thr:FAANES} in the complex setting and supports Remark \ref{rem:CVFOS} of Chapter \ref{cha:CVF}.
\begin{theorem}
A unital complex Banach algebra that is also a division ring is isometrically isomorphic to the complex numbers.
\label{thr:FAAGM}
\end{theorem}
\begin{proof}
Let $A$ be a unital complex Banach algebra that is also a division ring and let $a\in A$. Since in this case $\mbox{Sp}(a)$ is non-empty, there is some $\lambda\in\mbox{Sp}(a)$. Hence because $A$ is a division ring  $\lambda-a=0$ giving $a=\lambda$. More accurately we have $a=\lambda 1_{A}$ but because $A$ is unital we have $\|a\|_{A}=\|\lambda 1_{A}\|_{A}=|\lambda|_{\infty}\|1_{A}\|_{A}=|\lambda|_{\infty}$ and so the map from $A$ onto $\mathbb{C}$ given by $\lambda 1_{A}\mapsto\lambda$ is an isometric isomorphism.
\end{proof}
\begin{remark}
In the Archimedean setting it follows immediately from Theorem \ref{thr:FAAGM} that any complete valued field containing the complex numbers as a valued subfield will coincide with the complex numbers. Note that the proof of Theorem \ref{thr:FAAGM} is very well known.
\label{rem:FAAGM}
\end{remark}
In contrast to Theorem \ref{thr:FAANES} we have the following lemma. The result is certainly known but we give full details in lieu of a reference.
\begin{lemma}
Let $F$ be a complete valued field other than the complex numbers. Then there exists a unital Banach $F$-algebra $A$ such that $\mbox{Sp}(a)=\emptyset$ for some $a\in A$.
\label{lem:FAAES}
\end{lemma}
\begin{proof}
Let $F$ be a complete valued field other that the complex numbers. By Corollary \ref{cor:CVFEE} in the non-Archimedean setting, and since $\mathbb{R}$ is the only complete valued field other than $\mathbb{C}$ in the Archimedean setting, we can always find a complete valued field $L$ that is a proper extension of $F$. Let $a\in L\backslash F$ and note that $L$ is a unital Banach $F$-algebra. Then for every $\lambda\in F$ we have $\lambda-a\not=0$ and so $\lambda-a$ is invertible in $L$ since $L$ is a field. Hence $\mbox{Sp}(a)=\emptyset$.
\end{proof}
Whilst not considering every case, we now consider where the proof of Theorem \ref{thr:FAANES} fails when applying it to unital Banach $F$-algebras with $F\not=\mathbb{C}$. For $F=\mathbb{R}$ the Hahn-Banach theorem holds but Liouville's theorem does not with the trigonometric $\sin$ function restricted to $\mathbb{R}$ as an example of a non-constant, bounded, analytic function from $\mathbb{R}$ to $\mathbb{R}$. In the non-Archimedean setting we do have the ultrametric Liouville theorem, Theorem \ref{thr:FAAULT} for $F$ not locally compact, and there is also an ultrametric Hahn-Banach theorem for spherically complete fields, as follows.
\begin{definition}
An ultrametric space, see Definition \ref{def:CVFV}, is {\em spherically complete} if each nested sequence of balls has a non-empty intersection.
\label{def:FAASC}
\end{definition}
\begin{theorem}
Let $F$ be a spherically complete non-Archimedean field and let $V$ be an $F$-vector space, $s$ a seminorm on $V$ and $V_{0}\subseteq V$ a vector subspace. Then for every linear functional $\ell_{0}:V_{0}\rightarrow F$ such that $|\ell_{0}(v)|_{F}\leq s(v)$ for all $v\in V_{0}$ there is a linear functional $\ell:V\rightarrow F$ such that $\ell|_{V_{0}}=\ell_{0}$ and $|\ell(v)|_{F}\leq s(v)$ for all $v\in V$.
\label{thr:FAAHB}
\end{theorem}
\begin{remark}
We note that Theorem \ref{thr:FAAHB} is exactly the same as the Hahn-Banach theorem from the Archimedean setting, see \cite[p472]{Stout}, except with $\mathbb{R}$ and $\mathbb{C}$ replaced by any spherically complete non-Archimedean field. A proof can be found in both \cite[p51]{Schneider} and \cite[p288]{Schikhof} the latter of which further states that Theorem \ref{thr:FAAHB} becomes a falsity if $F$ is replaced by a non-spherically complete field. It is immediate that spherically complete ultrametric spaces are complete.
\label{rem:FAAHB}
\end{remark}
A proof of the following lemma can be found in \cite[p6]{Schneider}.
\begin{lemma}
All complete non-Archimedean fields with a discrete valuation are spherically complete. In particular if $F$ is a complete non-Archimedean fields that is locally compact then $F$ is spherically complete.
\label{lem:FAASC}
\end{lemma}
From the above details we see that for both Theorem \ref{thr:FAAULT} and Theorem \ref{thr:FAAHB} to be applicable we need a non-locally compact, spherically complete, non-Archimedean field. This restricts the possibilities since for example, for any prime $p$, a finite extension of $\mathbb{Q}_{p}$ is locally compact and $\mathbb{C}_{p}$ whilst not locally compact is also not spherically complete, see \cite[p5]{Schneider}. However, with reference to (ii) of Example \ref{exa:CVFA}, the complete non-Archimedean field $\mathbb{C}\{\{T\}\}$ is not locally compact since having an infinite residue field and it is spherically complete since its valuation is discrete. Moreover the totally ramified, see Remark \ref{rem:CVFEE}, simple extension $\mathbb{C}\{\{T\}\}(\sqrt{T})$ is a unital Banach $\mathbb{C}\{\{T\}\}$-algebra with complete valuation given by Theorem \ref{thr:CVFEE}. But by the proof of Lemma \ref{lem:FAAES} we have $\mbox{Sp}(\sqrt{T})=\emptyset$. So let's briefly review how the proof of Theorem \ref{thr:FAANES} from \cite[p11]{Stout} works and then consider where it fails for $\mathbb{C}\{\{T\}\}(\sqrt{T})$.\\
Let $A$ be a unital complex Banach algebra and let $a\in A$. Suppose towards a contradiction that $\mbox{Sp}(a)=\emptyset$. Then $\lambda-a$ is invertible for all $\lambda\in\mathbb{C}$. In particular $a^{-1}$ exists in $A$ and the map $\ell_{0}:\mathbb{C}a^{-1}\rightarrow\mathbb{C}$, given by $\ell_{0}(\lambda a^{-1}):=\lambda\alpha$ for a fixed $\alpha\in\mathbb{C}$ with $0<|\alpha|_{\infty}\leq\|a^{-1}\|_{A}$, is a continuous linear functional on the subspace $\mathbb{C}a^{-1}$ of $A$ to which the Hahn-Banach theorem can be applied directly. Hence there exists a continuous linear functional $\ell:A\rightarrow\mathbb{C}$ such that $\ell(-a^{-1})=-\alpha\not=0$. On the other hand for any continuous linear functional $\varphi:A\rightarrow\mathbb{C}$ we can define a function $f_{\varphi}:\mathbb{C}\rightarrow\mathbb{C}$ by
\begin{equation*}
f_{\varphi}(\lambda):=\varphi((\lambda-a)^{-1}).
\end{equation*}
The proof then shows that $f_{\varphi}$ is differentiable at every point of $\mathbb{C}$ and is therefore an entire function. Moreover $\lim_{\lambda\to\infty}f_{\varphi}(\lambda)=0$ since
\begin{equation*}
|f_{\varphi}(\lambda)|_{\infty}=\left|\frac{1}{\lambda}\varphi((1-\lambda^{-1}a)^{-1})\right|_{\infty}\leq\frac{1}{|\lambda|_{\infty}}\|\varphi\|_{\mbox{\footnotesize op}}\|(1-\lambda^{-1}a)^{-1}\|_{A},
\end{equation*}
where $\|\cdot\|_{\mbox{\footnotesize op}}$ is the standard operator norm. Hence, by Liouville theorem in the complex setting, $f_{\varphi}$ is the zero function. But we have $f_{\ell}(0)=-\alpha\not=0$, a contradiction, and so $\mbox{Sp}(a)\not=\emptyset$ as required. Note however that the function $f_{\varphi}$ is defined on $\mathbb{C}\backslash\mbox{Sp}(a)$.\\
Now for $\mathbb{C}\{\{T\}\}(\sqrt{T})$ the coordinate projection $P:\mathbb{C}\{\{T\}\}(\sqrt{T})\rightarrow\mathbb{C}\{\{T\}\}$ given by $P(\alpha+\beta\sqrt{T}):=\alpha$, where $\alpha,\beta\in\mathbb{C}\{\{T\}\}$, is a continuous linear functional analogous to an evaluation functional noting that convergence in $\mathbb{C}\{\{T\}\}(\sqrt{T})$ is coordinate-wise over $\mathbb{C}\{\{T\}\}$ by Remark \ref{rem:CVFEN}. Hence we can define a function $f_{P}:\mathbb{C}\{\{T\}\}\rightarrow\mathbb{C}\{\{T\}\}$ given by
\begin{equation*}
f_{P}(\lambda):=P((\lambda-\sqrt{T})^{-1})=P((\lambda+\sqrt{T})(\lambda^{2}-T)^{-1})=\lambda(\lambda^{2}-T)^{-1}.
\end{equation*}
The function $f_{P}$ is defined on all of $\mathbb{C}\{\{T\}\}$ since the roots of $\lambda^{2}-T$ are $\sqrt{T}$ and $-\sqrt{T}$. Furthermore $f_{P}$ is not constant and so it is the relative weakness of the ultrametric Liouville theorem in the non-Archimedean setting that allows the argument used in the proof of Theorem \ref{thr:FAANES} to fails in this case. Indeed we will now show that $f_{P}$ is not globally analytic on all of $\mathbb{C}\{\{T\}\}$. The first derivative of $f_{P}$ is
\begin{equation*}
f_{P}^{(1)}(\lambda)=(\lambda^{2}-T)^{-1}-2\lambda^{2}(\lambda^{2}-T)^{-2}
\end{equation*}
and so $f_{P}(0)=0$ and $f_{P}^{(1)}(0)=-\frac{1}{T}$. Continuing in this way we obtain the Taylor expansion of $f_{P}$ about zero as
\begin{equation*}
f_{P}(\lambda)=\sum_{n=0}^{\infty}\alpha_{n}\lambda^{n}=-\left(\frac{\lambda}{T}+\frac{\lambda^{3}}{T^{2}}+\frac{\lambda^{5}}{T^{3}}+\frac{\lambda^{7}}{T^{4}}+\cdots\right),\quad\mbox{for }|\lambda|_{T}<\rho,
\end{equation*}
where $\alpha_{n}:=\frac{f_{P}^{(n)}(0)}{n!}=-\frac{1-(-1)^{n}}{2}T^{-\frac{1}{2}(1+n)}\in\mathbb{C}\{\{T\}\}$ and
\begin{equation*}
\rho=\frac{1}{\limsup_{n\to\infty}\sqrt[n]{|\alpha_{n}|_{T}}}
\end{equation*}
is the radius of convergence of the Taylor series expansion. Hence we show that $\rho$ is finite. Using the rank 1 valuation logarithm, for $\sum_{n\in\mathbb{Z}}a_{n}T^{n}\in\mathbb{C}\{\{T\}\}^{\times}$ we have $|\sum_{n\in\mathbb{Z}}a_{n}T^{n}|_{T}=r^{-\min\{n:a_{n}\not=0\}}$ for some fixed $r>1$. Hence, noting that $\alpha_{2n}=0$ and $\alpha_{2n-1}=-T^{-n}$ for $n\in\mathbb{N}$, we have
\begin{equation*}
\limsup_{n\to\infty}\sqrt[n]{|\alpha_{n}|_{T}}=\lim_{n\to\infty}\sqrt[2n-1]{|\alpha_{2n-1}|_{T}}=\lim_{n\to\infty}\sqrt[2n-1]{r^{n}}=\lim_{n\to\infty}r^{\frac{n}{2n-1}}=r^{\frac{1}{2}}.
\end{equation*}
Hence $\rho=\frac{1}{\sqrt{r}}<1$ since $r>1$. In particular $f_{P}$ is only locally analytic on $\mathbb{C}\{\{T\}\}$ and not globally analytic, consistent with the ultrametric Liouville theorem not being applicable to $f_{P}$ as required.
\begin{definition}
Let $F$ be a complete valued field and let $A$ be a unital Banach $F$-algebra. Define $\mathcal{F}(A)$ as the set of all complete valued fields $L$ contained inside $A$ over which $A$ is also a unital Banach $L$-algebra.
\label{def:FAAMF}
\end{definition}
\begin{remark}
\label{rem:FAAFC}
Concerning the spectrum of an element.
\begin{enumerate}
\item[(i)]
It is tempting to conjecture that a generalisation of Theorem \ref{thr:FAANES} might hold for every complete valued field $F$ provided that, given $F$, we restrict the statement to those unital Banach $F$-algebras $A$ for which $F$ is a maximal element of $\mathcal{F}(A)$. This conjecture is false in both the non-commutative and commutative settings by Lemma \ref{lem:FAAHQ} below. However for a more general version of the conjecture one could permit the elements of $\mathcal{F}(A)$ to be complete normed division rings.
\item[(ii)]
Let $A$ be a unital real Banach algebra. In order to avoid an element $a\in A$ having empty spectrum Kaplansky gave the following alternative definition in this case,
\begin{equation*}
\mbox{Sp}_{\mathcal{K}}(a):=\{\alpha+i\beta\in\mathbb{C}:(a-\alpha)^{2}+\beta^{2}\mbox{ is not invertible in }A\}.
\end{equation*}
We won't investigate this definition here but for more details see \cite[p6]{Kulkarni-Limaye1992}.
\end{enumerate}
\end{remark}
\begin{lemma}
In both the non-commutative and commutative algebra settings one can find a complete valued field $F$, a unital Banach $F$-algebra $A$ and an element $a\in A$ such that $F$ is a maximal element of $\mathcal{F}(A)$ and $\mbox{Sp}(a)=\emptyset$.
\label{lem:FAAHQ}
\end{lemma}
\begin{proof}
Hamilton's real quaternions, $\mathbb{H}$, are an example of a non-commutative complete Archimedean division ring and unital real Banach algebra. Viewing $\mathbb{H}$ as a real vector space, the valuation on $\mathbb{H}$ is the Euclidean norm which is complete, Archimedean and indeed a valuation since being multiplicative on $\mathbb{H}$, see \cite[p56,p57]{Lam}. By the Gelfand-Mazur theorem, Theorem \ref{thr:FAAGM}, $\mathbb{H}$ is not a unital complex Banach algebra since being different to $\mathbb{C}$ and so $\mathbb{R}$ is maximal in $\mathcal{F}(\mathbb{H})$. Moreover for $a\in\mathbb{H}\backslash\mathbb{R}$ it is immediate that we have $\mbox{Sp}(a)=\emptyset$.\\
In the commutative setting consider the field of complex numbers $\mathbb{C}$ with the absolute valuation replaced by the $L_{1}$-norm as it applies to the real vector space $\mathbb{R}^{2}$, that is for $a=\alpha+i\beta\in\mathbb{C}$ we have $\|a\|_{1}:=|\alpha|_{\infty}+|\beta|_{\infty}$. Then $\mathbb{C}$ is complete with respect to $\|\cdot\|_{1}$ by the equivalence of norms on finite dimensional $\mathbb{R}$-vector spaces. Expressing complex numbers in their coordinate form it is easy to show that $\|\cdot\|_{1}$ is sub-multiplicative and so $(\mathbb{C},\|\cdot\|_{1})$ is a unital real Banach algebra. However $\|\cdot\|_{1}$ is not multiplicative since $\|(1+i)(1-i)\|_{1}=\|2\|_{1}=2<4=\|1+i\|_{1}\|1-i\|_{1}$ and so $\|\cdot\|_{1}$ is not a valuation on $\mathbb{C}$. Consequently $\mathbb{R}$ is maximal in $\mathcal{F}((\mathbb{C},\|\cdot\|_{1}))$ and over $\mathbb{R}$, $\mbox{Sp}(i)=\emptyset$ which completes the proof.
\end{proof}

	\chapter[Uniform algebras]{Uniform algebras}
\label{cha:UA}
In the first section of this chapter we survey some of the basic facts about complex uniform algebras and recall the close connection with the study of compact Hausdorff spaces, such as Swiss cheese sets, upon which such algebras of functions are defined. An inductive proof by the author of the Feinstein-Heath Swiss cheese ``classicalisation''
theorem is then presented. An article containing this proof has been published by the American Mathematical Society, see \cite{Mason}. In the second section of this chapter we turn our attention to non-complex analogs of uniform algebras. The constraints imposed by the various generalisations of the Stone-Weierstrass theorem are considered and the theory of real function algebras developed by Kulkarni and Limaye is introduced. We will establish the topological requirements of the spaces upon which algebras of functions in the non-Archimedean setting can be defined whilst qualifying as non-complex analogs of uniform algebras. These observations together with some of the details and examples from other chapters have been gathered together by the author into a survey paper which was subsequently accepted for publication by the American Mathematical Society, see \cite{Mason2011}.
\section{Complex uniform algebras}
\label{sec:UAC}
\begin{definition}
\label{def:UACUA}
Let $C_{\mathbb{C}}(X)$ be the unital complex Banach algebra of all continuous complex valued functions, defined on a compact Hausdorff space $X$, with pointwise operations and the sup norm given by
\begin{equation*}
\|f\|_{\infty}:=\sup_{x\in X}|f(x)|_{\infty}\quad\mbox{for all }f\in C_{\mathbb{C}}(X).
\end{equation*}
A {\em{uniform algebra}}, $A$, is a subalgebra of $C_{\mathbb{C}}(X)$ that is complete with respect to the sup norm, contains the constant functions making it a unital complex Banach algebra and separates the points of $X$ in the sense that for all $x_{1}, x_{2}\in X$ with $x_{1}\not= x_{2}$ there is $f\in A$ satisfying $f(x_{1})\not= f(x_{2})$.
\end{definition}
\begin{remark}
Introductions to uniform algebras can be found in \cite{Browder}, \cite{Gamelin} and \cite{Stout}. Some authors take Definition \ref{def:UACUA} to be a representation of uniform algebras and take a uniform algebra $A$ to be a unital complex Banach algebra with a square preserving norm, that is $\|a^{2}\|=\|a\|^{2}$ for all $a\in A$, which they sometimes then referred to as a uniform norm. This is quite legitimate since, as we will discuss at depth in the section on representation theory, the Gelfand transform shows us that every such algebra is isometrically isomorphic to an algebra conforming to Definition \ref{def:UACUA}. In this thesis we mainly introduce generalisations of Definition \ref{def:UACUA} over complete valued fields and then investigate the important representation theory results. Hence for us by {\em uniform norm} we will mean the sup norm.
\label{rem:UASA}
\end{remark}
It is very well known that in the complex setting, for suitable $X$, there exist uniform algebras that are proper subalgebras of $C_{\mathbb{C}}(X)$. However if $A$ is such a uniform algebra then $A$ is not self-adjoint, that is there is $f\in A$ with $\bar{f}\notin A$ where $\bar{f}$ denotes the complex conjugate of $f$. This result is the complex Stone-Weierstrass theorem, generalisations of which we will meet in Section \ref{sec:UANC}. We will also meet several analogs of the following example.
\begin{example}
A standard example is the {\em{disc algebra}} $A(\Delta)\subseteq C_{\mathbb{C}}(\Delta)$, of continuous functions analytic on the interior of $\Delta:=\{z\in\mathbb{C}:|z|\leq1\}$, which is as far from being self-adjoint as possible since if both $f$ and $\bar{f}$ are in $A(\Delta)$ then $f$ is constant, see \cite[p47]{Kulkarni-Limaye1992}. Also $P(\Delta)=A(\Delta)$ where $P(\Delta)$ is the uniform algebra of all functions on $\Delta$ that can be uniformly approximated by polynomials restricted to $\Delta$ with complex coefficients. This largely follows from Remark \ref{rem:FAAGL}, see \cite[p5]{Browder} or \cite[p2]{Stout}.
\label{exa:UADA}
\end{example}
For a compact Hausdorff space $X$ let $R(X)$ denote the uniform algebra of all functions on $X$ that can be uniformly approximated by rational functions from $C_{\mathbb{C}}(X)$. We also generalise to $X$ the uniform algebras introduced in Example \ref{exa:UADA} giving $A(X)$ and $P(X)$. In the theory of uniform approximation it is standard to ask for which $X$ is one or more of the following inclusions non-trivial
\begin{equation*}
P(X)\subseteq R(X)\subseteq A(X)\subseteq C_{\mathbb{C}}(X).
\end{equation*}
Whilst not always the case, this often only depend on $X$ up to homeomorphism. In particular many properties of uniform algebras are topological properties of the spaces upon which they are defined. Hence there is a strong connection between the study of uniform algebras and that of compact Hausdorff spaces. Therefore, in addition to being of interest in their own right, uniform algebras are important in the theory of uniform approximation; as examples of complex Banach algebras; in representation theory and in the study of compact Hausdorff space. With respect to the latter, we now turn our attention to the compact plane sets known as Swiss cheese sets.
\subsection{Swiss cheese sets in the complex plane}
\label{subsec:UASCS}
Throughout subsections \ref{subsec:UASCS} and \ref{subsec:UACT}, all discs in the complex plane are required to have finite positive radius. More generally let $\mathbb{N}_{0}:=\mathbb{N}\cup\{0\}$ from here on throughout the thesis. We begin with the following definitions taken from \cite{Feinstein-Heath}.
\begin{definition}
For a disc $D$ in the plane let $r(D)$ denote the radius of $D$.
\begin{enumerate}
\item[(i)]
A {\em{Swiss cheese}} is a pair ${\bf{D}}:=(\Delta,\mathcal{D})$ for which $\Delta$ is a closed disc and $\mathcal{D}$ is a countable or finite collection of open discs. A Swiss cheese ${\bf{D}}=(\Delta,\mathcal{D})$ is {\em{classical}} if the closures of the discs in $\mathcal{D}$ intersect neither one another nor $\mathbb{C}\backslash\mbox{\rm{int}}\Delta$, and $\sum_{D\in\mathcal{D}}r(D)<\infty$.
\item[(ii)]
The {\em{associated Swiss cheese set}} of a Swiss cheese ${\bf{D}}=(\Delta,\mathcal{D})$ is the plane set $X_{\bf{D}}:=\Delta\backslash\bigcup\mathcal{D}$.\\
A {\em{classical Swiss cheese set}} is a plane set $X$ for which there exists a classical Swiss cheese ${\bf{D}}=(\Delta,\mathcal{D})$ such that $X=X_{\bf{D}}$. 
\item[(iii)]
For a Swiss cheese ${\bf{D}}=(\Delta,\mathcal{D})$, we define $\delta({\bf{D}}):=r(\Delta)-\sum_{D\in\mathcal{D}}r(D)$ so that $\delta({\bf{D}})>-\infty$ if and only if $\sum_{D\in\mathcal{D}}r(D)<\infty$.
\end{enumerate}
\label{def:UASC}
\end{definition}
Figure \ref{fig:UASCS} provides an example for (ii) of Definition \ref{def:UASC}.
\begin{figure}[h]
\begin{center}
\includegraphics{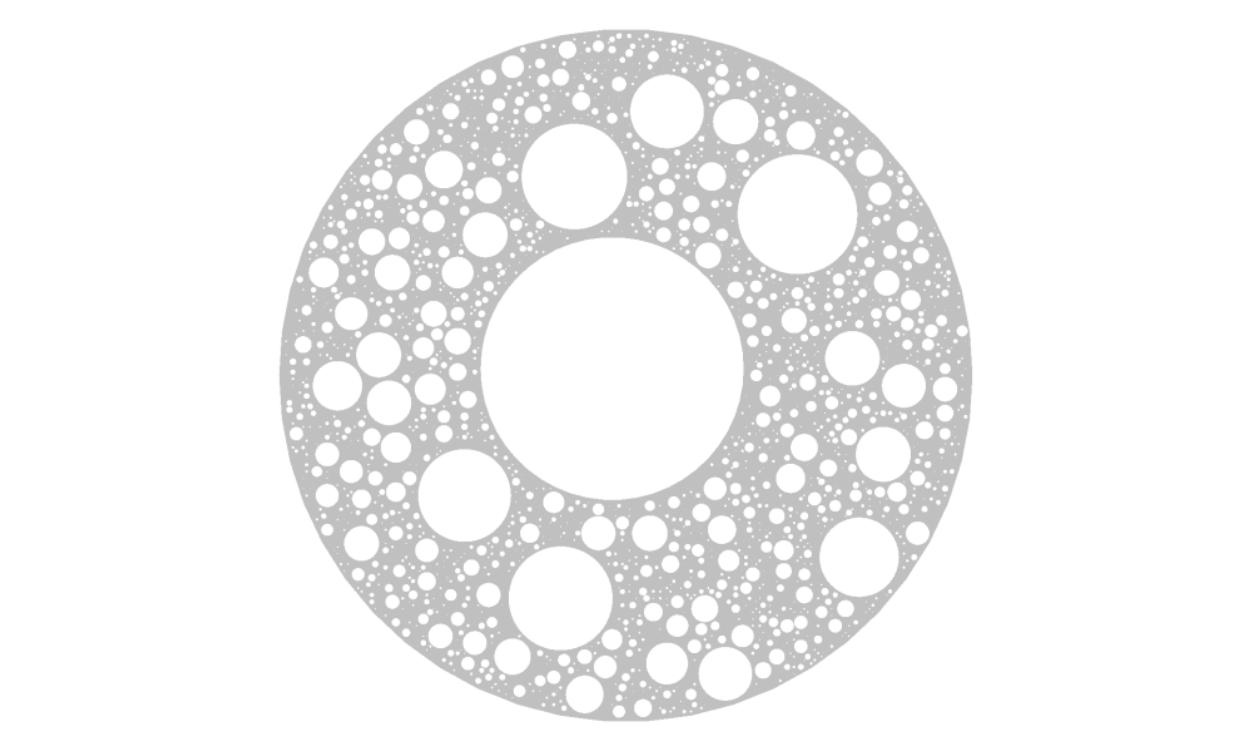}
\end{center}
\caption{A classical Swiss cheese set.}
\label{fig:UASCS}
\end{figure}
Swiss cheese sets are used extensively in the theory of uniform algebras since they provide many examples of uniform algebras with particular properties. For examples see \cite{Feinstein}, \cite[Ch2]{Gamelin} and \cite{Roth}. In particular \cite{Feinstein-Heath} includes a survey of the use of Swiss cheese constructions
in the theory of uniform algebras. The following example is from \cite{Roth}.
\begin{lemma}
For $X\subseteq\mathbb{C}$ non-empty and compact, let $A_{1}\subseteq A_{2}\subseteq C_{\mathbb{C}}(X)$ be uniform algebras with $A_{0}$ uniformly dense in $A_{1}$. Suppose we can find a continuous linear functional $\varphi:C_{\mathbb{C}}(X)\rightarrow\mathbb{C}$ such that $\varphi(A_{0})=\{0\}$ and $\varphi(f)=a\not=0$ for some $f\in A_{2}$. Then $A_{1}\not=A_{2}$.
\label{lem:UARO}
\end{lemma}
\begin{proof}
Let $q\in A_{0}$. Then
\begin{equation*}
0<|a|_{\infty}=|\varphi(f)-\varphi(q)|_{\infty}=|\varphi(f-q)|_{\infty}\leq\|f-q\|_{\infty}\|\varphi\|_{\mbox{\footnotesize op}}
\end{equation*}
giving $\|f-q\|_{\infty}\geq|a|_{\infty}\|\varphi\|_{\mbox{\footnotesize op}}^{-1}>0$ for all $q\in A_{0}$. Hence $f$ can not be uniformly approximated by elements of $A_{0}$. More simply, $\varphi(A_{1})=\{0\}$ by continuity.
\end{proof}
\begin{example}
\label{exa:UARO}
It is possible to have $R(X)\not=A(X)$. Let $D_{0}$ be a closed disc and let ${\bf{D}}=(D_{0},\mathcal{D})$ be a classical Swiss cheese with $\delta({\bf{D}})>0$ and $\mathcal{D}$ infinite. Let $(D_{n})$ be a sequence of open discs such that the map $n\mapsto D_{n}$ is a bijection from $\mathbb{N}$ to $\mathcal{D}$. For $n\in\mathbb{N}_{0}$ define $\gamma_{n}:[0,1]\rightarrow\mathbb{C}$ as the circular path
\begin{equation*}
\gamma_{n}(x):=r_{n}\exp(2\pi ix)+a_{n}
\end{equation*}
around the boundary $\partial D_{n}$. Now for a rational function $q\in C_{\mathbb{C}}(X_{\bf{D}})$ we note that on $\mathbb{C}$ the finitely many poles of $q$ lie in the open complement of $X_{\bf{D}}$ and so $X_{\bf{D}}$ is a subset of an open subset of $\mathbb{C}$ upon which $q$ is analytic. Hence by Cauchy's theorem, see \cite[p218]{Rudin}, we have $\varphi(q)=0$ for $\varphi:C_{\mathbb{C}}(X_{\bf{D}})\rightarrow\mathbb{C}$ defined by
\begin{equation*}
\varphi(f):=\int_{\gamma_{0}}f{\mathrm{d}}z-\sum_{n=1}^{\infty}\int_{\gamma_{n}}f{\mathrm{d}}z\quad\mbox{for }f\in C_{\mathbb{C}}(X_{\bf{D}}).
\end{equation*}
We now check that $\varphi$ is a bounded linear functional on $C_{\mathbb{C}}(X_{\bf{D}})$. The following uses the fundamental estimate. Let $f\in C_{\mathbb{C}}(X_{\bf{D}})$ with $\|f\|_{\infty}\leq 1$. Then,
\begin{align*}
|\varphi(f)|_{\infty} &=\left|\int_{\gamma_{0}}f(z){\mathrm{d}}z-\sum_{n=1}^{\infty}\int_{\gamma_{n}}f(z){\mathrm{d}}z\right|_{\infty}\\ &\leq\sum_{n=0}^{\infty}\left|\int_{\gamma_{n}}f(z){\mathrm{d}}z\right|_{\infty}\\
&\leq\sum_{n=0}^{\infty}\|f\|_{\infty}\int_{0}^{1}|\gamma'_{n}(x)|_{\infty}{\mathrm{d}}x\\
&\leq\sum_{n=0}^{\infty}\int_{0}^{1}|r_{n}2\pi i \exp(2\pi i x)|_{\infty}{\mathrm{d}}x\\
&=\sum_{n=0}^{\infty}r_{n}2\pi\int_{0}^{1}{\mathrm{d}}x=2\pi\left(\sum_{n=0}^{\infty}r_{n}\right)<4\pi r_{0}
\end{align*}
where $\sum_{n=0}^{\infty}r_{n}<2r_{0}$ since $\delta({\bf{D}})>0$. Now $4\pi r_{0}$ is an upper bound for the series of absolute terms, hence we have absolute convergence. Since absolute convergence implies convergence we have, for all $f\in C_{\mathbb{C}}(X_{\bf{D}})$,
\begin{equation*} \varphi(f)=\varphi\left(\frac{\|f\|_{\infty}}{\|f\|_{\infty}}f\right)=\|f\|_{\infty}\varphi\left(\frac{f}{\|f\|_{\infty}}\right)=\|f\|_{\infty}a_{f}\in\mathbb{C}
\end{equation*}
for some $a_{f}\in\mathbb{C}$. Moreover, our calculation shows that $\varphi$ is bounded with $\|\varphi\|_{\mbox{\footnotesize op}}<4\pi r_{0}$. The linearity of $\varphi$ follows from the linearity of integrating over a sum of terms and so $\varphi$ is a continuous linear functional on $C_{\mathbb{C}}(X_{\bf{D}})$. Next we note that the function $g:z\mapsto\bar{z}$ on $X_{\bf{D}}$ given by complex conjugation is an element of $C_{\mathbb{C}}(X_{\bf{D}})$. Cauchy's theorem does not imply that $\varphi(g)$ will be zero since $g$ is not analytic on any non-empty open subset of $\mathbb{C}$. We have
\begin{equation*}
\varphi(g)=\int_{\gamma_{0}}g(z){\mathrm{d}}z-\sum_{n=1}^{\infty}\int_{\gamma_{n}}g(z){\mathrm{d}}z=2\pi i\left(r_{0}^{2}-\sum_{n=1}^{\infty}r_{n}^{2}\right),
\end{equation*}
since for each $n\in\mathbb{N}_{0}$
\begin{align*}
\int_{\gamma_{n}}g(z){\mathrm{d}}z &=\int_{0}^{1}g(\gamma_{n})\gamma'_{n}{\mathrm{d}}x\\
&=\int_{0}^{1}(r_{n}\exp(-2\pi ix)+\bar{a}_{n})r_{n}2\pi i\exp(2\pi ix){\mathrm{d}}x\\
&=2\pi i r_{n}\int_{0}^{1}(r_{n}+\bar{a}_{n}\exp(2\pi ix)){\mathrm{d}}x\\
&=2\pi i r_{n}\left(r_{n}\int_{0}^{1}{\mathrm{d}}x+\bar{a}_{n}\int_{0}^{1}\exp(2\pi ix){\mathrm{d}}x\right)\\
&=2\pi i r_{n}\left(r_{n}+\bar{a}_{n}\left[\frac{1}{2\pi i}\exp(2\pi ix)\right]_{0}^{1}\right)\\
&=2\pi i r_{n}^{2}.
\end{align*}
Furthermore $\sum_{n=1}^{\infty}r_{n}^{2}\leq(\sum_{n=1}^{\infty}r_{n})^{2}<r_{0}^{2}$ since $\sum_{n=1}^{\infty}r_{n}<r_{0}$, by $\delta({\bf{D}})>0$, and so $\varphi(g)\not=0$. Therefore by Lemma \ref{lem:UARO} we have $R(X_{\bf{D}})\not=C_{\mathbb{C}}(X_{\bf{D}})$ and $g\not\in R(X_{\bf{D}})$. Certainly this result is immediate in the case where $X_{\bf{D}}$ has interior since then $g$ will not be an element of $A(X_{\bf{D}})$. However this is one of the occasions where the usefulness of Swiss cheese set constructions becomes evident since, with some consideration, it is straightforward to construct a classical Swiss cheese ${\bf{D}}=(D_{0},\mathcal{D})$ with $\delta({\bf{D}})>0$ such that $X_{\bf{D}}$ has empty interior. Since $A(X_{\bf{D}})$ is the uniform algebra of all elements from $C_{\mathbb{C}}(X_{\bf{D}})$ that are analytic on the interior of $X_{\bf{D}}$, in this case we have $A(X_{\bf{D}})=C_{\mathbb{C}}(X_{\bf{D}})$ and so by the above $R(X_{\bf{D}})\not=A(X_{\bf{D}})$ which completes this example.
\end{example}
Let ${\bf{D}}$ be a Swiss cheese as specified in Example \ref{exa:UARO} such that $X_{\bf{D}}$ has empty interior. The following subsection shows that in this case there is actually no need to require ${\bf{D}}$ to be classical in order that $R(X_{\bf{D}})\not=A(X_{\bf{D}})$.
\subsection{Classicalisation theorem}
\label{subsec:UACT}
In this subsection we give a new proof of an existing theorem by J. F. Feinstein and M. J. Heath, see \cite{Feinstein-Heath}. The theorem states that any Swiss cheese set defined by a Swiss cheese ${\bf{D}}$ with $\delta({\bf{D}})>0$ contains a Swiss cheese set as a subset defined by a classical Swiss cheese ${\bf{D^{'}}}$ with $\delta({\bf{D^{'}}})\geq\delta({\bf{D}})$. Feinstein and Heath begin their proof by developing a theory of allocation maps connected to such sets. A partial order on a family of these allocation maps is then introduced and Zorn's lemma applied. We take a more direct approach by using transfinite induction, cardinality and disc assignment functions, where a disc assignment function is a kind of labeled Swiss cheese that defines a Swiss cheese set. An explicit theory of allocation maps is no longer required although we are still using them implicitly. In this regard we will discuss the connections with the original proof of Feinstein and Heath. See \cite[p266]{Kelley} and \cite[p9]{Dales} for useful introductions to ordinals and transfinite induction which has been used in this subsection. We begin with the following definitions.
\begin{definition}
Let $\mathcal{O}$ be the set of all open discs and complements of closed discs in the complex plane.
\begin{enumerate}
\item[(i)]
A {\em{disc assignment function}} $d:S\rightarrow\mathcal{O}$ is a map from a subset $S\subseteq\mathbb{N}_{0}$, with $0\in S$, into $\mathcal{O}$ such that ${\bf{D}}_{d}:=(\mathbb{C}\backslash d(0),d(S\backslash\{0\}))$ is a Swiss cheese. We allow $S\backslash\{0\}$ to be empty since a Swiss cheese ${\bf{D}}=(\Delta,\mathcal{D})$ can have $\mathcal{D}=\emptyset$.
\item[(ii)]
For a disc assignment function $d:S\rightarrow\mathcal{O}$ and $i\in S$ we let $\bar{d}(i)$ denote the closure of $d(i)$ in $\mathbb{C}$, that is $\bar{d}(i):=\overline{d(i)}$. A disc assignment function $d:S\rightarrow\mathcal{O}$ is said to be {\em{classical}} if for all $(i,j)\in S^{2}$ with $i\not= j$ we have $\bar{d}(i)\cap\bar{d}(j)=\emptyset$ and $\sum_{n\in S\backslash\{0\}}r(d(n))<\infty$.
\item[(iii)]
For a disc assignment function $d:S\rightarrow\mathcal{O}$ we let $X_{d}$ denote the associated Swiss cheese set of the Swiss cheese ${\bf{D}}_{d}$.
\item[(iv)]
A disc assignment function $d:S\rightarrow\mathcal{O}$ is said to have the {\em{Feinstein-Heath condition}} when $\sum_{n\in S\backslash\{0\}}r(d(n))<r(\mathbb{C}\backslash d(0))$.
\item[(v)]
Define $H$ as the set of all disc assignment functions with the Feinstein-Heath condition.\\
For $h\in H$, $h:S\rightarrow\mathcal{O}$, define $\delta_{h}:=r(\mathbb{C}\backslash h(0))-\sum_{n\in S\backslash\{0\}}r(h(n))>0$.
\end{enumerate}
\label{def:UAPSC}
\end{definition}
Here is the Feinstein-Heath Swiss cheese ``Classicalisation'' theorem as it appears in \cite{Feinstein-Heath}.
\begin{theorem}
For every Swiss cheese ${\bf{D}}$ with $\delta({\bf{D}})>0$, there is a classical Swiss cheese ${\bf{D^{'}}}$ with $X_{\bf{D^{'}}}\subseteq X_{\bf{D}}$ and $\delta({\bf{D^{'}}})\geq\delta({\bf{D}})$.
\label{thr:UAPFHT}
\end{theorem}
From Definition \ref{def:UAPSC} we note that if a disc assignment function $d:S\rightarrow\mathcal{O}$ is classical then the Swiss cheese ${\bf{D}}_{d}$ will also be classical. Similarly if $d$ has the Feinstein-Heath condition then $\delta({\bf{D}}_{d})>0$. The converse of each of these implications will not hold in general because $d$ need not be injective.  However it is immediate that for every Swiss cheese ${\bf{D}}=(\Delta,\mathcal{D})$ with $\delta({\bf{D}})>0$ there exists an injective disc assignment function $h\in H$ such that ${\bf{D}}_{h}={\bf{D}}$. We note that every disc assignment function $h\in H$ has $\delta({\bf{D}}_{h})\ge\delta_{h}$ with equality if and only if $h$ is injective and that classical disc assignment functions are always injective. With these observations it easily follows that Theorem \ref{thr:UAPFHT} is equivalent to the following theorem involving disc assignment function.
\begin{theorem}
For every disc assignment function $h\in H$ there is a classical disc assignment function $h^{'}\in H$ with $X_{h^{'}}\subseteq X_{h}$ and $\delta_{h^{'}}\geq\delta_{h}$.
\label{thr:UAPHT}
\end{theorem}
Several lemmas from \cite{Feinstein-Heath} and \cite[\S 2.4.1]{Heath} will be used in the proof of Theorem \ref{thr:UAPHT} and we consider them now.
\begin{lemma}
Let $D_{1}$ and $D_{2}$ be open discs in $\mathbb{C}$ with radii $r(D_{1})$ and $r(D_{2})$ respectively such that $\bar{D}_{1}\cap\bar{D}_{2}\not=\emptyset$. Then there is an open disc $D$ with $D_{1}\cup D_{2}\subseteq D$ and with radius $r(D)\le r(D_{1})+r(D_{2})$.
\label{lem:UAP2.4.13}
\end{lemma}
Figure \ref{fig:UAPdisclemmas}, Example 1 exemplifies the application of Lemma \ref{lem:UAP2.4.13}.
\begin{lemma}
Let $D$ be an open disc and $\Delta$ be a closed disc such that $\bar{D}\not\subseteq\mbox{\textup{int}}\Delta$ and $\Delta\not\subseteq\bar{D}$. Then there is a closed disc $\Delta^{'}\subseteq\Delta$ with $D\cap\Delta^{'}=\emptyset$ and $r(\Delta^{'})\geq r(\Delta)-r(D)$.
\label{lem:UAP2.4.14}
\end{lemma}
Figure \ref{fig:UAPdisclemmas}, Example 2 exemplifies the application of Lemma \ref{lem:UAP2.4.14}.
\begin{figure}[h]
\begin{center}
\includegraphics{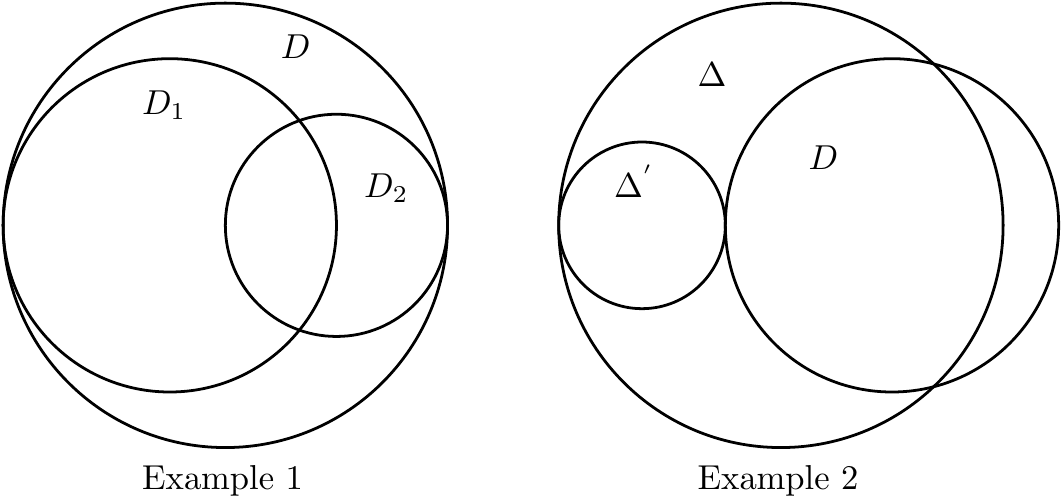}
\end{center}
\caption{Examples for lemmas \ref{lem:UAP2.4.13} and \ref{lem:UAP2.4.14}.}
\label{fig:UAPdisclemmas}
\end{figure}
\begin{lemma}
Let $\mathcal{F}$ be a non-empty, nested collection of open discs in $\mathbb{C}$, such that we have $\sup\{r(E):E\in\mathcal{F}\}<\infty$. Then $\bigcup\mathcal{F}$ is an open disc $D$. Further, for $\mathcal{F}$ ordered by inclusion, $r(D)=\lim_{E\in\mathcal{F}}r(E)=\sup_{E\in\mathcal{F}}r(E)$. 
\label{lem:UAP2.4.11} 
\end{lemma}
\begin{lemma}
Let $\mathcal{F}$ be a non-empty, nested collection of closed discs in $\mathbb{C}$, such that we have $\inf\{r(E):E\in\mathcal{F}\}>0$. Then $\bigcap\mathcal{F}$ is a closed disc $\Delta$. Further, for $\mathcal{F}$ ordered by reverse inclusion, $r(\Delta)=\lim_{E\in\mathcal{F}}r(E)=\inf_{E\in\mathcal{F}}r(E)$.
\label{lem:UAP2.4.12}
\end{lemma}
\begin{proof}[Proof of Theorem \ref{thr:UAPHT}]
At the heart of the proof of Theorem \ref{thr:UAPHT} is a completely defined map $f:H\rightarrow H$ which we now define case by case.
\begin{definition}
\label{def:UAPfHtoH}
Let $f:H\rightarrow H$ be the self map with the following construction.\\
{\em{Case 1:}} If $h\in H$ is a classical disc assignment function then define $f(h):=h$.\\
{\em{Case 2:}} If $h\in H$ is not classical then for $h:S\rightarrow\mathcal{O}$ let
\begin{equation*}
I_{h}:=\{(i,j)\in S^{2}:\bar{h}(i)\cap\bar{h}(j)\not=\emptyset, i\not= j\}.
\end{equation*}
We then have lexicographic ordering on $I_{h}$ given by 
\begin{equation*}
(i,j)\lesssim(i^{'},j^{'})\mbox{ if and only if $i<i^{'}$ or ($i=i^{'}$ and $j\le j^{'}$).}
\end{equation*}
Since this is a well-ordering on $I_{h}$, let $(n,m)$ be the minimum element of $I_{h}$ and hence note that $m\not= 0$ since $m>n$. We proceed toward defining $f(h):S^{'}\rightarrow\mathcal{O}$. 
\begin{equation*}
\mbox{Define $S^{'}:=S\backslash\{m\}$ and for $i\in S^{'}\backslash\{n\}$ we define $f(h)(i):=h(i)$.}
\end{equation*}
It remains for the definition of $f(h)(n)$ to be given and to this end we have the following two cases.\\ 
{\em{Case 2.1:}} $n\not=0$. In this case, by Definition \ref{def:UAPSC}, we note that both $h(m)$ and $h(n)$ are open discs. Associating $h(m)$ and $h(n)$ with $D_{1}$ and $D_{2}$ of Lemma \ref{lem:UAP2.4.13} we define $f(h)(n)$ to be the open disc satisfying the properties of $D$ of the lemma. Note in particular that,
\begin{equation}
h(m)\cup h(n)\subseteq f(h)(n)\mbox{ with }n<m.
\label{equ:UAPsset1}
\end{equation}
{\em{Case 2.2:}} $n=0$. In this case, by Definition \ref{def:UAPSC}, we note that $h(m)$ is an open disc and $h(0)$ is the complement of a closed disc. Associate $h(m)$ with $D$ from Lemma \ref{lem:UAP2.4.14} and put $\Delta :=\mathbb{C}\backslash h(0)$. Since $(0,m)\in I_{h}$ we have $\bar{h}(0)\cap\bar{h}(m)\not=\emptyset$ and so $\bar{h}(m)\not\subseteq\mbox{\textup{int}}\Delta$, noting $\mbox{\textup{int}}\Delta=\mathbb{C}\backslash\bar{h}(0)$. Further, since $h\in H$ we have $r(h(m))<r(\Delta)$ and so $\Delta\not\subseteq\bar{h}(m)$. Therefore the conditions of Lemma \ref{lem:UAP2.4.14} are satisfied for $h(m)$ and $\Delta$. Hence we define $f(h)(0)$ to be the complement of the closed disc satisfying the properties of $\Delta^{'}$ of Lemma \ref{lem:UAP2.4.14}. Note in particular that,
\begin{equation}
h(m)\cup h(0)\subseteq f(h)(0)\mbox{ with }0<m.
\label{equ:UAPsset2}
\end{equation}
\end{definition}
For this definition of the map $f$ we have yet to show that $f$ maps into $H$. We now show this together with certain other useful properties of $f$.
\begin{lemma}Let $h\in H$, then the following hold:
\begin{enumerate}
\item[(i)]
$f(h)\in H$ with $\delta_{f(h)}\geq\delta_{h}$;
\item[(ii)]
For $(h:S\rightarrow\mathcal{O})$ $\mapsto$ $(f(h):S^{'}\rightarrow\mathcal{O})$ we have $S^{'}\subseteq S$ with equality if and only if $h$ is classical. Otherwise $S^{'}=S\backslash\{m\}$ for some $m\in S\backslash\{0\}$;
\item[(iii)]
$X_{f(h)}\subseteq X_{h}$;
\item[(iv)]
For all $i\in S^{'}, h(i)\subseteq f(h)(i)$.
\end{enumerate}
\label{lem:UAPf}
\end{lemma}
\begin{proof}
We need only check (i) and (iii) for cases 2.1 and 2.2 of the definition of $f$, as everything else is immediate. Let $h\in H$.\\
(i) It is clear that $f(h)$ is a disc assignment function. It remains to check that $\delta_{f(h)}\geq\delta_{h}$.\\ 
For Case 2.1 we have, by Lemma \ref{lem:UAP2.4.13},
\begin{align*}
\delta_{h}&=r(\mathbb{C}\backslash h(0))-(r(h(m))+r(h(n)))-\sum_{i\in S\backslash\{0,m,n\}}r(h(i))\\
&\leq r(\mathbb{C}\backslash h(0))-r(f(h)(n))-\sum_{i\in S\backslash\{0,m,n\}}r(h(i))=\delta_{f(h)}.
\end{align*}
For Case 2.2 we have, by Lemma \ref{lem:UAP2.4.14},
\begin{align*}
\delta_{h}&=r(\mathbb{C}\backslash h(0))-r(h(m))-\sum_{i\in S\backslash\{0,m\}}r(h(i))\\
&\leq r(\mathbb{C}\backslash f(h)(0))-\sum_{i\in S\backslash\{0,m\}}r(h(i))=\delta_{f(h)}.
\end{align*}
(iii) Since $X_{h}=\mathbb{C}\backslash\bigcup_{i\in S}h(i)$ we require $\bigcup_{i\in S}h(i)\subseteq\bigcup_{i\in S^{'}}f(h)(i)$.\\
For Case 2.1 we have by Lemma \ref{lem:UAP2.4.13} that $h(m)\cup h(n)\subseteq f(h)(n)$, as shown at (\ref{equ:UAPsset1}), giving $\bigcup_{i\in S}h(i)\subseteq\bigcup_{i\in S^{'}}f(h)(i)$.\\ For Case 2.2 put $\Delta:=\mathbb{C}\backslash h(0)$ and $\Delta^{'}:=\mathbb{C}\backslash f(h)(0)$. We have by Lemma \ref{lem:UAP2.4.14} that $\Delta^{'}\subseteq\Delta$ and $h(m)\cap\Delta^{'}=\emptyset$. Hence $h(0)\cup h(m)\subseteq f(h)(0)$, as shown at (\ref{equ:UAPsset2}), and so $\bigcup_{i\in S}h(i)\subseteq\bigcup_{i\in S^{'}}f(h)(i)$ as required.
\end{proof}
We will use $f:H\rightarrow H$ to construct an ordinal sequence of disc assignment functions and then apply a cardinality argument to show that this ordinal sequence must stabilise at a classical disc assignment function. We construct the ordinal sequence so that it has the right properties.
\begin{definition}Let $h\in H$.
\begin{enumerate}
\item[(a)]
Define $h^{0}:S_{0}\rightarrow\bf{D}$ by $h^{0}:=h$.
\end{enumerate}
Now let $\alpha>0$ be an ordinal for which we have defined $h^{\beta}\in H$ for all $\beta<\alpha$.
\begin{enumerate}
\item[(b)]
If $\alpha$ is a successor ordinal then define $h^{\alpha}:S_{\alpha}\rightarrow\mathcal{O}$ by $h^{\alpha}:=f(h^{\alpha-1})$.
\item[(c)]
If $\alpha$ is a limit ordinal then define $h^{\alpha}:S_{\alpha}\rightarrow\mathcal{O}$ as follows. 
\begin{equation*}
\mbox{Set } S_{\alpha}:=\bigcap_{\beta<\alpha}S_{\beta}. \mbox{ Then for } n\in S_{\alpha} \mbox{ define } h^{\alpha}(n):=\bigcup_{\beta<\alpha}h^{\beta}(n).
\end{equation*}
\end{enumerate}
\label{def:UAPh ord}
\end{definition}
Suppose that for every ordinal $\alpha$ for which Definition \ref{def:UAPh ord} can be applied we have $h^{\alpha}\in H$. Then Definition \ref{def:UAPh ord} can be applied for every ordinal $\alpha$ by transfinite induction and therefore defines an ordinal sequence of disc assignment function. We will use transfinite induction to prove Lemma \ref{lem:UAPh ord} below which asserts that $h^{\alpha}$ is an element of $H$ as well as other useful properties of $h^{\alpha}$.
\begin{lemma}
Let $\alpha$ be an ordinal number and let $h\in H$. Then the following hold:
\begin{enumerate}
\item[($\alpha$,1)]
$h^{\alpha}\in H$ with $\delta_{h^{\alpha}}\geq\delta_{h}$;
\begin{enumerate}
\item[($\alpha$,1.1)]
$0\in S_{\alpha}$;
\item[($\alpha$,1.2)]
$h^{\alpha}(0)$ is the complement of a closed disc and\\ 
$h^{\alpha}(n)$ is an open disc for all $n\in S_{\alpha}\backslash\{0\}$;
\item[($\alpha$,1.3)]
$\sum_{n\in S_{\alpha}\backslash \{0\}}r(h^{\alpha}(n))\leq r(\mathbb{C}\backslash h^{\alpha}(0))-\delta_{h}$;
\end{enumerate}
\item[($\alpha$,2)]
For all $\beta\leq\alpha$ we have $S_{\alpha}\subseteq S_{\beta}$; 
\item[($\alpha$,3)]
For all $\beta\leq\alpha$ we have $X_{h^{\alpha}}\subseteq X_{h^{\beta}}$;
\item[($\alpha$,4)]
For all $n\in S_{\alpha}$, $\{h^{\beta}(n):\beta\leq\alpha\}$ is a nested increasing family of open sets.
\end{enumerate}
\label{lem:UAPh ord}
\end{lemma}
\begin{proof} We will use transfinite induction.\\
For $\alpha$ an ordinal number let $P(\alpha)$ be the proposition, Lemma \ref{lem:UAPh ord} holds at $\alpha$.\\
The base case $P(0)$ is immediate and our inductive hypothesis is that for all $\beta<\alpha$, $P(\beta)$ holds.\\
Now for $\alpha$ a successor ordinal we have $h^{\alpha}=f(h^{\alpha-1})$ and so $P(\alpha)$ is immediate by the inductive hypothesis and Lemma \ref{lem:UAPf}. Now suppose $\alpha$ is a limit ordinal. 
We have $S_{\alpha}:=\bigcap_{\beta<\alpha}S_{\beta}$ giving, for all $\beta\le\alpha$, $S_{\alpha}\subseteq S_{\beta}$. Hence ($\alpha$,2) holds. Also for all $\beta<\alpha$ we have $0\in S_{\beta}$ by ($\beta$,1.1). So $0\in S_{\alpha}$ showing that ($\alpha$,1.1) holds. To show ($\alpha$,1.2) we will use lemmas \ref{lem:UAP2.4.11} and \ref{lem:UAP2.4.12}.
\begin{enumerate}
\item[(i)]
Now for all $n\in S_{\alpha}\backslash\{0\}$, $\{h^{\beta}(n):\beta<\alpha\}$ is a nested increasing family of open discs by ($\beta$,1.2) and ($\beta$,4).
\item[(ii)]
Further, $\{\mathbb{C}\backslash h^{\beta}(0):\beta<\alpha\}$ is a nested decreasing family of closed discs by ($\beta$,1.2) and ($\beta$,4).
\item[(iii)]
Now for $n\in S_{\alpha}\backslash\{0\}$ and $\beta<\alpha$ we have\\ $r(h^{\beta}(n))\leq\sum_{m\in S_{\beta}\backslash\{0\}}r(h^{\beta}(m))=r(\mathbb{C}\backslash h^{\beta}(0))-\delta_{h^{\beta}}\leq r(\mathbb{C}\backslash h(0))-\delta_{h}$, by ($\beta$,1) and (ii). Hence $\sup\{r(h^{\beta}(n)):\beta<\alpha\}\leq r(\mathbb{C}\backslash h(0))-\delta_{h}$.
So by (i) and Lemma \ref{lem:UAP2.4.11} we have for $n\in S_{\alpha}\backslash\{0\}$ that 
\begin{equation*}
h^{\alpha}(n):=\bigcup_{\beta<\alpha}h^{\beta}(n) 
\end{equation*}
is an open disc with, 
\begin{equation*}
r(h^{\alpha}(n))=\sup_{\beta<\alpha}r(h^{\beta}(n))\leq r(\mathbb{C}\backslash h(0))-\delta_{h}.
\end{equation*}
\item[(iv)]
Now for $\beta<\alpha$ we have $r(\mathbb{C}\backslash h^{\beta}(0))\geq\delta_{h}$ by ($\beta$,1.3).\\
Hence $\inf\{r(\mathbb{C}\backslash h^{\beta}(0)):\beta<\alpha\}\geq\delta_{h}$. So by De Morgan, (ii) and Lemma \ref{lem:UAP2.4.12} we have 
\begin{equation*}
\mathbb{C}\backslash h^{\alpha}(0):=\mathbb{C}\backslash\bigcup_{\beta<\alpha}h^{\beta}(0)=\bigcap_{\beta<\alpha}\mathbb{C}\backslash h^{\beta}(0) 
\end{equation*}
is a closed disc with, 
\begin{equation*}
r(\mathbb{C}\backslash h^{\alpha}(0))=\inf_{\beta<\alpha}r(\mathbb{C}\backslash h^{\beta}(0))\geq\delta_{h}. 
\end{equation*}
Hence $h^{\alpha}(0)$ is the complement of a closed disc and so ($\alpha$,1.2) holds.
\end{enumerate}
We now show that ($\alpha$,4) holds. By ($\beta$,4) we have, for all $n\in S_{\alpha}$, $\{h^{\beta}(n):\beta<\alpha\}$ is a nested increasing family of open sets. We also have $h^{\alpha}(n)=\bigcup_{\beta<\alpha}h^{\beta}(n)$ so, for all $\beta\leq\alpha$, $h^{\beta}(n)\subseteq h^{\alpha}(n)$ and $h^{\alpha}(n)$ is an open set since ($\alpha$,1.2) holds. Hence ($\alpha$,4) holds. 
We will now show that ($\alpha$,1.3) holds. We first prove that, for all $\lambda<\alpha$, we have
\begin{equation}
\sum_{m\in S_{\alpha}\backslash\{0\}}r(h^{\alpha}(m))\leq r(\mathbb{C}\backslash h^{\lambda}(0))-\delta_{h}.
\label{equ:UAPinequ1}
\end{equation}
Let $\lambda<\alpha$, and suppose, towards a contradiction, that 
\begin{equation}
\sum_{m\in S_{\alpha}\backslash\{0\}}r(h^{\alpha}(m))>r(\mathbb{C}\backslash h^{\lambda}(0))-\delta_{h},
\label{equ:UAPcont}
\end{equation}
noting that the right hand side of (\ref{equ:UAPcont}) is non-negative by ($\lambda$,1.3).\\ 
Set
\begin{equation*}
\varepsilon:=\frac{1}{2}\left(\sum_{m\in S_{\alpha}\backslash\{0\}}r(h^{\alpha}(m))-(r(\mathbb{C}\backslash h^{\lambda}(0))-\delta_{h})\right)>0.
\end{equation*}
Then there exists $n\in S_{\alpha}\backslash\{0\}$ such that for $S_{\alpha}|_{1}^{n}:=\{m\in S_{\alpha}\backslash\{0\}:m\leq n\}$ we have 
\begin{equation}
\sum_{m\in S_{\alpha}|_{1}^{n}}r(h^{\alpha}(m))>r(\mathbb{C}\backslash h^{\lambda}(0))-\delta_{h}+\varepsilon>0. 
\label{equ:UAPinequ2}
\end{equation}
Further for each $m\in S_{\alpha}|_{1}^{n}$ we have, by (iii), $r(h^{\alpha}(m))=\sup_{\beta<\alpha}r(h^{\beta}(m))$. Hence for each $m\in S_{\alpha}|_{1}^{n}$ there exists $\beta_{m}<\alpha$ such that $r(h^{\beta_{m}}(m))\geq r(h^{\alpha}(m))-\frac{1}{2k}\varepsilon$, for $k:=|S_{\alpha}|_{1}^{n}|$, $k\not=0$ by (\ref{equ:UAPinequ2}). Let $\lambda^{'}:=\max\{\beta_{m}:m\in S_{\alpha}|_{1}^{n}\}<\alpha$ and note that this is a maximum over a finite set of elements since $S_{\alpha}|_{1}^{n}\subseteq\mathbb{N}$ is finite. Now for any $\gamma$ with $\max\{\lambda,\lambda^{'}\}\leq\gamma<\alpha$ we have,
\begin{align*}
\sum_{m\in S_{\gamma}\backslash\{0\}}r(h^{\gamma}(m))&\geq\sum_{m\in S_{\alpha}\backslash\{0\}}r(h^{\gamma}(m))& &(\mbox{since }S_{\alpha}\subseteq S_{\gamma})&\\
&\geq\sum_{m\in S_{\alpha}|_{1}^{n}}r(h^{\gamma}(m))& &&\\
&\geq\sum_{m\in S_{\alpha}|_{1}^{n}}r(h^{\beta_{m}}(m))& &(\mbox{by ($\gamma$,4)})&\\
&\geq\sum_{m\in S_{\alpha}|_{1}^{n}}(r(h^{\alpha}(m))-\frac{\varepsilon}{2k})& &(\mbox{by the above})&\\
&>r(\mathbb{C}\backslash h^{\lambda}(0))-\delta_{h}+\varepsilon-k\frac{\varepsilon}{2k}& &(\mbox{by (\ref{equ:UAPinequ2}) and }k:=|S_{\alpha}|_{1}^{n}|)&\\
&>r(\mathbb{C}\backslash h^{\lambda}(0))-\delta_{h}& &&\\
&\geq r(\mathbb{C}\backslash h^{\gamma}(0))-\delta_{h}& &(\mbox{by (ii)}).&
\end{align*}
This contradicts ($\gamma$,1.3). Hence we have shown that, for all $\lambda<\alpha$, (\ref{equ:UAPinequ1}) holds.\\ Now by (iv) we have $r(\mathbb{C}\backslash h^{\alpha}(0))=\inf_{\lambda<\alpha}r(\mathbb{C}\backslash h^{\lambda}(0))$.\\ Hence we have $\sum_{m\in S_{\alpha}\backslash\{0\}}r(h^{\alpha}(m))\leq r(\mathbb{C}\backslash h^{\alpha}(0))-\delta_{h}$ and so ($\alpha$,1.3) holds.

We now show that ($\alpha$,3) holds. We will show that for all ordinals $\beta<\alpha$,\\ $\bigcup_{i\in S_{\beta}}h^{\beta}(i)\subseteq\bigcup_{i\in S_{\alpha}}h^{\alpha}(i)$. Let $\beta<\alpha$ and $z\in\bigcup_{i\in S_{\beta}}h^{\beta}(i)$. Define,
\begin{equation*}
m:=\min\{i\in\mathbb{N}_{0}:\mbox{ there exists }\lambda<\alpha\mbox{ with } i\in S_{\lambda}\mbox{ and }z\in h^{\lambda}(i)\}.
\end{equation*}
By the definition of $m$ there exists $\zeta<\alpha$ with $m\in S_{\zeta}$ and $z\in h^{\zeta}(m)$. We claim that the set $\{\lambda<\alpha:m\not\in S_{\lambda}\}$ is empty. To prove this suppose towards a contradiction that we can define,
\begin{equation*}
\lambda^{'}:=\min\{\lambda<\alpha:m\not\in S_{\lambda}\}.
\end{equation*}
Then $\lambda^{'}>0$ since, by ($\zeta$,2), $S_{\zeta}\subseteq S_{0}$ with $m\in S_{\zeta}$. If $\lambda^{'}$ is a limit ordinal then $m\not\in S_{\lambda^{'}}=\bigcap_{\gamma<\lambda^{'}}S_{\gamma}$ giving $m\not\in S_{\gamma}$, for some $\gamma<\lambda^{'}$, and this contradicts the definition of $\lambda^{'}$. If $\lambda^{'}$ is a successor ordinal then $h^{\lambda^{'}}=f(h^{\lambda^{'}-1})$ with $m\in S_{\lambda^{'}-1}$ by the definition of $\lambda^{'}$. By $m\not\in S_{\lambda^{'}}$ and Definition \ref{def:UAPfHtoH} of $f:H\rightarrow H$, $h^{\lambda^{'}-1}$ is not classical. Therefore by (\ref{equ:UAPsset1}) and (\ref{equ:UAPsset2}) of Definition \ref{def:UAPfHtoH} there is $n\in S_{\lambda^{'}}$ with $n<m$ and $h^{\lambda^{'}-1}(m)\subseteq h^{\lambda^{'}}(n)$. Further for all $\lambda$ with $\lambda^{'}\leq\lambda<\alpha$ we have $m\not\in S_{\lambda}$ since $m\not\in S_{\lambda^{'}}$ and, by ($\lambda$,2), $S_{\lambda}\subseteq S_{\lambda^{'}}$. Hence we have $\zeta<\lambda^{'}$. Now, by ($\lambda^{'}-1$, 4), $\{h^{\gamma}(m):\gamma\leq\lambda^{'}-1\}$ is a nested increasing family of sets giving $z\in h^{\zeta}(m)\subseteq h^{\lambda^{'}-1}(m)\subseteq h^{\lambda^{'}}(n)$ with $n\in S_{\lambda^{'}}$. This contradicts the definition of $m$ since $n<m$. Hence we have shown that $\{\lambda<\alpha:m\not\in S_{\lambda}\}$ is empty giving $m\in S_{\alpha}=\bigcap_{\lambda<\alpha}S_{\lambda}$. Therefore, by Definition \ref{def:UAPh ord} and the definition of $\zeta$, we have $z\in h^{\zeta}(m)\subseteq\bigcup_{\lambda<\alpha}h^{\lambda}(m)=h^{\alpha}(m)\subseteq\bigcup_{i\in S_{\alpha}}h^{\alpha}(i)$ as required. Hence ($\alpha$,3) holds. Therefore we have shown, by the principal of transfinite induction, that $P(\alpha)$ holds and this concludes the proof of Lemma \ref{lem:UAPh ord}.
\end{proof}
Recall that our aim is to prove that for every $h\in H$ there is a classical disc assignment function $h^{'}\in H$ with $X_{h^{'}}\subseteq X_{h}$ and $\delta_{h^{'}}\geq\delta_{h}$. We have the following closing argument using cardinality. By ($\alpha$,2) of Lemma \ref{lem:UAPh ord} we obtain a nested ordinal sequence of domains $(S_{\alpha})$,\\
$\mathbb{N}_{0}\supseteq S\supseteq S_{1}\supseteq S_{2}\supseteq\cdots\supseteq S_{\omega}\supseteq S_{\omega+1}\supseteq\cdots\supseteq\{0\}$.\\
Now setting $S_{\alpha}^{c}:=\mathbb{N}_{0}\backslash S_{\alpha}$ gives a nested ordinal sequence $(S_{\alpha}^{c})$,\\
$\emptyset\subseteq S^{c}\subseteq S_{1}^{c}\subseteq S_{2}^{c}\subseteq\cdots\subseteq S_{\omega}^{c}\subseteq S_{\omega+1}^{c}\subseteq\cdots\subseteq\mathbb{N}$. 
\begin{lemma}
For the disc assignment function $h^{\beta}$ we have,\\
$h^{\beta}$ is classical if and only if $(S_{\alpha})$ has stabilised at $\beta$, i.e. $S_{\beta+1}=S_{\beta}$.
\label{lem:UAPstab}
\end{lemma}
\begin{proof}
The proof follows directly from (ii) of Lemma \ref{lem:UAPf}.
\end{proof}
Now let $\omega_{1}$ be the first uncountable ordinal. Suppose towards a contradiction that, for all $\beta<\omega_{1}$, $(S_{\alpha})$ has not stabilised at $\beta$. Then for each $\beta<\omega_{1}$ there exists some $n_{\beta+1}\in\mathbb{N}$ such that $n_{\beta+1}\in S_{\beta+1}^{c}$ but $n_{\beta+1}\not\in S_{\alpha}^{c}$ for all $\alpha\leq\beta$. Hence since there are uncountably many $\beta<\omega_{1}$ we have $S_{\omega_{1}}^{c}$ uncountable with $S_{\omega_{1}}^{c}\subseteq\mathbb{N}$, a contradiction. Therefore there exists $\beta<\omega_{1}$ such that $(S_{\alpha})$ has stabilised at $\beta$ and so, by Lemma \ref{lem:UAPstab}, $h^{\beta}$ is classical. Now by ($\beta$,1) of Lemma \ref{lem:UAPh ord} we have $h^{\beta}\in H$ with $\delta_{h^{\beta}}\geq\delta_{h}$ and by ($\beta$,3) we have $X_{h^{\beta}}\subseteq X_{h}$. In particular this completes the proof of Theorem \ref{thr:UAPHT} and the Feinstein-Heath Swiss cheese ``Classicalisation'' theorem.
\end{proof}
The proof of Theorem \ref{thr:UAPFHT} as presented above proceeded without reference to a theory of allocation maps. In the original proof of Feinstein and Heath, \cite{Feinstein-Heath}, allocation maps play a central role. We will recover a key allocation map from the original proof using the map $f:H\rightarrow H$ of Definition \ref{def:UAPfHtoH}. Here is the definition of an allocation map as it appears in \cite{Feinstein-Heath}.
\begin{definition}
Let ${\bf{D}}=(\Delta,\mathcal{D})$ be a Swiss cheese. We define 
\begin{equation*}
\widetilde{{\bf{D}}}=\mathcal{D}\cup\{\mathbb{C}\backslash\Delta\}.
\end{equation*}
Now let ${\bf{E}}=(\mathsf{E},\mathcal{E})$ be a second Swiss cheese, and let $f:\widetilde{{\bf{D}}}\rightarrow\widetilde{{\bf{E}}}$. We define $\mathcal{G}(f)=f^{-1}(\mathbb{C}\backslash\mathsf{E})\cap\mathcal{D}$. We say that $f$ is an {\bf{allocation map}} if the following hold:
\begin{enumerate}
\item[(A1)] for each $U\in\widetilde{{\bf{D}}}$, $U\subseteq f(U)$;
\item[(A2)]
\begin{equation*}
\sum_{D\in\mathcal{G}(f)}r(D)\geq r(\Delta)-r(\mathsf{E});
\end{equation*}
\item[(A3)] for each $E\in\mathcal{E}$,
\begin{equation*}
\sum_{D\in f^{-1}(E)}r(D)\geq r(E).
\end{equation*}
\end{enumerate}
\label{def:UAPAllo}
\end{definition}
Let ${\bf{D}}$ be the Swiss cheese of Theorem \ref{thr:UAPFHT} and let $\mathcal{S}({\bf{D}})$ be the family of allocation maps defined on $\widetilde{{\bf{D}}}$. In \cite{Feinstein-Heath} a partial order is applied to $\mathcal{S}({\bf{D}})$ and subsequently a maximal element $f_{\mbox{\scriptsize{max}}}$ is obtained using Zorn's lemma. The connection between allocation maps and Swiss cheeses is then exploited. Towards a contradiction the non-existence of the desired classical Swiss cheese ${\bf{D^{'}}}$ of Theorem \ref{thr:UAPFHT} is assumed. This assumption implies the existence of an allocation map $f'\in\mathcal{S}({\bf{D}})$ that is higher in the partial order applied to $\mathcal{S}({\bf{D}})$ than $f_{\mbox{\scriptsize{max}}}$, a contradiction. The result follows. It is at the last stage of the original proof where a connection to the new version can be found. In the construction of Feinstein and Heath the allocation map $f'$ factorizes as $f'=g\circ f_{\mbox{\scriptsize{max}}}$ where $g$ is also an allocation map. Let ${\bf{E}}=(\mathsf{E},\mathcal{E})$ be a non-classical Swiss cheese with $\delta({\bf{E}})>0$. Using the same method of construction that Feinstein and Heath use for $g$, an allocation map $g_{\mbox{\tiny{E}}}$ defined on $\widetilde{{\bf{E}}}$ can be obtained without contradiction. Clearly $\widetilde{{\bf{E}}}\not=f_{\mbox{\scriptsize{max}}}(\widetilde{{\bf{D}}})$. We will obtain $g_{\mbox{\tiny{E}}}$ using the map $f:H\rightarrow H$ of Definition \ref{def:UAPfHtoH}. Let $h\in H$, $h:S\rightarrow\mathcal{O}$, be an injective disc assignment function such that ${\bf{D}}_{h}={\bf{E}}$ and recall from Definition \ref{def:UAPfHtoH} that $f(h):S^{'}\rightarrow\mathcal{O}$ has $S^{'}=S\backslash\{m\}$ where $(n,m)$ is the minimum element of $I_{h}$. Set ${\bf{E'}}:={\bf{D}}_{f(h)}$. By Definitions \ref{def:UAPSC} and \ref{def:UAPAllo} we have
\begin{equation*}
\widetilde{{\bf{E}}}=\widetilde{{\bf{D}}_{h}}=h(S)\mbox{ and }\widetilde{{\bf{E'}}}=\widetilde{{\bf{D}}_{f(h)}}=f(h)(S^{'}).
\end{equation*}
Now define a map $\iota:\widetilde{{\bf{E}}}\rightarrow S^{'}$ by,
\begin{equation*}
\mbox{for } U\in\widetilde{{\bf{E}}}\mbox{, }\quad\iota(U):=\begin{cases} h^{-1}(U) &\mbox{ if }\quad h^{-1}(U)\not=m\\ n &\mbox{ if }\quad h^{-1}(U)=m \end{cases},
\end{equation*}
and note that this is well defined since $h$ is injective.
\begin{figure}[h]
\begin{equation*}
\xymatrix{
\widetilde{{\bf{E}}}\ar[r]^{\mbox{{\small{$g$}}}_{\mbox{\tiny{E}}}}\ar[d]_{\iota}&\widetilde{{\bf{E'}}}\\
S^{'}\ar[ru]_{f(h)}&
}
\end{equation*}
\caption{$g_{\mbox{\tiny{E}}}=f(h)\circ\iota$.}
\label{fig:UAPgEandf}
\end{figure}
The commutative diagram in Figure \ref{fig:UAPgEandf} show how $g_{\mbox{\tiny{E}}}$ is obtained using $f:H\rightarrow H$. The construction of $f$ in Definition \ref{def:UAPfHtoH} was developed from the construction that Feinstein and Heath used for $g$. The method of combining discs in Lemma \ref{lem:UAP2.4.13} also appears in \cite{Zhang}.
\begin{remark}
\label{rem:UASCS}
Concerning classicalisation.
\begin{enumerate}
\item[(i)]
Interestingly, as Heath points out in \cite{Feinstein-Heath}, every classical Swiss cheese set in $\mathbb{C}$ with empty interior is homeomorphic to the Sierpi\'{n}ski carpet. Hence up to homeomorphism there is only 1 Swiss cheese set of this type. In particular if $X_{\bf{D}}$ is a Swiss cheese contained in $\mathbb{C}$ with empty interior then either one of the conditions $X_{\bf{D}}$ is classical or $\delta({\bf{D}})>0$ is enough for $R(X_{\bf{D}})\not=A(X_{\bf{D}})$.
\item[(ii)]
I also anticipate the possibility of an analog of Theorem \ref{thr:UAPHT} on the sphere. Let $S\subseteq\mathbb{R}^{3}$ be a sphere of finite positive radius $r_{s}$ and center $c\in\mathbb{R}^{3}$. For $a,b\in S$ let $d_{s}(a,b):=r_{s}\angle acb$ be the length of the geodesic path in $S$ from $a$ to $b$. Now $d_{s}$ is a metric with respect to which we will define open and closed $S$-discs contained in $S$. With analogy to Definition \ref{def:UASC} let ${\bf{D}}_{s}:=(\Delta,\mathcal{D})$ be a Swiss cheese on $S$. Then either $\Delta=S$ or ${\bf{D}}_{s}^{'}:=(S,\mathcal{D}\cup\{S\backslash\Delta\})$ is a Swiss cheese on $S$, since $S\backslash\Delta$ is an open $S$-disc, for which $X_{{\bf{D}}_{s}^{'}}=X_{{\bf{D}}_{s}}$ in $S$. Further we have
\begin{equation*}
\delta({\bf{D}}_{s}):=r(\Delta)-\sum_{D\in\mathcal{D}}r(D)=\pi r_{s}-r(S\backslash\Delta)-\sum_{D\in\mathcal{D}}r(D)=\delta({\bf{D}}_{s}^{'})
\end{equation*}
and so for our choice of metric on $S$ we note that $\delta$ is independent of whether we use ${\bf{D}}_{s}$ or ${\bf{D}}_{s}^{'}$. Hence the situation for the sphere is a little simpler than that for the plane since we can allow all Swiss cheeses on the sphere to have the form ${\bf{D}}_{s}:=(S,\mathcal{D})$ and avoid the need to handle a special closed $S$-disc $\Delta$. Therefore on the sphere analogs of lemmas \ref{lem:UAP2.4.14} and \ref{lem:UAP2.4.12} are not required. We will not establish here whether the condition $\delta({\bf{D}}_{s})>0$ is sufficient for the analog of Theorem \ref{thr:UAPHT} on $S$ to hold since the next step in generalising this theorem should be to establish the class of all metric spaces for which a general version of the theorem holds. However the sphere is of particular interest in the context of uniform algebras since, less one point, the sphere is homeomorphic to the plane allowing many examples of uniform algebras to be defined on subsets of the sphere.
\end{enumerate}
\end{remark}
\section{Non-complex analogs of uniform algebras}
\label{sec:UANC}
The most obvious non-complex analog of Definition \ref{def:UACUA} is obtained by simply replacing the complex numbers in the definition by some other complete valued field $F$. In this case, whilst $C_{F}(X)$ will be complete and contain the constants, we need to take care concerning the topology on $X$ when $F$ is non-Archimedean, e.g. $C_{\mathbb{Q}_{p}}([0,1])$ only contains the constant.
\begin{theorem}
Let $F$ be a complete non-Archimedean valued field and let $C_{F}(X)$ be the unital Banach $F$-algebra of all continuous $F$-valued functions defined on a compact, Hausdorff space $X$. Then $C_{F}(X)$ separates the points of $X$ if and only if $X$ is totally disconnected.
\label{thr:UAXtop}
\end{theorem}
Before giving a proof of Theorem \ref{thr:UAXtop} we have the following version of Urysohn's lemma which will certainly already be known in some form because of its simplicity.
\begin{lemma}
Let $X$ be a totally disconnected, compact, Hausdorff space with finite subset $\{x,y_{1},y_{2},y_{3}, \cdots, y_{n}\}\subseteq X$, $x\not= y_{i}$ for all $i\in\{1, \cdots, n\}$ where $n\in\mathbb{N}$. Let $L$ be any non-empty topological space and $a,b\in L$. Then there exists a continuous map $h:X\longrightarrow L$ such that $h(x)=a$ and $h(y_{1})=h(y_{2})=h(y_{3})= \cdots =h(y_{n})=b$.
\label{lem:UAUrys}
\end{lemma}
\begin{proof}
Since $X$ is a Hausdorff space, for each $i\in\{1, \cdots, n\}$ there are disjoint open subsets $U_{i}$ and $V_{i}$ of $X$ with $x\in U_{i}$ and $y_{i}\in V_{i}$. Hence $U:=\bigcap_{i\in\{1, \cdots, n\}}U_{i}$ is an open subset of $X$ with $x\in U$ and $U\cap V_{i}=\emptyset$ for all $i\in\{1, \cdots, n\}$. Now since $X$ is a totally disconnected, compact, Hausdorff space, $x$ has a neighborhood base of clopen sets, see \cite[Theorem 29.7]{Willard} noting that $X$ is locally compact by Theorem \ref{thr:CVFHL}. Hence there is a clopen subset $W$ of $X$ with $x\in W\subseteq U$. The function $h:X\longrightarrow L$ given by $h(W):=\{a\}$ and $h(X\backslash W):=\{b\}$ is continuous as required.
\end{proof}
We now give the proof of Theorem \ref{thr:UAXtop}.
\begin{proof}
With reference to Lemma \ref{lem:UAUrys} it remains to show that $C_{F}(X)$ separates the points of $X$ only if $X$ is totally disconnected. Let $X$ be a compact, Hausdorff space such that $C_{F}(X)$ separates the points of $X$. Let $U$ be a non-empty connected subset of $X$ and let $f\in C_{F}(X)$. We note that $f(U)$ is a connected subset of $F$ since $f$ is continuous. Now, since $F$ is non-Archimedean it is totally disconnected i.e. its connected subsets are singletons. Hence $f(U)$ is a singleton and so $f$ is constant on $U$. Therefore, since $C_{F}(X)$ separates the points of $X$, $U$ is a singleton and $X$ is totally disconnected.
\end{proof}
We next consider the constraints on $C_{F}(X)$ revealed by generalisations of the Stone-Weierstrass theorem. In the real case the Stone-Weierstrass theorem for $C_{\mathbb{R}}(X)$ says that for every compact Hausdorff space $X$, $C_{\mathbb{R}}(X)$ is without a proper subalgebra that is uniformly closed, contains the real numbers and separates the points of $X$. A proof can be found in \cite[p50]{Kulkarni-Limaye1992}. The non-Archimedean case is given by a theorem of Kaplansky, see \cite[p157]{Berkovich} or \cite{Kaplansky}.
\begin{theorem}
Let $F$ be a complete non-Archimedean valued field, let $X$ be a totally disconnected compact Hausdorff space, and let $A$ be a $F$-subalgebra of $C_{F}(X)$ which satisfies the following conditions:
\begin{enumerate}
\item[(i)]
the elements of $A$ separate the points of $X$;
\item[(ii)]
for each $x\in X$ there exists $f\in A$ with $f(x)\not=0$.
\end{enumerate}
Then $A$ is everywhere dense in $C_{F}(X)$.
\label{thr:UAKapl}
\end{theorem}
Note that, in Theorem \ref{thr:UAKapl}, $A$ being a $F$-subalgebra of $C_{F}(X)$ means that $A$ is a subalgebra of $C_{F}(X)$ and a vector space over $F$. If we take $A$ to be unital then condition (ii) in Theorem \ref{thr:UAKapl} is automatically satisfied and the theorem is analogous to the real version of the Stone-Weierstrass theorem. In subsection \ref{subsec:UARFA} we will see that real function algebras are a useful example when considering non-complex analogs of uniform algebras with qualifying subalgebras.
\subsection{Real function algebras}
\label{subsec:UARFA}
Real function algebras were introduced by Kulkarni and Limaye in a paper from 1981, see \cite{Kulkarni-Limaye1981}. For a thorough text on the theory see \cite{Kulkarni-Limaye1992}. The following definition is a little more general than what we need in this subsection.
\begin{definition}
\label{def:UATIAI}
Let $X$ be a topological space and let $\tau:X\rightarrow X$ be a homeomorphism.
\begin{enumerate}
\item[(i)]
We will call $\tau$ a {\em topological involution on} $X$ if $\tau\circ\tau =\mbox{id}$ on $X$.
\item[(ii)]
We will call $\tau$ a {\em topological element of finite order on} $X$ if $\tau$ has finite order but with $\mbox{ord}(\tau)>2$.
\end{enumerate}
Let $F$ and $L$ be complete valued fields such that $L$ is a finite extension of $F$ as a valued field and let $g\in\mbox{Gal}(^{L}/_{F})$. Let $A$ either be an $F$-algebra or an $L$-algebra for which $\sigma:A\rightarrow A$ is a map satisfying $\mbox{ord}(\sigma)=\mbox{ord}(g)$ and for all $a,b\in A$ and scalars $\alpha$:
\begin{enumerate}
\item[]
$\sigma(a+b)=\sigma(a)+\sigma(b)$;
\item[]
$\sigma(ab)=\sigma(b)\sigma(a)$;
\item[]
$\sigma(\alpha a)=g(\alpha)\sigma(a)$.
\end{enumerate}
\begin{enumerate}
\item[(iii)]
We will call $\sigma$ a {\em algebraic involution on} $A$ if $\sigma\circ\sigma =\mbox{id}$ on $A$.
\item[(iv)]
We will call $\sigma$ a {\em algebraic element of finite order on} $A$ if $\sigma$ has $\mbox{ord}(\sigma)>2$.
\end{enumerate}
\end{definition}
Note, in Definition \ref{def:UATIAI} the requirement that $\tau$ be a homeomorphism is satisfied if $\tau$ is continuous.
\begin{definition}
\label{def:UARefa}
Let $X$ be a compact Hausdorff space and $\tau$ a topological involution on $X$. A {\em{real function algebra on}} $(X,\tau)$ is a real subalgebra $A$ of
\begin{equation*}
C(X,\tau):=\{f\in C_{\mathbb{C}}(X):f(\tau(x))=\bar{f}(x)\mbox{ for all }x\in X\}
\end{equation*}
that is complete with respect to the sup norm, contains the real numbers and separates the points of $X$.
\end{definition}
\begin{remark}
Concerning real function algebras.
\begin{enumerate}
\item[(i)]
Later, Theorem \ref{thr:CGGUA} will confirm that $C(X,\tau)$ in Definition \ref{def:UARefa} is itself always a real function algebra on $(X,\tau)$ and in some sense it is to real function algebras as $C_{\mathbb{C}}(X)$ is to complex uniform algebras.
\item[(ii)]
Let $X$ be a compact Hausdorff space and $Y$ a closed non-empty subset of $X$. Then $C_{Y}:=\{f\in C_{\mathbb{C}}(X):f(Y)\subseteq\mathbb{R}\}$ is a commutative real Banach algebra. As pointed out in \cite[p2]{Kulkarni-Limaye1992}, every such $C_{Y}$ can be transformed into a real function algebra but the converse of this is false. Hence Definition \ref{def:UARefa} is a more general object.
\item[(iii)]
With reference to Definition \ref{def:UARefa} we have $C(X,\tau)=\{f\in C_{\mathbb{C}}(X):\sigma(f)=f\}$ where $\sigma(f):=\bar{f}\circ\tau$. Moreover $\sigma$ is an algebraic involution on $C_{\mathbb{C}}(X)$ and every algebraic involution on $C_{\mathbb{C}}(X)$ arises from a topological involution on $X$ in this way, see \cite[p29]{Kulkarni-Limaye1992} for a proof.
\end{enumerate}
\label{rem:UARefa}
\end{remark}
The following example is a useful standard.
\begin{example}
\label{exa:UARdal}
Recall from Example \ref{exa:UADA} the disc algebra $A(\Delta)$ on the closed unit disc and let $\tau:\Delta\longrightarrow\Delta$ be the map $\tau(z):=\bar{z}$ given by complex conjugation, which we note is a Galois automorphism on $\mathbb{C}$. Now let 
\begin{equation*}
B(\Delta):=A(\Delta)\cap C(\Delta,\tau). 
\end{equation*}
We see that $B(\Delta)$ is complete since both $A(\Delta)$ and $C(\Delta,\tau)$ are, and similarly $B(\Delta)$ contains the real numbers. Further by the definition of $C(\Delta,\tau)$ and the fact that $A(\Delta)=P(\Delta)$ we have that $B(\Delta)$ is the $\mathbb{R}$-algebra of all uniform limits of polynomials on $\Delta$ with real coefficients. Hence $B(\Delta)$ separates the points of $\Delta$ since it contains the function $f(z):=z$. However whilst $\tau$ is in $C(\Delta,\tau)$ it is not an element of $A(\Delta)$. Therefore $B(\Delta)$ is a real function algebra on $(\Delta,\tau)$ and a proper subalgebra of $C(\Delta,\tau)$. It is referred to as the {\em{real disc algebra}}.
\end{example}
Finally for each compact Hausdorff space $X$, $C_{\mathbb{C}}(X)$ can be put into the form of a real function algebra as the following example shows. In particular $\mathbb{C}$ can be expressed as a real function algebra on a two point set.
\begin{example}
\label{exa:UATPS}
Let $X$ be a compact Hausdorff space, let $Y:=\{i,-i\}\subseteq\mathbb{C}$ have the trivial topology and give $X\times Y$ the product topology. We note that the subspace given by $X_{i}:=\{(x,y)\in X\times Y:y=i\}$ is homeomorphic to $X$ and similarly so is $X_{-i}$. Define a topological involution $\tau:X\times Y\rightarrow X\times Y$ by $(x,y)\mapsto (x,\bar{y})$. Then $C_{\mathbb{C}}(X_{i})$ is isometrically isomorphic to $C(X\times Y,\tau)$ by way of the mapping $f\mapsto h_{f}$ where
\begin{equation*}
h_{f}(z):=\begin{cases} f(z) &\mbox{ if }\quad z\in X_{i}\\ \bar{f}(\tau(z)) &\mbox{ if }\quad z\in X_{-i} \end{cases},\quad\mbox{for } f\in C_{\mathbb{C}}(X_{i}),
\end{equation*}
so that for $z\in X_{i}$ we have
\begin{equation*}
h_{f}(\tau(z))=\bar{f}(\tau(\tau(z)))=\bar{f}(z)=\bar{h_{f}}(z)
\end{equation*}
and for $z\in X_{-i}$
\begin{equation*}
h_{f}(\tau(z))=f(\tau(z))=\bar{\bar{f}}(\tau(z))=\bar{h_{f}}(z)
\end{equation*}
showing that $h_{f}\in C(X\times Y,\tau)$. The inverse mapping from $C(X\times Y,\tau)$ to $C_{\mathbb{C}}(X_{i})$ is given by the restriction map $h\mapsto h|_{X_{i}}$. One might suspect that such a mapping exists for every $C(Z,\tau)$ by restricting its elements to a compact subspace of equivalence class representatives for the forward orbits of $\tau$. But this is not the case in general since there can be $z\in Z$ with $\mbox{ord}(\tau,z)=1$ forcing all of the functions to be real valued at $z$ preventing the representation of the complex constants in $C(Z,\tau)$.
\end{example}

	\chapter[Commutative generalisation over complete valued fields]{Commutative generalisation over\\ complete valued fields}
\label{cha:CG}
If $J$ is a maximal ideal of a commutative unital complex Banach algebra $A$ then $J$ has codimension one since $A/J$ with the quotient norm is isometrically isomorphic to the complex numbers. This follows from the Gelfand-Mazur theorem noting that $A/J$ with the quotient norm is unital since $J$ is closed as a subset of $A$ and $J$ is different to $A$, see Lemma \ref{lem:RTCBR} and \cite[p16]{Stout}.\\
In contrast, for a complete non-Archimedean field $F$, if $I$ is a maximal ideal of a commutative unital Banach $F$-algebra then $I$ may have large finite or infinite codimension, note Corollary \ref{cor:CVFEE}. Hence, with Gelfand transform theory in mind, it makes sense to consider non-Archimedean analogs of uniform algebras in the form suggested by real function algebras where the functions take values in a complete extension of the underlying field of scalars. Moreover when there is a lattice of intermediate fields then these fields provide a way for a lattice of extensions of the algebra to occur. See \cite[Ch1]{Berkovich} and \cite[Ch15]{Escassut} for one form of the Gelfand transform in the non-Archimedean setting. This chapter introduces the main definitions of interest in the thesis. We will generalise the definitions made by Kulkarni and Limaye to all complete valued fields, show that the algebras obtained all qualify as generalisations of uniform algebras and that restricting attention to the Archimedean setting recovers the complex uniform algebras and real function algebras. Non-Archimedean examples and residue algebras are also introduced.
\section{Main definitions}
\label{sec:CGMD}
The following definition gives the requirements for those commutative algebras that are to be considered as generalisations of uniform algebras.
\newpage
\begin{definition}
Let $F$ and $L$ be complete valued fields such that $L$ is an extension of $F$ as a valued field. Let $X$ be a compact Hausdorff space and let $C_{L}(X)$ be the unital Banach $L$-algebra of all continuous $L$-valued functions on $X$ with pointwise operations and the sup norm. If a subset $A$ of $C_{L}(X)$ satisfies:
\begin{enumerate}
\item[(i)]
$A$ is an $F$-algebra under pointwise operations;
\item[(ii)]
$A$ is complete with respect to $\|\cdot\|_{\infty}$;
\item[(iii)]
$F\subseteq A$;
\item[(iv)]
$A$ separates the points of $X$,
\end{enumerate}
then we will call $A$ an $^{L}/_{F}$ {\em uniform algebra} or just a {\em uniform algebra} when convenient.
\label{def:CGLFUA}
\end{definition}
In the language of Definition \ref{def:CGLFUA}, an $^{L}/_{F}$ uniform algebra is a Banach $F$-algebra of $L$-valued functions, also every $^{L}/_{L}$ uniform algebra is an $^{L}/_{F}$ uniform algebra. We now generalise, in two parts, Kulkarni and Limaye's definition of a real function algebra.
\begin{definition}
Let $F$ and $L$ be complete valued fields such that $L$ is a finite extension of $F$ as a valued field. Let $X$ be a compact Hausdorff space and totally disconnected if $F$ is non-Archimedean. Define $C(X,\tau,g)\subseteq C_{L}(X)$ as the subset of elements $f\in C_{L}(X)$ for which the diagram in Figure \ref{fig:CGBFA} commutes.
\begin{figure}[h]
\begin{equation*}
\xymatrix{
X\ar@{->}[rr]^{f}\ar@{->}[dd]_{\tau}&&L\ar@{->}[dd]^{g}&&\mbox{(i) $g\in\mbox{Gal}(^{L}/_{F})$;}\hspace{28mm}\\
&&&\mbox{Where:}&\mbox{(ii) $\tau:X\rightarrow X$ with $\mbox{ord}(\tau)|\mbox{ord}(g)$;}\\
X\ar@{->}[rr]_{f}&&L&&\mbox{(iii) $g$ and $\tau$ are continuous.}\hspace{13.5mm}
}
\end{equation*}
\caption{Commutative diagram for $f\in C(X,\tau,g)$.}
\label{fig:CGBFA}
\end{figure}
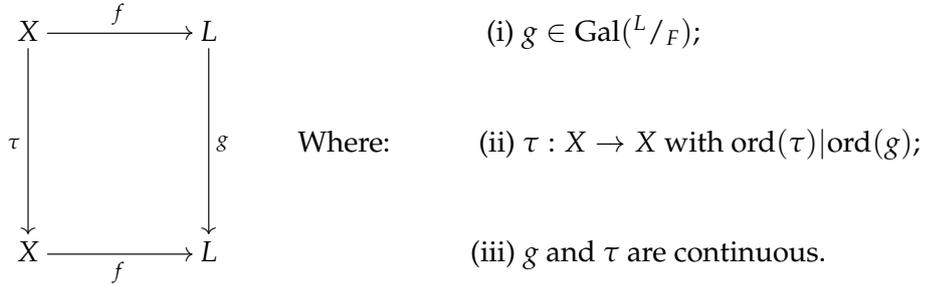

We will call $C(X,\tau,g):=\{f\in C_{L}(X): f(\tau(x))=g(f(x)) \mbox{ for all } x\in X \}$ the {\em basic $^{L}/_{L^{g}}$ function algebra on $(X,\tau,g)$} or just a {\em basic function algebra} when convenient.
\label{def:CGBFA}
\end{definition}
\begin{definition}
Let $F$ and $L$ be complete valued fields such that $L$ is a finite extension of $F$ as a valued field. Let $(X,\tau,g)$ conform to the conditions of Definition \ref{def:CGBFA} and let $A$ be a subset of the basic $^{L}/_{L^{g}}$ function algebra on $(X,\tau,g)$. If $A$ is also an $^{L}/_{L^{g}}$ uniform algebra then we will call $A$ an {\em $^{L}/_{L^{g}}$ function algebra on $(X,\tau,g)$}.
\label{def:CGLGFA}
\end{definition}
\begin{remark}
In definitions \ref{def:CGBFA} and \ref{def:CGLGFA} the valued field $L^{g}$ is complete by Remark \ref{rem:CVFEN}. The continuity of $g$ in Definition \ref{def:CGBFA} is only an observation since $g$ is an isometry on $L$ by Remark \ref{rem:CVFUT}. In fact $g$ also acts as an isometric automorphism on $C(X,\tau,g)$.
\label{rem:CGBFA}
\end{remark}
\section{Generalisation theorems}
\label{sec:CGGT}
With Definition \ref{def:CGLGFA} in mind the following theorem, which is the main theorem of this chapter, clarifies why an algebra conforming to the conditions of Definition \ref{def:CGBFA} is to be called a basic $^{L}/_{L^{g}}$ function algebra on $(X,\tau,g)$.
\begin{theorem}
Let $(X,\tau,g)$ conform to the conditions of Definition \ref{def:CGBFA}. Then the basic $^{L}/_{L^{g}}$ function algebra on $(X,\tau,g)$ is always an $^{L}/_{L^{g}}$ uniform algebra.
\label{thr:CGGUA}
\end{theorem}
\begin{remark}
We will see in the proof of Theorem \ref{thr:CGGUA} that ord$(\tau)|$ord$(g)$ is an optimum condition in Definition \ref{def:CGBFA} since if we do not include it in the definition then $C(X,\tau,g)$ separates the points of $X$ if and only if ord$(\tau)|$ord$(g)$ as per Figure \ref{fig:CGED}.
\label{rem:CGOC}
\end{remark}
\begin{figure}[h]
\begin{equation*}
\xymatrix{
&\mbox{ord}(\tau)|\mbox{ord}(g)\ar@2{->}[ld]_{1}&\\
\mbox{ord}(\tau,X)\subseteq\mbox{ord}(g,L)\ar@2{->}[rr]^{2}&&C(X,\tau,g)\mbox{ separates }X\ar@2{->}[lu]_{3}
}
\end{equation*}
\caption{Equivalence diagram for Definition \ref{def:CGBFA}.}
\label{fig:CGED}
\end{figure}
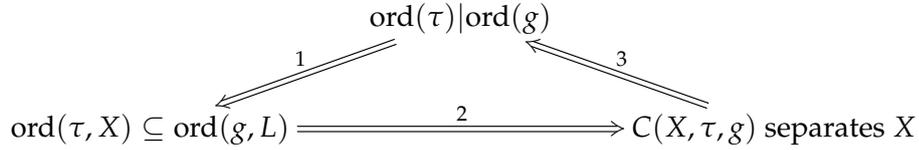
\begin{proof}[Proof of Theorem \ref{thr:CGGUA}]
Let $(X,\tau,g)$ conform to the conditions of Definition \ref{def:CGBFA}. It is immediate that $C(X,\tau,g)$ is a ring under pointwise operations and $L^{g}\subseteq C(X,\tau,g)$. We now show that $C(X,\tau,g)$ is complete with respect to the sup norm. First note that
\begin{equation*}
C(X,\tau,g)=\{f\in C_{L}(X):\sigma(f)=f\}
\end{equation*}
where $\sigma(f):=g^{(\mbox{ord}(g)-1)}\circ f\circ\tau$ is an isometry on $C_{L}(X)$ since $\tau$ is surjective and $g$ is an isometry on $L$ by Remark \ref{rem:CVFUT}. Further $\sigma$ is either an algebraic involution or a algebraic element of finite order on $C_{L}(X)$. Hence since $C_{L}(X)$ is commutative $\sigma$ is in fact an isometric automorphism on $C_{L}(X)$. Now let $(f_{n})$ be a Cauchy sequence in $C(X,\tau,g)$ and let $f$ be its limit in $C_{L}(X)$. Then for each $\varepsilon>0$ there exists $N\in\mathbb{N}$ such that for all $n>N$ we have $\|f-f_{n}\|_{\infty}=\|\sigma(f_{n})-\sigma(f)\|_{\infty}<\frac{\varepsilon}{2}$. But then $\|f-\sigma(f)\|_{\infty}=\|f-\sigma(f_{n})+\sigma(f_{n})-\sigma(f)\|_{\infty}\leq \|f-f_{n}\|_{\infty}+\|\sigma(f_{n})-\sigma(f)\|_{\infty}<\varepsilon$. Hence $\|f-\sigma(f)\|_{\infty}=0$ giving $\sigma(f)=f$ and so $f\in C(X,\tau,g)$. Hence $C(X,\tau,g)$ is complete. It remains to show that $C(X,\tau,g)$ separates the points of $X$ and to this end we now show each of the implications in Figure \ref{fig:CGED}.\\
{\em{1:}} Let $n\in\mbox{ord}(\tau,X)$. It is immediate that $n|\mbox{ord}(\tau)$ and since $\mbox{ord}(\tau)|\mbox{ord}(g)$ we have $n\in\mbox{ord}(g,L)$ by Lemma \ref{lem:CVFOSET}. We also note that the converse is immediate since for each $n\in\mbox{ord}(g,L)$ we have $n|\mbox{ord}(g)$ and so $\mbox{ord}(\tau)|\mbox{ord}(g)$.\\
{\em{2:}} Note that $\mbox{ord}(\sigma)=\mbox{ord}(g)$ and so, like a norm map, for every $h\in C_{L}(X)$ we have
\begin{align*}
&h\sigma(h)\sigma^{(2)}(h)\cdots\sigma^{(\mbox{ord}(g)-1)}(h)\in C(X,\tau,g)\quad\mbox{and}\\
&h+\sigma(h)+\sigma^{(2)}(h)+\cdots+\sigma^{(\mbox{ord}(g)-1)}(h)\in C(X,\tau,g).
\end{align*}
Now if $g=\mbox{id}$ is the identity then $C(X,\tau,g)=C_{L}(X)$ which separates the points of $X$ when $L$ is Archimedean by Urysohn's lemma, since $X$ is locally compact, and when $L$ is non-Archimedean by Theorem \ref{thr:UAXtop} since we required $X$ to be totally disconnected in this case. So now suppose $\mbox{ord}(g)>1$ and let $x,y\in X$ with $x\not=y$. We need to check two cases.\\
{\em{Case 1:}} In this case $y\not=\tau^{(n)}(x)$ for all $n\in\mathbb{N}$ with $n\leq\mbox{ord}(\tau,x)$. By Urysohn's lemma in the Archimedean setting or Lemma \ref{lem:UAUrys} otherwise, there is $h\in C_{L}(X)$ with $h(y)=0$ and $h(\tau^{(n)}(x))=1$ for all $n\in\mathbb{N}_{0}$. Let $f:=h\sigma(h)\sigma^{(2)}(h)\cdots\sigma^{(\mbox{ord}(g)-1)}(h)$ so that $f\in C(X,\tau,g)$ with $f(y)=0$, by construction, and $f(x)=1$ by construction given that $g(1)=1$. Then in this case $x$ and $y$ are separated by $f$.\\
{\em{Case 2:}} In this case $y=\tau^{(n)}(x)$ for some $n\in\mathbb{N}$ with $n<\mbox{ord}(\tau,x)$. Let $m:=\mbox{ord}(g)$ and $k:=\mbox{ord}(\tau,x)$ and note therefore that we have $1\leq n\leq k-1$, since $y\not=x$, and $m=km'$ for some $m'\in\mathbb{N}$. Further since $\mbox{ord}(\tau,X)\subseteq\mbox{ord}(g,L)$ there is $a\in L$ with $\mbox{ord}(g,a)=k$. By Urysohn's lemma in the Archimedean setting or Lemma \ref{lem:UAUrys} otherwise, there is $h\in C_{L}(X)$ with $h(x)=a$ and $h(\tau^{(i)}(x))=0$ for $1\leq i\leq k-1$. We will now check two sub-cases.\\
{\em{Case 2.1:}} The characteristic of the field $L$ is zero, i.e. $\mbox{char}(L)=0$.\\
Let $f:=h+\sigma(h)+\sigma^{(2)}(h)+\cdots+\sigma^{(\mbox{ord}(g)-1)}(h)$ so that we have $f\in C(X,\tau,g)$ with $f=h+g^{(m-1)}\circ h\circ\tau+g^{(m-2)}\circ h\circ\tau^{(2)}+\cdots+g\circ h\circ\tau^{(m-1)}$. This gives
\begin{align*}
f(x)=&h(x)+g^{(m-k)}\circ h(\tau^{(k)}(x))+g^{(m-2k)}\circ h(\tau^{(2k)}(x))+\cdots\\
     &\cdots+g^{(m-(m'-1)k)}\circ h(\tau^{((m'-1)k)}(x))\\
    =&h(x)+g^{((m'-1)k)}\circ h(\tau^{(k)}(x))+g^{((m'-2)k)}\circ h(\tau^{(2k)}(x))+\cdots\\
     &\cdots+g^{(k)}\circ h(\tau^{((m'-1)k)}(x))\\
    =&a+a+a+\cdots+a,\quad m'\mbox{ times,}\\
    =&m'a\quad\mbox{and}
\end{align*}
\begin{align*}
f(y)=&f(\tau^{(n)}(x))\\
    =&g^{(m-(k-n))}\circ h(\tau^{(k)}(x))+g^{(m-(2k-n))}\circ h(\tau^{(2k)}(x))+\cdots\\
     &\cdots+g^{(m-((m'-1)k-n))}\circ h(\tau^{((m'-1)k)}(x))+g^{(m-(m'k-n))}\circ h(\tau^{(m'k)}(x))\\
    =&g^{((m'-1)k+n)}\circ h(\tau^{(k)}(x))+g^{((m'-2)k+n)}\circ h(\tau^{(2k)}(x))+\cdots\\
     &\cdots+g^{(k+n)}\circ h(\tau^{((m'-1)k)}(x))+g^{(n)}\circ h(\tau^{(m)}(x))\\
    =&m'g^{(n)}(a)\quad\mbox{with }1\leq n\leq k-1.
\end{align*}
Hence since $\mbox{ord}(g,a)=k$ we have $f(x)\not=f(y)$.\\
{\em{Case 2.2:}} The characteristic of $L$ is $p$, i.e. $\mbox{char}(L)=p$, for some prime $p\in\mathbb{N}$. In this case the proof for Case 2.1 breaks down when $m'=p^{t}s$ with $s,t\in\mathbb{N}$, $p\nmid s$. So with respect to such circumstances define $f':=h\sigma^{(sk)}(h)\sigma^{(2sk)}(h)\cdots\sigma^{((p^{t}-1)sk)}(h)$ and
\begin{equation*}
f:=f'+\sigma(f')+\sigma^{(2)}(f')+\cdots+\sigma^{(sk-1)}(f').
\end{equation*}
We will now show that $\sigma(f)=f$ so that $f\in C(X,\tau,g)$ and note that this is satisfied if $\sigma^{(sk)}(f')=f'$ since $\sigma(f)=\sigma(f')+\sigma^{(2)}(f')+\sigma^{(3)}(f')+\cdots+\sigma^{(sk)}(f')$. Indeed we have that $\sigma^{(sk)}(f')=\sigma^{(sk)}(h)\sigma^{(2sk)}(h)\sigma^{(3sk)}(h)\cdots\sigma^{(p^{t}sk)}(h)$ with $\sigma^{(p^{t}sk)}(h)=\sigma^{(m)}(h)=h$ so that $\sigma^{(sk)}(f')=f'$ giving $f\in C(X,\tau,g)$. Now, for $0\leq i\leq sk-1$, we have
\begin{align*}
\sigma^{(i)}(f')(x)=&\sigma^{(i)}(h)(x)\sigma^{(sk+i)}(h)(x)\sigma^{(2sk+i)}(h)(x)\cdots\sigma^{((p^{t}-1)sk+i)}(h)(x)\\
                   =&\begin{cases} a^{p^{t}} &\mbox{ if }\quad k|i\\ 0 &\mbox{ if }\quad k\nmid i \end{cases}
\end{align*}
since $\sigma^{(kj)}(h)(x)=g^{(m-kj)}\circ h(\tau^{(kj)}(x))=g^{((p^{t}s-j)k)}\circ h(x)=a$ for $kj<m$, $j\in\mathbb{N}_{0}$, and $\sigma^{(j)}(h)(x)=g^{(m-j)}\circ h(\tau^{(j)}(x))=g^{(m-j)}(0)=0$ for $j<m$, $k\nmid j$. Hence
\begin{align*}
f(x)=&f'(x)+\sigma(f')(x)+\sigma^{(2)}(f')(x)+\cdots+\sigma^{(sk-1)}(f')(x)\\
    =&f'(x)+\sigma^{(k)}(f')(x)+\sigma^{(2k)}(f')(x)+\cdots+\sigma^{((s-1)k)}(f')(x)\\
    =&sa^{p^{t}}.
\end{align*}
But for $0\leq i\leq sk-1$ we also have
\begin{align*}
\sigma^{(i)}(f')(y)=&\sigma^{(i)}(h)(y)\sigma^{(sk+i)}(h)(y)\sigma^{(2sk+i)}(h)(y)\cdots\sigma^{((p^{t}-1)sk+i)}(h)(y)\\
                   =&\begin{cases} g^{(n)}(a)^{p^{t}} &\mbox{ if }\quad k|(i+n)\\ 0 &\mbox{ if }\quad k\nmid(i+n) \end{cases}
\end{align*}
since if $k|(i+n)$ then $i$ has the form $kj-n$ and the exponents of $\sigma$ therefore also have the form $kj-n<m$, $j\in\mathbb{N}$, giving
\begin{align*}
\sigma^{(kj-n)}(h)(y)=&g^{(m-(kj-n))}\circ h(\tau^{(kj-n)}(y))\\
                     =&g^{((p^{t}s-j)k+n)}\circ h(\tau^{(kj-n)}(\tau^{(n)}(x)))\\
                     =&g^{((p^{t}s-j)k+n)}\circ h(x)=g^{(n)}(a)
\end{align*}
and if $k\nmid(i+n)$ then for $j<m$ an exponents of $\sigma$ we also have $k\nmid(j+n)$ giving
\begin{align*}
\sigma^{(j)}(h)(y)=&g^{(m-j)}\circ h(\tau^{(j)}(y))\\
                  =&g^{(m-j)}\circ h(\tau^{(j+n)}(x))\\
                  =&g^{(m-j)}(0)=0.
\end{align*}
Hence
\begin{align*}
f(y)=&f'(y)+\sigma(f')(y)+\sigma^{(2)}(f')(y)+\cdots+\sigma^{(sk-1)}(f')(y)\\
    =&\sigma^{(k-n)}(f')(y)+\sigma^{(2k-n)}(f')(y)+\cdots+\sigma^{(sk-n)}(f')(y)\\
    =&sg^{(n)}(a)^{p^{t}}.
\end{align*}
Now since $p\nmid s$ we have $s\in L^{\times}$. Furthermore recall that since $\mbox{ord}(g,a)=k$ and $1\leq n\leq k-1$ we have $g^{(n)}(a)\not=a$. Therefore it remains to show that $g^{(n)}(a)^{p^{t}}\not=a^{p^{t}}$ in order to conclude that $f(y)\not=f(x)$. Recall that $p=\mbox{char}(L)\in\mathbb{N}$ is a prime. For $b\in L$ the Frobenius $\mbox{Frob}:L\rightarrow L$, $\mbox{Frob}(b):=b^{p}$, is an injective endomorphism on $L$. We show that the Frobenius is injective on $L$. Let $b_{1},b_{2}\in L$ with $b_{1}^{p}=b_{2}^{p}$.\\
The case $p>2$ gives $(b_{1}-b_{2})^{p}=b_{1}^{p}-b_{2}^{p}=0$.\\
The case $p=2$ gives $(b_{1}-b_{2})^{2}=b_{1}^{2}+b_{2}^{2}=2b_{1}^{2}=0$.\\
In each case $L$ is an integral domain and so $b_{1}-b_{2}=0$ giving $b_{1}=b_{2}$ as required. Therefore $\mbox{Frob}^{(t)}:L\rightarrow L$, $\mbox{Frob}^{(t)}(b):=b^{p^{t}}$, is also injective giving $g^{(n)}(a)^{p^{t}}\not=a^{p^{t}}$ since $g^{(n)}(a)\not=a$ and this finishes the proof of implications 2 in Figure \ref{fig:CGED}.\\
{\em{3:}} We now show implication 3 in Figure \ref{fig:CGED} by showing the contrapositive. Suppose $\mbox{ord}(\tau)\nmid\mbox{ord}(g)$. Then there exists some $x\in X$ such that $\tau^{(\mbox{ord}(g))}(x)\not=x$. Let $y:=\tau^{(\mbox{ord}(g))}(x)$. Now for all $f\in C(X,\tau,g)$ we have for all $i\in\mathbb{N}$ that
\begin{equation*}
f\circ\tau^{(i)}=f\circ\tau\circ\tau^{(i-1)}=g\circ f\circ\tau^{(i-1)}=\cdots=g^{(i)}\circ f.
\end{equation*}
Therefore $f(y)=f\circ\tau^{(\mbox{ord}(g))}(x)=g^{(\mbox{ord}(g))}\circ f(x)=f(x)$. Hence for all $f\in C(X,\tau,g)$ we have $f(x)=f(y)$ as required.\\
Hence having shown each of the implications in Figure \ref{fig:CGED}, $\mbox{ord}(\tau)|\mbox{ord}(g)$ is a necessary and sufficient condition in Definition \ref{def:CGBFA} in order that $C(X,\tau,g)$ separates the points of $X$ and this completes the proof of Theorem \ref{thr:CGGUA}.
\end{proof}
\begin{remark}
It's worth noting that for $\mbox{char}(L)=p$ the Frobenius is also a endomorphism on $C(X,\tau,g)$. Moreover for $L$ of any characteristic we have seen in the proof of Theorem \ref{thr:CGGUA} that $\sigma$, given by $\sigma(f):=g^{(\mbox{ord}(g)-1)}\circ f\circ\tau$, is an isometric automorphism on $C_{L}(X)$ with fixed elements $C(X,\tau,g)$.
\label{rem:CGFROB}
\end{remark}
With reference to the complex, real and non-Archimedean Stone-Weierstrass theorems from Chapter \ref{cha:UA}, the following combined Stone-Weierstrass theorem is immediate.
\begin{theorem}
Let $L$ be a complete valued field. Let $X$ conform to Definition \ref{def:CGBFA} and let $A$ be an $^{L}/_{L}$ function algebra on ($X$,id,id). Then either $A=C_{L}(X)$ or $L=\mathbb{C}$ and $A$ is not self adjoint, that is there is $f\in A$ with $\bar{f}\notin A$.
\label{thr:CGGSW}
\end{theorem}
\section{Examples}
\label{sec:CGEXA}
Our first example considers $^{L}/_{L^{g}}$ function algebras in the Archimedean setting.
\begin{example}
Let $F=\mathbb{R}$, $L=\mathbb{C}$ and $X$ be a compact Hausdorff space. We have $\mbox{Gal}(^{\mathbb{C}}/_{\mathbb{R}})=\langle\mbox{id}, \bar{z}\rangle$.\\
Setting $g=\mbox{id}$ in Definition \ref{def:CGBFA} forces $\tau$ to be the identity on $X$. In this case it's immediate that $C(X,\tau,g)=C_{\mathbb{C}}(X)$ and each $^{L}/_{L^{g}}$ function algebra on $(X,\tau,g)$ is a complex uniform algebra.\\
On the other hand, setting $g=\bar{z}$ forces $\tau$ to be a topological involution on $X$. In this case the $^{L}/_{L^{g}}$ function algebras on $(X,\tau,g)$ are precisely the real function algebras of Kulkarni and Limaye.
\label{exa:CGAEFA}
\end{example}
Our first non-Archimedean example is very straightforward involving the trivial valuation.
\begin{example}
\label{exa:CGTVFA}
Let $F=\mathbb{Q}$, but with the trivial valuation instead of the absolute valuation, and let $L=\mathbb{Q}(a)$ with the trivial valuation where $a=\exp(\frac{1}{10}2\pi i)$. With reference to Theorem \ref{thr:CVFIR}, and having factorised $x^{10}-1$ in $F[x]$, we have $\mbox{Irr}_{F,a}(x)=x^{4}-x^{3}+x^{2}-x+1$ which gives $[L,F]=\mbox{degIrr}_{F,a}(x)=4$. The roots of $\mbox{Irr}_{F,a}(x)$ are the elements of $S:=\{a,a^{3},a^{7},a^{9}\}$ and so, with reference to Definition \ref{def:CVFSN}, $L$ is a normal extension of $F$. Moreover with reference to Remark \ref{rem:CVFSN} $L$ is a separable extension of $F$ and so $L$ is also a Galois extension of $F$ with $\#\mbox{Gal}(^{L}/_{F})=[L,F]=4$ by Theorem \ref{thr:CVFGL}. Now for $g\in\mbox{Gal}(^{L}/_{F})$ we must have $g:S\rightarrow S$ since, for $b\in S$, $0=g(0)=g(\mbox{Irr}_{F,a}(b))=\mbox{Irr}_{F,a}(g(b))$. Putting $g(a):=a^{3}$ makes $g$ a generator of $\mbox{Gal}(^{L}/_{F})$ and we have
\begin{equation*}
g(a)=a^{3},\quad g^{(2)}(a)=a^{9},\quad g^{(3)}(a)=a^{7}\mbox{ and }g^{(4)}(a)=a.
\end{equation*}
Hence $L$ is a cyclic extension of $F$ meaning that $\mbox{Gal}(^{L}/_{F})$ is a cyclic group. Moreover
\begin{align*}
(a+a^{9}-a^{3}-a^{7})^{2}=&4-a^{2}-a^{4}-a^{6}-a^{8}\\
=&4-(a^{4}-a^{3}+a^{2}-a)\\
=&4-(\mbox{Irr}_{F,a}(a)-1)=5
\end{align*}
giving $a+a^{9}-a^{3}-a^{7}=\sqrt{5}$ noting that the real part of each term is positive. Further
\begin{align*}
g(\sqrt{5})=&g(a+a^{9}-a^{3}-a^{7})\\
=&g(a+g^{(2)}(a)-g(a)-g^{(3)}(a))\\
=&g(a)+g^{(3)}(a)-g^{(2)}(a)-a=-\sqrt{5}
\end{align*}
and so we have the intermediate field $\mathbb{Q}(a)^{\langle g^{(2)}\rangle}=\mathbb{Q}(\sqrt{5})$. Now let $S_{1}\subseteq\mathbb{N}\times\{1\}$, $S_{2}\subseteq\mathbb{N}\times\{\sqrt{5},-\sqrt{5}\}$, $S_{3}\subseteq\mathbb{N}\times\{a,a^{3},a^{7},a^{9}\}$ and $X:=S_{1}\cup S_{2}\cup S_{3}$ all be non-empty finite sets such that for $(x,y)\in X$ we have $(x,g(y))\in X$. Put the trivial topology on $X$ so that $X$ is a totally disconnected compact Hausdorff space noting that a set with the trivial topology is compact if and only if it is finite. Define a topological element of finite order $\tau$ on $X$ by $\tau((x,y)):=(x,g(y))$ and note that for our choice of topology every self map on $X$ is continuous as is every map from $X$ to $L$. Hence $C_{L}(X)$ is the $^{\mathbb{Q}(a)}/_{\mathbb{Q}}$ uniform algebra of all functions from $X$ to $L$ and we also have $\mbox{ord}(\tau)|\mbox{ord}(g)$ by construction. Hence with reference to Definition \ref{def:CGBFA} we have $C(X,\tau,g)$ as an example of a basic $^{\mathbb{Q}(a)}/_{\mathbb{Q}}$ function algebra. For $z\in X$ each $f\in C(X,\tau,g)$ is such that
\begin{align*}
\hspace{4cm}&f(z)\in\mathbb{Q}(a)        &\mbox{if }&z\in S_{3},\\
            &f(z)\in\mathbb{Q}(\sqrt{5}) &\mbox{if }&z\in S_{2}\mbox{ and}\\
            &f(z)\in\mathbb{Q}           &\mbox{if }&z\in S_{1}\mbox{ since}\hspace{4cm}
\end{align*}
$f(z)=f(\tau^{(\mbox{ord}(\tau,z))}(z))=g^{(\mbox{ord}(\tau,z))}(f(z))$ giving $\mbox{ord}(g,f(z))|\mbox{ord}(\tau,z)$. Furthermore $C(X,\tau,g)$ extends to $C(X,\tau^{(2)},g^{(2)})$ which is a basic $^{\mathbb{Q}(a)}/_{\mathbb{Q}(\sqrt{5})}$ function algebra. We will look at such extensions in the next section. Finally note that in general if we use the trivial valuation on $L$ then for every totally disconnected compact Hausdorff space $X$ the sup norm $\|\cdot\|_{\infty}$ is the trivial norm on $C_{L}(X)$.
\end{example}
We now look at some non-Archimedean examples involving an order two extension of the 5-adic numbers.
\begin{example}
\label{exa:CGEXONE}
Let $F:=\mathbb{Q}_{5}$ and $L:=\mathbb{Q}_{5}(\sqrt{2})$. Suppose towards a contradiction that $\sqrt{2}$ is already an element of $\mathbb{Q}_{5}$. With reference to Chapter \ref{cha:CVF}, we would have $1=|2|_{5}=|\sqrt{2}^{2}|_{5}=|\sqrt{2}|_{5}^{2}$ giving $|\sqrt{2}|_{5}=1$. But then $\sqrt{2}$ would have a 5-adic expansion over $\mathcal{R}_{5}:=\{0,1,\cdots,4\}$ of the form $\sum_{i=0}^{\infty}a_{i}5^{i}$ with $a_{0}\not=0$. Hence
\begin{equation}
a_{0}^{2} +2a_{0}\sum_{i=1}^{\infty}a_{i}5^{i}+\left(\sum_{i=1}^{\infty}a_{i}5^{i}\right)^{2}
\label{equ:CGsqrt}
\end{equation}
should be equal to 2. In particular $a_{0}^{2}$ should have the form $2+b$ where $b$ is a natural number, with a factor of 5, that cancels with the remaining terms of (\ref{equ:CGsqrt}). But since $a_{0}\in\{1,2,3,4\}$ we have $a_{0}^{2}\in\{1,4,4+5,1+3\cdot5\}$, a contradiction. Therefore we have $\mbox{Irr}_{F,\sqrt{2}}(x)=x^{2}-2$ giving $[L,F]=2$ and so $\mathbb{Q}_{5}(\sqrt{2})=\mathbb{Q}_{5}\oplus\sqrt{2}\mathbb{Q}_{5}$ as a $\mathbb{Q}_{5}$-vector space. It is immediate that $L$ is a Galois extension of $F$ with $\mbox{Gal}(^{L}/_{F})=\langle\mbox{id}, g\rangle$ where $g(\sqrt{2})=-\sqrt{2}$. The complete valuation on $F$ has a unique extension to a complete valuation on $L$, see Theorem \ref{thr:CVFEE}. Explicitly we have, for all $a\in L$,
\begin{equation*}
|a|_{L}=\sqrt{|a|_{L}|g(a)|_{L}}=\sqrt{|ag(a)|_{L}}=\sqrt{|ag(a)|_{5}},
\end{equation*}
noting that $ag(a)\in F$. Moreover in terms of the valuation logarithm $\nu_{5}$ on $F$ we have $\sqrt{|ag(a)|_{5}}=5^{-\frac{1}{2}\nu_{5}(ag(a))}$ and so the valuation logarithm for $L$ is $\omega(a):=\frac{1}{2}\nu_{5}(ag(a))$ for $a\in L$. Hence for $a+\sqrt{2}b\in L^{\times}$, with $a,b\in F$, we have $\omega(a+\sqrt{2}b)=\frac{1}{2}\nu_{5}(a^{2}-2b^{2})$. We show that this in fact gives
\begin{equation}
\omega(a+\sqrt{2}b)=\min\{\nu_{5}(a),\nu_{5}(b)\}.
\label{equ:CGINF}
\end{equation}
First recall that $\nu_{5}(0)=\infty$. If $b=0$ then $\omega(a)=\frac{1}{2}\nu_{5}(a^{2})=\frac{1}{2}2\nu_{5}(a)=\nu_{5}(a)$.\\
If $a=0$ then $\omega(\sqrt{2}b)=\frac{1}{2}\nu_{5}(-2b^{2})=\frac{1}{2}(\nu_{5}(-2)+2\nu_{5}(b))=\frac{1}{2}2\nu_{5}(b)=\nu_{5}(b)$, noting that $\nu_{5}(-2)=0$ since $-2=3+\sum_{i=1}^{\infty}4\cdot5^{i}$.\\
If $a,b\in F^{\times}$ and $\nu_{5}(a)\not=\nu_{5}(b)$ then by the above $\nu_{5}(a^{2})\not=\nu_{5}(-2b^{2})$. Hence, by Lemma \ref{lem:CVFEQ}, $\omega(a+\sqrt{2}b)=\frac{1}{2}\nu_{5}(a^{2}-2b^{2})=\frac{1}{2}\min\{\nu_{5}(a^{2}),\nu_{5}(-2b^{2})\}=\min\{\nu_{5}(a),\nu_{5}(b)\}$.\\
If $a,b\in F^{\times}$ and $\nu_{5}(a)=\nu_{5}(b)=n$ for some $n\in\mathbb{Z}$ then the expansion $a=\sum_{i=n}^{\infty}a_{i}5^{i}$ over $\mathcal{R}_{5}$ has $a_{n}\in\{1,2,3,4\}$ and so the expansion $a^{2}=\sum_{i=2n}^{\infty}a'_{i}5^{i}$ has $\overline{a'_{2n}}=\overline{a_{n}^{2}}$ in the residue field $\overline{F}=\mathbb{F}_{5}$ giving $a'_{2n}\in\{1,4\}$. Similarly the expansion $b=\sum_{i=n}^{\infty}b_{i}5^{i}$ has $b_{n}\in\{1,2,3,4\}$ and so $-2b^{2}=\left(3+\sum_{i=1}^{\infty}4\cdot5^{i}\right)\left(\sum_{i=n}^{\infty}b_{i}5^{i}\right)^{2}=\sum_{i=2n}^{\infty}b'_{i}5^{i}$ has $\overline{b'_{2n}}=\overline{3b_{n}^{2}}$ in $\overline{F}$ giving $b'_{2n}\in\{2,3\}$. Hence the expansion $a^{2}-2b^{2}=\sum_{i=2n}^{\infty}c_{i}5^{i}$ has $\overline{c_{2n}}=\overline{a'_{2n}+b'_{2n}}$ in $\overline{F}$ giving $c_{2n}\in\{1,2,3,4\}$. In particular $c_{2n}\not=0$ and so
\begin{equation*}
\omega(a+\sqrt{2}b)=\frac{1}{2}\nu_{5}(a^{2}-2b^{2})=\frac{1}{2}\nu_{5}(\sum_{i=2n}^{\infty}c_{i}5^{i})=\frac{1}{2}2n=n
\end{equation*}
and this completes the proof of (\ref{equ:CGINF}). With reference to Remark \ref{rem:CVFEE} it follows that $L$ is an unramified extension of $F$ with $|a+\sqrt{2}b|_{L}=\max\{|a|_{L},|b|_{L}\}=\max\{|a|_{F},|b|_{F}\}$ for $a,b\in F$. Further it follows easily from (\ref{equ:CGINF}) that $\mathcal{R}_{L}:=\{a+\sqrt{2}b:a,b\in\{0,\cdots,4\}\}$ is a set of representatives in $L$ of the elements in the residue field $\overline{L}$. Hence $\overline{L}=\mathbb{F}_{25}$ since $\#\mathcal{R}_{L}=25$ and $[\overline{L},\overline{F}]=2$. Therefore, by Theorem \ref{thr:CVFHB}, $L$ is locally compact and the unit ball $\Delta_{L}:=\{x\in L:|x|_{L}\leq1\}=\{x\in L:\omega(x)\geq0\}$ is a totally disconnected compact Hausdorff space with respect to $|\cdot|_{L}$. Further if we take $\tau_{1}$ to be the restriction of $g$ to $\Delta_{L}$ then, since $g$ is an isometry on $L$, $\tau_{1}$ is a topological involution on $\Delta_{L}$ and the basic $^{L}/_{F}$ function algebra
\begin{equation*}
C(\Delta_{L},\tau_{1},g)=\{f\in C_{L}(\Delta_{L}):f(\tau_{1}(x))=g(f(x))\mbox{ for all }x\in\Delta_{L}\} 
\end{equation*}
is a non-Archimedean analog of the real disc algebra. Now let $f(x)=\sum_{n=0}^{\infty}a_{n}x^{n}$ be a power series in $C(\Delta_{L},\tau_{1},g)$. Then for $x\in\Delta_{L}$ and $\sigma$ from Remark \ref{rem:CGFROB} we have
\begin{equation*}
\sum_{n=0}^{\infty}a_{n}x^{n}=f(x)=\sigma(f)(x)=g\left(\sum_{n=0}^{\infty}a_{n}g(x)^{n}\right)=\sum_{n=0}^{\infty}g(a_{n})x^{n}
\end{equation*}
where the last equality follows because the two series have identical sequences of partial sums. Hence similarly we have, for $x\in\Delta_{L}$, $\sum_{n=0}^{\infty}(a_{n}-g(a_{n}))x^{n}=0$. In the general case of such circumstance we can not immediately assume that all the pairs of coefficients $a_{n}$ and $g(a_{n})$ are equal since $\Delta_{L}$ could be a set of roots of the series $\sum_{n=0}^{\infty}(a_{n}-g(a_{n}))x^{n}$ whilst there being an element of $L$ in the region of convergence of the series that is not a root. However since $0\in\Delta_{L}$ we have $a_{0}=g(a_{0})$. Now let $m\in\mathbb{N}$ be such that for all $i\in\mathbb{N}_{0}$ with $i<m$ we have $a_{i}=g(a_{i})$. Then $x^{m}\sum_{n=m}^{\infty}(a_{n}-g(a_{n}))x^{n-m}=0$ on $\Delta_{L}$ and $\sum_{n=m}^{\infty}(a_{n}-g(a_{n}))x^{n-m}=0$ on $\Delta_{L}\backslash\{0\}$. Let $\rho$ be the radius of convergence of $\sum_{n=m}^{\infty}(a_{n}-g(a_{n}))x^{n-m}$. Then with reference to Theorem \ref{thr:FAACP}, since $1\in\Delta_{L}\backslash\{0\}$ with $|1|_{L}=1$ and $\sum_{n=m}^{\infty}(a_{n}-g(a_{n}))x^{n-m}$ converges on $\Delta_{L}\backslash\{0\}$ we have $\rho\geq1$. Hence $\sum_{n=m}^{\infty}(a_{n}-g(a_{n}))x^{n-m}$ converges uniformly on the ball $\bar{B}_{\frac{1}{5}}(0)=\{x\in L:\omega(x)\geq 1\}$ by Theorem \ref{thr:FAACP}. Therefore $\sum_{n=m}^{\infty}(a_{n}-g(a_{n}))x^{n-m}$ is continuous on $\bar{B}_{\frac{1}{5}}(0)$ and so $\sum_{n=m}^{\infty}(a_{n}-g(a_{n}))x^{n-m}=0$ at $0\in\bar{B}_{\frac{1}{5}}(0)$. Hence $a_{m}=g(a_{m})$ and by induction we have shown that $a_{n}=g(a_{n})$ for all $n\in\mathbb{N}_{0}$. In particular all the power series in $C(\Delta_{L},\tau_{1},g)$ only have $F$ valued coefficients. However since $\Delta_{L}\not\subseteq F$ these functions take values in $L$.
\end{example}
Whilst the last example of a basic function algebra included many globally analytic functions the only globally analytic functions in the following example are constants. However many locally analytic functions are included.
\begin{example}
\label{exa:CGEXTWO}
Let $F$, $L$, $\Delta_{L}$, $\omega$ and $g$ be as in Example \ref{exa:CGEXONE} and therefore note that $\omega|_{\Delta_{L}}:\Delta_{L}\rightarrow\mathbb{N}_{0}\cup\{\infty\}$. Define $\tau_{2}(0):=0$ and for $x\in\Delta_{L}\backslash\{0\}$, 
\begin{equation}
\tau_{2}(x):= \left\{ \begin{array} {l@{\quad\mbox{if}\quad}l}
5x & 2\mid\omega(x) \\
5^{-1}x & 2\nmid\omega(x).
\end{array} \right.
\label{equ:CGTAUTWO}
\end{equation}
Let $x\in\Delta_{L}$ with $\omega(x)\in\mathbb{N}_{0}$. Then $\omega(\tau_{2}(x))=\omega(5x)=\omega(x)+\omega(5)=\omega(x)+1$ if $2|\omega(x)$. Similarly $\omega(\tau_{2}(x))=\omega(x)-1$ if $2\nmid\omega(x)$. Hence when $\omega(x)\in\mathbb{N}_{0}$ we have $\omega(\tau_{2}(x))\in\mathbb{N}_{0}$ giving $\tau_{2}(x)\in\Delta_{L}$. Further $\tau_{2}:\Delta_{L}\rightarrow\Delta_{L}$ since $\tau_{2}(0)=0$. Moreover $\mbox{ord}(\tau_{2})=2$ and so to show that $\tau_{2}$ is a topological involution on $\Delta_{L}$ it remains to show that $\tau_{2}$ is continuous. Let $x\in\Delta_{L}$ and $(x_{n})$ be a sequence in $\Delta_{L}$ such that $\lim_{n\to\infty}x_{n}=x$. Let $\varepsilon>0$. For $x\not=0$ there exists $N\in\mathbb{N}$ such that for all $n\geq N$ we have $\omega(x_{n})=\omega(x)$ since convergence in $L$ is from the side, see Lemma \ref{lem:FAACS}. With reference to (\ref{equ:CGTAUTWO}) this gives, for all $n\geq N$, $\tau_{2}(x_{n})=\tau_{2}(x_{n})x_{n}^{-1}x_{n}=\tau_{2}(x)x^{-1}x_{n}$. Further since $\lim_{n\to\infty}x_{n}=x$ there exists $M\in\mathbb{N}$ such that, for all $m\geq M$, $|\tau_{2}(x)x^{-1}|_{L}|x-x_{m}|_{L}<\varepsilon$. Hence for all $n\geq\max\{M,N\}$ we have
\begin{equation*}
|\tau_{2}(x)-\tau_{2}(x_{n})|_{L}=|\tau_{2}(x)x^{-1}(x-x_{n})|_{L}=|\tau_{2}(x)x^{-1}|_{L}|x-x_{n}|_{L}<\varepsilon.
\end{equation*}
On the other hand for $x=0$ note that $\omega(\tau_{2}(x_{n}))\geq\omega(x_{n})-1$ for all $n\in\mathbb{N}$. In this case since $\lim_{n\to\infty}x_{n}=0$ there exists $N'\in\mathbb{N}$ such that for all $n\geq N'$ we have $5|x_{n}|_{L}<\varepsilon$ giving
\begin{equation*}
|\tau_{2}(x_{n})|_{L}=5^{-\omega(\tau_{2}(x_{n}))}\leq5^{-(\omega(x_{n})-1)}=5|x_{n}|_{L}<\varepsilon
\end{equation*}
as required. Hence $\tau_{2}$ is a topological involution on $\Delta_{L}$. We now consider the basic $^{L}/_{F}$ function algebra
\begin{equation*}
C(\Delta_{L},\tau_{2},g)=\{f\in C_{L}(\Delta_{L}):f(\tau_{2}(x))=g(f(x))\mbox{ for all }x\in\Delta_{L}\}.
\end{equation*}
We begin by proving that the only power series in $C(\Delta_{L},\tau_{2},g)$ are the constants belonging to $F$. Let $f(x):=\sum_{n=0}^{\infty}a_{n}x^{n}$ be a power series in $C(\Delta_{L},\tau_{2},g)$. Since $\tau_{2}(0)=0$ and $f(\tau_{2}(0))=g(f(0))$ we have $a_{0}=g(a_{0})$ giving $a_{0}\in F$ and so $a_{0}\in C(\Delta_{L},\tau_{2},g)$. Hence $h:=f-a_{0}$ is also in $C(\Delta_{L},\tau_{2},g)$. Suppose towards a contradiction that $h$ is not identically zero on $\Delta_{L}$. Since $1\in\Delta_{L}$, $\sum_{n=1}^{\infty}a_{n}$ converges and so by Lemma \ref{lem:FAACS} we have $\lim_{n\to\infty}\omega(a_{n})=\infty$. Hence we can define $M:=\min\{\omega(a_{n}):n\in\mathbb{N}\}$. Also let $m:=\min\{n\in\mathbb{N}:a_{n}\not=0\}$. Now since $\Delta_{L}=\{x\in L:\omega(x)\geq0\}$ we can find $y\in\Delta_{L}\backslash\{0\}$ such that $2|\omega(y)$ and $M+\omega(y)>\omega(a_{m})$. Hence for every $n>m$ we have
\begin{equation*}
\omega(a_{m}y^{m})=\omega(a_{m})+m\omega(y)<M+\omega(y)+m\omega(y)\leq\omega(a_{n})+n\omega(y)=\omega(a_{n}y^{n}).
\end{equation*}
So, by Lemma \ref{lem:FAACS}, $\omega\left(\sum_{n=m+1}^{\infty}a_{n}y^{n}\right)\geq\min\{\omega(a_{n}y^{n}):n\geq m+1\}>\omega(a_{m}y^{m})$. Hence $\omega\left(\sum_{n=m}^{\infty}a_{n}y^{n}\right)=\omega\left(a_{m}y^{m}+\sum_{n=m+1}^{\infty}a_{n}y^{n}\right)=\omega(a_{m}y^{m})$ by Lemma \ref{lem:CVFEQ}. Similarly for every $n>m$ we have
\begin{align*}
\omega(a_{m}5^{m}y^{m})=&\omega(a_{m})+m(\omega(y)+1)\\
                       <&M+\omega(y)+1+m(\omega(y)+1)\\
                    \leq&\omega(a_{n})+n(\omega(y)+1)=\omega(a_{n}5^{n}y^{n}),
\end{align*}
giving $\omega\left(\sum_{n=m}^{\infty}a_{n}5^{n}y^{n}\right)=\omega(a_{m}5^{m}y^{m})=\omega(5^{m})+\omega(a_{m}y^{m})=m+\omega(a_{m}y^{m})$. Now $h(\tau_{2}(y))=g(h(y))$ and $2|\omega(y)$ hence $\sum_{n=m}^{\infty}a_{n}5^{n}y^{n}=g\left(\sum_{n=m}^{\infty}a_{n}y^{n}\right)$. But, since $g$ is an isometry, $\omega\left(g\left(\sum_{n=m}^{\infty}a_{n}y^{n}\right)\right)=\omega(a_{m}y^{m})$ and $\omega\left(\sum_{n=m}^{\infty}a_{n}5^{n}y^{n}\right)=m+\omega(a_{m}y^{m})$ with $m\in\mathbb{N}$ which is a contradiction. Therefore $h$ is identically zero on $\Delta_{L}$ as required. However, whilst the only power series in $C(\Delta_{L},\tau_{2},g)$ are constants belonging to $F$, it is easy to construct locally analytic elements of $C(\Delta_{L},\tau_{2},g)$ using power series. Define
\begin{equation*}
\mathcal{C}(n):=\{x\in\Delta_{L}:\omega(x)=n\}\mbox{ for }n\in\omega(\Delta_{L}),
\end{equation*}
let $(e_{n})_{n\in\mathbb{N}}$ be the even sequence in $\omega(\Delta_{L})$ given by $e_{n}:=2(n-1)$ and let $a\in F$. Now let $(f_{n})_{n\in\mathbb{N}}$ be a sequence of power series with the following properties:
\begin{enumerate}
\item[(i)]
for all $n\in\mathbb{N}$ the coefficients of $f_{n}$ are elements of $L$;
\item[(ii)]
for all $n\in\mathbb{N}$ we have $5^{-e_{n}}<\rho_{n}$ where $\rho_{n}$ is the radius of convergence of $f_{n}$ so that $f_{n}$ is convergent on $\mathcal{C}(e_{n})$;
\item[(iii)]
we have $\lim_{n\to\infty}\inf\{\omega(f_{n}(x)-a):x\in\mathcal{C}(e_{n})\}=\infty$.
\end{enumerate}
We then define $f:\Delta_{L}\rightarrow L$ by
\begin{equation*}
f(x):= \left\{ \begin{array} {l@{\quad\mbox{if}\quad}l}
f_{n}(x) & x\in\mathcal{C}(e_{n})\\
g(f_{n}(\tau_{2}(x))) & x\in\mathcal{C}(e_{n}+1)\\
a & x=0.
\end{array} \right.
\end{equation*}
We show that $f$ is continuous. Let $x\in\Delta_{L}$ and let $(x_{n})$ be a sequence in $\Delta_{L}$ such that $\lim_{n\to\infty}x_{n}=x$.\\
If $x\not=0$ then by Lemma \ref{lem:FAACS} there is $N\in\mathbb{N}$ such that for all $n\geq N$ we have $\omega(x_{n})=\omega(x)$. If for some $m\in\mathbb{N}$ we have $x\in\mathcal{C}(e_{m})$ then $f(x)=f_{m}(x)$ and for all $n\geq N$ we have $f(x_{n})=f_{m}(x_{n})$ since $x_{n}\in\mathcal{C}(e_{m})$. Hence by the continuity of $f_{m}$ on $\mathcal{C}(e_{m})$ we have $\lim_{n\to\infty}f(x_{n})=f(x)$.\\
If for some $m\in\mathbb{N}$ we have $x\in\mathcal{C}(e_{m}+1)$ then $f(x)=g(f_{m}(\tau_{2}(x)))$, with $\tau_{2}(x)\in\mathcal{C}(e_{m})$, and for all $n\geq N$ we have $f(x_{n})=g(f_{m}(\tau_{2}(x_{n})))$, since $x_{n}\in\mathcal{C}(e_{m}+1)$, with $\tau_{2}(x_{n})\in\mathcal{C}(e_{m})$. Now $\tau_{2}$ is continuous on $\Delta_{L}$, $f_{m}$ is continuous on $\mathcal{C}(e_{m})$ and $g$ is an isometry on $L$. Hence again we have $\lim_{n\to\infty}f(x_{n})=f(x)$.\\
If $x=0$ then by the definition of $f$ we have $f(x)=a$. Let $\varepsilon<\infty$. We need to show that there is $N\in\mathbb{N}$ such that for all $n\geq N$ we have $\omega(f(x_{n})-a)>\varepsilon$. By property (iii) given in the construction of $f$ there is $M\in\mathbb{N}$ such that for all $m\geq M$ we have $\inf\{\omega(f_{m}(y)-a):y\in\mathcal{C}(e_{m})\}>\varepsilon$. Since $\lim_{n\to\infty}\omega(x_{n})=\infty$ there is $N\in\mathbb{N}$ such that for all $n\geq N$ we have $\omega(x_{n})\geq e_{M}$. So let $n\geq N$. Then either $x_{n}=0$, noting that $\omega(0)=\infty$, or there is $m\geq M$ with either $x_{n}\in\mathcal{C}(e_{m})$ or $x_{n}\in\mathcal{C}(e_{m}+1)$.\\
For $x_{n}=0$ we have $\omega(f(0)-a)=\omega(a-a)=\infty>\varepsilon$.\\
For $x_{n}\in\mathcal{C}(e_{m})$ we have $\omega(f(x_{n})-a)=\omega(f_{m}(x_{n})-a)>\varepsilon$ since $m\geq M$.\\
For $x_{n}\in\mathcal{C}(e_{m}+1)$ define $y:=\tau_{2}(x_{n})$ and note that $y\in\mathcal{C}(e_{m})$. Then
\begin{align*}
\omega(f(x_{n})-a)=&\omega(g(f_{m}(\tau_{2}(x_{n})))-a)&&\\
                  =&\omega(g(f_{m}(y)-a)) &&(\mbox{since }a\in F)\\
                  =&\omega(f_{m}(y)-a) &&(\mbox{since }g\mbox{ is an isometry on }L)\\
                  >&\varepsilon &&(\mbox{since }m\geq M\mbox{ and }y\in\mathcal{C}(e_{m})).
\end{align*}
Hence $f$ is continuous. We now show that $f\in C(\Delta_{L},\tau_{2},g)$. Let $x\in\Delta_{L}$.\\
For $x=0$ we have $f(\tau_{2}(0))=f(0)=a=g(a)=g(f(0))$ since $a\in F$.\\
For $x\in\mathcal{C}(e_{n})$, for some $n\in\mathbb{N}$, we have $f(x)=f_{n}(x)$. Define $y:=\tau_{2}(x)$ giving $y\in\mathcal{C}(e_{n}+1)$. Then we have
\begin{equation*}
f(\tau_{2}(x))=f(y)=g(f_{n}(\tau_{2}(y)))=g(f_{n}(\tau_{2}(\tau_{2}(x))))=g(f_{n}(x))=g(f(x)).
\end{equation*}
For $x\in\mathcal{C}(e_{n}+1)$, for some $n\in\mathbb{N}$, we have $f(x)=g(f_{n}(\tau_{2}(x)))$. Put $y:=\tau_{2}(x)$ giving $y\in\mathcal{C}(e_{n})$. Then we have
\begin{equation*}
f(\tau_{2}(x))=f(y)=f_{n}(y)=f_{n}(\tau_{2}(x))=g(g(f_{n}(\tau_{2}(x))))=g(f(x)).
\end{equation*}
Hence $f\in C(\Delta_{L},\tau_{2},g)$ as required. Now suppose there is $N\in\mathbb{N}$ such that for all $n\geq N$ we have $f_{n}=a$. Then $f$ will be locally analytic on $\Delta_{L}$ noting that convergence in $\Delta_{L}$ is from the side, in particular see the proof of Lemma \ref{lem:FAACS}.
\end{example}
\begin{remark}
Concerning examples \ref{exa:CGEXONE} and \ref{exa:CGEXTWO}.
\begin{enumerate}
\item[(i)]
Since $g$ is an isometry on $\Delta_{L}$ and $5\in F$ we note that $\tau_{1}\circ\tau_{2}$ is also a topological involution on $\Delta_{L}$ with $\mbox{ord}(\tau_{1}\circ\tau_{2})=2$. For $f\in C(\Delta_{L},\tau_{1},g)\cap C(\Delta_{L},\tau_{2},g)$ and $x\in\Delta_{L}$ we have $f(\tau_{1}\circ\tau_{2}(x))=g(f(\tau_{2}(x)))=g(g(f(x)))=f(x)$ which gives $f\in C(\Delta_{L},\tau_{1}\circ\tau_{2},\mbox{id})$. But, with reference to Figure \ref{fig:CGED}, $C(\Delta_{L},\tau_{1}\circ\tau_{2},\mbox{id})$ is not a basic function algebra since it fails to separate the points of $\Delta_{L}$ noting that we have $\mbox{ord}(\tau_{1}\circ\tau_{2})\nmid\mbox{ord}(\mbox{id})$. Hence $C(\Delta_{L},\tau_{1},g)\cap C(\Delta_{L},\tau_{2},g)$ is not a $^{L}/_{F}$ function algebra on $(\Delta_{L},\tau_{1},g)$. However $C(\Delta_{L},\tau_{1}\circ\tau_{2},g)$ is a basic function algebra.
\item[(ii)]
We note that by Theorem \ref{thr:CGGSW} every element $f\in C(\Delta_{L},\tau_{2},g)$ can be uniformly approximated by polynomials belonging to $C_{L}(\Delta_{L})$. However, apart from the elements of $F$, none of these polynomials belong to $C(\Delta_{L},\tau_{2},g)$.
\end{enumerate}
\end{remark}
In the next section we look at ways of obtaining more basic function algebras.
\section{Non-Archimedean new basic function algebras from old}
\label{sec:CGNANFO}
\subsection{Basic extensions}
\label{subsec:CGBE}
The following theorem concerns extensions of basic function algebras resulting from the field structure involved.
\begin{theorem}
\label{thr:CGBET}
Let the basic $^{L}/_{L^{g}}$ function algebra on $(X,\tau,g)$ be such that $\mbox{Gal}(^{L}/_{L^{g}})$ and $\langle\mbox{id}\rangle$ are respectively at the top and bottom of a lattice of groups with intermediate elements. Then $C_{L}(X)$ and $C(X,\tau,g)$ are respectively at the top and bottom of a particular lattice of basic function algebras with intermediate elements and there is a one-one correspondence between the subgroups of $\mbox{Gal}(^{L}/_{L^{g}})$ and the elements of this lattice which we will call the lattice of basic extensions of $C(X,\tau,g)$.
\end{theorem}
\begin{proof}
With reference to Remark \ref{rem:CGFROB}, the automorphism $\sigma(f)=g^{(\mbox{ord}(g)-1)}\circ f\circ\tau$, for $f\in C_{L}(X)$, is such that $C_{L}(X)^{\langle\sigma\rangle}=C(X,\tau,g)$ where
\begin{equation*}
C_{L}(X)^{\langle\sigma\rangle}:=\{f\in C_{L}(X):\sigma(f)=f\}.
\end{equation*}
Now by the fundamental theorem of Galois theory we have $\mbox{Gal}(^{L}/_{L^{g}})=\langle g\rangle$ and so $\mbox{Gal}(^{L}/_{L^{g}})$ is a cyclic group. Moreover we have $\mbox{ord}(\sigma)=\mbox{ord}(g)$ giving $\langle\sigma\rangle\cong\langle g\rangle$ as cyclic groups. It is standard from group theory that a subgroup of a cyclic group is cyclic. In particular, for $n|\mbox{ord}(\sigma)$, $\langle\sigma^{(n)}\rangle$ is the unique cyclic subgroup of $\langle\sigma\rangle$ of size $\#\langle\sigma^{(n)}\rangle=\mbox{ord}(\sigma^{(n)})=\frac{\mbox{ord}(\sigma)}{n}$. Moreover for $G$ a subgroup of $\langle\sigma\rangle$ we have $\frac{\mbox{ord}(\sigma)}{\#G}\in\mathbb{N}$, by Lagrange's theorem, and so for $n=\frac{\mbox{ord}(\sigma)}{\#G}$ we have $\langle\sigma^{(n)}\rangle=G$ with $n|\mbox{ord}(\sigma)$. Hence we define a map $\varsigma:\{\langle\sigma^{(n)}\rangle:n|\mbox{ord}(\sigma)\}\rightarrow\{C_{L}(X)^{\langle\sigma^{(n)}\rangle}:n|\mbox{ord}(\sigma)\}$ by
\begin{equation*}
\varsigma(\langle\sigma^{(n)}\rangle):=C_{L}(X)^{\langle\sigma^{(n)}\rangle}:=\{f\in C_{L}(X):\sigma^{(n)}(f)=f\}=C(X,\tau^{(n)},g^{(n)}).
\end{equation*}
Now let $n|\mbox{ord}(\sigma)$. Since $\mbox{ord}(\tau)|\mbox{ord}(g)$ we also have $\mbox{ord}(\tau,X)\subseteq\mbox{ord}(g,L)$, see Figure \ref{fig:CGED}. Hence $\mbox{ord}(\tau^{(n)},X)\subseteq\mbox{ord}(g^{(n)},L)$ giving $\mbox{ord}(\tau^{(n)})|\mbox{ord}(g^{(n)})$ and so $\varsigma(\langle\sigma^{(n)}\rangle)$ is a basic function algebra. Moreover since
\begin{equation*}
C(X,\tau^{(n)},g^{(n)})=\{f\in C_{L}(X):f(\tau^{(n)}(x))=g^{(n)}(f(x))\mbox{ for all }x\in X\}
\end{equation*}
the constants in $\varsigma(\langle\sigma^{(n)}\rangle)$ are the elements of the field $L^{\langle g^{(n)}\rangle}$ and so $\varsigma$ is injective by the fundamental theorem of Galois theory. Finally it is immediate that the elements of $\{C_{L}(X)^{\langle\sigma^{(n)}\rangle}:n|\mbox{ord}(\sigma)\}$ form a lattice as described in the theorem and this completes the proof.
\end{proof}
\begin{example}
\label{exa:CGBEL}
Let $F=\mathbb{Q}$ and let $L=\mathbb{Q}(a)$ where $a=\exp(\frac{1}{14}2\pi i)$. Having factorised $x^{14}-1$ in $F[x]$, we have $\mbox{Irr}_{F,a}(x)=x^{6}-x^{5}+x^{4}-x^{3}+x^{2}-x+1$ with roots $S:=\{a,a^{3},a^{5},a^{9},a^{11},a^{13}\}$. Hence $L$ is the splitting field of $\mbox{Irr}_{F,a}(x)$ over $F$ and so $L$ is a Galois extension of $F$ with $\#\mbox{Gal}(^{L}/_{F})=[L,F]=6$ by Theorem \ref{thr:CVFGL}. In fact putting $g(a):=a^{3}$ makes $g$ a generator of $\mbox{Gal}(^{L}/_{F})$ and so $L$ is a cyclic extension of $F$ and $F=L^{g}$ since $L$ is a Galois extension. We can take $C(X,\tau,g)$ to be a basic $^{L}/_{L^{g}}$ function algebra constructed by analogy with Example \ref{exa:CGTVFA} where $X\subseteq\mathbb{N}\times L$ is non-empty and finite with $\tau((x,y))=(x,g(y))\in X$ for all $(x,y)\in X$. In this case Figure \ref{fig:CGBE} shows the lattice of basic extensions of $C(X,\tau,g)$ as given by Theorem \ref{thr:CGBET}.
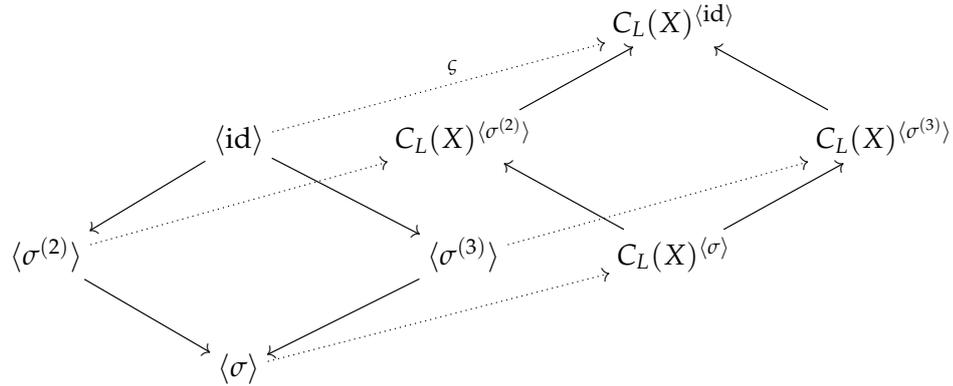
\begin{figure}[h]
\begin{equation*}
\xymatrix{
&&&C_{L}(X)^{\langle\mbox{\footnotesize{id}}\rangle}&\\
&\langle\mbox{id}\rangle\ar@{.>}[rru]^{\varsigma}\ar@{->}[ld]\ar@{->}[rd]&C_{L}(X)^{\langle\sigma^{(2)}\rangle}\ar@{->}[ru]&&C_{L}(X)^{\langle\sigma^{(3)}\rangle}\ar@{->}[lu]\\
\langle\sigma^{(2)}\rangle\ar@{.>}[rru]\ar@{->}[rd]&\hspace{20mm}&\langle\sigma^{(3)}\rangle\ar@{.>}[rru]\ar@{->}[ld]&C_{L}(X)^{\langle\sigma\rangle}\ar@{->}[lu]\ar@{->}[ru]&\\
&\langle\sigma\rangle\ar@{.>}[rru]&&&
}
\end{equation*}
\caption{Lattice of basic extensions.}
\label{fig:CGBE}
\end{figure}
Finally we note that in this example $F=\mathbb{Q}$, requiring the trivial valuation, was chosen just to keep things simple.
\end{example}
\subsection{Residue algebras}
\label{subsec:CGRA}
We begin with an analog of Definition \ref{def:CVFRF}.
\begin{definition}
\label{def:CGRA}
For $C(X,\tau,g)$ a basic $^{L}/_{L^{g}}$ function algebra in the non-Archimedean setting, with valuation logarithm $\omega$ on $L$, we define:
\begin{enumerate}
\item[(i)]
$\mathcal{O}(X,\tau,g):=\{f\in C(X,\tau,g):\inf_{x\in X}\omega(f(x))\geq0,\mbox{\small{ equivalently }}\|f\|_{\infty}\leq1\}$;
\item[(ii)]
$\mathcal{O}^{\times}(X,\tau,g):=\{f\in C(X,\tau,g):\omega(f(x))=0,\mbox{\small{ equivalently }}|f(x)|_{L}=1,\forall x\in X\}$;
\item[(iii)]
$\mathcal{J}(X,\tau,g):=\{f\in C(X,\tau,g):\inf_{x\in X}\omega(f(x))>0,\mbox{\small{ equivalently }}\|f\|_{\infty}<1\}$;
\item[(iv)]
$\mathcal{M}^{y}(X,\tau,g):=\{f\in\mathcal{O}(X,\tau,g):\omega(f(y))>0,\mbox{\small{ equivalently }}|f(y)|_{L}<1\}$ for $y\in X$.
\end{enumerate}
\end{definition}
In this subsection we will mainly be interested in the following two theorems and their proofs. The main theorem is Theorem \ref{thr:CGRAT} which concerns the residue algebra of particular basic function algebras. Before proving these theorems we will need to prove several other results that are also of interest in their own right.
\begin{theorem}
\label{thr:CGRUI}
If $C(X,\tau,g)$ is a basic $^{L}/_{L^{g}}$ function algebra in the non-Archimedean setting then:
\begin{enumerate}
\item[(i)]
$\mathcal{O}(X,\tau,g)$ is a ring;
\item[(ii)]
$\mathcal{O}^{\times}(X,\tau,g)$ is the multiplicative group of units of $\mathcal{O}(X,\tau,g)$;
\item[(iii)]
$\mathcal{J}(X,\tau,g)$ is an ideal of $\mathcal{O}(X,\tau,g)$;
\item[(iv)]
$\mathcal{M}^{y}(X,\tau,g)$ is a maximal ideal of $\mathcal{O}(X,\tau,g)$ for each $y\in X$.
\end{enumerate}
\end{theorem}
\begin{theorem}
\label{thr:CGRAT}
Let $F$ be a locally compact complete non-Archimedean field of characteristic zero with nontrivial valuation. Let $L$ be a finite unramified extension of $F$ with $L^{g}=F$ for some $g\in\mbox{Gal}(^{L}/_{F})$ and let $C(X,\tau,g)$ be a basic $^{L}/_{F}$ function algebra. Then there is an isometric isomorphism
\begin{equation*}
\mathcal{O}(X,\tau,g)/\mathcal{J}(X,\tau,g)\cong C(X,\tau,\bar{g})
\end{equation*}
where $C(X,\tau,\bar{g})$ is the basic $^{\overline{L}}/_{\overline{F}}$ function algebra on $(X,\tau,\bar{g})$. Here $\overline{F}$ and $\overline{L}$ are respectively the residue field of $F$ and $L$ whilst $\bar{g}$ is the residue automorphism on $\overline{L}$ induced by $g$. More generally $L$ need not be an unramified extension of $F$ for the above to hold provided that we impose the condition $\mbox{ord}(\tau)|\mbox{ord}(\bar{g})$ directly.
\end{theorem}
\begin{remark}
\label{rem:CGRAT}
Concerning Theorem \ref{thr:CGRAT}.
\begin{enumerate}
\item[(i)]
The conditions in Theorem \ref{thr:CGRAT} imply that $F$ contains a $p$-adic field, up to a positive exponent of the valuation, since by Theorem \ref{thr:CVFHB} the residue field $\overline{F}$ is finite and so the valuation on $F$ when restricted to $\mathbb{Q}$ can not be trivial.
\item[(ii)]
Since $\overline{L}$ is finite the valuation on $\overline{L}$ is the trivial valuation. In general the quotient norm on a residue field is the trivial valuation. In particular for $\bar{a}\in\overline{L}$ with residue class representative $a\in\mathcal{O}_{L}$ we have, by Lemma \ref{lem:CVFEQ}, that
\begin{equation*}
\min\{|a-b|_{L}:b\in\mathcal{M}_{L}\}=\left\{ \begin{array} {l@{\quad\mbox{if}\quad}l}
1 & a\notin\mathcal{M}_{L} \\
0 & a\in\mathcal{M}_{L}.
\end{array} \right.
\end{equation*}
\item[(iii)]
To be thorough for the reader we show that $\bar{g}$ is well defined, although this is covered in \cite[p52]{Fesenko}. Let $\bar{a}\in\overline{L}$ with residue class representative $a\in\mathcal{O}_{L}$. For $g\in\mbox{Gal}(^{L}/_{F})$ we obtain $\bar{g}\in\mbox{Gal}(^{\overline{L}}/_{\overline{F}})$ by $\bar{g}(\bar{a}):=\overline{g(a)}$. Now let $b\in\mathcal{O}_{L}$ with $b\not=a$ but $\bar{b}=\bar{a}$ so that $a-b\in\mathcal{M}_{L}$. By Remark \ref{rem:CVFUT}, $g$ is an isometry on $L$ and so $g(a)-g(b)=g(a-b)\in\mathcal{M}_{L}$ giving $\bar{g}(\bar{b})=\overline{g(b)}=\overline{g(a)}=\bar{g}(\bar{a})$ and so $\bar{g}$ is well defined.
\item[(iv)]
The map $g\mapsto\bar{g}$ is a homomorphism from $\mbox{Gal}(^{L}/_{F})$ to $\mbox{Gal}(^{\overline{L}}/_{\overline{F}})$. Indeed for $\bar{a}\in\overline{L}$ and $g_{1},g_{2}\in\mbox{Gal}(^{L}/_{F})$ we have
\begin{equation*}
\overline{g_{1}\circ g_{2}}(\bar{a})=\overline{g_{1}\circ g_{2}(a)}=\overline{g_{1}(g_{2}(a))}=\bar{g}_{1}\left(\overline{g_{2}(a)}\right)=\bar{g}_{1}(\bar{g}_{2}(\bar{a}))=\bar{g}_{1}\circ\bar{g}_{2}(\bar{a}).
\end{equation*}
Under the conditions of Theorem \ref{thr:CGRAT} this homomorphism becomes an isomorphism as per Lemma \ref{lem:CGORDG} below, see \cite[p52]{Fesenko}. In particular this ensures that $C(X,\tau,\bar{g})$, in Theorem \ref{thr:CGRAT}, is a basic function algebra since $\mbox{ord}(\bar{g})=\mbox{ord}(g)$ gives $\mbox{ord}(\tau)|\mbox{ord}(\bar{g})$.
\end{enumerate}
\end{remark}
\begin{lemma}
\label{lem:CGORDG}
Let $F$ be a local field, as per Remark \ref{rem:CVFHB}, and let $L$ be a finite unramified Galois extension of $F$. Then $\mbox{Gal}(^{L}/_{F})\cong\mbox{Gal}(^{\overline{L}}/_{\overline{F}})$ giving $\mbox{ord}(\bar{g})=\mbox{ord}(g)$ for all $g\in\mbox{Gal}(^{L}/_{F})$.
\end{lemma}
The following definition and lemma will be useful when proving Theorem \ref{thr:CGRUI}. Note that the first part of Lemma \ref{lem:CGVLF} makes sense even though we have yet to show that $\mathcal{O}(X,\tau,g)$ is a ring.
\begin{definition}
\label{def:CGVLF}
Let $L$ be a complete valued field with valuation logarithm $\omega$ and let $X$ be a totally disconnected compact Hausdorff space.
\begin{enumerate}
\item[(i)]
We call a map $\iota:X\rightarrow\omega(\mathcal{O}_{L})$ a {\em value level function}.
\item[(ii)]
We place a partial order on the set of all value level function by setting
\begin{equation*}
\iota_{1}\geq\iota_{2}\mbox{ if and only if for all }x\in X\mbox{ we have }\iota_{1}(x)\geq\iota_{2}(x).
\end{equation*}
\end{enumerate}
\end{definition}
\begin{lemma}
\label{lem:CGVLF}
Let $C(X,\tau,g)$ be a basic $^{L}/_{L^{g}}$ function algebra in the non-Archimedean setting with valuation logarithm $\omega$ on $L$ and let $\iota:X\rightarrow\omega(\mathcal{O}_{L})$ be a value level function. Then:
\begin{enumerate}
\item[(i)]
$\mathcal{M}_{\iota}(X,\tau,g):=\{f\in C(X,\tau,g):\omega(f(x))\geq\iota(x)\mbox{ for all }x\in X\}$ is an ideal of $\mathcal{O}(X,\tau,g)$;
\item[(ii)]
for $\iota'$ another value level function with $\iota\geq\iota'$ we have $\mathcal{M}_{\iota}(X,\tau,g)\subseteq\mathcal{M}_{\iota'}(X,\tau,g)$.
\end{enumerate}
\end{lemma}
\begin{proof}
For (i), let $f_{1},f_{2}\in\mathcal{M}_{\iota}(X,\tau,g)$ and $f\in\mathcal{O}(X,\tau,g)$. Then for each $x\in X$ we have $\omega(f_{1}(x)+f_{2}(x))\geq\min\{\omega(f_{1}(x)),\omega(f_{2}(x))\}\geq\iota(x)$ giving $f_{1}+f_{2}\in\mathcal{M}_{\iota}(X,\tau,g)$ and $\omega(f_{1}(x)f(x))=\omega(f_{1}(x))+\omega(f(x))\geq\omega(f_{1}(x))\geq\iota(x)$ giving $f_{1}f\in\mathcal{M}_{\iota}(X,\tau,g)$ as required. For (ii), this is immediate.
\end{proof}
Note, the mapping $\iota\mapsto\mathcal{M}_{\iota}(X,\tau,g)$ is not assumed to be injective. We now prove Theorem \ref{thr:CGRUI}.
\begin{proof}[Proof of Theorem \ref{thr:CGRUI}]
For (i), note that since $\omega(1)=0$, $\omega(0)=\infty$ and $1,0\in F$ we have $1,0\in\mathcal{O}(X,\tau,g)$. Further $\mathcal{O}(X,\tau,g)$ is closed under multiplication and addition by Lemma \ref{lem:CGVLF} since for the value level function that is constantly zero we have $\mathcal{O}(X,\tau,g)=\mathcal{M}_{0}(X,\tau,g)$. Hence $\mathcal{O}(X,\tau,g)$ is a ring.\\
For (ii), we need to show that $\mathcal{O}^{\times}(X,\tau,g)=\mathcal{O}(X,\tau,g)^{\times}$. Let $f\in\mathcal{O}(X,\tau,g)^{\times}$. Then for all $x\in X$ we have $\omega(f(x))\geq0$ and $\omega(f^{-1}(x))\geq0$ since $f,f^{-1}\in\mathcal{O}(X,\tau,g)^{\times}$ but we also have $\omega(f^{-1}(x))=\omega((f(x))^{-1})=-\omega(f(x))$ giving $\omega(f(x))=0$. Hence $\mathcal{O}(X,\tau,g)^{\times}\subseteq\mathcal{O}^{\times}(X,\tau,g)$. Now let $f\in\mathcal{O}^{\times}(X,\tau,g)$. We have
\begin{equation}
\label{equ:CGUNIT}
\omega(f^{-1}(x))=-\omega(f(x))=0
\end{equation}
for all $x\in X$ and so it remains to show that $f^{-1}$ is an element of $C(X,\tau,g)$. We have $1=g(1)=g(f^{-1}f)=g(f^{-1})g(f)$ giving $g(f^{-1})=(g(f))^{-1}$ and so
\begin{equation*}
f^{-1}(\tau)=(f(\tau))^{-1}=(g(f))^{-1}=g(f^{-1}).
\end{equation*}
For continuity let $x\in X$ and $(x_{n})$ be a sequence in $X$ with $\lim_{n\to\infty}x_{n}=x$ in $X$. Then by (\ref{equ:CGUNIT}) we have
\begin{align*}
\omega(f^{-1}(x_{n})-f^{-1}(x))=&\omega(f^{-1}(x_{n})-f^{-1}(x))+\omega(f(x))\\
=&\omega((f^{-1}(x_{n})-f^{-1}(x))f(x))\\
=&\omega(f^{-1}(x_{n})f(x)-1)\\
=&\omega(f^{-1}(x_{n})(f(x)-f(x_{n})))\\
=&\omega(f^{-1}(x_{n}))+\omega(f(x)-f(x_{n}))\\
=&\omega(f(x)-f(x_{n})).
\end{align*}
Therefore by the continuity of $f$ we have $\lim_{n\to\infty}f^{-1}(x_{n})=f^{-1}(x)$ from which it follows that $\mathcal{O}^{\times}(X,\tau,g)\subseteq\mathcal{O}(X,\tau,g)^{\times}$ as required.\\
For (iii), taking into account that the valuation on $L$ could be dense, $\mathcal{J}(X,\tau,g)$ is an ideal of $\mathcal{O}(X,\tau,g)$ by Lemma \ref{lem:CGVLF} since $\mathcal{J}(X,\tau,g)=\bigcup_{n\in\mathbb{N}}\mathcal{M}_{\frac{1}{n}}(X,\tau,g)$ noting that $1\not\in\mathcal{M}_{\frac{1}{n}}(X,\tau,g)$ for all $n\in\mathbb{N}$ and that a union of nested ideals is an ideal.\\
For (iv), $\mathcal{M}^{y}(X,\tau,g)$ is an ideal of $\mathcal{O}(X,\tau,g)$ by Lemma \ref{lem:CGVLF} since $\mathcal{M}^{y}(X,\tau,g)=\bigcup_{n\in\mathbb{N}}\mathcal{M}_{\frac{1}{n}\chi_{\{y\}}}(X,\tau,g)$ where $\chi_{\{y\}}$ is the indicator function. Also by Lemma \ref{lem:CGVLF}, since $\frac{1}{n}\geq\frac{1}{n}\chi_{\{y\}}$ for all $n\in\mathbb{N}$, we have $\mathcal{J}(X,\tau,g)\subseteq\mathcal{M}^{y}(X,\tau,g)$. We now show that $\mathcal{M}^{y}(X,\tau,g)$ is a maximal ideal of $\mathcal{O}(X,\tau,g)$. Let $\mathcal{I}(X,\tau,g)$ be a, not necessarily proper, ideal of $\mathcal{O}(X,\tau,g)$ with $\mathcal{M}^{y}(X,\tau,g)\subsetneqq\mathcal{I}(X,\tau,g)$. Then there is $f\in\mathcal{I}(X,\tau,g)$ with $\omega(f(y))=0$. Define on $X$
\begin{equation*}
f'(x):=\left\{ \begin{array} {l@{\quad\mbox{if}\quad}l}
0 & \omega(f(x))=0 \\
1 & \omega(f(x))>0.
\end{array} \right.
\end{equation*}
We show that $f'$ is an element of $\mathcal{M}^{y}(X,\tau,g)$ and so $f'\in\mathcal{I}(X,\tau,g)$. For continuity let $x\in X$ and $(x_{n})$ be a sequence of elements of $X$ with $\lim_{n\to\infty}x_{n}=x$ in $X$. Since $f$ is continuous we have $\lim_{n\to\infty}f(x_{n})=f(x)$ with respect to $\omega$. Hence if $f(x)=0$ then there exists $N\in\mathbb{N}$ such that for all $n\geq N$ we have $\omega(f(x_{n}))=\omega(f(x_{n})-f(x))>0$. If $f(x)\not=0$ then since convergence in $L$ is from the side, see Lemma \ref{lem:FAACS}, there exists $N\in\mathbb{N}$ such that for all $n\geq N$ we have $\omega(f(x_{n}))=\omega(f(x))$. Hence in every case there exists $N\in\mathbb{N}$ such that for all $n\geq N$ we have $f'(x_{n})=f'(x)$ and so $f'$ is continuous. We need to show that $f'(\tau(x))=g(f'(x))$. Since $g$ is an isometry on $L$ we have $\omega(f(\tau(x)))=\omega(g(f(x)))=\omega(f(x))$ giving
\begin{align*}
f'(\tau(x))=&\left\{ \begin{array} {l@{\quad\mbox{if}\quad}l}
0 & \omega(f(\tau(x)))=0 \\
1 & \omega(f(\tau(x)))>0
\end{array} \right.\\
=&f'(x)\\
=&g(f'(x))
\end{align*}
noting that $f'$ takes values only in $\{0,1\}\subseteq F$. Now $\omega(1)=0$ and $\omega(0)=\infty$ so that for all $x\in X$ we have $\omega(f'(x))\geq0$ giving $f'\in\mathcal{O}(X,\tau,g)$. Further since $\omega(f(y))=0$ we have $\omega(f'(y))=\omega(0)=\infty$ and so we have shown that $f'\in\mathcal{M}^{y}(X,\tau,g)\subsetneqq\mathcal{I}(X,\tau,g)$. Now since $\mathcal{I}(X,\tau,g)$ is an ideal we have $f+f'\in\mathcal{I}(X,\tau,g)$. Moreover by the definition of $f'$, for each $x\in X$, if $\omega(f(x))=0$ then $\omega(f'(x))=\omega(0)=\infty$ and if $\omega(f(x))>0$ then $\omega(f'(x))=\omega(1)=0$. Hence for all $x\in X$, $\omega(f(x)+f'(x))=0$ by Lemma \ref{lem:CVFEQ} and so $f+f'\in\mathcal{O}^{\times}(X,\tau,g)$ giving $\mathcal{I}(X,\tau,g)=\mathcal{O}(X,\tau,g)$. Therefore $\mathcal{M}^{y}(X,\tau,g)$ is a maximal ideal of $\mathcal{O}(X,\tau,g)$ and this completes the proof of Theorem \ref{thr:CGRUI}.
\end{proof}
The following two lemmas will be used in the proof of Theorem \ref{thr:CGRAT}. The first of these, Lemma \ref{lem:CGSLCF}, will be known but we provide a proof in the absence of a reference. The second, Lemma \ref{lem:CGIRC}, may be new, since I have not seen it in the literature, however it could be known to some number theorists.
\begin{lemma}
\label{lem:CGSLCF}
Let $F$ be a complete non-Archimedean field with a nontrivial, discrete valuation and valuation logarithm $\nu$. Let $\pi$ be a prime element and $\mathcal{R}$ be a set of residue class representatives for $F$, as shown in Theorem \ref{thr:CVFSE}. Then, for $X$ a compact Hausdorff space, each $f\in C_{F}(X)$ has a unique expansion, as a series of locally constant $\mathcal{R}$-valued functions, of the form
\begin{equation*}
f=\sum_{i=n}^{\infty}f_{i}\pi^{i},\quad\mbox{for some }n\in\mathbb{Z}.
\end{equation*}
Moreover, for $j\geq n$ and $x,y\in X$ with $\nu(f(x)-f(y))>j\nu(\pi)$, we have $f_{i}(x)=f_{i}(y)$ for all $i$ in the interval $n\leq i\leq j$.
\end{lemma}
\begin{proof}
Let $f\in C_{F}(X)$ and note that since $X$ is compact, $f$ is bounded. Hence there is $n\in\mathbb{Z}$ such that, for all $x\in X$, $\nu(f(x))\geq n\nu(\pi)$. Therefore by allowing terms to be zero where necessary and by using the unique $\pi$-power series expansion over $\mathcal{R}$ for elements of $F^{\times}$, as shown in Theorem \ref{thr:CVFSE}, we have for each $x\in X$
\begin{equation*}
f(x)=\sum_{i=n}^{\infty}f_{i}(x)\pi^{i}\in F.
\end{equation*}
Hence for each $i\geq n$ we have obtained a function $f_{i}:X\rightarrow\mathcal{R}$ and the resulting expansion $f=\sum_{i=n}^{\infty}f_{i}\pi^{i}$ is unique. Now for $j\geq n$ let $x,y\in X$ be such that we have $\nu(f(x)-f(y))>j\nu(\pi)$. If we do not have $f_{k}(x)=f_{k}(y)$ for all $k\geq n$ then let $k\geq n$ be the first integer for which $f_{k}(x)\not=f_{k}(y)$. Therefore $f_{k}(x)$ and $f_{k}(y)$ are representatives in $\mathcal{O}_{F}$ of two different residue classes. Hence $f_{k}(x)-f_{k}(y)\not\in\mathcal{M}_{F}$ showing that $\nu(f_{k}(x)-f_{k}(y))=0$. Therefore by Lemma \ref{lem:CVFEQ} and the definition of $k$ we have
\begin{align*}
k\nu(\pi)=&\nu(f_{k}(x)-f_{k}(y))+\nu(\pi^{k})\\
=&\nu((f_{k}(x)-f_{k}(y))\pi^{k})\\
=&\nu\left((f_{k}(x)-f_{k}(y))\pi^{k}+\sum_{i=k+1}^{\infty}f_{i}(x)\pi^{i}-\sum_{i=k+1}^{\infty}f_{i}(y)\pi^{i}\right)\\
=&\nu\left(\sum_{i=n}^{\infty}f_{i}(x)\pi^{i}-\sum_{i=n}^{\infty}f_{i}(y)\pi^{i}\right)\\
=&\nu(f(x)-f(y))>j\nu(\pi).
\end{align*}
Hence $k>j$ giving $f_{i}(x)=f_{i}(y)$ for all $i$ in the interval $n\leq i\leq j$. Finally we show, for all $j\geq n$, that $f_{j}$ is a locally constant function. For $x\in X$ define the following ball in $F$
\begin{equation*}
B_{j\nu(\pi)}(f(x)):=\{a\in F:\nu(f(x)-a)>j\nu(\pi)\}.
\end{equation*}
Then since $f$ is continuous there exists an open subset $U$ of $X$ with $x\in U$ such that $f(U)\subseteq B_{j\nu(\pi)}(f(x))$. Hence for each $y\in U$ we have $\nu(f(x)-f(y))>j\nu(\pi)$ and so $f_{j}(x)=f_{j}(y)$. In particular $f_{j}$ is constant on $U$ and this completes the proof of Lemma \ref{lem:CGSLCF}.
\end{proof}
Lemma \ref{lem:CGSLCF} has the following corollary which, for locally constant functions, goes slightly further than Theorem \ref{thr:CGGSW} since it does not assume that $X$ is totally disconnected.
\begin{corollary}
\label{cor:CGSLCF}
Let $F$ and $X$ be as in Lemma \ref{lem:CGSLCF} and let $\mathrm{LC}_{F}(X)$ be the set of all locally constant $F$ valued functions defined on $X$. Then $\mathrm{LC}_{F}(X)$ is uniformly dense in $C_{F}(X)$.
\end{corollary}
\begin{proof}
For $f\in C_{F}(X)$ let $f=\sum_{i=n}^{\infty}f_{i}\pi^{i}$ be the expansion from Lemma \ref{lem:CGSLCF}. We note that a finite sum of locally constant functions is locally constant. Hence, for each $m\geq n$, $f_{i\leq m}:=\sum_{i=n}^{m}f_{i}\pi^{i}$ is an element of $\mathrm{LC}_{F}(X)$. Let $\varepsilon<\infty$ and $m>\frac{\varepsilon}{\nu(\pi)}$. Then we have $\inf_{x\in X}\nu(f(x)-f_{i\leq m}(x))=\inf_{x\in X}\nu(\sum_{i=m+1}^{\infty}f_{i}(x)\pi^{i})\geq\nu(\pi^{m+1})>m\nu(\pi)>\varepsilon$ as required.
\end{proof}
Here is the second of the two lemmas that will be used in the proof of Theorem \ref{thr:CGRAT}.
\begin{lemma}
\label{lem:CGIRC}
Let $F$ and $L$ be non-Archimedean fields with $L$ a finite extension of $F$ as a valued field such that the following holds:
\begin{enumerate}
\item[(i)]
we have $\mathbb{Q}\subseteq F$ and the valuation logarithm $\nu$ on $F$ when restricted to $\mathbb{Q}$ is a $p$-adic valuation logarithm;
\item[(ii)]
the residue field $\overline{F}$ is finite and so $\overline{L}$ is also finite;
\item[(iii)]
the elements of $\mbox{Gal}(^{L}/_{F})$ are isometric on $L$, noting that this is automatically satisfied if $F$ is complete.
\end{enumerate}
Then for each $g\in\mbox{Gal}(^{L}/_{F})$ there exists a set $\mathcal{R}_{L,g}\subseteq\mathcal{O}_{L}^{\times}\cup\{0\}$ of residue class representatives for $L$ such that the restriction of $g$ to $\mathcal{R}_{L,g}$ is an endofunction $g|_{\mathcal{R}_{L,g}}:\mathcal{R}_{L,g}\rightarrow\mathcal{R}_{L,g}$.
\end{lemma}
\begin{proof}
Let $\omega$ be the extension of $\nu$ to $L$ and let $\mathcal{R}_{L}$ with $0\in\mathcal{R}_{L}$ be an arbitrary set of residue class representatives for $L$. Fix $g\in\mbox{Gal}(^{L}/_{F})$ and for $a\in\mathcal{O}_{L}^{\times}$ denote the orbit of $a$ with respect to $g$ by
\begin{equation*}
\langle g\rangle(a):=\{g^{(n)}(a):n\in\{1,\cdots,\mbox{ord}(g,a)\}\}.
\end{equation*}
Also denote $\overline{\langle g\rangle(a)}:=\{\overline{g^{(n)}(a)}:n\in\{1,\cdots,\mbox{ord}(g,a)\}\}=\langle \bar{g}\rangle(\bar{a})\subseteq\overline{L}$. We will show that $\mathcal{R}_{L,g}\subseteq\mathcal{O}_{L}^{\times}\cup\{0\}$ can be constructed from $\mathcal{R}_{L}$. Clearly we can let $0$ represent $\bar{0}$ and so include $0$ in $\mathcal{R}_{L,g}$. More generally we need to make sure that:
\begin{enumerate}
\item[(1)]
for each $a'\in\mathcal{R}_{L}$ there is precisely one element $a\in\mathcal{R}_{L,g}$ such that $\bar{a}=\overline{a'}$, that is $a=a'+b$ for some $b\in L$ with $\omega(b)>0$;
\item[(2)]
for each $a\in\mathcal{R}_{L,g}$ we have $g(a)\in\mathcal{R}_{L,g}$.
\end{enumerate}
To this end we will show that the following useful facts hold for Lemma \ref{lem:CGIRC}.
\begin{enumerate}
\item[(a)]
For $a_{1},a_{2}\in\mathcal{O}_{L}^{\times}$ either $\overline{\langle g\rangle(a_{1})}\cap\overline{\langle g\rangle(a_{2})}=\emptyset$ or $\overline{\langle g\rangle(a_{1})}=\overline{\langle g\rangle(a_{2})}$. Clearly, since $\mbox{ord}(g)$ is finite, if $\langle g\rangle(a_{1})\cap\langle g\rangle(a_{2})\not=\emptyset$ then $\langle g\rangle(a_{1})=\langle g\rangle(a_{2})$.
\item[(b)]
Let $a'\in\mathcal{R}_{L}\backslash\{0\}$. Then there exists $a\in\mathcal{O}_{L}^{\times}$ with $\bar{a}=\overline{a'}$ such that if $a_{1},a_{2}\in\langle g\rangle(a)$ with $a_{1}\not=a_{2}$ then $\overline{a_{1}}\not=\overline{a_{2}}$. Further since $g$ is an isometry we have $\omega(a_{1})=0$ for all $a_{1}\in\langle g\rangle(a)$. This ensures that every residue class that has a representative in $\langle g\rangle(a)$ has only one representative in $\langle g\rangle(a)$.
\end{enumerate}
Hence by applying (a) and (b) above we obtain $\mathcal{R}_{L,g}$ as a disjoint union of the orbits of finitely many elements form $\mathcal{O}_{L}^{\times}\cup\{0\}$. Note if $\mathcal{R}_{F}$ is a set of residue class representatives for $F$ and $\mathcal{R}_{F}\subseteq\mathcal{R}_{L}$ then with the above construction we can choose to have $\mathcal{R}_{F}\subseteq\mathcal{R}_{L,g}$ since $g$ restricts to the identity map on $\mathcal{R}_{F}$. Also note that (b) is not in general satisfied for all $a\in\mathcal{O}_{L}^{\times}$ with $\bar{a}=\overline{a'}$. Indeed, in the case of Example \ref{exa:CGEXONE} where $L=\mathbb{Q}_{5}(\sqrt{2})$, for $a=1+5\sqrt{2}$ we have $g(a)=1-5\sqrt{2}\not=a$ and yet $\overline{g(a)}=\bar{a}=\bar{1}$. We will now prove that (a) and (b) above hold.\\
For (a) it is enough to confirm that $\overline{\langle g\rangle(a)}=\langle\bar{g}\rangle(\bar{a})$ for all $a\in\mathcal{O}_{L}^{\times}$. Since $g$ is an isometry, (iii) and (iv) of Remark \ref{rem:CGRAT} are applicable and so we have for each $n\in\mathbb{N}$ that $\overline{g^{(n)}(a)}=\overline{g^{(n)}}(\bar{a})=\bar{g}^{(n)}(\bar{a})$. Hence the result follows.\\
For (b) we first note that, for each $a'\in\mathcal{R}_{L}\backslash\{0\}$, $g$ maps residue class to residue class. That is $g$ restricts to a bijection $g|_{\overline{a'}}:\overline{a'}\rightarrow\overline{g(a')}$ and $g$ also restricts to a bijection $g|_{\overline{g(a')}}:\overline{g(a')}\rightarrow\overline{g^{(2)}(a')}$ and so forth. This is because $g$ restricts to a bijective endofunction on $\mathcal{M}_{L}$ since $g$ has finite order and is an isometry. Now for (b) to hold we need to check that for each $a'\in\mathcal{R}_{L}\backslash\{0\}$ there is an $a\in\overline{a'}$ such that when the forward orbit of $a$, with respect to $g$, returns to a residue class it has visited before then it returns to the same element of that residue class. Let $a'\in\mathcal{R}_{L}\backslash\{0\}$ and let $n$ be the first element of $\{1,2,\cdots,\mbox{ord}(g,a')\}$ such that there exists an $i\in\{0,1,\cdots,n-1\}$ with $g^{(i)}(a')$ in the same residue class as $g^{(n)}(a')$. Hence $\omega(g^{(i)}(a')-g^{(n)}(a'))>0$ and since $g$ is an isometry we have $\omega(a'-g^{(n-i)}(a'))>0$ giving $i=0$ by the definition of $n$. Therefore $g^{(n)}(a')=a'+b$ for some $b\in\mathcal{M}_{L}$ and $g^{(n)}$ restricts to $g^{(n)}|_{\overline{a'}}:\overline{a'}\rightarrow\overline{a'}$.\\
Hence for (b) to hold it is enough to show that there is $a\in\overline{a'}$ which is a fixed point with respect to $g^{(n)}$. To this end we more generally show that for each $g\in\mbox{Gal}(^{L}/_{F})$ with $g|_{\overline{a'}}:\overline{a'}\rightarrow\overline{a'}$ there is a fixed point $a\in\overline{a'}$ of $g$. So for such a $g\in\mbox{Gal}(^{L}/_{F})$ let $m:=\mbox{ord}(g,a')$. Now recall that we have $\mathbb{Q}\subseteq F$ and that $\nu$ on $F$ when restricted to $\mathbb{Q}$ is a $p$-adic valuation logarithm for some prime $p$. Hence we have two cases, $p\nmid m$ and $p|m$.\\
Suppose $p\nmid m$. We have $g(a')=a'+b$ for some $b\in\mathcal{M}_{L}$. Further then we have
\begin{align*}
g^{(2)}(a')=&g(a'+b)=g(a')+g(b)=a'+b+g(b),\\
g^{(3)}(a')=&g(a'+b+g(b))=g(a')+g(b)+g^{(2)}(b)=a'+b+g(b)+g^{(2)}(b),\\
     \vdots &\\
g^{(m-1)}(a')=&a'+b+g(b)+g^{(2)}(b)+\cdots+g^{(m-2)}(b).
\end{align*}
Hence consider
\begin{align*}
a:=&\frac{1}{m}(a'+g(a')+g^{(2)}(a')+\cdots+g^{(m-1)}(a'))\\
  =&\frac{1}{m}(ma'+(m-1)b+(m-2)g(b)+\cdots+(m-(m-1))g^{(m-2)}(b)).
\end{align*}
Since $\mathbb{Q}\subseteq F$ we have $g(\frac{1}{m})=\frac{1}{m}$ giving $g(a)=a$. Moreover since $p\nmid m$ we have $\omega(m^{-1})=\nu(m^{-1})=-\nu(m)=0$. Therefore
\begin{align*}
\omega(a-a')=&\omega\left(\frac{1}{m}((m-1)b+(m-2)g(b)+(m-3)g^{(2)}(b)+\cdots+g^{(m-2)}(b))\right)\\
            =&0+\omega((m-1)b+(m-2)g(b)+(m-3)g^{(2)}(b)+\cdots+g^{(m-2)}(b))\\
        \geq &\min\{\omega((m-1)b),\omega((m-2)g(b)),\omega((m-3)g^{(2)}(b)),\cdots,\omega(g^{(m-2)}(b))\}\\
            =&\omega(g^{(m-2)}(b))=\omega(b)>0.
\end{align*}
Hence $a$ is an element of $\overline{a'}$ with $g(a)=a$ as required.\\
Suppose $p|m$. Then there is $n,m'\in\mathbb{N}$ such that $m=p^{n}m'$ with $p\nmid m'$. Hence $\mbox{ord}(g^{(p^{n-1}m')},a')=p$. Now suppose that the following holds.
\begin{enumerate}
\item[(b2)]
Let $a_{0}\in\mathcal{O}_{L}^{\times}$. Then for each $g'\in\mbox{Gal}(^{L}/_{F})$ with $g'|_{\overline{a_{0}}}:\overline{a_{0}}\rightarrow\overline{a_{0}}$ and $\mbox{ord}(g',a_{0})=p$ there is a fixed point $a_{1}\in\overline{a_{0}}$ of $g'$.
\end{enumerate}
Then by applying (b2) there is $a_{1}\in\overline{a'}$ which is a fixed point of $g^{(p^{n-1}m')}$ and so we have $\mbox{ord}(g,a_{1})|p^{n-1}m'$. By repeated application of (b2) we can obtain an element $a_{n}\in\overline{a'}$ such that $\mbox{ord}(g,a_{n})|m'$. Now, since the set $\mathcal{R}_{L}$ was an arbitrary set of residue class representatives for $L$ and $a_{n}$ represents $\overline{a'}$, we can apply the case $p\nmid m$ for $a_{n}$ to obtain $a\in\overline{a_{n}}=\overline{a'}$ with $g(a)=a$ as required.\\
It remains to show that (b2) holds. So let $a_{0}\in\mathcal{O}_{L}^{\times}$ and $g'\in\mbox{Gal}(^{L}/_{F})$ satisfy the conditions of (b2). Hence for some $b\in\mathcal{M}_{L}$ we have
\begin{align*}
g'(a_{0})=&a_{0}+b,\\
g'^{(2)}(a_{0})=&a_{0}+b+g'(b),\\
     \vdots &\\
g'^{(p-1)}(a_{0})=&a_{0}+b+g'(b)+\cdots+g'^{(p-2)}(b).
\end{align*}
Define $b_{1}:=b$, $b_{2}:=b+g'(b)$, $\cdots$, $b_{p-1}:=b+g'(b)+\cdots+g'^{(p-2)}(b)$ and note that for all $i\in\{1,\cdots,p-1\}$ we have
\begin{equation}
\label{equ:CGWBI}
\omega(b_{i})\geq\min\{\omega(b),\omega(g'(b)),\cdots,\omega(g'^{(p-2)}(b))\}=\omega(b)>0.
\end{equation}
Now since $\omega|_{\mathbb{Q}}$ is the $p$-adic valuation logarithm $\nu_{p}$ we have $\mathbb{F}_{p}\subseteq\overline{L}$. Therefore since $\overline{L}$ is a finite field we have $\#\overline{L}=p^{k}$ for some $k\in\mathbb{N}$. Hence we consider
\begin{align*}
a_{1}:=&(a_{0}g'(a_{0})g'^{(2)}(a_{0})\cdots g'^{(p-1)}(a_{0}))^{p^{k-1}}\\
      =&(a_{0}(a_{0}+b_{1})(a_{0}+b_{2})\cdots(a_{0}+b_{p-1}))^{p^{k-1}}\\ =&(a_{0}^{p}+a_{0}b_{1}(a_{0}+b_{2})\cdots(a_{0}+b_{p-1})+a_{0}^{2}b_{2}(a_{0}+b_{3})\cdots(a_{0}+b_{p-1})+\cdots\\
       &\cdots+a_{0}^{p-1}b_{p-1})^{p^{k-1}}.
\end{align*}
Now, by Lemma \ref{lem:CVFEQ} and (\ref{equ:CGWBI}), we have
\begin{align*}
\omega(a_{0}b_{1}(a_{0}+b_{2})\cdots(a_{0}+b_{p-1}))=&\omega(a_{0})+\omega(b_{1})+\omega(a_{0}+b_{2})+\cdots\\
&\cdots+\omega(a_{0}+b_{p-1})\\
=&0+\omega(b_{1})+0+\cdots+0\\
\geq&\omega(b)>0.
\end{align*}
The same inequality holds for later terms in the above expansion on $a_{1}$, hence for
\begin{equation*}
c:=a_{0}b_{1}(a_{0}+b_{2})\cdots(a_{0}+b_{p-1})+a_{0}^{2}b_{2}(a_{0}+b_{3})\cdots(a_{0}+b_{p-1})+\cdots+a_{0}^{p-1}b_{p-1}
\end{equation*}
we have $\omega(c)>0$. This gives
\begin{equation*}
a_{1}=(a_{0}^{p}+c)^{p^{k-1}}=a_{0}^{p^{k}}+\sum_{i=1}^{p^{k-1}}\binom{p^{k-1}}{i}a_{0}^{p(p^{k-1}-i)}c^{i}
\end{equation*}
such that for each $i\in\{1,\cdots,p^{k-1}\}$ we have
\begin{align*}
\omega\left(\binom{p^{k-1}}{i}a_{0}^{p(p^{k-1}-i)}c^{i}\right)=&\omega\left(\binom{p^{k-1}}{i}\right)+p(p^{k-1}-i)\omega(a_{0})+i\omega(c)\\
=&\omega\left(\binom{p^{k-1}}{i}\right)+0+i\omega(c)>0
\end{align*}
noting that $\omega\left(\binom{p^{k-1}}{i}\right)\geq0$ since $\omega|_{\mathbb{Q}}$ is the $p$-adic valuation logarithm $\nu_{p}$. Hence for $c':=\sum_{i=1}^{p^{k-1}}\binom{p^{k-1}}{i}a_{0}^{p(p^{k-1}-i)}c^{i}$ we have $\omega(c')>0$. Further since $\#\overline{L}^{\times}=p^{k}-1$ we have $\overline{a_{0}}^{p^{k}-1}=\bar{1}$ by Lagrange's theorem. In particular $\overline{a_{1}}=\overline{a_{0}^{p^{k}}}+\overline{c'}=\overline{a_{0}}^{p^{k}}+\bar{0}=\overline{a_{0}}$ giving $a_{1}\in\overline{a_{0}}$ and since $a_{1}=(a_{0}g'(a_{0})g'^{(2)}(a_{0})\cdots g'^{(p-1)}(a_{0}))^{p^{k-1}}$ with $\mbox{ord}(g',a_{0})=p$ we have $g'(a_{1})=a_{1}$ as required. This completes the proof of Lemma \ref{lem:CGIRC}.
\end{proof}
We will now prove Theorem \ref{thr:CGRAT}.
\begin{proof}[Proof of Theorem \ref{thr:CGRAT}]
For $f\in\mathcal{O}(X,\tau,g)$ let $\tilde{f}:=f+\mathcal{J}(X,\tau,g)$ denote the quotient class to which $f$ belongs and let $\pi$ be a prime element of $L$. Note, $\mathcal{O}(X,\tau,g)/\mathcal{J}(X,\tau,g)$ is endowed with the usual quotient operations and the quotient norm which in this case gives the trivial valuation. We begin by establishing a set $\mathcal{R}(X,\tau,g)\subseteq\mathcal{O}(X,\tau,g)$ of quotient class representatives for $\mathcal{O}(X,\tau,g)/\mathcal{J}(X,\tau,g)$. By Lemma \ref{lem:CGIRC} there is a set $\mathcal{R}_{L,g}$ of residue class representatives for $\overline{L}$ such that $g|_{\mathcal{R}_{L,g}}:\mathcal{R}_{L,g}\rightarrow\mathcal{R}_{L,g}$. Furthermore by Lemma \ref{lem:CGSLCF} every $f\in C_{L}(X)$ has a unique expansion of the form
\begin{equation}
\label{equ:CGXRL}
f=\sum_{i=n}^{\infty}f_{i}\pi^{i},\quad\mbox{for some }n\in\mathbb{Z},
\end{equation}
where, for each $i\geq n$, $f_{i}:X\rightarrow\mathcal{R}_{L,g}$ is a locally constant function. Hence, using expansion (\ref{equ:CGXRL}), for $f\in\mathcal{O}(X,\tau,g)$ we have
\begin{equation*}
f_{0}\circ\tau+h\circ\tau=f\circ\tau=g\circ f=g\circ f_{0}+g\circ h
\end{equation*}
where $h:=\sum_{i=1}^{\infty}f_{i}\pi^{i}$ with $\omega(h(x))>0$ for all $x\in X$. Now note that $\tau:X\rightarrow X$ and $g|_{\mathcal{R}_{L,g}}:\mathcal{R}_{L,g}\rightarrow\mathcal{R}_{L,g}$ give $f_{0}\circ\tau:X\rightarrow\mathcal{R}_{L,g}$ and $g\circ f_{0}:X\rightarrow\mathcal{R}_{L,g}$. Further since $g$ is an isometry on $L$ we have $\omega(h\circ\tau(x))>0$ and $\omega(g\circ h(x))>0$ for all $x\in X$. Hence since the expansion of $f\circ\tau$ in the for of (\ref{equ:CGXRL}) is unique we have $f_{0}\circ\tau=g\circ f_{0}$ and $h\circ\tau=g\circ h$. Moreover $f_{0}$ is continuous since locally constant and, for $x\in X$,
\begin{equation*}
\omega(f_{0}(x))=\left\{ \begin{array} {l@{\quad\mbox{if}\quad}l}
\infty & f_{0}(x)=0 \\
0 & f_{0}(x)\not=0.
\end{array} \right.
\end{equation*}
In particular we have $f_{0}\in\mathcal{O}(X,\tau,g)$. Hence we also have $h=f-f_{0}\in\mathcal{O}(X,\tau,g)$ since $\mathcal{O}(X,\tau,g)$ is a ring. But since $\omega(h(x))>0$ for all $x\in X$ we in fact have $h\in\mathcal{J}(X,\tau,g)$ giving
\begin{equation*}
\tilde{f}=\widetilde{f_{0}}.
\end{equation*}
Now by the uniqueness of expansions in the form of (\ref{equ:CGXRL}) and since $\mathcal{J}(X,\tau,g)$ is an ideal we have for any other element $f'=f'_{0}+h'\in\mathcal{O}(X,\tau,g)$ that $\widetilde{f'}=\tilde{f}$ if and only if $f'_{0}=f_{0}$. Hence using expansion (\ref{equ:CGXRL}) we define
\begin{equation*}
\mathcal{R}(X,\tau,g):=\left\{f_{0}:f=\sum_{i=0}^{\infty}f_{i}\pi^{i}\in\mathcal{O}(X,\tau,g)\right\}
\end{equation*}
noting that $0\in\mathcal{R}(X,\tau,g)$ since $0\in\mathcal{R}_{L,g}$ and $0\in\mathcal{O}(X,\tau,g)$. We now define a map $\phi:\mathcal{O}(X,\tau,g)/\mathcal{J}(X,\tau,g)\rightarrow C(X,\tau,\bar{g})$ by
\begin{equation*}
\phi(\tilde{f})=\phi(\widetilde{f_{0}})=\phi(f_{0}+\mathcal{J}(X,\tau,g)):=\overline{f_{0}}
\end{equation*}
where for $x\in X$ we define $\overline{f_{0}}(x):=\overline{f_{0}(x)}=f_{0}(x)+\mathcal{M}_{L}$. We show that $\phi$ is a ring isomorphism by checking that:
\begin{enumerate}
\item[(i)]
for all $\tilde{f}\in\mathcal{O}(X,\tau,g)/\mathcal{J}(X,\tau,g)$ we have $\overline{f_{0}}\in C(X,\tau,\bar{g})$;
\item[(ii)]
$\phi$ is multiplicative, linear and $\phi(\tilde{1})=\bar{1}$;
\item[(iii)]
$\mbox{ker}(\phi)=\{\tilde{0}\}$ ensuring that $\phi$ is injective;
\item[(iv)]
$\phi$ is surjective.
\end{enumerate}
For (i), since $f_{0}$ is a locally constant function on $X$ we have $\overline{f_{0}}\in C_{\overline{L}}(X)$. Furthermore we have already shown above that $f_{0}\circ\tau=g\circ f_{0}$. Hence for each $x\in X$ we have $\overline{f_{0}}(\tau(x))=\overline{f_{0}(\tau(x))}=\overline{g(f_{0}(x))}=\bar{g}\left(\overline{f_{0}(x)}\right)=\bar{g}\left(\overline{f_{0}}(x)\right)$ and so $\overline{f_{0}}\in C(X,\tau,\bar{g})$.\\
For (ii), let $\tilde{f},\widetilde{f'}\in\mathcal{O}(X,\tau,g)/\mathcal{J}(X,\tau,g)$. We show that $\phi$ is multiplicative. Set $h:=f_{0}f'_{0}$ giving $f_{0}f'_{0}=h_{0}+h'$ with $h'\in\mathcal{J}(X,\tau,g)$ and $f_{0},f'_{0},h_{0}\in\mathcal{R}(X,\tau,g)$. Hence for each $x\in X$ we have $h'(x)\in\mathcal{M}_{L}$. Therefore for each $x\in X$ we have
\begin{align*}
\phi(\tilde{f}\widetilde{f'})(x)=\phi(\widetilde{f_{0}}\widetilde{f'_{0}})(x)=&\phi(\widetilde{f_{0}f'_{0}})(x)\\
=&\phi(\widetilde{h_{0}})(x)\\
=&\overline{h_{0}}(x)\\
=&\overline{h_{0}(x)}\\
=&\overline{h_{0}(x)+h'(x)}\\
=&\overline{f_{0}(x)f'_{0}(x)}\\
=&\overline{f_{0}}(x)\overline{f'_{0}}(x)\\
=&(\phi(\widetilde{f_{0}})\phi(\widetilde{f'_{0}}))(x)=(\phi(\tilde{f})\phi(\widetilde{f'}))(x).
\end{align*}
Linearity, $\phi(\tilde{f}+\widetilde{f'})=\phi(\tilde{f})+\phi(\widetilde{f'})$, is shown in much the same way. Showing that $\phi(\tilde{1})=\bar{1}$ is almost immediate. Let $1_{0}$ be the representative in $\mathcal{R}_{L,g}$ of $\bar{1}$. Then we have $1_{0}\in\mathcal{R}(X,\tau,g)$ giving $\phi(\tilde{1})=\phi(\widetilde{1_{0}})=\overline{1_{0}}=\bar{1}$ as required. In fact we can always choose $\mathcal{R}_{L,g}$ such that $1_{0}=1$.\\
For (iii), let $f\in\mathcal{O}(X,\tau,g)$. If for all $x\in X$ we have $\phi(\widetilde{f_{0}})(x)=\overline{f_{0}(x)}=\bar{0}$ then $\omega(f_{0}(x))>0$ for all $x\in X$ giving $f_{0}\in\mathcal{J}(X,\tau,g)$. Hence $\widetilde{f_{0}}=\tilde{0}$ and so $\mbox{ker}(\phi)=\{\tilde{0}\}$. In fact since $f_{0}$ is an element of $\mathcal{R}(X,\tau,g)$ we have $f_{0}=0$ in this case.\\
For (iv), given $\bar{f}\in C(X,\tau,\bar{g})$ and $x\in X$ we have $\bar{f}(x)=a_{0}(x)+\mathcal{M}_{L}$ for some element $a_{0}(x)$ of $\mathcal{R}_{L,g}$ since $\mathcal{R}_{L,g}$ is a set of residue class representatives for $\overline{L}$. Since the valuation on $\overline{L}$ is the trivial valuation, $\bar{f}$ is a locally constant function and hence, when viewed as a function on $X$, so is $a_{0}$. Therefore $a_{0}$ is a continuous $\mathcal{O}_{L}$ valued function noting that $\mathcal{R}_{L,g}\subseteq\mathcal{O}_{L}$. Further since $\bar{f}\in C(X,\tau,\bar{g})$ we have for each $x\in X$ that
\begin{equation*}
\overline{a_{0}(\tau(x))}=\bar{f}(\tau(x))=\bar{g}(\bar{f}(x))=\bar{g}\left(\overline{a_{0}(x)}\right)=\overline{g(a_{0}(x))}.
\end{equation*}
Now because $g|_{\mathcal{R}_{L,g}}:\mathcal{R}_{L,g}\rightarrow\mathcal{R}_{L,g}$ we have $g(a_{0}(x))\in\mathcal{R}_{L,g}$ giving $a_{0}(\tau(x))=g(a_{0}(x))$ for all $x\in X$. Hence, as a function on $X$, $a_{0}\in\mathcal{O}(X,\tau,g)$ and so $a_{0}\in\mathcal{R}(X,\tau,g)$ with $\phi(\widetilde{a_{0}})=\overline{a_{0}}=\bar{f}$ as required. Finally, since the valuation on $\overline{L}$ is the trivial valuation, the sup norm on $C(X,\tau,\bar{g})$ is the trivial norm. Therefore it is immediate that $\phi$ is an isometry and this completes the proof of Theorem \ref{thr:CGRAT}.
\end{proof}
The last result of this section follows easily from the preceding results.
\begin{corollary}
\label{cor:CGLCUD}
Let $F$, $L$ and $g\in\mbox{Gal}(^{L}/_{F})$ conform to Lemma \ref{lem:CGIRC} with $F$ and $L$ having complete nontrivial discrete valuations and let $C(X,\tau,g)$ be a basic $^{L}/_{L^{g}}$ function algebra. Further let $\mathcal{R}(X,\tau,g)\subseteq\mathcal{O}(X,\tau,g)$ be the subset of all locally constant $\mathcal{R}_{L,g}$ valued functions. If there is a prime element $\pi$ of $L$ such that $g(\pi)=\pi$ then each $f\in C(X,\tau,g)\backslash\{0\}$ has a unique series expansion of the form
\begin{equation*}
f=\sum_{i=n}^{\infty}f_{i}\pi^{i},\quad\mbox{for some }n\in\mathbb{Z},
\end{equation*}
where for each $i\geq n$ we have $f_{i}\in\mathcal{R}(X,\tau,g)$. In particular the subset of all locally constant functions, $\mathrm{LC}(X,\tau,g)\subseteq C(X,\tau,g)$, is uniformly dense in $C(X,\tau,g)$.
\end{corollary}
\begin{remark}
\label{rem:CGLCUD}
For $L$ an unramified extension of $F$, every prime element $\pi\in F$ is a prime element of $L$ with $g(\pi)=\pi$. In particular Corollary \ref{cor:CGLCUD} holds when $L$ is a finite unramified extension of $\mathbb{Q}_{p}$ as is the case for examples \ref{exa:CGEXONE} and \ref{exa:CGEXTWO}.
\end{remark}
\begin{proof}[Proof of Corollary \ref{cor:CGLCUD}]
Let $f$ be an element of $C(X,\tau,g)\backslash\{0\}$ and let $\pi$ be a prime element of $L$ with $g(\pi)=\pi$. By Lemma \ref{lem:CGSLCF} $f$ has a unique series expansion of the form
\begin{equation*}
f=\sum_{i=n}^{\infty}f_{i}\pi^{i},\quad\mbox{for some }n\in\mathbb{Z},
\end{equation*}
with $f_{n}\not=0$ and $f_{i}:X\rightarrow\mathcal{R}_{L,g}$ a locally constant function for all $i\geq n$. Hence
\begin{equation*}
\sum_{i=n}^{\infty}f_{i}\circ\tau\pi^{i}=f\circ\tau=g\circ f=\sum_{i=n}^{\infty}(g\circ f_{i})(g(\pi))^{i}=\sum_{i=n}^{\infty}g\circ f_{i}\pi^{i}.
\end{equation*}
Therefore since $g$ restricts to an endofunction on $\mathcal{R}_{L,g}$ and by the uniqueness of the expansion we have $f_{i}\circ\tau=g\circ f_{i}$ for all $i\geq n$. Hence $f_{i}$ is an element of $\mathcal{R}(X,\tau,g)$ for all $i\geq n$ and for each $m\in\mathbb{N}$ we have $\sum_{i=n}^{n+m-1}f_{i}\pi^{i}\in C(X,\tau,g)$. Finally $\left(\sum_{i=n}^{n+m-1}f_{i}\pi^{i}\right)_{m}$ is a sequence of locally constant functions which converges uniformly to $f$ as required.
\end{proof}
This brings us to the end of Chapter \ref{cha:CG}. In the next chapter we will see that $^{L}/_{L^{g}}$ function algebras have a part to play in representation theory.
	\chapter[Representation theory]{Representation theory}
\label{cha:RT}
The first section of this chapter introduces several results from the theory of Banach rings and Banach $F$-algebras that we will use later in the chapter. These results have been taken from \cite[Ch1]{Berkovich}. However I have provided a thorough proof of each result in order to give significantly more detail than \cite{Berkovich} since some of them may not be widely known. The second section begins by recalling which Banach $F$-algebras can be represented by complex uniform algebras or real function algebras in the Archimedean setting and one such result in the non-Archimedean setting provided by \cite{Berkovich} is also noted. We then develop this theory further by identifying a large class of Banach $F$-algebras that can be represented by $^{L}/_{L^{g}}$ function algebras. The resulting representation theorem is the main result of interest in this chapter and the rest of the chapter is given over to the proof of the theorem.
\section{Further Banach rings and Banach {\it F}-algebras}
\label{sec:RTBR}
Since the definition of a Banach ring was given in Definition \ref{def:FAABA} we begin with the first lemma.
\begin{lemma}
\label{lem:RTRRT}
Let $R$ be a Banach ring and let $r\in\mathbb{R}$ be positive. Define
\begin{equation*}
R\langle r^{-1}T\rangle:=\{f=\sum_{i=0}^{\infty}a_{i}T^{i}:a_{i}\in R\mbox{ and }\sum_{i=0}^{\infty}\|a_{i}\|_{R}r^{i}<\infty\}.
\end{equation*}
Then with the Cauchy product and usual addition:
\begin{enumerate}
\item[(i)]
we have that $R\langle r^{-1}T\rangle$ is a Banach ring with respect to the norm
\begin{equation*}
\|f\|_{R,r}:=\sum_{i=0}^{\infty}\|a_{i}\|_{R}r^{i};
\end{equation*}
\item[(ii)]
for $a\in R$ we have $1-aT$ invertible in $R\langle r^{-1}T\rangle$ if and only if $\sum_{i=0}^{\infty}\|a^{i}\|_{R}r^{i}<\infty$.
\end{enumerate}
\end{lemma}
\begin{proof}
For (i), let $f_{1}=\sum_{i=0}^{\infty}a_{i}T^{i}$ and $f_{2}=\sum_{i=0}^{\infty}b_{i}T^{i}$ be elements of $R\langle r^{-1}T\rangle$. Then
\begin{align*}
\|f_{1}+f_{2}\|_{R,r}=&\left\|\sum_{i=0}^{\infty}(a_{i}+b_{i})T^{i}\right\|_{R,r}\\
=&\sum_{i=0}^{\infty}\|a_{i}+b_{i}\|_{R}r^{i}\\
\leq&\sum_{i=0}^{\infty}\left(\|a_{i}\|_{R}+\|b_{i}\|_{R}\right)r^{i}\\
=&\sum_{i=0}^{\infty}\|a_{i}\|_{R}r^{i}+\sum_{i=0}^{\infty}\|b_{i}\|_{R}r^{i}\\
=&\|f_{1}\|_{R,r}+\|f_{2}\|_{R,r}<\infty
\end{align*}
showing that $R\langle r^{-1}T\rangle$ is closed under addition and that the triangle inequality holds for $\|\cdot\|_{R,r}$. Clearly $\|f\|_{R,r}=0$ if and only if $f=0$. Further
\begin{align*}
\|f_{1}f_{2}\|_{R,r}=&\left\|\sum_{i=0}^{\infty}\left(\sum_{k=0}^{i}a_{k}b_{i-k}\right)T^{i}\right\|_{R,r}\\
=&\sum_{i=0}^{\infty}\left\|\sum_{k=0}^{i}a_{k}b_{i-k}\right\|_{R}r^{i}\\
\leq&\sum_{i=0}^{\infty}\left(\sum_{k=0}^{i}\|a_{k}\|_{R}\|b_{i-k}\|_{R}\right)r^{i}\\
=&\left(\sum_{i=0}^{\infty}\|a_{i}\|_{R}r^{i}\right)\left(\sum_{i=0}^{\infty}\|b_{i}\|_{R}r^{i}\right),\mbox{ by Mertens' Theorem, see \cite[p204]{Apostol}}\\
=&\|f_{1}\|_{R,r}\|f_{2}\|_{R,r}<\infty
\end{align*}
showing that $R\langle r^{-1}T\rangle$ is closed under multiplication and $\|\cdot\|_{R,r}$ is sub-multiplicative. Furthermore we have $1_{R,r}=1_{R}T^{0}$ which gives $\|1_{R,r}\|_{R,r}=\|1_{R}\|_{R}r^{0}=1$ and similarly $\|-1_{R,r}\|_{R,r}=\|-1_{R}\|_{R}r^{0}=1$.\\
We now show that $R\langle r^{-1}T\rangle$ is complete. Let $\left(\sum_{i=0}^{\infty}a_{i,n}T^{i}\right)_{n}$ be a Cauchy sequence in $R\langle r^{-1}T\rangle$. Then for $k\in \mathbb{N}_{0}$ and $\varepsilon>0$ there exists $M\in\mathbb{N}$ such that for all $m,m'\geq M$ we have $\|\sum_{i=0}^{\infty}a_{i,m}T^{i}-\sum_{i=0}^{\infty}a_{i,m'}T^{i}\|_{R,r}=\sum_{i=0}^{\infty}\|a_{i,m}-a_{i,m'}\|_{R}r^{i}<\varepsilon r^{k}$. Hence for all $m,m'\geq M$ we have $\|a_{k,m}-a_{k,m'}\|_{R}<\varepsilon$ and so for all $k\in \mathbb{N}_{0}$, $(a_{k,n})_{n}$ is a Cauchy sequence in $R$. Since $R$ is a Banach ring $(a_{k,n})_{n}$ converges to some $b_{k}\in R$.\\
We show that $\sum_{i=0}^{\infty}b_{i}T^{i}$ is an element of $R\langle r^{-1}T\rangle$. Let $\varepsilon_{0}>0$. Then there exists $M\in\mathbb{N}$ such that for all $m\geq M$ we have
\begin{align*}
\left\|\sum_{i=0}^{\infty}a_{i,m}T^{i}\right\|_{R,r}\leq&\left\|\sum_{i=0}^{\infty}a_{i,m}T^{i}-\sum_{i=0}^{\infty}a_{i,M}T^{i}\right\|_{R,r}+\left\|\sum_{i=0}^{\infty}a_{i,M}T^{i}\right\|_{R,r}\\
<&\varepsilon_{0}+\left\|\sum_{i=0}^{\infty}a_{i,M}T^{i}\right\|_{R,r}<\infty.
\end{align*}
Now let $N\in\mathbb{N}_{0}$ and $\varepsilon>0$. Since for each $k\in\mathbb{N}_{0}$, $(a_{k,n})_{n}$ is a Cauchy sequence in $R$ with limit $b_{k}$, there is $M'\in\mathbb{N}$ such that for all $m'\geq M'$ we have $\sum_{i=0}^{N}\|b_{i}-a_{i,m'}\|_{R}r^{i}<\varepsilon$. Hence letting $m_{0}\geq\max\{M,M'\}$ gives
\begin{align*}
\left\|\sum_{i=0}^{N}b_{i}T^{i}\right\|_{R,r}=&\left\|\sum_{i=0}^{N}b_{i}T^{i}-\sum_{i=0}^{N}a_{i,m_{0}}T^{i}+\sum_{i=0}^{N}a_{i,m_{0}}T^{i}\right\|_{R,r}\\
\leq&\left\|\sum_{i=0}^{N}(b_{i}-a_{i,m_{0}})T^{i}\right\|_{R,r}+\left\|\sum_{i=0}^{\infty}a_{i,m_{0}}T^{i}\right\|_{R,r}\\
<&\varepsilon+\varepsilon_{0}+\left\|\sum_{i=0}^{\infty}a_{i,M}T^{i}\right\|_{R,r}.
\end{align*}
Since $\varepsilon>0$ was arbitrary we have $\|\sum_{i=0}^{N}b_{i}T^{i}\|_{R,r}\leq\varepsilon_{0}+\|\sum_{i=0}^{\infty}a_{i,M}T^{i}\|_{R,r}$. Since this holds for each $N\in\mathbb{N}_{0}$ we have $\|\sum_{i=0}^{\infty}b_{i}T^{i}\|_{R,r}\leq\varepsilon_{0}+\|\sum_{i=0}^{\infty}a_{i,M}T^{i}\|_{R,r}$ giving $\sum_{i=0}^{\infty}b_{i}T^{i}\in R\langle r^{-1}T\rangle$ as required.\\
Let $\varepsilon>0$. We will show, for large enough $n\in\mathbb{N}$, that $\|\sum_{i=0}^{\infty}a_{i,n}T^{i}-\sum_{i=0}^{\infty}b_{i}T^{i}\|_{R,r}<\varepsilon$ and so $R\langle r^{-1}T\rangle$ is complete. Since $\left(\sum_{i=0}^{\infty}a_{i,n}T^{i}\right)_{n}$ is a Cauchy sequence there exists $M_{1}\in\mathbb{N}$ such that for all $m,n\geq M_{1}$ we have $\|\sum_{i=0}^{\infty}(a_{i,n}-a_{i,m})T^{i}\|_{R,r}<\varepsilon/4$.\\
Let $n\geq M_{1}$. Since $\sum_{i=0}^{\infty}a_{i,n}T^{i}$ and $\sum_{i=0}^{\infty}b_{i}T^{i}$ are elements of $R\langle r^{-1}T\rangle$ there exists $N\in\mathbb{N}$ such that $\|\sum_{i=N+1}^{\infty}a_{i,n}T^{i}\|_{R,r}<\varepsilon/4$ and $\|\sum_{i=N+1}^{\infty}b_{i}T^{i}\|_{R,r}<\varepsilon/4$.\\
Since for each $i\in\mathbb{N}_{0}$, $(a_{i,m})_{m}$ is a Cauchy sequence in $R$ with limit $b_{i}$, there is $M_{2}\in\mathbb{N}$ such that for all $m\geq M_{2}$ we have $\|\sum_{i=0}^{N}(a_{i,m}-b_{i})T^{i}\|_{R,r}=\sum_{i=0}^{N}\|a_{i,m}-b_{i}\|_{R}r^{i}<\varepsilon/4$. Let $m=\max\{M_{1},M_{2}\}$ and define $c_{n}:=\|\sum_{i=0}^{\infty}a_{i,n}T^{i}-\sum_{i=0}^{\infty}b_{i}T^{i}\|_{R,r}$, then
\begin{align*}
c_{n}=&\left\|\sum_{i=N+1}^{\infty}a_{i,n}T^{i}+\sum_{i=0}^{N}(a_{i,n}-a_{i,m})T^{i}+\sum_{i=0}^{N}(a_{i,m}-b_{i})T^{i}-\sum_{i=N+1}^{\infty}b_{i}T^{i}\right\|_{R,r}\\
<&\varepsilon/4+\varepsilon/4+\varepsilon/4+\varepsilon/4=\varepsilon\mbox{ as required.}
\end{align*}
For (ii), the result is obvious for $a=0$ and so suppose $a\not=0$. If $\sum_{i=0}^{\infty}\|a^{i}\|_{R}r^{i}<\infty$ then $\sum_{i=0}^{\infty}a^{i}T^{i}$ is an element of $R\langle r^{-1}T\rangle$ and by the definition of the Cauchy product we have
\begin{align*}
\left(\sum_{i=0}^{\infty}a^{i}T^{i}\right)(1-aT)=&(a^{0}1_{R})T^{0}+\sum_{i=1}^{\infty}(a^{i}1_{R}+a^{i-1}(-a))T^{i}\\
=&1_{R}T^{0}+\sum_{i=1}^{\infty}a^{i-1}(a+(-a))T^{i}\\
=&1_{R}T^{0}=1_{R,r}=1.
\end{align*}
Similarly this holds for $(1-aT)(\sum_{i=0}^{\infty}a^{i}T^{i})$ and so $1-aT$ is invertible.\\
Now conversely if $1-aT$ is invertible in $R\langle r^{-1}T\rangle$ then for $\sum_{i=0}^{\infty}b_{i}T^{i}$ the inverse of $1-aT$ in $R\langle r^{-1}T\rangle$ we have by the definition of the Cauchy product
\begin{equation*}
1_{R,r}=\left(\sum_{i=0}^{\infty}b_{i}T^{i}\right)(1-aT)=(b_{0}1_{R})T^{0}+\sum_{i=1}^{\infty}(b_{i}1_{R}+b_{i-1}(-a))T^{i}.
\end{equation*}
Hence $b_{0}=1_{R}=a^{0}$ and, for each $i\in\mathbb{N}$, $0=b_{i}+b_{i-1}(-a)$ giving
\begin{equation*}
b_{i-1}a=b_{i}+b_{i-1}a+b_{i-1}(-a)=b_{i}+b_{i-1}(a+(-a))=b_{i}.
\end{equation*}
Therefore for each $i\in\mathbb{N}$, $b_{i}=b_{i-1}a$ with $b_{0}=1_{R}$ giving $b_{i}=a^{i}$ by induction. Hence $\sum_{i=0}^{\infty}a^{i}T^{i}=\sum_{i=0}^{\infty}b_{i}T^{i}$ is an element of $R\langle r^{-1}T\rangle$ and so $\sum_{i=0}^{\infty}\|a^{i}\|_{R}r^{i}<\infty$ as required.
\end{proof}
\begin{remark}
\label{rem:RTRRT}
Since $R\langle r^{-1}T\rangle$ extends $R$ as a ring and by the definition of the Cauchy product it is immediate that $R\langle r^{-1}T\rangle$ is commutative if and only if $R$ is commutative. Similarly by the definition of the norm $\|\cdot\|_{A,r}$ and the Cauchy product, if $A$ is a unital Banach $F$-algebra then $A\langle r^{-1}T\rangle$ is also a unital Banach $F$-algebra. These details are easily checked.
\end{remark}
The following definitions will be used many times in this chapter.
\begin{definition}
\label{def:RTBMS}
Let $R$ be a Banach ring. A {\em bounded multiplicative seminorm} on $R$ is a map $|\cdot|:R\rightarrow\mathbb{R}$ taking non-negative values that is:
\begin{enumerate}
\item[(1)]
bounded, $|a|\leq\|a\|_{R}$ for all $a\in R$, but not constantly zero on $R$;
\item[(2)]
multiplicative, $|ab|=|a||b|$ for all $a,b\in R$ and hence $|1_{R}|=1$ by setting $a=1_{R}$ and $b\not\in\mbox{ker}(|\cdot|)$;
\item[(3)]
a seminorm and so $|\cdot|$ also satisfies the triangle inequality and $0\in\mbox{ker}(|\cdot|)$ but the kernel is not assumed to be a singleton.
\end{enumerate}
\end{definition}
\begin{definition}
\label{def:RTMOA}
Let $F$ be a complete non-Archimedean field and let $A$ be a commutative unital Banach $F$-algebra. In this chapter $\mathcal{M}_{0}(A)$ will denote the set of all proper closed prime ideals of $A$ that are the kernels of bounded multiplicative seminorms on $A$. For $x_{0}\in\mathcal{M}_{0}(A)$, or any proper closed ideal of $A$, we will denote the quotient norm on $A/x_{0}$ by $|\cdot|_{x_{0}}$ that is $|a+x_{0}|_{x_{0}}:=\inf\{\|a+b\|_{A}:b\in x_{0}\}$ for $a\in A$.
\end{definition}
We now proceed with a number of Lemmas. In particular towards the end of Section \ref{sec:RTBR} it will be show that $\mathcal{M}_{0}(A)$ is always nonempty.
\begin{lemma}
\label{lem:RTBMS}
Let $A$ be a unital Banach $F$-algebra. If $|\cdot|$ is a bounded multiplicative seminorm on $A$ as a Banach ring then we have $|\alpha|=|\alpha|_{F}$ for all $\alpha\in F$. Hence since $|\cdot|$ is multiplicative it is also a vector space seminorm, that is $|\alpha a|=|\alpha|_{F}|a|$ for all $a\in A$ and $\alpha\in F$.
\end{lemma}
\begin{proof}
For $\alpha\in F^{\times}$ we first note that $1=|1_{A}|=|\alpha\alpha^{-1}|=|\alpha||\alpha^{-1}|$ and so $|\alpha|\not=0$ and $|\alpha^{-1}|=|\alpha|^{-1}$. Similarly $|\alpha^{-1}|_{F}=|\alpha|_{F}^{-1}$ since $|\cdot|_{F}$ is a valuation. Moreover since $|\cdot|$ is bounded we have $|\alpha|\leq\|\alpha 1_{A}\|_{A}=|\alpha|_{F}\|1_{A}\|_{A}=|\alpha|_{F}$. Since this holds for all $\alpha\in F^{\times}$ we also have $|\alpha^{-1}|\leq |\alpha^{-1}|_{F}$ giving $|\alpha|_{F}\leq|\alpha|$ and so $|\alpha|=|\alpha|_{F}$ for all $\alpha\in F$.
\end{proof}
\begin{lemma}
\label{lem:RTQN}
Let $F$ and $A$ be as in Definition \ref{def:RTMOA}. For $x_{0}\in\mathcal{M}_{0}(A)$, or any proper closed ideal of $A$, the following holds:
\begin{enumerate}
\item[(i)]
the quotient ring $A/x_{0}$ has $F\subseteq A/x_{0}$ and is an integral domain if $x_{0}$ is prime;
\item[(ii)]
the quotient norm is such that $|\alpha+x_{0}|_{x_{0}}=|\alpha|_{F}$ for all $\alpha\in F$;
\item[(iii)]
the quotient norm $|\cdot|_{x_{0}}$ is an $F$-vector space norm on $A/x_{0}$, opposed to being merely a seminorm, and it is sub-multiplicative;
\item[(iv)]
if $\|\cdot\|_{A}$ is square preserving, that is $\|a^{2}\|_{A}=\|a\|_{A}^{2}$ for all $a\in A$, then both $\|\cdot\|_{A}$ and $|\cdot|_{x_{0}}$ observe the strong triangle inequality noting that $F$ is non-Archimedean;
\item[(v)]
by way of the map $a\mapsto |a+x_{0}|_{x_{0}}$, as a seminorm on $A$, $|\cdot|_{x_{0}}$ is bounded.
\end{enumerate}
\end{lemma}
\begin{proof}
For (i), if $a_{1}+x_{0}$, $a_{2}+x_{0}\in A/x_{0}$ with $(a_{1}+x_{0})(a_{2}+x_{0})=a_{1}a_{2}+x_{0}=0+x_{0}$ then we have $a_{1}a_{2}\in x_{0}$. Hence if $x_{0}$ is a prime ideal of $A$ then at least one of $a_{1}+x_{0}$ and $a_{2}+x_{0}$ is equal to $0+x_{0}$ and so $A/x_{0}$ is an integral domain. It is immediate that $A/x_{0}$ has a subset that is an isomorphic copy of $F$ since $x_{0}$ is a proper ideal of $A$.\\
For (ii), we first show that $|1+x_{0}|_{x_{0}}=1$. Note that $|1+x_{0}|_{x_{0}}\leq\|1\|_{A}=1$ since $0\in x_{0}$. So now suppose toward a contradiction that there is $b\in x_{0}$ such that $\|1+b\|_{A}<1$. We have for all $n\in\mathbb{N}$ that $b_{n}:=-((1+b)^{n}-1)$ is an element of $x_{0}$ since $x_{0}$ is an ideal of $A$. But $\|1-b_{n}\|_{A}=\|(1+b)^{n}\|_{A}\leq\|1+b\|_{A}^{n}$ with $\lim_{n\to\infty}\|1+b\|_{A}^{n}=0$ and so $1$ is an element of $x_{0}$ since $x_{0}$ is closed which contradicts $x_{0}$ being a proper ideal of $A$. We conclude that $\|1+b\|_{A}\geq1$ for all $b\in x_{0}$ and so $|1+x_{0}|_{x_{0}}\geq1$. Hence $|1+x_{0}|_{x_{0}}=1$ by the above. Now for $\alpha\in F^{\times}$ we have $x_{0}=\alpha x_{0}$ since $\alpha$ is invertible where $\alpha x_{0}:=\{\alpha b:b\in x_{0}\}$. Hence
\begin{align*}
|\alpha+x_{0}|_{x_{0}}=&\inf\{\|\alpha+b\|_{A}:b\in x_{0}\}\\
=&\inf\{\|\alpha+\alpha b\|_{A}:b\in x_{0}\}\\
=&\inf\{|\alpha|_{F}\|1+b\|_{A}:b\in x_{0}\}\\
=&|\alpha|_{F}|1+x_{0}|_{x_{0}}=|\alpha|_{F}
\end{align*}
as required. In a similar way for $a\in A$ one shows that $|\alpha a+x_{0}|_{x_{0}}=|\alpha|_{F}|a+x_{0}|_{x_{0}}$.\\
For (iii), we note that $|\cdot|_{x_{0}}$ is a norm on $A/x_{0}$ because $x_{0}$ is closed as a subset of $A$ so that for $a\in A\backslash x_{0}$ there is $\varepsilon>0$ with $\|a+b\|_{A}\geq\varepsilon$ for all $b\in x_{0}$ giving $|a+x_{0}|_{x_{0}}\geq\varepsilon$. We now show that $|\cdot|_{x_{0}}$ is sub-multiplicative. For $a_{1},a_{2}\in A$ we have
\begin{align*}
|a_{1}a_{2}+x_{0}|_{x_{0}}=&\inf\{\|a_{1}a_{2}+b\|_{A}:b\in x_{0}\}\\
\leq&\inf\{\|a_{1}a_{2}+a_{1}b_{2}+a_{2}b_{1}+b_{1}b_{2}\|_{A}:b_{1},b_{2}\in x_{0}\}\\
\leq&\inf\{\|a_{1}+b_{1}\|_{A}\|a_{2}+b_{2}\|_{A}:b_{1},b_{2}\in x_{0}\}\\
=&|a_{1}+x_{0}|_{x_{0}}|a_{2}+x_{0}|_{x_{0}}.
\end{align*}
For (iv), suppose $\|\cdot\|_{A}$ is square preserving. In this case the proof of Theorem \ref{thr:CVFCHAR} also works for $A$ and so $\|\cdot\|_{A}$ observes the strong triangle inequality, see \cite[p18]{Schikhof} for details. Hence for $a_{1},a_{2}\in A$ we also have
\begin{align*}
|a_{1}+a_{2}+x_{0}|_{x_{0}}=&\inf\{\|a_{1}+a_{2}+b\|_{A}:b\in x_{0}\}\\
=&\inf\{\|a_{1}+b_{1}+a_{2}+b_{2}\|_{A}:b_{1},b_{2}\in x_{0}\}\\
\leq&\inf\{\max\{\|a_{1}+b_{1}\|_{A},\|a_{2}+b_{2}\|_{A}\}:b_{1},b_{2}\in x_{0}\}\\
=&\max\{\inf\{\|a_{1}+b_{1}\|_{A}:b_{1}\in x_{0}\},\inf\{\|a_{2}+b_{2}\|_{A}:b_{2}\in x_{0}\}\}\\
=&\max\{|a_{1}+x_{0}|_{x_{0}},|a_{2}+x_{0}|_{x_{0}}\}.
\end{align*}
For (v), since we have $0\in x_{0}$ it is immediate that $|a+x_{0}|_{x_{0}}\leq\|a\|_{A}$ for all $a\in A$.
\end{proof}
\begin{lemma}
\label{lem:RTCBR}
Let $R$ be a commutative Banach ring. Then:
\begin{enumerate}
\item[(i)]
if $a\in R$ has $\|1-a\|_{R}<1$ then $a$ is invertible in $R$;
\item[(ii)]
for $I$ a proper ideal of $R$ the closure $J$ of $I$, as a subset of $R$, is a proper ideal of $R$;
\item[(iii)]
each non-invertible element of $R$ is an element of some maximal ideal of $R$. The maximal ideals of $R$ are proper, closed and prime.
\end{enumerate}
\end{lemma}
\begin{proof}
For (i), for $a\in R$ with $\|1-a\|_{R}<1$ let $\delta>0$ be such that $\|1-a\|_{R}<\delta<1$. Then setting $b:=1-a$ gives $\|b^{n}\|_{R}\leq\|b\|_{R}^{n}<\delta^{n}<1$ for all $n\in\mathbb{N}$. Therefore we have $\sum_{n=0}^{m}\|b^{n}\|_{R}<\sum_{n=0}^{\infty}\delta^{n}=\frac{1}{1-\delta}$ for each $m\in\mathbb{N}$ and so $\sum_{n=0}^{\infty}b^{n}\in R$ since $R$ is complete. Moreover
\begin{align*}
\left\|(1-b)\sum_{n=0}^{\infty}b^{n}-1\right\|_{R}=&\left\|(1-b)\sum_{n=0}^{\infty}b^{n}-(1-b)\sum_{n=0}^{m}b^{n}+(1-b)\sum_{n=0}^{m}b^{n}-1\right\|_{R}\\
\leq&\left(\sum_{n=m+1}^{\infty}\|b^{n}\|_{R}\right)\|1-b\|_{R}+\|b^{m+1}\|_{R}.
\end{align*}
Hence $a$ is invertible in $R$ since $1-b=1-(1-a)=a$ and
\begin{equation*}
\lim_{m\to\infty}\left(\left(\sum_{n=m+1}^{\infty}\|b^{n}\|_{R}\right)\|1-b\|_{R}+\|b^{m+1}\|_{R}\right)=0.
\end{equation*}
For (ii), let $I$ be a proper ideal of $R$ and $J$ its closure as a subset. For $a,b\in J$ there are sequences $(a_{n})$, $(b_{n})$ of elements of $I$ converging to $a$ and $b$ respectively with respect to $\|\cdot\|_{R}$. Hence for $a'\in R$, $(a'a_{n})$ is a sequence in $I$ with $\|a'a-a'a_{n}\|_{R}\leq\|a'\|_{R}\|a-a_{n}\|_{R}$ for each $n\in\mathbb{N}$ and so $a'a\in J$ since $\lim_{n\to\infty}\|a-a_{n}\|_{R}=0$. Similarly $(a_{n}+b_{n})$ is a sequence in $I$ with $\|(a+b)-(a_{n}+b_{n})\|_{R}\leq\|a-a_{n}\|_{R}+\|b-b_{n}\|_{R}$ for each $n\in\mathbb{N}$ and so $a+b\in J$. Hence $J$ is an ideal of $R$. Now since $I$ is a proper ideal of $R$ each element $a_{n}$ of the sequence $(a_{n})$ is not invertible and so $\|1-a_{n}\|_{R}\geq1$ for all $n\in\mathbb{N}$ by (i). Hence $1\leq\|1-a+a-a_{n}\|_{R}\leq\|1-a\|_{R}+\|a-a_{n}\|_{R}$ for all $n\in\mathbb{N}$ and so $\|1-a\|_{R}\geq1$ for all $a\in J$ giving $1\not\in J$. Hence $J$ is proper.\\
For (iii), let $a$ be a non-invertible element of $R$ noting that we can always take $a=0$. The principal ideal $I_{a}:=aR$ is proper since for all $b\in R$, $ab\not=1$. By Zorn's lemma $I_{a}$ is a subset of some maximal ideal $J_{a}$ of $R$. Every maximal ideal $J$ of $R$ is proper and prime, noting that $R/J$ is a field or by other means, and closed as a subset of $R$ by (ii).
\end{proof}
\begin{remark}
\label{rem:RTCBR}
We note that if $A$ is a commutative unital Banach $F$-algebra then Lemma \ref{lem:RTCBR} applies to $A$ and $A\langle r^{-1}T\rangle$ for each $r>0$.
\end{remark}
\begin{lemma}
\label{lem:RTAUX}
Let $F$ be a complete non-Archimedean field and let $A$ be a commutative unital Banach $F$-algebra with maximal ideal $m_{0}$. Let $S(A)$ denote the set of all norms on the field $A/m_{0}$ that are also unital bounded seminorms on $A$ as a Banach ring. That is if $|\cdot|$ is an element of $S(A)$ then $|1|=1$ and $|\cdot|$ conforms to Definition \ref{def:RTBMS} except it need not be multiplicative merely sub-multiplicative. It follows that:
\begin{enumerate}
\item[(i)]
the set $S(A)$ is non-empty;
\item[(ii)]
for $|\cdot|\in S(A)$, $\overline{A/m_{0}}$ the completion of $A/m_{0}$ with respect to $|\cdot|$, $r>0$ and $a\in A/m_{0}$, if $a-T$ is non-invertible in $\overline{A/m_{0}}\langle r^{-1}T\rangle$ then there is $|\cdot|'\in S(A)$ with $|a|'\leq r$ and $|b|'\leq|b|$ for all $b\in A/m_{0}$.
\end{enumerate}
\end{lemma}
\begin{proof}
For (i), we note that the quotient norm $|\cdot|_{m_{0}}$ is an element of $S(A)$ since (iii) of Lemma \ref{lem:RTCBR} shows that (ii), (iii) and (v) of Lemma \ref{lem:RTQN} apply to $|\cdot|_{m_{0}}$.\\
For (ii), suppose that $a-T$ is non-invertible in $\overline{A/m_{0}}\langle r^{-1}T\rangle$. Then $a-T$ is an element of some maximal ideal $J$ of $\overline{A/m_{0}}\langle r^{-1}T\rangle$ by Lemma \ref{lem:RTCBR}. Hence the quotient norm $|\cdot|_{J}$ on $\overline{A/m_{0}}\langle r^{-1}T\rangle/J$ is an element of $S(\overline{A/m_{0}}\langle r^{-1}T\rangle)$ by Lemma \ref{lem:RTQN}. Therefore since $J$ is closed and $A/m_{0}$ is a field, $|a'|':=|a'+J|_{J}$, for $a'\in A/m_{0}$, defines a norm on $A/m_{0}$. Since $|\cdot|_{J}$ is an element of $S(\overline{A/m_{0}}\langle r^{-1}T\rangle)$ we have that $|\cdot|'$ is unital as a seminorm on $A$. Similarly since $|\cdot|_{J}$ is bounded as a seminorm on $\overline{A/m_{0}}\langle r^{-1}T\rangle$ we have for all $a'\in A$ that
\begin{equation}
\label{equ:RTAUX}
|a'+m_{0}|'=|(a'+m_{0})+J|_{J}\leq\|a'+m_{0}\|_{\overline{A/m_{0}},r}=|a'+m_{0}|\leq\|a'\|_{A}
\end{equation}
noting that $|\cdot|$ is an element of $S(A)$. Hence $|\cdot|'$ is bounded as a seminorm on $A$ and so $|\cdot|'$ is an element of $S(A)$. Further since $a-T$ is an element of $J$ we have
\begin{equation*}
|a|'=|a+J|_{J}=|T+J|_{J}\leq\|T\|_{\overline{A/m_{0}},r}=r.
\end{equation*}
Finally we have $|b|'\leq|b|$ for all $b\in A/m_{0}$ by (\ref{equ:RTAUX}).
\end{proof}
The following Lemma will be particularly useful in Section \ref{sec:RTR}.
\begin{lemma}
\label{lem:RTNMT}
Let $F$ be a complete non-Archimedean field and let $A$ be a commutative unital Banach $F$-algebra. With reference to Definition \ref{def:RTMOA} the following holds:
\begin{enumerate}
\item[(i)]
the set $\mathcal{M}_{0}(A)$ is non-empty since every maximal ideal of $A$ is an element of $\mathcal{M}_{0}(A)$;
\item[(ii)]
an element $a\in A$ is invertible if and only if $a+x_{0}\not=0+x_{0}$ for all $x_{0}\in\mathcal{M}_{0}(A)$.
\end{enumerate}
\end{lemma}
\begin{proof}
Whilst this proof provides more detail, much of the following has been taken from \cite[Ch1]{Berkovich}. For (i), let $m_{0}$ be a maximal ideal of $A$. Hence the quotient ring $A/m_{0}$ is a field. Let $S(A)$ be as in Lemma \ref{lem:RTAUX} and note therefore that $S(A)$ is non-empty. We put a partial order on $S(A)$ by $|\cdot|\lesssim|\cdot|'$ if and only if $|a+m_{0}|\leq|a+m_{0}|'$ for all $a+m_{0}\in A/m_{0}$. Now let $E$ be a chain in $S(A)$, that is $E$ is a subset of $S(A)$ such that $\lesssim$ restricts to a total order on $E$. Define a map $|\cdot|_{E}:A/m_{0}\rightarrow\mathbb{R}$ by
\begin{equation*}
|a+m_{0}|_{E}:=\inf\{|a+m_{0}|:|\cdot|\in E\}.
\end{equation*}
We will show that $|\cdot|_{E}$ is a lower bound for $E$ in $S(A)$. It is immediate from the definition of $|\cdot|_{E}$ that it is unital and bounded since all of the elements of $E$ are. Hence it suffices to show that $|\cdot|_{E}$ is a sub-multiplicative norm on $A/m_{0}$. Clearly $|0+m_{0}|_{E}=0$ so, simplifying our notation slightly, let $a$ be an element of $A/m_{0}^{\times}$. We show that $|a|_{E}\not=0$. Let $|\cdot|$ be an element of $E$ and suppose towards a contradiction that there is $|\cdot|'\in E$ such that $|a|'<\min\{|a|,|a^{-1}|^{-1}\}$. Then
\begin{equation*}
1=|1|'=|aa^{-1}|'\leq|a|'|a^{-1}|'<|a^{-1}|^{-1}|a^{-1}|'.
\end{equation*}
Hence by the above we have $|a^{-1}|<|a^{-1}|'$ and $|a|'<|a|$ giving $|\cdot|'\not\lesssim|\cdot|$ and $|\cdot|\not\lesssim|\cdot|'$ which contradicts both $|\cdot|$ and $|\cdot|'$ being elements of $E$. Therefore for all $|\cdot|'\in E$ we have $|a|'\geq\min\{|a|,|a^{-1}|^{-1}\}$. In particular $|a|_{E}\not=0$. Now for $a,b\in A/m_{0}$ we have
\begin{align*}
|a+b|_{E}=&\inf\{|a+b|:|\cdot|\in E\}\\
\leq&\inf\{|a|+|b|:|\cdot|\in E\}\\
=&\inf\{|a|+|b|':|\cdot|,|\cdot|'\in E\},\quad\dag\\
=&|a|_{E}+|b|_{E},
\end{align*}
where line $\dag$ follows from the line above it because if $|\cdot|\lesssim |\cdot|'$ then $|a|+|b|\leq |a|+|b|'$. Hence the triangle inequality holds for $|\cdot|_{E}$. Similarly we have $|ab|_{E}\leq |a|_{E}|b|_{E}$ and so $|\cdot|_{E}$ is sub-multiplicative as required. Hence $|\cdot|_{E}$ is a lower bound for $E$ in $S(A)$. Therefore by Zorn's lemma there exists a minimal element of $S(A)$ with respect to $\lesssim$. Let $|\cdot|$ be a minimal element of $S(A)$ and denote by $\overline{A/m_{0}}$ the completion of $A/m_{0}$ with respect to $|\cdot|$. We will show that $|\cdot|$ is multiplicative on $A/m_{0}$ and hence satisfies (i) of Lemma \ref{lem:RTNMT}. Note that for now we should only take $\overline{A/m_{0}}$ to be an integral domain and not a field since we can't apply Theorem \ref{thr:CVFCOM}.\\
Now since $A/m_{0}$ is a field $|\cdot|$ will be multiplicative if $|a^{-1}|=|a|^{-1}$ for all $a\in A/m_{0}^{\times}$ since for $a,b\in A/m_{0}^{\times}$ with $|a^{-1}|=|a|^{-1}$ we have $|b|=|baa^{-1}|\leq|ba||a^{-1}|=|ba||a|^{-1}$ giving $|a||b|\leq|ab|$ and since $|\cdot|$ is sub-multiplicative we have $|ab|=|a||b|$. Hence we will show that $|a^{-1}|=|a|^{-1}$ for all $a\in A/m_{0}^{\times}$. To this end we first show that $|\cdot|$ is power multiplicative that is $|a^{n}|=|a|^{n}$ for all $a\in A/m_{0}$ and $n\in\mathbb{N}$. Suppose towards a contradiction that there is $a\in A/m_{0}$ with $|a^{n}|<|a|^{n}$ for some $n>1$. We claim that $a-T$ is non-invertible in the Banach ring $\overline{A/m_{0}}\langle r^{-1}T\rangle$ with $r:=\sqrt[n]{|a^{n}|}$. By Lemma \ref{lem:RTRRT} it suffices to show that the series $\sum_{i=0}^{\infty}|a^{-i}|r^{i}$ does not converge. Expressing $i$ as $i=pn+q$, for some $q\in\{0,\cdots,n-1\}$, we have $|a^{i}|\leq|a^{n}|^{p}|a^{q}|$ and $|a^{i}|^{-1}\leq|a^{-i}|$ since $1=|a^{i}a^{-i}|\leq|a^{i}||a^{-i}|$. Therefore
\begin{equation*}
|a^{-i}|r^{i}\geq|a^{i}|^{-1}|a^{n}|^{p+\frac{q}{n}}\geq\frac{|a^{n}|^{p}|a^{n}|^{\frac{q}{n}}}{|a^{n}|^{p}|a^{q}|}=\frac{|a^{n}|^{\frac{q}{n}}}{|a^{q}|}.
\end{equation*}
Hence $|a^{-i}|r^{i}\geq\min\{\frac{|a^{n}|^{\frac{q}{n}}}{|a^{q}|}:q\in\{0,\cdots,n-1\}\}>0$ for all $i\geq0$. Therefore $a-T$ is non-invertible in $\overline{A/m_{0}}\langle r^{-1}T\rangle$ with $r:=\sqrt[n]{|a^{n}|}$. Now by Lemma \ref{lem:RTAUX} there exists $|\cdot|'\in S(A)$ such that $|a|'\leq r$ and $|b|'\leq|b|$ for all $b\in A/m_{0}$. But, since $|a^{n}|<|a|^{n}$, this gives $|a|'\leq r=\sqrt[n]{|a^{n}|}<|a|$ which contradicts $|\cdot|$ being a minimal element of $S(A)$. Hence we have shown that $|a^{n}|=|a|^{n}$ for all $a\in A/m_{0}$ and $n\in\mathbb{N}$.\\
Now suppose towards a contradiction that there exists an element $a\in A/m_{0}^{\times}$ with $|a|^{-1}<|a^{-1}|$. We claim that $a-T$ is non-invertible in $\overline{A/m_{0}}\langle r^{-1}T\rangle$ with $r:=|a^{-1}|^{-1}$. Again by Lemma \ref{lem:RTRRT} it suffices to show that the series $\sum_{i=0}^{\infty}|a^{-i}|r^{i}$ does not converge. Indeed since $|\cdot|$ is power multiplicative we have
\begin{equation*}
|a^{-i}|r^{i}=|(a^{-1})^{i}|r^{i}=|a^{-1}|^{i}(|a^{-1}|^{-1})^{i}=|a^{-1}|^{0}=1.
\end{equation*}
Hence $a-T$ is non-invertible in $\overline{A/m_{0}}\langle r^{-1}T\rangle$ with $r:=|a^{-1}|^{-1}$. Now again by Lemma \ref{lem:RTAUX} there exists $|\cdot|'\in S(A)$ such that $|a|'\leq r$ and $|b|'\leq|b|$ for all $b\in A/m_{0}$. But, since $|a|^{-1}<|a^{-1}|$, this gives $|a|'\leq r=|a^{-1}|^{-1}<|a|$ which contradicts $|\cdot|$ being a minimal element of $S(A)$. Hence we have shown that $|a^{-1}|=|a|^{-1}$ for all $a\in A/m_{0}^{\times}$ and so $|\cdot|$ is multiplicative. Finally $m_{0}$ is the kernel of $|\cdot|$ and as a maximal ideal of $A$ it is proper, closed and prime by Lemma \ref{lem:RTCBR}. In particular since $m_{0}$ was an arbitrary maximal ideal of $A$ every maximal ideal of $A$ is an element of $\mathcal{M}_{0}(A)$.\\
For (ii), for $a$ an invertible element of $A$ we have $a\not\in x_{0}$ for all $x_{0}\in\mathcal{M}_{0}(A)$ since $x_{0}$ is a proper ideal of $A$. Hence $a+x_{0}\not=0+x_{0}$ in $A/x_{0}$ for all $x_{0}\in\mathcal{M}_{0}(A)$. On the other hand for $a$ a non-invertible element of $A$ we have by Lemma \ref{lem:RTCBR} that $a$ is an element of a maximal ideal $J_{a}$ of $A$. By (i) above, $J_{a}$ is an element of $\mathcal{M}_{0}(A)$ and $a+J_{a}=0+J_{a}$ in $A/J_{a}$. Therefore for $a$ a non-invertible element of $A$ we do not have $a+x_{0}\not=0+x_{0}$ in $A/x_{0}$ for all $x_{0}\in\mathcal{M}_{0}(A)$.
\end{proof}
With the preceding theory in place we can now turn our attention to the main topic of this chapter.
\section{Representations}
\label{sec:RTR}
\subsection{Established theorems}
\label{subsec:RTRET}
The particular well known representation theorems in the Archimedean setting that we will find an analog of in the non-Archimedean setting are as follows. See \cite[p35]{Kulkarni-Limaye1992} for details of Theorem \ref{thr:RTREPR}.
\begin{theorem}
\label{thr:RTREPC}
Let $A$ be a commutative unital complex Banach algebra with $\|a^{2}\|_{A}=\|a\|_{A}^{2}$ for all $a\in A$. Then $A$ is isometrically isomorphic to a uniform algebra on some compact Hausdorff space $X$, in other words a $^{\mathbb{C}}/_{\mathbb{C}}$ function algebra on $(X,\mbox{id},\mbox{id})$.
\end{theorem}
\begin{theorem}
\label{thr:RTREPR}
Let $A$ be a commutative unital real Banach algebra with $\|a^{2}\|_{A}=\|a\|_{A}^{2}$ for all $a\in A$. Then $A$ is isometrically isomorphic to a real function algebra on some compact Hausdorff space $X$ with topological involution $\tau$ on $X$, in other words a $^{\mathbb{C}}/_{\mathbb{R}}$ function algebra on $(X,\tau,\bar{z})$.
\end{theorem}
We will now recall some of the theory behind Theorem \ref{thr:RTREPC}. For more details see \cite[p29]{Stout} or \cite[p4,p11]{Gamelin}. The space $X$ is the character space $\mbox{Car}(A)$ which as a set is the set of all non-zero, complex-valued, multiplicative $\mathbb{C}$-linear functionals on $A$. It turns out that the characters on $A$ are all automatically continuous. Note, in the case of Theorem \ref{thr:RTREPR} the functionals are complex-valued but $\mathbb{R}$-linear and $\tau$ maps each such functional to its complex conjugate. For a commutative unital complex Banach algebra $A$ the Gelfand transform is a homomorphism from $A$ to a space of complex valued functions $\hat{A}$ defined by $a\mapsto\hat{a}$ where $\hat{a}(\varphi):=\varphi(a)$ for all $a\in A$ and $\varphi\in\mbox{Car}(A)$. The topology on $\mbox{Car}(A)$ is the initial topology given by the family of functions $\hat{A}$. Known in this case as the Gelfand topology it is the weakest topology on $\mbox{Car}(A)$ such that all the elements of $\hat{A}$ are continuous giving $\hat{A}\subseteq C_{\mathbb{C}}(\mbox{Car}(A))$. The norm given to $\hat{A}$ is the sup norm.\\
Now for a commutative unital complex Banach algebra $A$ the set of maximal ideals of $A$ and the set of kernels of the elements of $\mbox{Car}(A)$ agree. In Theorem \ref{thr:RTREPC}, $\|\cdot\|_{A}$ being square preserving ensures that $A$ is semisimple, that is that the Jacobson radical of $A$ is $\{0\}$ where the Jacobson radical is the intersection of all maximal ideals of $A$ and so the intersection of all the kernels of elements of $\mbox{Car}(A)$. Forcing $A$ to be semisimple ensures that the Gelfand transform is injective since if $A$ is semisimple then the kernel of the Gelfand transform is $\{0\}$. Similarly to confirm that the Gelfand transform is injective it is enough to show that it is an isometry. Given Theorem \ref{thr:RTREPC} it is immediate that a commutative unital complex Banach algebra $A$ is isometrically isomorphic to a uniform algebra if and only if its norm is square preserving since the sup norm has this property. Hence Theorem \ref{thr:RTREPC} provides a characterisation of uniform algebras.\\
Now in the non-Archimedean setting Berkovich, the author of \cite{Berkovich}, takes the following approach involving Definition \ref{def:RTMACH}.
\begin{definition}
\label{def:RTMACH}
Let $F$ be a complete non-Archimedean field and let $A$ be a commutative unital Banach $F$-algebra. Define $\mathcal{M}_{1}(A)$ to be the set of all bounded multiplicative seminorms on $A$. Further a {\em{character on}} $A$ is a non-zero, multiplicative $F$-linear functional on $A$ that takes values in some complete field extending $F$ as a valued field.
\end{definition}
For an appropriate topology, $\mathcal{M}_{1}(A)$ plays the role for $A$ in Definition \ref{def:RTMACH} that the maximal ideal space, equivalently the character space, plays in the Archimedean setting. For $|\cdot|\in \mathcal{M}_{1}(A)$ let $x_{0}:=\mbox{ker}(|\cdot|)$. Then $x_{0}$ is a proper closed prime ideal of $A$. Hence the quotient ring $A/x_{0}$ is an integral domain. Lemma \ref{lem:RTQRV} is useful here.
\begin{lemma}
\label{lem:RTQRV}
Let $F$ be a complete valued field and let $A$ be a commutative unital Banach $F$-algebra. For $|\cdot|$ a bounded multiplicative seminorm on $A$ with kernel $x_{0}$ the value $|a|$ of $a\in A$ only depends on the quotient class in $A/x_{0}$ to which $a$ belongs. Hence $|\cdot|$ is well defined when used as a valuation on $A/x_{0}$ by setting $|a+x_{0}|:=|a|$. Further $x_{0}$ is a closed subset of $A$.
\end{lemma}
\begin{proof}
For $a\in A$ and $b\in x_{0}$ we have $|a|=|a+b-b|\leq|a+b|+|b|=|a+b|$ and $|a+b|\leq|a|+|b|=|a|$ hence $|a+b|=|a|$ as required. Furthermore this also gives an easy way of seeing that $x_{0}$ is a closed subset of $A$. Let $a$ be an element of $A\backslash x_{0}$ then for all $b\in x_{0}$ we have $|a|=|a-b|\leq\|a-b\|_{A}$ since $|\cdot|$ is bounded and so $x_{0}$ is closed.
\end{proof}
Now by Lemma \ref{lem:RTQRV} we can take $|\cdot|$ to be a valuation on $A/x_{0}$ and hence extend it to a valuation on the field of fractions $\mbox{Frac}(A/x_{0})$. Hence an element $|\cdot|\in \mathcal{M}_{1}(A)$ defines a character on $A$ by sending the elements of $A$ to their image in the completion of $\mbox{Frac}(A/x_{0})$ with respect to $|\cdot|$. With these details in place we have the following theorem by Berkovich, see \cite[p157]{Berkovich}.
\begin{theorem}
\label{thr:RTREPF}
Let $F$ be a complete non-Archimedean field. Let $A$ be a commutative unital Banach $F$-algebra with $\|a^{2}\|_{A}=\|a\|_{A}^{2}$ for all $a\in A$. Suppose that all of the characters of $A$ take values in $F$. Then:
\begin{enumerate}
\item[(i)]
the space $\mathcal{M}_{1}(A)$ is totally disconnected;
\item[(ii)]
the Gelfand transform gives an isomorphism from $A$ to $C_{F}(\mathcal{M}_{1}(A))$.
\end{enumerate}
\end{theorem}
As we move on to the next subsection it's worth pointing out that the Gelfand theory presented in \cite{Berkovich} does not make use of any definition such as that of $^{L}/_{L^{g}}$ function algebras.
\subsection{Motivation}
\label{subsec:RTRMOT}
For $A$ a commutative unital complex Banach algebra it is straightforward to confirm that there is a one-one correspondence between the elements of $\mbox{Car}(A)$ and the elements of the maximal ideal space. Since $A$ is unital the complex constants are elements of $A$ and for $\varphi\in\mbox{Car}(A)$, $\varphi$ restricts to the identity on $\mathbb{C}$. Hence by the first isomorphism theorem for rings we have
\begin{equation}
\label{equ:RTCONG}
A/\mbox{ker}(\varphi)\cong\varphi(A)=\mathbb{C}
\end{equation}
showing that $\mbox{ker}(\varphi)$ is a maximal ideal of $A$. Therefore, by also noting the prelude to Chapter \ref{cha:CG}, the set of maximal ideals of $A$ and the set of kernels of the elements of $\mbox{Car}(A)$ do indeed agree. It remains to show that no two characters on $A$ have the same kernel and this marks an important difference with the theory we are about to present. First though let $\varphi,\phi$ be elements of $\mbox{Car}(A)$ with $\mbox{ker}(\varphi)=\mbox{ker}(\phi)$. We note that for each $a\in A$ there is a unique $\alpha\in\mathbb{C}$ representing the quotient class $a+\mbox{ker}(\varphi)$ by (\ref{equ:RTCONG}). Hence for some $b\in\mbox{ker}(\varphi)$ we have $a+b=\alpha$ giving
\begin{equation*}
\varphi(a)=\varphi(a)+\varphi(b)=\varphi(a+b)=\varphi(\alpha)=\alpha=\phi(\alpha)=\phi(a+b)=\phi(a)+\phi(b)=\phi(a)
\end{equation*}
and so no two characters on $A$ have the same kernel.\\
Now let $F$ be a complete non-Archimedean field. We wish to identify sufficient conditions for a commutative unital Banach $F$-algebra to be represented by some $^{L}/_{F}$ function algebra. In this respect the following lemma is informative and motivates an appropriate choice of character space in Subsection \ref{subsec:RTRFBD}.
\begin{lemma}
\label{lem:RTMOTMA}
For $A$ an $^{L}/_{L^{g}}$ function algebra on $(X,\tau,g)$, where $L$ can be Archimedean or non-Archimedean and $A$ is not assumed to be basic, define a family of maps on $A$ by
\begin{equation*}
|f|_{A,x}:=|f(x)|_{L}\quad\mbox{for }x\in X\mbox{ and }f\in A.
\end{equation*}
Then for each $x\in X$:
\begin{enumerate}
\item[(i)]
the map $|\cdot|_{A,x}$ is a bounded multiplicative seminorm on $A$;
\item[(ii)]
the kernel $\mbox{\rm{ker}}(|\cdot|_{A,x})$, which is the same as $\mbox{\rm{ker}}(\hat{x})$ where $\hat{x}$ is the evaluation character $\hat{x}$(f):=f(x) on $A$, is not only a proper closed prime ideal of $A$ but it is also a maximal ideal;
\item[(iii)]
we have $\mbox{\rm{ker}}\left(\widehat{\tau(x)}\right)=\mbox{\rm{ker}}(\hat{x})$ even if $\tau$ is not the identity and in general different evaluation characters can have the same kernel.
\end{enumerate}
\end{lemma}
\begin{proof}
For (i), it is immediate that $|\cdot|_{A,x}$ is a bounded multiplicative seminorm on $A$ since the norm on $A$ is the sup norm and $|\cdot|_{L}$ is a valuation on $L$.\\
For (ii), it is immediate that $\mbox{ker}(|\cdot|_{A,x})$ is a proper ideal of $A$ noting that $|1|_{A,x}=|1|_{L}=1$. It remains to show that $\mbox{ker}(|\cdot|_{A,x})$ is a maximal ideal of $A$ noting Lemma \ref{lem:RTCBR}. To this end we show that the quotient ring $A/\mbox{ker}(\hat{x})$ is a field. We first note that $L^{g}\subseteq\hat{x}(A)\subseteq L$ and that $\hat{x}(A)$ is a ring and so an integral domain. Further by the first isomorphism theorem for rings we have $A/\mbox{ker}(\hat{x})\cong\hat{x}(A)$ and so $A/\mbox{ker}(\hat{x})$ contains an embedding of $L^{g}$ and each element $a\in A/\mbox{ker}(\hat{x})$ is an element of an algebraic extension of $L^{g}$ since $L$ is a finite extension of $L^{g}$. Therefore for $a\in A/\mbox{ker}(\hat{x})$ with $a\not=0$ we have by Lemma \ref{lem:CVFPOL} that $L^{g}(a)=L^{g}[a]$ where $L^{g}(a)$ is a simple extension of $L^{g}$ and $L^{g}[X]$ is the ring of polynomials over $L^{g}$. Hence, since $L^{g}[a]\subseteq A/\mbox{ker}(\hat{x})$, the inverse $a^{-1}$ is an element of $A/\mbox{ker}(\hat{x})$ which is therefore a field as required.\\
For (iii), we note that for all $f\in A$ and $x\in X$ we have $f(\tau(x))=g(f(x))$ since $f$ is an element of $C(X,\tau,g)$. Further since $g\in\mbox{Gal}(^{L}/_{F})$ we have $g(f(x))=0$ if and only if $f(x)=0$ and so $\mbox{ker}\left(\widehat{\tau(x)}\right)=\mbox{ker}(\hat{x})$. However in general $f(x)$ need not be equal to $g(f(x))$ and so different evaluation characters can have the same kernel.
\end{proof}
\subsection{Representation under finite basic dimension}
\label{subsec:RTRFBD}
This subsection will involve the use of Definition \ref{def:RTFISO}.
\begin{definition}
\label{def:RTFISO}
Suppose $F_{1}$ and $F_{2}$ are extensions of a field $F$ such that there exists an isomorphism $\varphi:F_{1}\rightarrow F_{2}$ with $\varphi(a)=a$ for all $a\in F$. Then $\varphi$ is called an $F$-{\em isomorphism} and $F_{1}$ and $F_{2}$ are called $F$-{\em isomorphic} or with the same meaning $F$-{\em conjugate}. Similarly if $F$ is complete then we can talk of $F$-{\em isomorphic} Banach $F$-algebras etc.
\end{definition}
The following definition and theorem will be the focus of attention for the rest of this chapter.
\begin{definition}
\label{def:RTFBD}
Let $F$ be a complete valued field and let $A$ be a commutative unital Banach $F$-algebra. We say that $A$ has {\em finite basic dimension} if there exists a finite extension $L$ of $F$ extending $F$ as a valued field such that:
\begin{enumerate}
\item[(i)]
for each proper closed prime ideal $J$ of $A$, that is the kernel of a bounded multiplicative seminorm on $A$, the field of fractions $\mbox{Frac}(A/J)$ is $F$-isomorphic to a subfield of $L$;
\item[(ii)]
there is $g\in\mbox{Gal}(^{L}/_{F})$ with $L^{g}=F$.
\end{enumerate}
Cases where $L=F$ are allowed.
\end{definition}
The purpose of Definition \ref{def:RTFBD} is to generalise to the non-Archimedean setting conditions that are innately present in the Archimedean case due to the Gelfand Mazur theorem. We will discuss this in Remark \ref{rem:RTFBD}.
\begin{theorem}
\label{thr:RTREPLF}
Let $F$ be a locally compact complete non-Archimedean field with nontrivial valuation. Let $A$ be a commutative unital Banach $F$-algebra with $\|a^{2}\|_{A}=\|a\|_{A}^{2}$ for all $a\in A$ and finite basic dimension. Then:
\begin{enumerate}
\item[(i)]
for some finite extension $L$ of $F$ extending $F$ as a valued field, a character space $\mathcal{M}(A)$ of $L$ valued, multiplicative $F$-linear functionals can be defined;
\item[(ii)]
the space $\mathcal{M}(A)$ is a totally disconnected compact Hausdorff space;
\item[(iii)]
$A$ is isometrically $F$-isomorphic to a $^{L}/_{F}$ function algebra on $(\mathcal{M}(A),g,g)$ for some $g\in\mbox{Gal}(^{L}/_{F})$.
\end{enumerate}
\end{theorem}
\begin{remark}
\label{rem:RTFBD}
Concerning the condition of finite basic dimension.
\begin{enumerate}
\item[(i)]
We first note that all commutative unital complex Banach algebras and commutative unital real Banach algebras have finite basic dimension. To see this let $A$ be such an algebra and let $J$ be a proper closed prime ideal of $A$ such that $J$ is the kernel of a bounded multiplicative seminorm $|\cdot|$ on $A$. Then, by Lemma \ref{lem:RTBMS} and Lemma \ref{lem:RTQRV}, $|\cdot|$ extends the absolute valuation on $\mathbb{R}$ to a valuation on the integral domain $A/J$. Extending $|\cdot|$ to a valuation on $\mbox{Frac}(A/J)$ gives either $\mathbb{R}$ or $\mathbb{C}$ by the Gelfand Mazur theorem and noting Theorem \ref{thr:CVFCOM}. Finally with consideration of $\mbox{Gal}(^{\mathbb{C}}/_{\mathbb{R}})$ the result follows. Hence we note that with little modification Theorem \ref{thr:RTREPC}, Theorem \ref{thr:RTREPR} and Theorem \ref{thr:RTREPLF} could be combined into a single theorem.
\item[(ii)]
Now the argument in (i) was deliberately a little naive noting that for every commutative unital Banach $F$-algebra $A$ with finite basic dimension the kernel of every bounded multiplicative seminorm on $A$ is a maximal ideal of $A$. This follows easily from Lemma \ref{lem:CVFPOL} since such a kernel $J$ is a proper closed prime ideal of $A$ and the elements of the quotient ring $A/J$ are algebraic over $F$ and so $A/J$ is a field.
\item[(iii)]
Finally if $A$ is a commutative unital Banach $F$-algebra then in general the set of maximal ideals of $A$ is a subset of the set of kernels of bounded multiplicative seminorms on $A$ by Lemma \ref{lem:RTNMT}. Hence Theorem \ref{thr:RTREPLF} might be strengthened if we can find a proof that accepts changing (i) in Definition \ref{def:RTFBD} to the condition that for each maximal ideal $J$ of $A$ the field $A/J$ is $F$-isomorphic to a subfield of $L$. This is something for the future. The change only makes a difference for cases where there is a bounded multiplicative seminorm on $A$ with kernel $J$ such that $A/J$ has elements that are transcendental over $F$ since otherwise $J$ is a maximal ideal of $A$.
\end{enumerate}
\end{remark}
\begin{proof}[Proof of Theorem \ref{thr:RTREPLF}]
Let $\mathcal{M}_{0}(A)$ be as in Definition \ref{def:RTMOA}. Now $A$ has finite basic dimension so for each $x_{0}\in\mathcal{M}_{0}(A)$ the quotient ring $A/x_{0}$ is a field by Remark \ref{rem:RTFBD}. Further there is a finite extension $L$ of $F$ extending $F$ as a valued field such that for all $x_{0}\in\mathcal{M}_{0}(A)$ the field $A/x_{0}$ is $F$-isomorphic to a subfield of $L$. Moreover for $|\cdot|$ a bounded multiplicative seminorm on $A$ with kernel $x_{0}$ the map $|a+x_{0}|_{A/x_{0}}:=|a|$, for $a\in A$, defines a valuation on $A/x_{0}$ extending the valuation on $F$ by Lemma \ref{lem:RTQRV} and Lemma \ref{lem:RTBMS}. We note that since $L$ and $A/x_{0}$ are both finite extensions of $F$ they are complete valued fields. Further since $|\cdot|_{A/x_{0}}$ is defined by a bounded multiplicative seminorm on $A$ we have
\begin{equation}
\label{equ:RTBOU}
|a+x_{0}|_{A/x_{0}}\leq\|a\|_{A}\quad\mbox{for all }a\in A.
\end{equation}
We now progress towards defining the character space of $A$. Define $\mathcal{M}(A)$ as the set of all pairs $x:=(x_{0},\varphi)$ where $x_{0}\in\mathcal{M}_{0}(A)$ and $\varphi$ is an $F$-isomorphism from $A/x_{0}$ to a subfield of $L$ extending $F$. Then to each $x=(x_{0},\varphi)\in\mathcal{M}(A)$ we associated a map $\hat{x}:A\rightarrow L$ given by $\hat{x}(a):=\varphi(a+x_{0})$ for all $a\in A$. Note that for each element $x=(x_{0},\varphi)\in\mathcal{M}(A)$ we have
\begin{equation}
\label{equ:RTEQU}
|a+x_{0}|_{A/x_{0}}=|\hat{x}(a)|_{L}\quad\mbox{for all }a\in A
\end{equation}
by the uniqueness of the valuation on $A/x_{0}$ extending the valuation on $F$, see Theorem \ref{thr:CVFUT}. In particular each $F$-isomorphism from $A/x_{0}$ to a subfield of $L$ extending $F$ is an isometry and similarly we recall that each element of $\mbox{Gal}(^{L}/_{F})$ is isometric. Now for the element $g\in\mbox{Gal}(^{L}/_{F})$ with $L^{g}=F$, or indeed any other element of $\mbox{Gal}(^{L}/_{F})$, we note that $g$ can be considered as a map of finite order $g:\mathcal{M}(A)\rightarrow\mathcal{M}(A)$ given by $g((x_{0},\varphi)):=(x_{0},g\circ\varphi)$. In particular for $x=(x_{0},\varphi_{1})\in\mathcal{M}(A)$ we have $g\circ\hat{x}=\widehat{g(x)}$ and so there is $y=(y_{0},\varphi_{2})\in\mathcal{M}(A)$ with $y_{0}=x_{0}$ such that the diagram in Figure \ref{fig:RTCMUT} commutes.
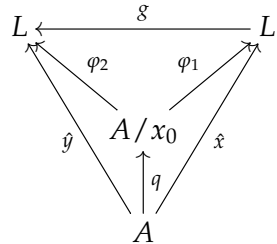
\begin{figure}[h]
\begin{equation*}
\xymatrix{
L&&L\ar@{->}[ll]_{g}\\
&A/x_{0}\ar@{->}[ru]^{\varphi_{1}}\ar@{->}[lu]_{\varphi_{2}}&\\
&A\ar@{->}[ruu]_{\hat{x}}\ar@{->}[u]_{q}\ar@{->}[luu]^{\hat{y}}&
}
\end{equation*}
\caption{Commutative diagram for the characters associated to {\it{x}} and {\it{y}}.}
\label{fig:RTCMUT}
\end{figure}
Note that in the case of Figure \ref{fig:RTCMUT} the fields $\hat{x}(A)$ and $\hat{y}(A)$ are $F$-conjugate and could actually be the same subfield of $L$ if the restriction $g|_{\hat{x}(A)}$ is an element of $\mbox{Gal}(^{\hat{x}(A)}/_{F})$. Now by construction for each $x\in\mathcal{M}(A)$ the map $\hat{x}$ is a non-zero, $L$-valued, multiplicative $F$-linear functional on $A$. Hence $\hat{x}$ is continuous since we have
\begin{equation}
\label{equ:RTXBOU}
|\hat{x}(a)|_{L}\leq\|a\|_{A}\quad\mbox{for all }a\in A
\end{equation}
by (\ref{equ:RTEQU}) and (\ref{equ:RTBOU}). We now set up the Gelfand transform in the usual manner by defining a map
\begin{equation*}
\widehat{\cdot}:A\rightarrow\hat{A},\quad a\mapsto\hat{a},
\end{equation*}
where the elements of $\hat{A}$ are the functions $\hat{a}:\mathcal{M}(A)\rightarrow L$ given by $\hat{a}(x):=\hat{x}(a)$. We equip $\hat{A}$ with the binary operations of pointwise addition and multiplication and put the sup norm
\begin{equation*}
\|\hat{a}\|_{\infty}:=\sup_{x\in\mathcal{M}(A)}|\hat{a}(x)|_{L}\quad\mbox{for all }\hat{a}\in\hat{A}
\end{equation*}
on $\hat{A}$ making $\hat{A}$ a commutative unital normed $F$-algebra. Note that with these binary operations it is immediate that the Gelfand transform is an $F$-homomorphism and so $\hat{A}$ is closed under addition and multiplication. Later we will show that $\widehat{\cdot}:A\rightarrow\hat{A}$ is an isometry and so it is also injective. It then follows that $\hat{A}$ is a Banach $F$-algebra since $A$ and $\hat{A}$ are isometrically $F$-isomorphic.\\
Now we equip $\mathcal{M}(A)$ with the Gelfand topology which is the initial topology of $\hat{A}$. Hence the elements of $\hat{A}$ are continuous $L$-valued functions on the space $\mathcal{M}(A)$. We show that $\hat{A}$ separates the points of $\mathcal{M}(A)$ and that $\mathcal{M}(A)$ is a compact Hausdorff space. Let $x$ and $y$ be elements of $\mathcal{M}(A)$ with $x=(x_{0},\varphi)$, $y=(y_{0},\phi)$ and $x\not=y$. If $x_{0}\not=y_{0}$ then there is $a\in x_{0}\cup y_{0}$ such that $a\not\in x_{0}\cap y_{0}$ for which precisely one of $\hat{a}(x)=\hat{x}(a)$ and $\hat{a}(y)=\hat{y}(a)$ is zero. If $x_{0}=y_{0}$ then $\varphi\not=\phi$ on $A/x_{0}$. Hence there is some $a\in A$ such that $\varphi(a+x_{0})\not=\phi(a+x_{0})$ giving
\begin{equation*}
\hat{a}(x)=\hat{x}(a)=\varphi(a+x_{0})\not=\phi(a+x_{0})=\hat{y}(a)=\hat{a}(y)
\end{equation*}
and so $\hat{A}$ separates the points of $\mathcal{M}(A)$. We now show that $\mathcal{M}(A)$ is Hausdorff, note in fact that the proof is standard. Let $x$ and $y$ be elements of $\mathcal{M}(A)$ with $x\not=y$. Since $\hat{A}$ separates the points of $\mathcal{M}(A)$ there is $\hat{a}\in\hat{A}$ such that $\hat{a}(x)\not=\hat{a}(y)$. Further $L$ is Hausdorff and so there are disjoint open subsets $U_{1}$ and $U_{2}$ of $L$ such that $\hat{a}(x)\in U_{1}$ and $\hat{a}(y)\in U_{2}$. Since the topology on $\mathcal{M}(A)$ is the initial topology of $\hat{A}$ the preimage $\hat{a}^{-1}(U_{1})$ is an open neighborhood of $x$ in $\mathcal{M}(A)$ and the preimage $\hat{a}^{-1}(U_{2})$ is an open neighborhood of $y$ in $\mathcal{M}(A)$. Moreover $\hat{a}^{-1}(U_{1})$ and $\hat{a}^{-1}(U_{2})$ are disjoint because $U_{1}$ and $U_{2}$ are, as required.\\
The following, showing that $\mathcal{M}(A)$ is compact, is an adaptation of part of the proof of Theorem \ref{thr:RTREPR} from \cite[p23]{Kulkarni-Limaye1992}. For each $a\in A$ define $L_{a}:=\{\alpha\in L:|\alpha|_{L}\leq\|a\|_{A}\}$ and $L_{A}:=\prod_{a\in A}L_{a}$ with the product topology. Each $L_{a}$ is compact by Theorem \ref{thr:CVFHB} noting that $L$ is locally compact by Remark \ref{rem:CVFEE}. Hence $L_{A}$ is compact by Tychonoff's Theorem. Now by (\ref{equ:RTXBOU}) we have $|\hat{x}(a)|_{L}\leq\|a\|_{A}$ for all $x\in\mathcal{M}(A)$ and $a\in A$. Therefore for each $x\in\mathcal{M}(A)$ we have $\hat{x}(a)\in L_{a}$ and so $\hat{x}$ is a point of $L_{A}$ and $\mathcal{M}(A)$ can be considered as a subset of $L_{A}$. Now the product topology on $L_{A}$ is the initial topology of the family of coordinate projections $P_{a}:L_{A}\rightarrow L_{a}$, $a\in A$. Since we have $P_{a}|_{\mathcal{M}(A)}=\hat{a}$ the Gelfand topology on $\mathcal{M}(A)$ is the initial topology of the family of coordinate projections restricted to $\mathcal{M}(A)$. Hence the topology on $\mathcal{M}(A)$ is the relative topology of $\mathcal{M}(A)$ as a subspace of $L_{A}$. Since $L_{A}$ is compact, any subspace of $L_{A}$ that is closed as a subset is also compact. Hence it remains to show that $\mathcal{M}(A)$ is a closed subset of $L_{A}$. Let $\varphi\in L_{A}$ be in the closure of $\mathcal{M}(A)$. Hence we have $|\varphi(a)|_{L}\leq\|a\|_{A}$ for all $a\in A$ and there is a net $(x_{\lambda})$ in $\mathcal{M}(A)$ converging to $\varphi$. Now since $L_{A}$ has the product topology, convergence in $L_{A}$ is coordinate-wise, see \cite[\S 8]{Willard}. Therefore for $a,b\in A$ we have
\begin{equation*}
\varphi(a+b)=\lim\hat{x}_{\lambda}(a+b)=\lim(\hat{x}_{\lambda}(a)+\hat{x}_{\lambda}(b))=\varphi(a)+\varphi(b).
\end{equation*}
Similarly, $\varphi(ab)=\varphi(a)\varphi(b)$ and $\varphi(\alpha)=\alpha$ for all $a,b\in A$ and $\alpha\in F$. Now since $\varphi$ takes values in $L$ and $L$ is a finite extension of $F$ we have that $\varphi(A)$ is a subfield of $L$ extending $F$ by Lemma \ref{lem:CVFPOL}. Hence since $A/\mbox{ker}(\varphi)\cong\varphi(A)$, by the first isomorphism theorem for rings, the kernel of $\varphi$ is a maximal ideal of $A$. Therefore $\mbox{ker}(\varphi)$ is an element of $\mathcal{M}_{0}(A)$. Further $\varphi$ defines an $F$-isomorphism from $A/\mbox{ker}(\varphi)$ to a subfield of $L$ extending $F$ by $\varphi'(a+\mbox{ker}(\varphi)):=\varphi(a)$. Hence we have obtained $y:=(\mbox{ker}(\varphi),\varphi')$ which is an element of $\mathcal{M}(A)$ with $\hat{y}=\varphi$ and so $\mathcal{M}(A)$ is closed as a subset on $L_{A}$.\\
We will now show that $g:\mathcal{M}(A)\rightarrow\mathcal{M}(A)$ is continuous. The set of preimages
\begin{equation*}
\mathcal{S}:=\{\hat{a}^{-1}(U):\hat{a}\in\hat{A}\mbox{ and }U\subseteq L\mbox{ is open}\}
\end{equation*}
is a sub-base for the Gelfand topology on $\mathcal{M}(A)$. To show that $g:\mathcal{M}(A)\rightarrow\mathcal{M}(A)$ is continuous it is enough to show that for each $V\in\mathcal{S}$ the preimage $g^{-1}(V)$ is also an element of $\mathcal{S}$. We note that $g:\mathcal{M}(A)\rightarrow\mathcal{M}(A)$ is a bijection since $g$ has finite order. So let $V=\hat{a}^{-1}(U)$ be an element of $\mathcal{S}$ for some $\hat{a}\in A$ and open $U\subseteq L$. We have $x=(x_{0},\varphi)\in\mathcal{M}(A)$ an element of $V$ if and only if $\hat{a}(x)=\hat{x}(a)=\varphi(a+x_{0})$ is an element of $U$. Now consider the elements of the preimage $g^{-1}(V)$ and note that they are the elements $y=(y_{0},\phi)\in\mathcal{M}(A)$ such that $g(y)=(y_{0},g\circ\phi)\in V$. These are precisely the elements of $\mathcal{M}(A)$ such that
\begin{equation*}
\hat{a}(y)=\hat{y}(a)=\phi(a+y_{0})\in g^{(\mbox{ord}(g)-1)}(U).
\end{equation*}
And so $g^{-1}(V)=\hat{a}^{-1}\left(g^{(\mbox{ord}(g)-1)}(U)\right)$ and since $g$ is an isometry on $L$ we note that $g^{(\mbox{ord}(g)-1)}(U)$ is an open subset of $L$. Hence $g^{-1}(V)$ is an element of $\mathcal{S}$ as required.\\
We now show that the Gelfand transform is an isometry. Note that the following adapts material that can be found in \cite[Ch1]{Berkovich}. Let $a$ be an element of $A$. By (\ref{equ:RTXBOU}) we have $|\hat{a}(x)|_{L}=|\hat{x}(a)|_{L}\leq\|a\|_{A}$ for all $x\in\mathcal{M}(A)$ and so $\|\hat{a}\|_{\infty}\leq\|a\|_{A}$. For the reverse inequality let $\varepsilon>0$ and set $r:=\|\hat{a}\|_{\infty}+\varepsilon$. Then for all $x_{0}\in\mathcal{M}_{0}(A)$ we have
\begin{equation}
\label{equ:RTLER}
|a+x_{0}|_{A/x_{0}}=|\hat{x}(a)|_{L}=|\hat{a}(x)|_{L}\leq\|\hat{a}\|_{\infty}<r
\end{equation}
for some $x=(x_{0},\varphi)\in\mathcal{M}(A)$ by $(\ref{equ:RTEQU})$ and noting that $A$ has finite basic dimension. Now consider the commutative unital Banach $F$-algebra $A\langle rT\rangle$. Let $\mathcal{M}_{0}(A\langle rT\rangle)$ be the set of all proper closed prime ideals of $A\langle rT\rangle$ that are the kernels of bounded multiplicative seminorms on $A\langle rT\rangle$. Note that $\mathcal{M}_{0}(A\langle rT\rangle)$ is non-empty by Lemma \ref{lem:RTNMT}. We recall that the elements of $A\langle rT\rangle$ are of the form $\sum_{i=0}^{\infty}a_{i}T^{i}$ with
\begin{equation*}
\left\|\sum_{i=0}^{\infty}a_{i}T^{i}\right\|_{A,r^{-1}}=\sum_{i=0}^{\infty}\|a_{i}\|_{A}(r^{-1})^{i}=\sum_{i=0}^{\infty}\|a_{i}\|_{A}r^{-i}<\infty
\end{equation*}
and $a_{i}\in A$ for all $i\in\mathbb{N}_{0}$. Hence $A$ is a subring of $A\langle rT\rangle$ since for each $b\in A$ we have $b=bT^{0}$ an element of $A\langle rT\rangle$. Now for $y_{0}\in\mathcal{M}_{0}(A\langle rT\rangle)$ let $|\cdot|$ be a bounded multiplicative seminorm on $A\langle rT\rangle$ with $y_{0}=\mbox{ker}(|\cdot|)$. Since $|\cdot|$ is bounded we have
\begin{equation}
\label{equ:RTLERI}
|T|\leq\|T\|_{A,r^{-1}}=r^{-1}.
\end{equation}
Moreover since for $b\in A$ we have $\|bT^{0}\|_{A,r^{-1}}=\|b\|_{A}(r^{-1})^{0}=\|b\|_{A}$, the restriction $|\cdot||_{A}$ is a bounded multiplicative seminorm on $A$. Further $m_{0}:=\mbox{ker}(|\cdot||_{A})$ is closed as a subset of $A$ by Lemma \ref{lem:RTQRV} and so $m_{0}$ is an element of $\mathcal{M}_{0}(A)$. Hence $m_{0}$ is a maximal ideal of $A$ by remark \ref{rem:RTFBD}. In particular $|b+m_{0}|_{A/m_{0}}:=|b|$, for $b\in A$, is the unique valuation on $A/m_{0}$ extending the valuation on $F$ as we have seen earlier in this proof for other elements of $\mathcal{M}_{0}(A)$. Therefore by (\ref{equ:RTLER}) and (\ref{equ:RTLERI}) we have
\begin{equation*}
|aT|=|a||T|\leq|a+m_{0}|_{A/m_{0}}r^{-1}<rr^{-1}=1.
\end{equation*}
Furthermore $1=|1|\leq|1-aT|+|aT|$ and so we have $|1-aT|\geq1-|aT|>0$. Therefore $1-aT$ is not an element of $y_{0}$ since $y_{0}$ is the kernel of $|\cdot|$. Since $y_{0}$ was any element of $\mathcal{M}_{0}(A\langle rT\rangle)$ we have $1-aT\not\in y_{0}$ for all $y_{0}\in\mathcal{M}_{0}(A\langle rT\rangle)$. Hence by Lemma \ref{lem:RTNMT} we note that $1-aT$ is invertible in $A\langle rT\rangle$. Therefore by Lemma \ref{lem:RTRRT} the series $\sum_{i=0}^{\infty}\|a^{i}\|_{A}r^{-i}$ converges. In particular we can find $N\in\mathbb{N}$ such that for all $n>N$ we have $\|a^{2^{n}}\|_{A}r^{-2^{n}}<\frac{1}{2}$ giving $(\|a\|_{A}r^{-1})^{2^{n}}<\frac{1}{2}$ since $\|\cdot\|_{A}$ is square preserving. Hence $\|a\|_{A}<r=\|\hat{a}\|_{\infty}+\varepsilon$ and since $\varepsilon>0$ was arbitrary we have $\|a\|_{A}\leq\|\hat{a}\|_{\infty}$ and so $\|a\|_{A}=\|\hat{a}\|_{\infty}$ as required.\\
What remains to be shown is that the elements of $\hat{A}$ are also elements of $C(\mathcal{M}(A),g,g)$ and that $\mathcal{M}(A)$ is totally disconnected. For $\hat{a}\in\hat{A}$ and $x=(x_{0},\varphi)\in\mathcal{M}(A)$ we have
\begin{align*}
\hat{a}(g(x))=\hat{a}((x_{0},g\circ\varphi))=&\widehat{(x_{0},g\circ\varphi)}(a)\\
=&g\circ\varphi(a+x_{0})\\
=&g(\varphi(a+x_{0}))\\
=&g\left(\widehat{(x_{0},\varphi)}(a)\right)=g(\hat{x}(a))=g(\hat{a}(x))
\end{align*}
and so $\hat{a}$ is an element of $C(\mathcal{M}(A),g,g)$. Finally it is immediate that $\mathcal{M}(A)$ is totally disconnected since $\hat{A}$ separates the points of $\mathcal{M}(A)$, the elements of $\hat{A}$ are continuous functions from $\mathcal{M}(A)$ to $L$, the image of a connected component is connected for continuous functions and $L$ is totally disconnected. In particular see the proof of Theorem \ref{thr:UAXtop}. This completes the proof of Theorem \ref{thr:RTREPLF}.
\end{proof}
In the next chapter we will survey some existing results in the Archimedean non-commutative setting and also consider the possibility of their generalisation to the non-Archimedean setting. We will then finish by noting some of the open questions arising from the Thesis.
	\chapter[Non-commutative generalisation and open questions]{Non-commutative generalisation and open questions}
\label{cha:NG}
In recent years a theory of non-commutative real function algebras has been developed by Jarosz, see \cite{Jarosz} and \cite{Abel-Jarosz}. In the first section of this short chapter we survey and remark upon some of this non-commutative Archimedean theory and consider the possibility of non-commutative non-Archimedean analogs. In the second section we note some of the open questions arising from the thesis.
\section{Non-commutative generalisation}
\label{sec:NGNCG}
\subsection{Non-commutative real function algebras}
\label{subsec:NGNCRFA}
In the recent theory of non-commutative real function algebras the continuous functions involved take values in Hamilton's real quaternions, $\mathbb{H}$, which are an example of a non-commutative complete Archimedean division ring and $\mathbb{R}$-algebra. Viewing $\mathbb{H}$ as a real vector space, the valuation on $\mathbb{H}$ is the Euclidean norm which is complete, Archimedean and indeed a valuation since being multiplicative on $\mathbb{H}$. To put $\mathbb{H}$ into context, as in the case of complete Archimedean fields, there are very few unital division algebras over the reals with the Euclidean norm as a valuation. Up to isomorphism they are $\mathbb{R}$, $\mathbb{C}$, $\mathbb{H}$ and the octonions $\mathbb{O}$. We note that the octonions are non-associative. The proof that there are no other unital division algebras over the reals with the Euclidean norm as a valuation is given by Hurwitz's 1, 2, 4, 8 Theorem, see \cite[Ch1]{Shapiro} and \cite{Lewis}. In particular for such an algebra $\mathbb{A}$ the square of the Euclidean norm is a regular quadratic form on $\mathbb{A}$ and since for $\mathbb{A}$ the Euclidean norm is a valuation it is multiplicative. This shows that $\mathbb{A}$ is a real composition algebra to which Hurwitz's 1, 2, 4, 8 Theorem can be applied.\\
Here we only briefly consider non-commutative real function algebras and hence the reader is also referred to \cite{Jarosz}. Note I am unaware of any such developments involving the octonions. Here is Jarosz's analog of $C(X,\tau)$ from Definition \ref{def:UARefa}.
\begin{definition}
\label{def:NGNCRFA}
Let $\mbox{Gal}(^{\mathbb{H}}/_{\mathbb{R}})$ be the group of all automorphisms on $\mathbb{H}$ that are the identity on $\mathbb{R}$. Let $X$ be a compact space and $\mbox{Hom}(X)$ be the group of homeomorphisms on $X$. For a group homomorphism $\Phi:\mbox{Gal}(^{\mathbb{H}}/_{\mathbb{R}})\rightarrow\mbox{Hom}(X)$, $\Phi(T)=\Phi_{T}$, we define
\begin{equation*}
C_{\mathbb{H}}(X,\Phi):=\{f\in C_{\mathbb{H}}(X):f(\Phi_{T}(x))=T(f(x))\mbox{ for all }x\in X\mbox{ and } T\in\mbox{Gal}(^{\mathbb{H}}/_{\mathbb{R}})\}.
\end{equation*}
\end{definition}
\begin{remark}
\label{rem:NGPHI}
Concerning Definition \ref{def:NGNCRFA}.
\begin{enumerate}
\item[(i)]
The groups $\mbox{Gal}(^{\mathbb{H}}/_{\mathbb{R}})$ and $\mbox{Hom}(X)$ in Definition \ref{def:NGNCRFA} have composition as their group operation. We note that the map $\ast:\mbox{Gal}(^{\mathbb{H}}/_{\mathbb{R}})\times C_{\mathbb{H}}(X)\rightarrow C_{\mathbb{H}}(X)$ given by $T\ast f:=T^{-1}(f(\Phi_{T}(x)))$ is similar to a group action on $C_{\mathbb{H}}(X)$ only with the usual associativity replaced by $T_{1}\circ T_{2}\ast f=T_{2}\ast T_{1}\ast f$.
\item[(ii)]
There is an interesting similarity between Definition \ref{def:NGNCRFA} and Definition \ref{def:CGBFA} of Basic function algebras. Let $X$ be a compact Hausdorff space, $F$ a complete valued field and $L$ a finite extension of $F$. Further let $\langle g\rangle$ be the cyclic group generated by some $g\in\mbox{Gal}(^{L}/_{F})$ and similarly let $\langle\tau\rangle$ be the cyclic group generated by some homeomorphism $\tau:X\rightarrow X$. Then there exists a surjective group homomorphism $\Phi:\langle g\rangle\rightarrow\langle\tau\rangle$ if and only if $\mbox{ord}(\tau)|\mbox{ord}(g)$. To see this suppose such a surjective group homomorphism exists. Then there are $m,n\in\mathbb{N}$ such that $\Phi(g^{(m)})=\mbox{id}$ and $\Phi(g^{(n)})=\tau$. This gives
\begin{align*}
\tau^{(\mbox{ord}(g))}=\Phi(g^{(n)})^{(\mbox{ord}(g))}=&\Phi(g^{(n\mbox{ord}(g))})\\
=&\Phi(\mbox{id})\\
=&\mbox{id}\circ\Phi(\mbox{id})\\
=&\Phi(g^{(m)})\circ\Phi(\mbox{id})\\
=&\Phi(g^{(m)}\circ\mbox{id})=\Phi(g^{(m)})=\mbox{id}
\end{align*}
and so $\mbox{ord}(\tau)|\mbox{ord}(g)$. Conversely if $\mbox{ord}(\tau)|\mbox{ord}(g)$ then $\Phi$ defined by $\Phi(g):=\tau$ will do. It is an interesting question then whether the definition of basic function algebras can be further generalised by utilizing group homomorphisms as Definition \ref{def:NGNCRFA} suggests noting that $\Phi$ is onto for some subgroup of $\mbox{Hom}(X)$. In particular, with reference to Definition \ref{def:CGBFA}, we have considered basic $^{L}/_{L^{g}}$ function algebras where $g$ is an element of $\mbox{Gal}(^{L}/_{F})$. We note that $L$ is a cyclic extension of $L^{g}$ by the fundamental theorem of Galois theory. Therefore it is interesting to consider the possibility of basic $^{L}/_{F}$ function algebras where $L$ is a Galois extension of $F$ but not necessarily a cyclic extension. Such group homomorphisms might also be useful in cases involving infinite extensions of $F$.
\item[(iii)]
Turning our attention back to the non-commutative setting, as a conjecture I suggests that Definition \ref{def:NGNCRFA} may also be useful if $\mbox{Gal}(^{\mathbb{H}}/_{\mathbb{R}})$ is replaced by a subgroup, particularly when considering extensions of the algebra.
\end{enumerate}
\end{remark}
Definition \ref{def:NGNCRFA} has been used by Jarosz in the representation of non-commutative real Banach algebras with square preserving norm as follows.
\begin{definition}
\label{def:NGFNON}
A real algebra $A$ is {\em{fully non-commutative}} if every nonzero multiplicative, linear functional $\varphi:A\rightarrow\mathbb{H}$ is surjective.
\end{definition}
\begin{theorem}
\label{thr:NGREPHR}
Let $A$ be a non-commutative real Banach algebra with $\|a^{2}\|_{A}=\|a\|_{A}^{2}$ for all $a\in A$. Then there is a compact set $X$ and an isomorphism $\Phi:\mbox{Gal}(^{\mathbb{H}}/_{\mathbb{R}})\rightarrow\mbox{Hom}(X)$ such that $A$ is isometrically isomorphic with a subalgebra $\hat{A}$ of $C_{\mathbb{H}}(X,\Phi)$. Furthermore $a\in A$ is invertible if and only if the corresponding element $\hat{a}\in\hat{A}$ does not vanish on $X$. If $A$ is fully non-commutative then $\hat{A}=C_{\mathbb{H}}(X,\Phi)$.
\end{theorem}
Jarosz also gives the following Stone-Weierstrass theorem type result.
\begin{theorem}
\label{thr:NGNCSW}
Let $X$ be a compact Hausdorff space and let $A$ be a fully non-commutative closed subalgebra of $C_{\mathbb{H}}(X)$. Then $A=C_{\mathbb{H}}(X)$ if and only if $A$ strongly separates the points of $X$, that is for all $x_{1},x_{2}\in X$ with $x_{1}\not=x_{2}$ there is $f\in A$ satisfying $f(x_{1})\not=f(x_{2})=0$.
\end{theorem}
\subsection{Non-commutative non-Archimedean analogs}
\label{subsec:NGNCNA}
Non-commutative, non-Archimedean analogs of uniform algebras have yet to be seen. Hence in this subsection we give an example of a non-commutative extension of a complete non-Archimedean field which would be appropriate when considering such analogs of uniform algebras. We first have the following definition from the general theory of quaternion algebras. The main reference for this subsection is \cite[Ch3]{Lam} but \cite{Lewis} is also useful.
\begin{definition}
\label{def:NGGTQA}
Let $F$ be a field, with characteristic not equal to 2, and $s,t\in F^{\times}$ where $s=t$ is allowed. We define the {\em{quaternion $F$-algebra}} $(\frac{s,t}{F})$ as follows. As a 4-dimensional vector space over $F$ we define 
\begin{equation*}
\left(\frac{s,t}{F}\right):=\{a+bi+cj+dk:a,b,c,d\in F\}
\end{equation*}
with $\{1,i,j,k\}$ as a natural basis giving the standard coordinate-wise addition and scalar multiplication. As an $F$-algebra, multiplication in $(\frac{s,t}{F})$ is given by 
\begin{equation*}
i^{2}=s,\quad j^{2}=t,\quad k^{2}=ij=-ji
\end{equation*}
together with the usual distributive law and multiplication in $F$.
\end{definition}
Hamilton's real quaternions, $\mathbb{H}:=(\frac{-1,-1}{\mathbb{R}})$ with the Euclidean norm, is an example of a non-commutative, complete valued, Archimedean division algebra over $\mathbb{R}$. It is not the case that every quaternion algebra $(\frac{s,t}{F})$ will be a division algebra, although there are many examples that are. For our purposes we have the following example.
\begin{example}
\label{exa:NGNCNA}
Using $\mathbb{Q}_{5}$, the complete non-Archimedean field of 5-adic numbers, define
\begin{equation*}
\mathbb{H}_{5}:=\left(\frac{5,2}{\mathbb{Q}_{5}}\right). 
\end{equation*}
Then for $q,r\in\mathbb{H}_{5}$, $q=a+bi+cj+dk$, the conjugation on $\mathbb{H}_{5}$ given by
\begin{equation*}
\bar{q}:=a-bi-cj-dk 
\end{equation*}
is such that $\overline{q+r}=\bar{q}+\bar{r}$, $\overline{qr}=\bar{r}\bar{q}$, $\bar{q}q=q\bar{q}=a^{2}-5b^{2}-2c^{2}+10d^{2}$ with $\bar{q}q\in\mathbb{Q}_{5}$. Further
\begin{equation*}
|q|_{\mathbb{H}_{5}}:=\sqrt{|\bar{q}q|_{5}} 
\end{equation*}
is a complete non-Archimedean valuation on $\mathbb{H}_{5}$, where $|\cdot|_{5}$ is the 5-adic valuation on $\mathbb{Q}_{5}$. In particular $\mathbb{H}_{5}$, together with $|\cdot|_{\mathbb{H}_{5}}$, is an example of a non-commutative, complete valued, non-Archimedean division algebra over $\mathbb{Q}_{5}$. When showing this directly it is useful to know that for $a,b,c,d\in\mathbb{Q}_{5}$ we have
\begin{equation*}
\nu_{5}(a^{2}-5b^{2}-2c^{2}+10d^{2})=\mbox{min}\{\nu_{5}(a^{2}),\nu_{5}(5b^{2}),\nu_{5}(2c^{2}),\nu_{5}(10d^{2})\} 
\end{equation*}
where $\nu_{5}$ is the 5-adic valuation logarithm as defined in Example \ref{exa:CVFPN}. Given the above, we will confirm that $|\cdot|_{\mathbb{H}_{5}}$ is multiplicative. For more details please see the suggested references \cite[Ch3]{Lam} and \cite{Lewis}. Let $q,r\in\mathbb{H}_{5}$ and note that we have $\bar{r}\bar{q}qr=\bar{q}q\bar{r}r$ since $\bar{q}q$ is an element of $\mathbb{Q}_{5}$. Therefore
\begin{align*}
|qr|_{\mathbb{H}_{5}}=\sqrt{|\overline{qr}qr|_{5}}=&\sqrt{|\bar{r}\bar{q}qr|_{5}}\\
=&\sqrt{|\bar{q}q\bar{r}r|_{5}}\\
=&\sqrt{|\bar{q}q|_{5}|\bar{r}r|_{5}}=\sqrt{|\bar{q}q|_{5}}\sqrt{|\bar{r}r|_{5}}=|q|_{\mathbb{H}_{5}}|r|_{\mathbb{H}_{5}}
\end{align*}
as required.
\end{example}
More generally for the $p$-adic field $\mathbb{Q}_{p}$ the quaternion algebra $(\frac{p,u}{\mathbb{Q}_{p}})$ will be a division algebra as long as $u$ is a unit of $\{a\in\mathbb{Q}_{p}:|a|_{p}\leq 1\}$, i.e. $|u|_{p}=1$, and $\mathbb{Q}_{p}(\sqrt{u})$ is a quadratic extension of $\mathbb{Q}_{p}$.
\section{Open questions}
\label{sec:NGOQ}
There are many open questions related to the content of this thesis and I had intended to investigate more of them but there was no time. Many of these questions come from the need to generalise established Archimedean results whilst others arise from the developing theory itself. We now consider some of these questions and note that several of them appear to be quite accessible.
\begin{enumerate}
\item[(Q1)]
J. Wermer gave the following theorem in 1963.
\begin{theorem}
\label{thr:NGWER}
Let $X$ be a compact Hausdorff space, $A\subseteq C_{\mathbb{C}}(X)$ a complex uniform algebra and $\Re(A):=\{\Re(f):f=\Re(f)+i\Im(f)\in A\}$ the set of the real components of the functions in $A$. If $\Re(A)$ is a ring then $A=C_{\mathbb{C}}(X)$.
\end{theorem}
The following analog of Theorem \ref{thr:NGWER} for real function algebras was given by S. H. Kulkarni and N. Srinivasan in \cite{Kulkarni-Srinivasan}, although I have not used their notation.
\begin{theorem}
\label{thr:NGWKS}
Let $X$ be a compact Hausdorff space, $\tau$ a topological involution on $X$ and $A$ a $^{\mathbb{C}}/_{\mathbb{R}}$ function algebra on $(X,\tau,\bar{z})$, i.e. a real function algebra. If $\Re(A)$ is a ring then $A=C(X,\tau,\bar{z})$.
\end{theorem}
It is interesting to know whether Theorem \ref{thr:NGWKS} can be generalised to all $^{L}/_{L^{g}}$ function algebras on $(X,\tau,g)$. Of course the result would be trivial if, in the non-Archimedean setting, the basic $^{L}/_{L^{g}}$ function algebra on $(X,\tau,g)$ is the only $^{L}/_{L^{g}}$ function algebra on $(X,\tau,g)$. With analogy to $\Re(A)$ above, in this case we should ask whether the set of $L^{g}$ components of the functions in $C(X,\tau,g)$ form a ring.
\item[(Q2)]
As alluded to in (Q1) we have not given an example in the non-Archimedean setting of a $^{L}/_{L^{g}}$ function algebra on $(X,\tau,g)$ that is not basic. We need to know whether the basic function algebras are the only such examples. Theorem \ref{thr:UAKapl}, Kaplansky's version of the Stone Weierstrass Theorem, may be important here. Further even if in the non-Archimedean setting there is a $^{L}/_{L^{g}}$ function algebra on $(X,\tau,g)$ that is not basic, such an algebra might still be isometrically isomorphic to some Basic function algebras.
\item[(Q3)]
With reference to Theorem \ref{thr:RTREPLF} we note that there are plenty of examples of commutative, unital Banach $F$-algebras with finite basic dimension in the non-Archimedean setting. Indeed if $K$ is not only a finite Galois extension of $F$ but also a cyclic extension then taking $A:=K$ gives such an algebra. In this case the character space $\mathcal{M}(A)$ will be finite with each element given by an element of $\mbox{Gal}(^{K}/_{F})$, see the proof of Theorem \ref{thr:RTREPLF} for details. However such examples are not particularly interesting and it would be good to know whether all $^{L}/_{L^{g}}$ function algebras on $(X,\tau,g)$ have finite basic dimension so that Theorem \ref{thr:RTREPLF} becomes closer to a characterisation result. We recall that all commutative unital complex Banach algebras and commutative unital real Banach algebras have finite basic dimension, see Remark \ref{rem:RTFBD}.
\item[(Q4)]
With reference to Definition \ref{def:CGBFA} of the basic $^{L}/_{L^{g}}$ function algebra on $(X,\tau,g)$ the map $\sigma(f)=g^{(\mbox{ord}(g)-1)}\circ f\circ\tau$ on $C_{L}(X)$ is such that each $f\in C_{L}(X)$ is an element of $C(X,\tau,g)$ if and only if $\sigma(f)=f$. We have seen that $\sigma$ is either an algebraic involution on $C_{L}(X)$ or a algebraic element of finite order on $C_{L}(X)$. It should be established whether every such involution and element of finite order on $C_{L}(X)$ has the form of $\sigma$ for some $g$ and $\tau$. This is the case for real function algebras, see \cite[p29]{Kulkarni-Limaye1992}.
\item[(Q5)]
As described in Remark \ref{rem:NGPHI} it might be possible to generalise the definition of Basic function algebras by involving a group homomorphism in the definition. The algebras currently given by Definition \ref{def:CGBFA} could then appropriately be referred to as cyclic basic function algebras given that the group $\mbox{Gal}(^{L}/_{L^{g}})$ is cyclic. Further the possibility of generalising the definition of Basic function algebras to the case where the functions take values in some infinite extension of the underlying field over which the algebra is a vector space should also be considered. The involvement of a group homomorphism might also be useful in this case as well as some more of the theory from \cite{Berkovich}.
\item[(Q6)]
As seen in Subsection \ref{subsec:NGNCNA} the general theory of quaternion algebras provides the necessary structures for generalising the theory of non-commutative real function algebras to the non-Archimedean setting. Further with reference to Subsection \ref{subsec:CGBE} it would be interesting to see what sort of lattice of basic extensions the non-commutative real function algebras have. We can also look at this in the non-Archimedean setting along with the residue algebra.
\item[(Q7)]
A proof of the following theorem can be found in \cite[p18]{Kulkarni-Limaye1992}.
\begin{theorem}
\label{thr:NGCOM}
Let $A$ be a unital Banach algebra in the Archimedean setting satisfying one of the following conditions:
\begin{enumerate}
\item[(i)]
the algebra $A$ is a complex algebra and there exists some positive constant $c$ such that $\|a\|_{A}^{2}\leq c\|a^{2}\|_{A}$ for all $a\in A$;
\item[(ii)]
the algebra $A$ is a real algebra and there exists some positive constant $c$ such that $\|a\|_{A}^{2}\leq c\|a^{2}+b^{2}\|_{A}$ for all $a,b\in A$ with $ab=ba$.
\end{enumerate}
Then $A$ is commutative.
\end{theorem}
It would be interesting to establish whether there is such a theorem for all unital Banach $F$-algebras. If not then perhaps some special cases are possible in the non-Archimedean setting. The proof of theorem \ref{thr:NGCOM} uses Liouville's theorem and some spectral theory in the Archimedean setting. Both of these are different in the non-Archimedean setting, see Theorem \ref{thr:FAAULT} and Subsection \ref{subsec:FAASE}.
\item[(Q8)]
It might be interesting to investigate the isomorphism classes of basic function algebras. That is for a given basic function algebra $A$ are there other basic function algebras that are isometrically isomorphic to $A$.
\item[(Q9)]
It is interesting to consider whether the Kaplansky spectrum of Remark \ref{rem:FAAFC} can be used for some cases in the non-Archimedean setting and, if so, whether it is one such definition in some larger family of definitions of spectrum applicable in the non-Archimedean setting.
\item[(Q10)]
More broadly the established theory of Banach algebras provides a large supply of topics that can be considered for generalisation over complete valued fields. In addition to several of the other references included in this thesis \cite{Dales} will be of much interest when considering such possibilities. One obvious example is the generalisation of automatic continuity results. That is what conditions on a Banach $F$-algebra force homomorphisms from, or to, that algebra to be continuous. There is one such result in this thesis noting that in Theorem \ref{thr:RTREPLF} the elements of $\mathcal{M}(A)$ are automatically continuous. Further \cite{Bachman} may also be of interest concerning function algebras.
\item[(Q11)]
As mentioned in Remark \ref{rem:UASCS} there is a possible generalisation of the Swiss cheese classicalisation theorem to the Riemann sphere and possibly to a more general class of metric spaces.
\item[(Q12)]
It might be interesting to consider generalising over all complete valued fields the theory of algebraic extensions of commutative unital normed algebras. See the survey paper \cite{Dawson} for details.
\item[(Q13)]
The possibility of generalising $C^{*}$-Algebras over complete valued fields is interesting but perhaps not straightforward. The Levi-Civita field might be of interest here since it is totally ordered such that the order topology agrees with the valuation topology. Hence it might be possible to define positive elements in this case. Perhaps the algebraic elements of finite order mentioned in (Q4) are relevant. Also there is a monograph by Goodearl from 1982 that considers real $C^{*}$-Algebras that might be of use. The possibility of a non-Archimedean theory of Von Neumann algebras might also be a good place to start.
\end{enumerate}

	\phantomsection
	\addcontentsline{toc}{chapter}{References}
	\nocite{*}
	\bibliography{Thesisbibdata}
\end{document}